\documentclass[11 pt,a4paper,leqno,english]{smfbook}

\usepackage{latexsym,amsmath,amscd,amsfonts,amssymb,mathrsfs}

\usepackage{smfthm}

\numberwithin{equation}{chapter}

\usepackage[francais,english]{babel}
\input cyracc.def
\font\tencyr=wncyr10
\def\cyr{\tencyr\cyracc}

\usepackage[all]{xy}

\makeindex

\author{Jean-Pierre Ramis}
\address{Institut de France (Acad{\'e}mie des Sciences)
and Laboratoire {\'E}mile Picard, 
Institut de Math{\'e}matiques, CNRS UMR 5219, 
U.F.R. M.I.G., Universit{\'e} Paul Sabatier (Toulouse 3),
31062 Toulouse CEDEX 9}
\email{ramis.jean-pierre@wanadoo.fr}

\author{Jacques Sauloy}
\address{Laboratoire {\'E}mile Picard, 
Institut de Math{\'e}matiques, CNRS UMR 5219, 
U.F.R. M.I.G., Universit{\'e} Paul Sabatier (Toulouse 3),
31062 Toulouse CEDEX 9}
\email{sauloy@math.univ-toulouse.fr}
\urladdr{http://www.math.univ-toulouse.fr/~sauloy/}

\author{Changgui Zhang}
\address{Laboratoire Paul Painlev{\'e}, UMR-CNRS 8524,
UFR Math{\'e}matiques Pures et Appliqu{\'e}es, 
Universit{\'e} des Sciences et Technologies de Lille (Lille 1),
Cit{\'e} Scientique, 59655 Villeneuve d'Ascq Cedex}
\email{Changgui.Zhang@math.univ-lille1.fr}
\urladdr{http://math.univ-lille1.fr/~zhang/}

\title{Local analytic classification of $q$-difference equations}
\alttitle{Classification analytique locale des \'equations aux $q$-diff\'erences}

\def\C{{\mathbf C}}           
\def\Q{{\mathbf Q}}           
\def\Z{{\mathbf Z}}           
\def\R{{\mathbf R}}           
\def\N{{\mathbf N}}           

\def\Ker{\text{Ker~}}         
\def\Id{\text{Id}}            
\def\rk{\text{rk~}}           
\def\Mat{{\text{Mat}}}        
\def\GL{{\text{GL}}}          

\def\sq{\sigma_q}             
\def\Eq{{\mathbf{E}_{q}}}     
\def\F{{\mathcal{F}}}         

\def\B{{\mathcal{B}}}         
\def\Bq{{\mathcal{B}_{q,1}}}  

\def\O{{\mathcal{O}}}         
\def\M{{\mathcal{M}}}         

\def\D{{\mathcal{D}}}         
\def\Dq{{\D_{q}}}             
\def\Dqs{{\D_{q,s}}}          

\def\gr{{\text{gr}}}          
\def\G{{\mathfrak{G}}}        
\def\g{{\mathfrak{g}}}        

\def\Ra{{\C\{z\}}}            
\def\Ka{{\C(\{z\})}}          
\def\Rf{{\C[[z]]}}            
\def\Rfs{{\Rf_{q;s}}}         
\def\Kf{{\C((z))}}            
\def\Kfs{{\Kf_{q;s}}}         

\def\DMod{{Mod_{\D}}}          
\def\DKMod{{Mod_{\D_{K,\sigma}}}}  
\def\DMa{{DiffMod(\Ka,\sq)}}  
\def\DMf{{DiffMod(\Kf,\sq)}}  
\def\DMfs{{DiffMod(\Kfs,\sq)}} 
\def\DM{{DiffMod(K,\sigma)}}   

\def\lmod{\left |}            
\def\rmod{\right |}           
\def\lnorm{\left |\!|}        
\def\rnorm{\right |\!|}       
\def\tq{~ | ~}              
\def\ie{\emph{i.e.}}
\def\eg{\emph{e.g.}}
\def\cf{\emph{cf.}}

\def\Lin{\mathcal{L}}         
\def\Hom{\text{Hom}}          
\def\Homint{\underline{\Hom}} 
\def\Aut{\text{Aut}}          
\def\End{\text{End}}          
\def\Ext{\text{Ext}}          
\def\1{\underline{1}}         

\def\EE{{\mathbb E}} 
\def\AA{{\mathbb A}} 
\def\BB{{\mathbb B}}
\def\OO{{\mathbb O}}
\def\CC{{{\mathcal C}}}
\def\RR{{{\mathcal R}}}
\def\CA{{{\mathcal C}_{A_0}}}
\def\aa{{{\mathcal A}}}
\def\ff{{{\mathcal{F}}}}

\def\U{{\mathbf{U}}}          
\def\ii{{\text{i}}}           
\def\UU{{\mathfrak{U}}}       
\def\VV{{\mathfrak{V}}}       
\def\Cc{{\tilde{C}}}          
\def\Ee{{\mathscr{E}}}        
\def\Aa{{\mathscr{A}}}        
\def\Bb{{\mathscr{B}}}        

\def\thq{{\theta_{q}}}        
\def\thqc{{\theta_{q,c}}}      
\def\tsh{\text{\cyr Ch}}      
\def\div{{\text{div}}}        

\begin{document}
\frontmatter

\begin{abstract}
We essentially achieve Birkhoff's program for $q$-difference 
equations by giving three different descriptions of the moduli 
space of isoformal analytic classes. This involves an extension 
of Birkhoff-Guenther normal forms, $q$-analogues of the so-called 
Birkhoff-Malgrange-Sibuya theorems and a new theory of summation. 
The results were announced in \cite{RSZ1,RSZ2} and in various
seminars and conferences between 2004 and 2006.
\end{abstract}

\begin{altabstract}
Nous achevons pour l'essentiel le programme de Birkhoff pour la
classification des \'equations aux $q$-diff\'erences en donnant
trois descriptions distinctes de l'espace des modules des classes
analytiques isoformelles. Cela passe par une extension des formes
normales de Birkhoff-Guenther, des $q$-analogues des th\'eor\`emes
dits de Birkhoff-Malgrange-Sibuya et une nouvelle th\'eorie de la
sommation. Ces r\'esultats ont \'et\'e annonc\'es dans \cite{RSZ1,RSZ2}
ainsi que dans divers s\'eminaires et conf\'erences de 2004 \`a 2006.
\end{altabstract}



\altkeywords{\'Equations aux $q$-diff\'erences, classification analytique,
ph\'enom\`ene de Stokes, sommation des s\'eries divergentes}

\maketitle

\tableofcontents

\mainmatter



\chapter{Introduction}


\section{The problem}


\subsection{The generalized Riemann problem and the allied problems}
\label{subsection:GRP}

This paper is a contribution to a large program stated and begun by
G.D. Birkhoff in the beginning of XXth century \cite{Birkhoff1}:
{\it the generalized Riemann problem\index{Riemann problem (generalized)} 
for linear differential equations
and the allied problems for linear difference and $q$-difference equations}.
Such problems are now called {\it Riemann-Hilbert-Birkhoff problems}.
Today the state of achievement of the program  {\it as it was formulated
by Birkhoff} in \cite{Birkhoff1} is the following:

\begin{itemize}
\item[--]
For linear {\it differential} equations the problem is completely
closed (both in the regular-singular case and in the irregular case).
\item[--]
For linear {\it $q$-difference} equations ($\lmod q\rmod\neq 1$),
taking account of preceding results due to Birkhoff \cite{Birkhoff1},
the second author \cite{JSAIF} and van der Put-Reversat \cite{vdPR}
in the regular-singular case, and to Birkhoff-Guenther \cite{Birkhoff3}
in the {\it irregular case}, the present work {\it essentially closes
the problem} and moreover answers related questions formulated later
by Birkhoff in a joint work with his PhD student P.E. Guenther
\cite{Birkhoff3} (\cf\ below \ref{subsection:BG}).
\item[--]
For linear {\it difference} equations the problem is closed for
global regular-singular equations (Birkhoff, J. Roques \cite{Roques})
and there are some important results in the irregular case \cite{Imm},
\cite{BFI}.
\end{itemize}


\subsection{The Birkhoff-Guenther program}
\label{subsection:BG}

We quote the conclusion of \cite{Birkhoff3}, it contains a program
which is one of our central motivations for the present work.
We shall call it {\it Birkhoff-Guenther program}.
\index{Birkhoff-Guenther program} \\

{\it  Up to the present time, the theory of $q$-difference equations
has lagged noticeabely behind the sister theories of linear difference
and $q$-difference equations. In the opinion of the authors, the use
of the canonical system, as formulated above in a special case, is
destined to carry the theory of $q$-difference equations to a comparable
degree of completeness. This program includes in particular the complete
theory of convergence and divergence of formal series, the explicit
determination of the essential invariants (constants in the canonical
form), the inverse Riemann theory both for the neighborhood of $x=\infty$
and in the complete plane (case of rational coefficients), explicit
integral representations of the solutions, and finally the definition
of $q$-sigma periodic matrices, so far defined essentially only in the
case $n=1$. Because of its extensiveness this material cannot be
presented here.} \\

The paper \cite{Birkhoff3} appeared in 1941 and Birkhoff died in 1944;
as far as we know ``this material" never appeared and the corresponding
questions remained opened and forgotten for a long time.


\subsection{What this paper could contain but does not}
\label{subsection:WTPC}

Before describing the contents of the paper in the following paragraph,
we shall first say briefly  what it could contain but does not.\\

The kernel of the present work is the detailed proofs of some results 
announced in \cite{RSZ1}, \cite{RSZ2} and in various seminars and 
conferences between 2004 and 2006, but there are also some new results 
in the same spirit and some examples.    \\

In this paper, we limit ourselves to the case $\lmod q\rmod\neq 1$.
The problem of classification in the case $\lmod q\rmod=1$ involves
diophantine conditions, it remains open but for the only exception
\cite{LDV2}. Likewise, we do not study problems of confluence of our
invariants towards invariants of differential equations, that is of
$q$-Stokes invariants\index{$q$-Stokes invariants}
towards classical Stokes invariants (\cf\ in
this direction \cite{LuZ}, \cite{zh4}). \\

In this work, following Birkhoff, we classify analytically equations 
admitting a fixed normal form, a moduli problem. There is another way 
to classify equations: in terms of representations of a \emph{``fundamental 
group"}, in Riemann's spirit. This is related to the Galois theory of 
$q$-difference equations, we will not develop this topic here, limiting 
ourselves to the following remarks, even if the two types of classification 
are strongly related. \\

The initial work of Riemann was a ``description" of hypergeometric
differential equations as two-dimensional \emph{representations} of
the free non-abelian \emph{group} generated by two elements. Later
Hilbert asked for a classification of meromorphic linear differential
equations in terms of {\it finite dimensional representations of
free groups}. Apparently the idea of Birkhoff was to get classifications
of meromorphic linear differential, difference and $q$-difference
equations in terms of elementary linear algebra and combinatorics
but {\it not} in terms of group representations. For many reasons
it is interesting to work in the line of Riemann and Hilbert and to
translate Birkhoff style invariants in terms of group representations.
The corresponding groups will be {\it fundamental groups} of the Riemann
sphere minus a finite set or {\it generalized fundamental groups}.
It is possible to define categories of linear differential, difference
and $q$-difference equations, these categories are tannakian categories;
then applying the fundamental theorem of Tannakian categories we can
interpret them in terms of finite dimensional representations of
pro-algebraic groups, the Tannakian groups. The fundamental groups
and the (hypothetic) generalized fundamental groups will be Zariski
dense subgroups of the Tannakian groups. \\

In the case of differential equations the work was achieved by
the first author, the corresponding group is the {\it wild fundamental
group}\index{wild fundamental group} 
which is Zariski dense in the Tannakian group.  In the case
of difference equations almost nothing is known. In the case of
$q$-difference equations, the situation is the following.

\begin{enumerate}
\item
In the \emph{local regular singular case} the work was achieved by
the second author in \cite{JSGAL}, the generalized fundamental group
is abelian, its semi-simple part is abelian free on two generators and
its unipotent part is isomorphic to the additive group $\C$.
\item
In the \emph{global regular singular case} only the abelian case
is understood  \cite{JSGAL}, using the geometric class field theory.
In the general case some non abelian class field theory is needed.
\item
In the \emph{local irregular case}, using the results of the present
paper the first and second authors recently got a generalized
fundamental group \cite{RS1,RS2,RS3}.
\item
Using case 3, the solution of the \emph{global general case} should
follows easily from the solution of case 2.
\end{enumerate}

\subsubsection{About the asumption that the slopes are integral}
\label{subsubsection:integralslopes}

The ``abstract'' part of our paper does not require any assumption
on the slopes\index{slopes (integrality assumption)}
: that means general structure theorems, \eg\
theorem \ref{theo:schemadesmodules} and proposition
\ref{prop:schemadesmodulesqGevrey}. The same is true of the
decomposition of the local Galois group of an irregular equation
as a semi-direct product of the formal group by a unipotent group
in \cite{RS1,RS2}. \\

However, all our \emph{explicit} constructions (Birkhoff-Guenther
normal form, privileged cocycles, discrete summation) rest on the
knowledge of a normal form for pure modules, that we have found
only in the case of integral slopes. \\

In \cite{vdPR} van der Put and Reversat classified pure modules with 
non integral slopes. The extension of our results to the case of non
integral slopes, using \cite{vdPR}, does not seem to have been done - 
and would be very useful\footnote{Long after the present work had been 
submitted for publication, this task was undertaken by Virginie Bugeaud,
see \cite{VB}.}.


\section{Contents of the paper}
\label{section:contents}

We shall classify analytically isoformal $q$-difference equations.
That is, we consider analytic $q$-difference modules together with
an isomorphism of their formalisation with a given formal $q$-difference
module, and we consider equivalences preserving this additional structure.
This is in adequation to the corresponding problem for differential
equations, see for instance \cite{BV}. \\

The isoformal analytic classes form an affine space, we shall give 
\emph{three descriptions} of this space respectively in chapters  
\ref{chapter:isoformalclasses}, \ref{chapter:qMalgrangeSibuya} 
and \ref{chapter:geometryofF(M)} (this third case is a particular 
case of the second and is based upon chapter 
\ref{chapter:asymptotictheory}) using different constructions of 
the analytic invariants. A \emph{direct and explicit} comparison 
between the second and the third description is straightforward 
but such a comparison between the first and the second or the
third construction is quite subtle; in the present paper it will 
be clear for the ``abelian case", of two integral slopes; for the 
general case we refer the reader to \cite{RS1,RS2} and also to 
work in progress \cite{JSmoduli}. \\

Chapter \ref{chapter:generalnonsense} deals with the general
setting of the problem (section \ref{section:catqdiffmods})
and introduces two fundamental tools: the Newton polygon 
and the slope filtration (section \ref{section:Newtonetfiltration}). 
From this, the problem of analytic isoformal classification, that
we state in section \ref{section:analyticisoformalclassification},
admits a purely algebraic formulation: the isograded classification
of filtered difference modules. Simple examples are tackled in the
same section to give an idea of the landscape. This is also the
occasion to introduce the $q$-Borel transformation. The corresponding 
algebraic theory is developped in a more general setting in the 
appendix \ref{chapter:ISOGRAD}. \\

The first attack at the classification problem for $q$-difference
equations comes in chapter \ref{chapter:isoformalclasses}. Section
\ref{section:isoformalclasses} specializes the results of appendix
\ref{chapter:ISOGRAD} to $q$-differences, and section 
\ref{section:indextheorems} provides the proof of some related
index computations. One finds that the space of classes is an
affine scheme (theorem \ref{theo:schemadesmodules}) and computes
its dimension. This is a rather abstract result. In order to
provide explicit descriptions (normal forms, coordinates, 
invariants ...) from section \ref{section:explicitdescription}
up to the end of the paper, we assume that \emph{the slopes of 
the Newton polygon are integers}. This allows for the more precise 
theorem \ref{theo:kintegralslopesspace} and the existence of 
Birkhoff-Guenther normal forms originating in \cite{Birkhoff3}. 
This also makes easier the explicit computations of the following 
chapters. Then some precisions about $q$-Gevrey classification are 
given in section \ref{section:qGevreyinterpolation}. \\

Analytic isoformal classification by normal forms is a special 
feature of the $q$-difference case, such a thing does not exist 
for differential equations. To tackle this case Birkhoff introduced 
functional $1$-cochains using Poincar{\'e} asymptotics 
\cite{Birkhoff0,Birkhoff1}. Later, in the seventies, Malgrange 
interpreted Birkhoff cochains using sheaves on a circle $S^1$ 
(the real blow up of the origin of the complex plane). 
Here, we modify these constructions in order to deal with 
the $q$-difference case, introducing a new asymptotic theory 
and replacing the circle $S^1$ by the elliptic curve 
$\Eq = \C^{*}/q^{\Z}$. \\

In this spirit, chapter \ref{chapter:qMalgrangeSibuya} tackles 
the extension to $q$-difference equations of the so called 
Birkhoff-Malgrange-Sibuya theorems. In section 
\ref{section:asymptotics} is outlined an asymptotic theory 
adapted to $q$-difference equations but weaker than that of 
section \ref{section:asymptoticsreldivisor} of the next chapter:
the difference is the same as between classical Gevrey versus
Poincar{\'e} asymptotics. The counterpart of the Poincar{\'e}
version of Borel-Ritt is theorem \ref{theo:BorelRittfaible};
also, comparison with Whitney conditions is described in lemma
\ref{prop:asymptotiqueetWhitney}. Indeed, the geometric methods 
of section \ref{section:fundamentalisomorphism} rest on 
the integrability theorem of Newlander-Niremberg. They allow
the proof of the first main theorem, the $q$-analogue of the
abstract Birkhoff-Malgrange-Sibuya theorem (theorem 
\ref{theo:firstqMalgrangeSibuya}). Then, in section
\ref{section:theisoformalclassificationJP}, it is applied to 
the classification problem for $q$-difference equations and 
one obtains the $q$-analogue of the concrete Birkhoff-Malgrange-Sibuya 
theorem (theorem \ref{theo:secondqMalgrangeSibuya}): there is
a natural bijection from the space $\F(M_{0})$ of analytic isoformal 
classes in the formal class of $M_{0}$ with the first cohomology set
$H^{1}(\Eq,\Lambda_{I}(M_{0}))$ of the ``Stokes sheaf''. The latter is 
the sheaf of automorphisms of $M_{0}$ infinitely tangent to identity, 
a sheaf of unipotent groups over the elliptic curve $\Eq = \C^{*}/q^{\Z}$.
The proof of theorem \ref{theo:secondqMalgrangeSibuya} appeals to the 
fundamental theorem of existence of asymptotic solutions, previously 
proved in section \ref{section:existasymptsols}. \\

Chapter \ref{chapter:asymptotictheory} aims at developping a summation 
process for $q$-Gevrey divergent series. After some preparatory material 
in section \ref{section:preparatorymaterial}, an asymptotic theory 
``with estimates" well suited for $q$-difference equations is expounded 
in section \ref{section:asymptoticsreldivisor}. Here, the sectors of 
the classical theory are replaced by preimages in $\C^{*}$ of the Zariski 
open sets of the elliptic curve $\C^{*}/q^{\Z}$, that is, complements of 
finite unions of discrete $q$-spirals in $\C^{*}$; and the growth 
conditions at the boundary of the sectors are replaced by polarity 
conditions along the discrete spirals. The $q$-Gevrey analogue of 
the classical theorems are stated and proved in sections 
\ref{section:BorelRitt} and 
\ref{section:relativedivisorsandmultisummable}: the counterpart
of the Gevrey version of Borel-Ritt is theorem \ref{theo:BorelRitt}
and multisummability conditions appear in theorems \ref{theo:mniveaux}
and \ref{theo:masymptotique}. The theory is then applied
to $q$-difference equations and to their classification in 
section \ref{section:analyticclassificationCZ}, where is proved
the summability of solutions (theorem \ref{theo:sommabilite}),
the existence of asymptotic solutions coming as a consequence
(theorem \ref{theo:solutionsasymptotiques}) and the description
of Stokes phenomenon as an application (theorem \ref{theo:Stokes}). \\

Chapter \ref{chapter:geometryofF(M)} deals with some complementary
information on the geometry of the space $\F(M_{0})$ of analytic
isoformal classes, through its identification with the cohomology 
set $H^{1}(\Eq,\Lambda_{I}(M_{0}))$ obtained in chapter 
\ref{chapter:qMalgrangeSibuya}. Theorem \ref{theo:secondqMalgrangeSibuya}
of chapter \ref{chapter:qMalgrangeSibuya} implicitly attaches cocycles 
to analytic isoformal classes and theorem \ref{theo:Stokes} of chapter 
\ref{chapter:asymptotictheory} shows how to obtain them by a summation
process. In section \ref{section:privilegedcocycles}, we give yet another
construction of ``privileged cocycles'' (from \cite{JSStokes}) and study
their properties. In section \ref{section:devissageqGevrey}, we show
how the devissage of the sheaf $\Lambda_{I}(M_{0})$ by holomorphic
vector bundles over $\Eq$ allows to identify $H^{1}(\Eq,\Lambda_{I}(M_{0}))$ 
with an affine space, and relate it to the corresponding result of theorem 
\ref{theo:kintegralslopesspace}. In section \ref{section:vectorbundles},
we recall how holomorphic vector bundles over $\Eq$ appear naturally
in the theory of $q$-difference equations and we apply them to an
interpretation of the formula for the dimension of $\F(M_{0})$. \\

Chapter \ref{chapter:examplesStokes} provides some elementary examples
motivated by their relation to $q$-special functions, either linked 
to modular functions or to confluent basic hypergeometric series. \\

The appendix is devoted to the study of the following purely algebraic
problem: to classify (in a given abelian category) all finitely filtered
objects with fixed associated graded object, up to equivalence compatible
with the graduation. Since the space under study looks like some generalized 
extension module, it relies on homological algebra and index computations.
The problem is solved here under restrictive assumptions, but which are
sufficient to apply them to our specific situation. Note that, to ease
the reading of chapter \ref{chapter:generalnonsense}, there are some
redundancies, so that the reader does not have to swallow the whole 
appendix before reading that chapter !

\subsection{About the wealth of summation processes}

It is important to notice that the construction of $q$-analogs 
of classical objects\index{$q$-analogs (are not canonical)} 
(special functions\ldots) is not canonical, 
there are in general several ``good" $q$-analogs. So there are 
several $q$-analogs of the Borel-Laplace 
summation\index{summation processes (wealth of)} 
(and 
multisummation)\footnote{Note however that the ``algebraic summation'' 
introduced in \cite{JSStokes} and recalled and used here in chapter 
\ref{chapter:geometryofF(M)} is actually just a special case of 
the general summation process described in chapter
\ref{chapter:asymptotictheory}: when one restricts to solutions of
$q$-difference equations, the algorithm of summation admits a more
elementary expression.}:
there are several choices for Borel and Laplace 
kernels (depending on a choice of $q$-analog of the exponential 
function) and several choices of the integration contours (continuous 
or discrete in Jackson style) \cite{RamisPanorama}, \cite{rz}, 
\cite{zh1,zh2,zh3}. The paper \cite{LuZ} studies a case where
distinct processes give the same result, while \cite{CZMORD} dwells
on the difference of such results to obtain modularity properties. \\

Our choice of summation here seems quite "optimal": the entries of our 
Stokes matrices are \emph{elliptic functions} (\cf\ the ``$q$-sigma 
periodic matrices" of Birkhoff-Guenther program), moreover Stokes 
matrices are meromorphic in the parameter of ``$q$-direction of 
summation"; this is essential for applications to $q$-difference 
Galois theory (\cf\ \cite{RS1,RS2,RS3}). Unfortunately we did not 
obtain explicit integral formulae for this summation (except for 
some particular cases), in contrast with what happens for other 
summations introduced before by the third author.

\section{General notations}
\label{section:generalnotations}

Generally speaking, in the text, the sentence $A := B$ means that 
the term $A$ is defined by formula $B$. Numbered equations are 
numbered according to the chapter they appear in: thus, equation 
(2.3) is the third numbered equation in chapter 2. \\

But for some changes, the notations are the same as in \cite{RSZ1}, 
\cite{RSZ2}, etc. Here are the most useful ones\footnote{The reader
will also find on page \pageref{indexnotations} an index of notations
followed by a terminological index.}. \\

\label{not1}
We write $\Ra$ the ring of convergent power series (holomorphic
germs at $0 \in \C$) and $\Ka$ its field of fractions (meromorphic
germs). Likewise, we write $\Rf$ the ring of formal power series 
and $\Kf$ its field of fractions. \\

We fix once for all a complex number $q \in \C$\label{not2}
such that 
$\lmod q \rmod > 1$. Then, the dilatation operator $\sq$ 
is an automorphism of any of the above rings and fields,
well defined by the formula:
$$
\label{not3}
\sq f(z) := f(q z).
$$
Other rings and fields of functions on which $\sq$ operates
will be introduced in the course of the paper. This operator 
also acts coefficientwise on vectors, matrices\ldots over any 
of these rings and fields. \\

\label{not4}
We write $\Eq$ the complex torus (or elliptic curve) 
$\Eq := \C^{*}/q^{\Z}$ and $p: \C^{*} \rightarrow \Eq$
the natural projection. For all $\lambda \in \C^*$, we write 
$[\lambda;q] := \lambda q^{\Z} \subset \C^{*}$ the \emph{discrete 
logarithmic $q$-spiral}\index{$q$-spiral (discrete logarithmic)}
through the point $\lambda \in \C^{*}$. 
All the points of $[\lambda;q]$ have the same image 
$\overline{\lambda} := p(\lambda) \in \Eq$ and we may identify 
$[\lambda;q] = p^{-1}\left(\overline{\lambda}\right)$ with 
$\overline{\lambda}$. \\

A linear analytic (resp. formal) 
$q$-difference equation\index{$q$-difference equation} (implicitly:
at $0 \in \C$) is an equation:
\begin{equation}
\label{not5}
\sq X = A X, 
\end{equation}
where $A \in \GL_{n}(\Ka)$ (resp. $A \in \GL_{n}(\Kf)$).


\subsection{Theta Functions}
\label{subsection:thetafunctions}

Jacobi theta functions
\index{Jacobi theta function} 
\index{theta function}
pervade the theory of $q$-difference equations.
We shall mostly have use for two slightly different forms of them. \\

In chapter \ref{chapter:asymptotictheory}, we shall use:
\label{not6}
\begin{equation}
\label{eqn:tripleproduit}
\theta(z;q) := \sum_{n\in\Z} q^{-n(n-1)/2} z^n =
\prod_{n\in\N}(1 - q^{-n-1})(1 + q^{-n} z)(1 + q^{-n-1}/z).
\end{equation} 
The second equality is Jacobi's celebrated triple product 
formula\index{Jacobi's triple product formula}\index{triple product formula}.
When obvious, the dependency in $q$ will be omitted and we shall write 
$\theta(z)$ instead of $\theta(z;q)$. One has: 
$$
\theta(qz) = q~z~\theta(z) \text{~and~} \theta(z) = \theta(1/q z).
$$

In chapter \ref{chapter:geometryofF(M)}, we shall rather use:
\label{not7}
\begin{equation}
\label{eqn:thq}
\thq(z) := \sum_{n \in \Z} q^{-n(n+1)/2} z^{n} 
        = \prod_{n\in\N}(1 - q^{-n-1})(1 + q^{-n-1} z)(1 + q^{-n}/z).
\end{equation}
One has of course $\thq(z) = \theta(q^{-1}z;q)$ and:
$$
\thq(q z) = z \thq(z) = \thq(1/z).
$$
Both functions are analytic over the whole of $\C^*$ and vanish on 
the discrete $q$-spiral $- q^{\Z}$ with simple zeroes. \\

In chapter \ref{chapter:examplesStokes}, we shall use in alternance
both forms, according to which fits better the needs of the computation.
Indeed, in section \ref{section:QuadForms}, we shall even use the
``classical'' forms $\theta_{i,j}$, $i,j \in \{0,1\}$ to be found in
the standard theory of special functions.


\subsection{$q$-Gevrey levels}
\label{subsection:qGevreylevels}

As in \cite{Bezivin},\cite{RamisGrowth}, we introduce the space 
of formal series of 
\emph{$q$-Gevrey order $s$}\index{$q$-Gevrey order (series)}:
$$
\Rfs := \left\{\sum a_{n} z^{n} \in \Rf \tq 
\exists A > 0 ~:~ a_{n} = O\left(A^{n} q^{s n^{2}/2}\right)\right\}.
$$
We also say that $f \in \Rfs$\label{not8}
is of 
\emph{$q$-Gevrey level $1/s$}\index{$q$-Gevrey level (series)}.
It understood that $\Rf_{q;0} = \Ra$ and $\Rf_{q;\infty} = \Rf$. Note
however that it is \emph{not} true that 
$\bigcap\limits_{s > 0} \Rf_{q;s} = \Ra$, as shows the example of the
classical Euler series $\sum n! z^n$, which belongs to the former
but not to the later (and the same is true for any divergent series
having a finite Gevrey level in the sense of the classical theory
of ordinary differential equations); nor that 
$\bigcup\limits_{s > 0} \Rf_{q;s} = \Rf$, as shows the example of 
$\sum n^n z^n$, which belongs to the latter but not to the former. \\

Similar considerations apply to spaces of Laurent formal series:
$$
\Kfs := \Rfs[1/z].
$$
More generally, one can speak of $q$-Gevrey sequences of complex
numbers. Let $k \in \R^* \cup \{\infty\}$ and $s := \dfrac{1}{k}$. 
A sequence $(a_n) \in \C^{\N}$ is 
\emph{$q$-Gevrey of order $s$}\index{$q$-Gevrey order (sequence)} if 
it is dominated by a sequence of the form 
$\left(C A^n \lmod q \rmod^{n^2/(2k)}\right)$, for some constants 
$C,A > 0$. \\
Note that this terminology is all about sequences, or coefficients of 
series. Extension to $q$-Gevrey 
asymptotics\index{$q$-Gevrey asymptotics}\index{asymptotics ($q$-Gevrey)}
is explained in definition
\ref{defi:ALambda}, while the $q$-Gevrey 
interpolation\index{$q$-Gevrey interpolation}
by growth of decay 
of functions is dealt with in definitions \ref{defi:LambdaF} and
\ref{defi:diviseurrelatif}.

\section*{Acknowledgements}

During the final redaction of this work, the first author beneffited 
of a support of the ANR \emph{Galois} (Programme Blanc) and was invited 
by the departement of mathematics of Tokyo University (College of Arts 
and Science, Komaba). During the same period, the second author 
benefitted of a semester CRCT by the Universit\'e Toulouse 3 and
was successively invited by University of Valladolid (Departamento
de Algebra, Geometria and Topologia), Universit\'e des Sciences
et Technologies de Lille (Laboratoire Paul Painlev\'e) and
Institut Math\'ematique de Jussieu (\'Equipe de Topologie et
G\'eom\'etrie Alg\'ebriques); he wants to thank warmly these
three departments for their nice working atmosphere. \\
We are happy to thank here our friends and colleagues Lucia Di Vizio
and Anne Duval who showed active interest in our work, offering many
valuable suggestions. \\
Last, we heartily thank the referee of prevous submissions of this paper, 
who encouraged us to add some motivating examples (these make up the contents 
of chapter \ref{chapter:examplesStokes}) and to extend it to the present 
form; and we thank even more heartily the referee of the \emph{present} 
submission to Asterisque, who spent a great lot of time and effort to help 
us make it readable.


 
\chapter{Some general nonsense}
\label{chapter:generalnonsense}


\section{The category of $q$-difference modules}
\label{section:catqdiffmods}

General references for this section are \cite{vdPS}, \cite{JSFIL}.


\subsection{Some general facts about difference modules}
\label{subsection:generalitiesdiffmods}

Here, we also refer to the classical litterature about difference fields, 
like \cite{Cohn} and \cite{Franke} (see also \cite{Ore}). A detailed
proof of many elementary algebraic facts can be found in \cite{LDVS}.
We shall also use specific results proved in appendix \ref{chapter:ISOGRAD}. \\

We call \emph{difference field}\index{difference field}\footnote{Much 
of what follows will be extended to the case of difference \emph{rings} 
in appendix \ref{chapter:ISOGRAD}, where the basic linear constructions 
will be described in great detail; the reader is encouraged to refer to 
the appendix only when necessary.} a pair $(K,\sigma)$, where $K$ is 
a (commutative) field and $\sigma$ a field automorphism of $K$. 
We write indifferently $\sigma(x)$ or $\sigma x$ the action 
of $\sigma$ on $x \in K$. One can then form the Ore ring of 
difference operators:
$$
\label{not9}
\D_{K,\sigma} := K \langle T,T^{-1} \rangle
$$
characterized by the twisted commutation relation:
$$
\forall k \in \Z ~,~ x \in K ~,~ T^{k} x = \sigma^{k}(x) T^{k}.
$$
We shall rather write somewhat improperly 
$\D_{K,\sigma} := K \langle \sigma,\sigma^{-1} \rangle$ and, for short,
$\D := \D_{K,\sigma}$ in this section. The center of 
$\D$ is the ``field of constants'':
$$
\label{not10}
K^{\sigma} := \{x \in K \tq \sigma(x) = x\}.
$$
The ring $\D$ is left euclidean and any ideal is generated by a unique
entire unitary polynomial 
$P = \sigma^{n} + a_{1} \sigma^{n-1} + \cdots + a_{n}$. \\

Any (left) $\D$-module $M \in \DMod$ \label{not11}
can be seen as
a $K$-vector space $E$ and the left multiplication $x \mapsto \sigma.x$
as a \emph{semi-linear} (or $\sigma$-linear) automorphism 
\index{semi-linear automorphism}
\index{$\sigma$-linear automorphism}
$\Phi: E \rightarrow E$, which means
that $\Phi(\lambda x) = \sigma(\lambda) \Phi(x)$; and any pair 
$(E,\Phi)$ \label{not12}
of a $K$-vector space $E$ and a semi-linear automorphism $\Phi$ of $E$
defines a $\D$-module; we just write $M = (E,\Phi)$. Morphisms
from $(E,\Phi)$ to $(E',\Phi')$ in $\DMod$ are linear maps
$u \in \Lin_{K}(E,E')$ such that $\Phi' \circ u = u \circ \Phi$. \\

The $\D$-module $M$ has finite length if, and only if, $E$ is a finite 
dimensional $K$-vector space. A finite length $\D$-module is called a 
\emph{difference module} \index{difference module} 
over $(K,\sigma)$, or over $K$ for short. The 
full subcategory of $\DMod$ whose objects are of difference modules is 
written $DiffMod(K,\sigma)$. \label{not13}
The categories $\DMod$ and $DiffMod(K,\sigma)$
are abelian and $K^{\sigma}$-linear. \\

By choosing a basis of $E$, we can identify any difference module
\label{not14}
with some $(K^{n},\Phi_{A})$, where $A \in \GL_{n}(K)$ and 
$\Phi_{A}(X) := A^{-1} \sigma X$ (with the natural operation of 
$\sigma$ on $K^{n})$; the reason for using $A^{-1}$ will become clear soon. 
If $B \in \GL_{p}(K)$, then morphisms from $(K^{n},\Phi_{A})$ to
$(K^{p},\Phi_{B})$ can be identified with matrices $F \in \Mat_{p,n}(K)$ such 
that $(\sigma F) A = B F$ (and composition amounts to the product of matrices).

\subsubsection{Unity}

An important particular object is the \emph{unity} \index{unity} 
\label{not15} 
$\1$, which may be 
described either as $(K,\sigma)$ or as $\D/\D P$ with $P = \sigma - 1$. 
For any difference module $M = (E,\Phi)$ the $K^{\sigma}$-vector 
space $\Hom(\1,M)$ can be identified with the kernel of the $K^{\sigma}$-linear 
map $\Phi - Id: E \rightarrow E$; in case $M =(K^{n},\Phi_{A})$, this boils 
down to the space $\{X \in K^{n} \tq \sigma X = A X\}$ of solutions of a 
``$\sigma$-difference system''; whence our definition of $\Phi_{A}$. The 
\emph{functor of solutions} \index{functor of solutions} 
$M \leadsto \Gamma(M) := \Hom(\1,M)$, from
the category $DiffMod(K,\sigma)$ to the category of finite dimensional
(see remark \ref{rema:fibrefunctors}) vector spaces over $K^{\sigma}$ is 
left exact and $K^{\sigma}$-linear. We shall have use for its right derived 
functors $\Gamma^{i}(M) = \Ext^{i}(\1,M)$ (\emph{see} \cite{BAH}).

\subsubsection{Internal Hom} \index{internal Hom}

Let $M = (E,\Phi)$ and $N = (F,\Psi)$. The map 
$T_{\Phi,\Psi}: f \mapsto \Psi \circ f \circ \Phi^{-1}$ is a semi-linear 
automorphism of the $K$-vector space $\Lin_{K}(E,F)$, whence a difference 
\label{not16} 
module $\Homint(M,N) := \bigl(\Lin_{K}(E,F),T_{\Phi,\Psi}\bigr)$. 
Then one has $\Homint(\1,M) = M$ and
$\Gamma\bigl(\Homint(M,N)\bigr) = \Hom(M,N)$. 
The \emph{dual} \index{dual}
of $M$ is $M^{\vee} := \Homint(M,\1)$, so that 
$\Hom(M,\1) = \Gamma(M^{\vee})$. For instance, the dual
of $M =(K^{n},\Phi_{A})$ is $M^{\vee} =(K^{n},\Phi_{A^{\vee}})$, where
$A^{\vee} := {}^{t} A^{-1}$. 

\subsubsection{Tensor product} \index{tensor product}

Let $M = (E,\Phi)$ and $N = (F,\Psi)$. The map 
$\Phi \otimes \Psi: x \otimes y \mapsto \Phi(x) \otimes \Psi(y)$ from
$E \otimes_{K} F$ to itself is well defined and it is a semi-linear
automorphism, whence a difference module 
\label{not17} 
$M \otimes N = (E \otimes_{K} F,\Phi \otimes \Psi)$. 
The obvious morphism yields the 
adjunction relation \index{adjunction relation}: 
$$
\Hom\bigl(M,\Homint(N,P)\bigr) = \Hom(M \otimes N,P).
$$
We also have functorial isomorphisms $\1 \otimes M = M$ and
$\Homint(M,N) = M^{\vee} \otimes N$. The classical 
computation of the rank \index{rank} through 
$\1 \rightarrow M^{\vee} \otimes M \rightarrow \1$ yields $\dim_{K} E$
as it should. We write $\rk M$ this number.

\subsubsection{Extension of scalars} \index{extension of scalars} 

An extension difference field $(K',\sigma')$ of $(K,\sigma)$ 
consists in an extension $K'$ of $K$ and an automorphism $\sigma'$ 
of $K'$ which restricts to $\sigma$ on $K$. Any difference module 
$M = (E,\Phi)$ over $(K,\sigma)$ then gives rise to a difference 
module $M' = (E',\Phi')$ over $(K',\sigma')$, where 
$E' := K' \otimes_{K} E$ and $\Phi' := \sigma' \otimes \Phi$ is 
defined the same way as above. We then write $\Gamma_{K'}(M)$ 
the ${K'}^{\sigma'}$-vector space $\Gamma(M')$. The \emph{functor of 
solutions (with values) in $K'$} 
\index{functor of solutions with values in $K'$}
is defined as $M \leadsto \Gamma(M')$;
it is left exact and $K^{\sigma}$-linear. The functor $M \leadsto M'$ 
from $DiffMod(K,\sigma)$ to $DiffMod(K',\sigma')$ is compatible 
with unity, internal Hom, tensor product and dual. The image 
of $M =(K^{n},\Phi_{A})$ is $M' =({K'}^{n},\Phi_{A})$.


\subsection{$q$-difference modules} \index{$q$-difference modules}
\label{subsection:qdiffmods}

\label{not18}
We now restrict our attention to the difference fields $\bigl(\Ka,\sq\bigr)$ 
and $\bigl(\Kf,\sq\bigr)$, and to the corresponding categories of analytic, 
resp. formal, $q$-difference modules: they are $\C$-linear abelian categories 
since $\Ka^{\sq} = \Kf^{\sq} = \C$. In both settings, we consider the 
$q$-difference module $M =(K^{n},\Phi_{A})$ as an abstract model for the 
linear $q$-difference system \index{$q$-difference system} 
$\sq X = A X$. Isomorphisms from (the system
with matrix) $A$ to (the system with matrix) $B$ in either category
correspond to analytic, resp. formal, 
gauge transformations, 
\index{gauge transformation}
\ie\ matrices $F \in \GL_{n}(\Ka)$, resp. $F \in \GL_{n}(\Kf)$, such that
$B = F[A] := (\sq F) A F^{-1}$. We write $\Dq$ indifferently
for $\D_{\Ka,\sq}$ or $\D_{\Kf,\sq}$ when the distinction is irrelevant.

\begin{lemm}[Cyclic vector lemma] \index{cyclic vector lemma}
\label{lemm:cyclicvector}
Any (analytic or formal) $q$-difference module is isomorphic to a module
$\Dq/\Dq P$ for some unitary entire $q$-difference operator $P$.
\end{lemm}
\begin{proof}
See \cite{LDV}, \cite{JSAIF}.
\end{proof}

Note however that the ideal $\Dq P$, and thus its generator $P$, are by
no means unique, since they depend on the choice of a cyclic vector. For 
instance, it is an easy exercice to show that the operators $P := \sq - a$ 
and $P' := \sq - a'$ (with $a,a' \in K^{*}$) give rise to isomorphic modules
$\Dq/\Dq P \simeq \Dq/\Dq P'$ if, and only if, $a'/a$ belongs to the subgroup 
$\{\frac{\sigma b}{b} \tq b \in K^{*}\}$ of $K^{*}$ (\cite{LDVS}, exemple 3.9).

\begin{theo}
\label{theo:quasitannakiennes}
The categories $\DMa$ and $\DMf$ are abelian $\C$-linear rigid tensor 
categories. 
\end{theo}
\begin{proof}
See for instance \cite{vdPS}, \cite{JSGAL}. (From proposition 3.1 of 
\cite{RS1}, it actually follows that $\DMa$ is a neutral tannakian 
category, and the same is true for $\DMf$ by \cite{vdPS} or \cite{vdPR},
but we won't use these facts.)
\end{proof}

\subsubsection{The functor of solutions} \index{functor of solutions}
\label{subsubsection:functorofsolutions}

It is customary in $\D$-module theory to call \emph{solution} 
of $M$ a morphism $M \rightarrow \1$. Indeed, a solution of 
$\D/\D P$ is then an element of $\Ker P$. We took the dual 
convention, yielding a \emph{covariant} $\C$-linear left 
exact functor $\Gamma$ for the following reason. To any 
analytic $q$-difference module $M = (\Ka^{n},\Phi_{A})$, 
one can associate a holomorphic vector bundle $F_{A}$ over 
the elliptic curve (or complex torus) $\Eq := \C^{*}/q^{\Z}$ 
\index{elliptic curve $\Eq$}
in such a way that the space of global sections of $F_{A}$ 
can be identified with $\Gamma(M)$. We thus think 
of $\Gamma$ as a functor of global sections, and, 
the functor $M \leadsto F_{A}$ being exact, the $\Gamma^{i}$ 
can be defined through sheaf cohomology. \\

There is an interesting relationship between the space 
$\Gamma(M) = \Hom(\1,M)$ of solutions of $M$ and the space
$\Gamma(M^{\vee}) = \Hom(M,\1)$ of ``cosolutions'' of $M$. Starting 
from a unitary entire (analytic or formal) $q$-difference operator
$P = \sq^{n} + a_{1} \sq^{n-1} + \cdots + a_{n} \in \Dq$,
one studies the solutions of the $q$-difference equation:
$$
P.f := \sq^{n} f + a_{1} \sq^{n-1} f + \cdots + a_{n} f = 0
$$
by vectorializing it into the form:
$$
\sq X = A X, \text{~where~} 
X = \begin{pmatrix} f \\ \sq f \\ \vdots \\ 
\sq^{n-2} f \\ \sq^{n-1} f \end{pmatrix}
\text{~and~}
A = \begin{pmatrix}
0 & 1 & 0 & \ldots & 0 \\
0 & 0 & 1 & \ldots & 0 \\
\vdots & \vdots & \vdots & \ddots & \vdots \\
0 & 0 & 0 & \ldots & 1 \\
- a_{n} & - a_{n-1} & - a_{n-2} & - \ldots & - a_{1} 
\end{pmatrix}.
$$
Solutions of $P$ then correspond to solutions of $M = (\Ka^{n},\Phi_{A})$ 
or $(\Kf^{n},\Phi_{A})$. Now, by lemma \ref{lemm:cyclicvector}, one
has $M = \Dq/\Dq Q$ for some unitary entire $q$-difference operator $Q$.
Any such polynomial $Q$ is \emph{dual} \index{dual polynomial} 
to $P$. An explicit formula for 
a particular dual polynomial is given in \cite[prop. 2.1.10]{JSFIL}. \\

From the derived functors $\Gamma^{i}(M) = \Ext^{i}(\1,M)$, one can
recover general Ext-modules: 

\begin{prop}
\label{prop:GammaetExt}
There are functorial isomorphisms:
$$
\Ext^{i}(M,N) \simeq \Gamma^{i}(M^{\vee} \otimes N).
$$
\end{prop}
\begin{proof}
 The covariant functor $N \leadsto \Hom(M,N)$ is obtained by composing
the exact covariant functor $N \leadsto M^{\vee} \otimes N$ with the
left exact covariant functor $\Gamma$. Moreover, considered as a 
left exact functor over $\DMod$, it is isomorphic to the functor
$N \leadsto \Homint(M,N)$, which sends injective modules 
to injective modules. Indeed, the adjunction formula:
$$
\Hom(P,\Homint(M,N)) \simeq \Hom(P \otimes M,N)
$$ 
makes sense and remains valid for arbitrary modules in $\DMod$ for 
definitions of tensor product and internal Hom in $\DMod$ similar to 
the ones given above: see theorem \ref{theo:adjunction} in paragraph
\ref{subsection:internalhom} of section \ref{section:ALDIFF} of the 
appendix; and it implies that the flatness of $M$ over $K$ is enough 
to ensure the exactness of $P \leadsto \Hom(P,\Homint(M,N))$ whenever 
$N$ is injective. One can then derive $N \leadsto \Hom(M,N)$ using 
\cite[chap. III, \S 7, th. 1]{GM}. 
\end{proof} 

\begin{rema}
\label{rema:fibrefunctors}
For any difference field extension $(K,\sigma)$ of $(\Ka,\sq)$, the 
$K^{\sigma}$-vector space of solutions of $M$ in $K$ has dimension 
$\dim_{K^{\sigma}} \Gamma_{K}(M) \leq \rk M$: this follows from the $q$-wronskian 
lemma (\emph{see} \cite{LDV}). One can show that the functor $\Gamma_{K}$ 
is a fibre functor if, and only if, $(K,\sigma)$ is a \emph{universal field 
of solutions}, \ie\ such that one always have the equality
$\dim_{K^{\sigma}} \Gamma_{K}(M) = \rk M$. The field $\M(\C^{*})$ of meromorphic
functions over $\C^{*}$, with the automorphism $\sq$, is a universal field
of solutions, thus providing a fibre functor over 
$\M(\C^{*})^{\sq} = \M(\Eq)$ (field of elliptic functions). Actually, no
subfield of $\M(\C^{*})$ gives a fibre functor over $\C$ (this follows
from the argument at the beginning of \cite[0.3]{JSGAL}). In \cite{vdPS}, 
van der Put and Singer use an algebra of symbolic solutions which is 
reduced but not integral. A transcendental construction of a fibre functor 
over $\C$ is described in \cite{RS1}.
\end{rema}


\section{The Newton polygon and the slope filtration}
\label{section:Newtonetfiltration}

We summarize results from \cite{JSFIL}\footnote{Note however that, from 
\cite{RS1}, we changed our terminology: the slopes of a $q$-difference 
module are the \emph{opposites} of what they used to be (see herebelow
subsection \ref{subsection:Newtonpolygon}); and what we call a pure 
isoclinic, resp. pure module was previously called pure, resp. tamely 
irregular (see herebelow subsection \ref{subsection:puremodules}).}.


\subsection{The Newton polygon}
\label{subsection:Newtonpolygon}

The contents of this section are valid as well in the analytic as
in the formal setting. To the (analytic or formal) $q$-difference operator
$P = \sum a_{i} \sq^{i} \in \Dq$, we associate the \emph{Newton polygon}
\index{Newton polygon of a $q$-difference operator}
\label{not19}
$N(P)$, defined as the convex hull of 
$\{(i,j) \in \Z^{2} \tq j \geq v_{0}(a_{i})\}$, where $v_{0}$
denotes the $z$-adic valuation in $\Ka,\Kf$. Multiplying $P$ by a unit
$a \sq^{k}$ of $\Dq$ just translates the Newton polygon by a vector of
$\Z^{2}$, and we shall actually consider $N(P)$ as defined up to such a 
translation (or, which amounts to the same, choose a unitary entire $P$).
The relevant information therefore consists in the lower part of the boundary
of $N(P)$, made up of vectors 
$(r_{1},d_{1}),\ldots,(r_{k},d_{k})$, $r_{i} \in \N^{*}$, $d_{i} \in\Z$. Going 
from left to right, the \emph{slopes} 
\index{slopes}
$\mu_{i} := \dfrac{d_{i}}{r_{i}}$ are 
rational numbers such that $\mu_{1} < \cdots < \mu_{k}$. Their set is 
written $S(P)$ and $r_{i}$ is the \emph{multiplicity} 
\index{multiplicity}
of $\mu_{i} \in S(P)$. 
The most convenient object is however the \emph{Newton function}
\index{Newton function}
$r_{P}: \Q \rightarrow \N$, such that $\mu_{i} \mapsto r_{i}$ and null
out of $S(P)$. 

\label{not20}
\begin{theo}
\label{theo:Newtonpolygon}
(i) For a given $q$-difference module $M$, all unitary entire $P$ such 
that $M \simeq \Dq/ \Dq P$ have the same Newton polygon; 
the \emph{Newton polygon} of $M$
\index{Newton polygon of a $q$-difference module}
 is defined as $N(M) := N(P)$, and we put
$S(M) := S(P)$, $r_{M} := r_{P}$. \\
(ii) The Newton polygon is additive: for any exact sequence,
$$
0 \rightarrow M' \rightarrow M \rightarrow M'' \rightarrow 0
\Longrightarrow r_{M} = r_{M'} + r_{M''}.
$$
(iii) The Newton polygon is multiplicative: 
$$
\forall \mu \in \Q ~,~ r_{M_{1} \otimes M_{1}}(\mu) = 
\sum_{\mu_{1} + \mu_{2} = \mu} r_{M_{1}}(\mu_{1}) r_{M_{2}}(\mu_{2}).
\text{Also:~} r_{M^{\vee}}(\mu) = r_{M}(- \mu).
$$
(iv) Let $\ell \in \N^{*}$, introduce a new variable $z' := z^{1/\ell}$
(ramification) 
\index{ramification}
and make $K' := \Kf[z'] = \C((z'))$ or $\Ka[z'] = \C(\{z'\})$
a difference field extension by putting $\sigma'(z') = q' z'$, where $q'$ is 
any $\ell^{\text{th}}$ root of $q$. The Newton polygon of the $q'$-difference 
module $M' := K' \otimes M$ (computed w.r.t. variable $z'$) is given by the 
formula:
$$
r_{M}(\mu) = r_{M'}(\ell \mu).
$$
\end{theo}

To follow the next example, recall from paragraph
\ref{subsubsection:functorofsolutions} how we vectorialize a scalar
$q$-difference equation $P.f = 0$, $P \in \Dq$, by a $q$-difference
system $\sq X = A X$, and then associate to this system a $q$-difference
module $(\Ka^n,\Phi_A)$. Remember that the latter is isomorphic to some
$\Dq/\Dq Q$ (by the cyclic vector lemma and left euclideanity of $\Dq$),
but that $Q$ is not equal to $P$, it is dual to $P$.

\begin{exem}
\label{exem:Tshakaloff}
Vectorializing an analytic equation of order two yields:
$$
\sq^{2} f + a_{1} \sq f + a_{2} f = 0 \Longleftrightarrow \sq X = A X, 
\text{~with~} X = \begin{pmatrix} f \\ \sq f \end{pmatrix} \text{~and~} 
A = \begin{pmatrix} 0 & 1 \\ -a_{2} & - a_{1} \end{pmatrix}.
$$
The associated module is $(\Ka^{2},\Phi_{A})$. Putting 
$e := \begin{pmatrix} 1 \\ 0 \end{pmatrix}$, one has 
$\Phi_{A}(e) = \begin{pmatrix} -a_{1}/a_{2} \\ 1 \end{pmatrix}$ 
and: 
$$
\Phi_{A}^{2}(e) = \dfrac{-1}{a_{2}} e + \dfrac{- \sq a_{1}}{\sq a_{2}} \Phi_{A}(e)
\Longrightarrow M \simeq \Dq/ \Dq \hat{L}, \text{~where~}
\hat{L} := \sq^{2} + \dfrac{\sq a_{1}}{\sq a_{2}} \sq + \dfrac{1}{a_{2}},
$$
$\hat{L}$ thus being a dual of $L := \sq^{2} + a_{1} \sq + a_{2} \in \Dq$. \\
Note that if we started from vector
$e' := \begin{pmatrix} 0 \\ 1 \end{pmatrix}$, we would compute likewise:
$$
\Phi_{A}^{2}(e') = 
\dfrac{-1}{\sq a_{2}} e' + \dfrac{- a_{1}}{\sq a_{2}} \Phi_{A}(e')
\Longrightarrow M \simeq \Dq/ \Dq \hat{L'}, \text{~where~}
\hat{L'} := \sq^{2} + \dfrac{a_{1}}{\sq a_{2}} \sq + \dfrac{1}{\sq a_{2}},
$$
another dual of $L$. However, they give the same Newton polygon, since
$v_{0}(a_{1}/\sq a_{2}) = v_{0}(\sq a_{1}/\sq a_{2})$ and
$v_{0}(1/\sq a_{2}) = v_{0}(1/a_{2})$. \\
We now specialize to the equation satisfied by a $q$-analogue of the 
Euler series, the so-called \emph{Tshakaloff\footnote{We use the cyrillic
letter $\tsh$ (``tsh'') to denote it.}series}:
\index{Tshakaloff series}
\label{not21}
\begin{equation}
\label{eqn:Tshakaloff}
\tsh(z) := \sum_{n \geq 0} q^{n(n-1)/2} z^{n}.
\end{equation}
Then $\hat{\phi} := \tsh$ satisfies:
$$
\hat{\phi} = 1 + z \sq \hat{\phi} \Longrightarrow L.\hat{\phi} = 0,
$$
where:
$$
q z L := (\sq - 1)(z \sq - 1) = q z \sq^{2} - (1+z) \sq + 1.
$$
By our previous definitions, $S(L) = \{0,1\}$ (both multiplicities
equal $1$). The second computation of a dual above implies that:
$$
M \simeq \Dq/\Dq \hat{L}, \text{~where~}
\hat{L} := \sq^{2} - q(1+z) \sq + q^{2} z = (\sq - q z) (\sq - q),
$$
whence $S(M) = S(\hat{L}) = \{-1,0\}$ (both multiplicities equal $1$). \\
This computation relies on the obvious vectorialisation with matrix
$A = \begin{pmatrix} 0 & 1 \\ -1/qz & (1+z)/qz \end{pmatrix}$. However,
we also have:
$$
L.f = 0 \Longleftrightarrow \sq Y = B Y,
\text{~with~} Y = \begin{pmatrix} f \\ z \sq f - f \end{pmatrix} 
\text{~and~} B = \begin{pmatrix} z^{-1} & z^{-1} \\ 0 & 1 \end{pmatrix}.
$$
The fact that matrix $A$ is analytically equivalent to an upper triangular
matrix comes from the analytic factorisation of $\hat{L}$; the exponents
of $z$ on the diagonal are the slopes: this, as we shall see, is a general 
fact when the slopes are integral.
\end{exem}


\subsection{Pure modules}
\label{subsection:puremodules}

We call \emph{pure isoclinic (of slope $\mu$)} 
\index{pure isoclinic}
a module $M$ such that
$S(M) = \{\mu\}$ and \emph{pure} 
\index{pure}
a direct sum of pure isoclinic modules.
We call \emph{fuchsian} 
\index{fuchsian}
a pure isoclinic module of slope $0$.
The following description is valid whether $K = \Kf$ or $\Ka$.

\begin{lemm}
\label{lemm:puremodules}
(i) A pure isoclinic module of slope $\mu$ over $K$ can be written: 
$$
M \simeq \Dq/\Dq P, \text{~~where~~}
P = a_{n} \sq^{n} + a_{n-1} \sq^{n-1} + \cdots + a_{0} \in \Dq
\text~~{with:}
$$
$$
a_{0} a_{n} \neq 0 ~;~ 
\forall i \in \{1,n-1\} ~,~ v_{0}(a_{i}) \geq v_{0}(a_{0}) + \mu i
\text{~and~} v_{0}(a_{n}) = v_{0}(a_{0}) + \mu n.
$$
(ii) If $\mu \in \Z$, it further admits the following description:
$$
M = (K^{n},\Phi_{z^{\mu} A}) \text{~with~} A \in \GL_{n}(\C).
$$
\end{lemm}
\begin{proof}
The first description is immediate from the definitions. The second is
proved in \cite[th\'eor\`eme 6.15]{LDVS}.
\end{proof}


\subsubsection{Pure modules with integral slopes}

In particular, any fuchsian module is equivalent to some module 
$(K^{n},\Phi_{A})$ with $A \in \GL_{n}(\C)$. One may moreover require 
that $A$ has all its eigenvalues in the fundamental annulus:
\index{fundamental annulus}
$$
\forall \lambda \in \text{Sp} A \;,\; 
1 \leq \lmod \lambda \rmod < \lmod q \rmod.
$$
Indeed, one can do this using gauge transforms by shearing matrices,
like in \cite[1.1.1]{JSAIF}.  \\

Last, two fuchsian modules $(K^{n},\Phi_{A})$, $(K^{n'},\Phi_{A'})$ with
$A \in \GL_{n}(\C)$, $A' \in \GL_{n'}(\C)$ having all their eigenvalues 
in the fundamental annulus are isomorphic if and only if the matrices 
$A,A'$ are similar (so that $n = n'$). \\

It follows that any pure isoclinic module of slope $\mu$ is equivalent 
to some module $(K^{n},\Phi_{z^{\mu} A})$ with $A \in \GL_{n}(\C)$, the matrix
$A$ having all its eigenvalues in the fundamental annulus; and that two
such modules $(K^{n},\Phi_{z^{\mu} A})$, $(K^{n'},\Phi_{z^{\mu} A'})$ are 
isomorphic if and only if the matrices $A,A'$ are similar.


\subsubsection{Pure modules of arbitrary slopes}

The classification of pure modules of arbitrary (not necessarily 
integral) slopes 
\index{slopes (arbitrary)}
was obtained by van der Put and Reversat in \cite{vdPR}.
\index{Put@van der Put}
\index{van der Put}
\index{Reversat} 
It is cousin to the classification by Atiyah of vector bundles over an 
elliptic curve, which it allows to recover in a simple and elegant way.
Although we shall not need it, we briefly recall that result. \\

The first step is the classification of irreducible (that is, simple)
modules. An irreducible $q$-difference module $M$ is automaticaly pure 
isoclinic of slope say $\mu$. We write $\mu = d/r$ with $d,r$ coprime
and may assume that $r \geq 2$ (the case $r = 1$ is already known). 
Let $K' := K[z^{1/r}] = K[z']$ and $q'$ an arbitrary $r$th root of $q$. 
Then $M$ is isomorphic to the module obtained by restriction of scalars
from some $q'$-difference module $M'$ of rank $1$ over $K'$; and $M'$ is 
isomorphic to $(K',\Phi_{c {z'}^{d}})$ for a unique $c \in \C^{*}$ such that 
$1 \leq \lmod c \rmod < \lmod q' \rmod$. Actually:
$$
M \simeq E(r,d,c^{r}) := 
\Dq/\Dq \left(\sq^{r} - {q'}^{- dr(r-1)/2} c^{-r} z^{-d}\right).
$$
Moreover, for $E(r,d,a)$ and $E(r',d',a')$ to be isomorphic, it is
necessary (and sufficient) that $(r',d',a') = (r,d,a)$. \\

Then van der Put and Reversat prove that an indecomposable module $M$
(that is, $M$ is not a non trivial direct sum) comes from successive 
extensions of isomorphic irreducible modules; indeed, it has the form 
$M \simeq E(r,d,a) \otimes (K^{m},\Phi_{U})$ for some indecomposable
unipotent constant matrix $U$. Last, any pure isoclinic module is a 
direct sum of indecomposable modules in an essentially unique way.


\subsection{The slope filtration}
\label{subsection:slopefiltration}
\index{slope filtration}

Submodules, quotient modules, sums\ldots of pure isoclinic modules of
a given slope keep the same property. It follows that each module
$M$ admits a biggest pure submodule $M'$ of slope $\mu$ (this means
that $M'$ contains all submodules of $M$ that are pure of slope $\mu$); 
one then has \emph{a priori} $\rk M' \leq r_{M}(\mu)$. Like in the case 
of differential equations, one wants to ``break the Newton polygon'' 
and find pure submodules of maximal possible rank. In the formal case, 
the results are similar, but in the analytic case, we get a bonus.

\subsubsection{Formal case} 
 
Any $q$-difference operator $P$ over $\Kf$ admits, for all $\mu \in S(P)$, 
a factorisation $P = Q R$ with $S(Q) = \{\mu\}$ and 
$S(R) = S(P) \setminus \{\mu\}$. As a consequence, writing $M_{(\mu)}$
\label{not22}
the biggest pure submodule of slope $\mu$ of $M$, one has 
$\rk M_{(\mu)} = r_{M}(\mu)$ and $S(M/M_{(\mu)}) = S(M) \setminus \{\mu\}$.
Moreover, all modules are pure:
$$
M = \bigoplus_{\mu \in S(M)} M_{(\mu)}.
$$
The above splitting is canonical, functorial (preserved by morphisms)
and compatible with all tensor operations (tensor product, internal
Hom, dual).

\subsubsection{Analytic case}

Here, from old results due (independantly) to Adams and to Birkhoff
and Guenther, one draws that, if $\mu := \min S(P)$ is the smallest 
slope, then $P$ admits, a factorisation $P = Q R$ with $S(Q) = \{\mu\}$ 
and $S(R) = S(P) \setminus \{\mu\}$. As a consequence, the biggest pure 
submodule of slope $\mu$ of $M$, call it $M'$, again satisfies 
$\rk M' = r_{M}(\mu)$, so that $S(M/M') = S(M) \setminus \{\mu\}$.

\label{not23}
\begin{theo}
\label{theo:slopefiltration}
(i) Each $q$-difference module $M$ over $\Ka$ admits a unique filtration
with pure isoclinic quotients. It is an ascending filtration $(M_{\leq \mu})$ 
characterized by the following properties:
$$
S(M_{\leq \mu}) = S(M) \cap \left]- \infty,\mu\right] \text{~and~}
S(M/M_{\leq \mu}) = S(M) \cap \left]\mu,+\infty\right[.
$$
(ii) This filtration is strictly functorial, \ie\ all morphisms are strict
(see explanation after the ``proof''). \\
(iii) Writing $M_{< \mu} := \bigcup_{\mu' < \mu} M_{\leq \mu'}$ and
$M_{(\mu)} := M_{\leq \mu}/M_{< \mu}$ (which is pure isoclinic of slope $\mu$
and rank $r_{M}(\mu)$), the functor:
$$
M \leadsto \gr\ M := \bigoplus_{\mu \in S(M)} M_{(\mu)}
$$
\index{functor $\gr$}
is exact, $\C$-linear, faithful and $\otimes$-compatible.
\end{theo}
\begin{proof}
See \cite[3.2,3.3]{JSFIL}.
\end{proof}

Note that assertion (ii) says two things. First, a morphism
$f: M \rightarrow N$ of $q$-difference modules automatically respects
the canonical filtration, \ie\ $f(M_{\leq \mu}) \subset N_{\leq \mu}$ for 
all $\mu$. This implies in particular that $f$ induces morphisms of 
$q$-difference modules $M_{(\mu)} \rightarrow N_{(\mu)}$, so that 
$M \leadsto \gr M$ is indeed a functor. \\

Second, recall from \cite[p. II.2]{ALM} that a morphism $f: M \rightarrow N$ 
is said to be \emph{strict} with respect to the filtrations $(M_{\leq \mu})$, 
$(N_{\leq \mu})$, if $f(M_{\leq \mu}) = N_{\leq \mu} \cap f(M)$ for all $\mu$; 
equivalently: the two filtrations on $f(M)$ respectively induced by that 
of $N$ (by restriction) and by that of $M$ (through $f$) are identical.
it is then a classical consequence (and a nice exercice in linear algebra)
that the functor $M \leadsto \gr M$ is indeed exact.


\section{Application to the analytic isoformal classification}
\label{section:analyticisoformalclassification}
\index{analytic isoformal classification}
\index{analytic classification}

The formal classification of analytic $q$-difference modules is
equivalent to the classification of pure modules: the formal class 
of $M$ is equal to the formal class of $\gr M$, and the latter is
essentially the same thing as the analytic class 
\index{analytic class}
of $\gr M$. The classification of pure modules has been described 
in paragraph \ref{subsection:puremodules}. \\

Let us fix a formal class, 
\index{formal class}
that is, a pure module:
$$
M_{0} := P_{1} \oplus \cdots \oplus P_{k}.
$$
Here, each $P_{i}$  is pure isoclinic of slope $\mu_{i} \in \Q$ and 
rank $r_{i} \in \N^{*}$ and we assume that $\mu_{1} < \cdots < \mu_{k}$. 
All modules $M$ such that $\gr M \simeq M_{0}$ have the same Newton 
polygon $N(M) = N(M_{0})$. They constitute a formal class and we want 
to classify them analytically. The following definition is inspired 
by \cite{BV}.

\label{not24}
\begin{defi}
\label{defi:analyticisoformalclass}
We write $\F(M_{0})$ or $\F(P_{1},\ldots,P_{k})$ for the set of equivalence 
classes of pairs $(M,g)$ of an analytic $q$-difference module $M$ and an
isomorphism $g: \gr(M) \rightarrow M_{0}$, where $(M,g)$ is said to be 
equivalent to $(M',g')$ if there exists a morphism $f: M \rightarrow M'$ 
such that $g = g' \circ \gr(f)$.
\end{defi}

Note that $f$ is automatically an isomorphism. The goal of this paper
is to describe precisely $\F(P_{1},\ldots,P_{k})$. This definition
obviously admits a purely algebraic generalization, which is related 
to some interesting problems in homological algebra. In the appendix, 
we shall describe in some detail the necessary formalism and results 
in that direction; this will give us an adequate frame to formulate 
and prove our first structure theorems for the space $\F(M_{0})$ of 
isoformal analytic classes (theorems \ref{theo:schemadesmodules} and 
\ref{theo:kintegralslopesspace}).

\begin{rema}
\label{rema:badquotient1}
The group $\prod\limits_{1 \leq i \leq k} \Aut(P_{i})$ naturally operates
on $\F(P_{1},\ldots,P_{k})$ in the following way: if $(\phi_{i})$ is
an element of the group, thus defining $\phi \in \Aut(M_{0})$, then
send the class of the isomorphism $g: \gr(M) \rightarrow M_{0}$ to the 
class of $\phi \circ g$. The quotient of $\F(P_{1},\ldots,P_{k})$ by
this action is the set of analytic classes within the formal class 
of $M_{0}$, but it is not naturally made into a space of moduli.
\index{moduli (space of)}
An example is given at the end of the next paragraph (remark 
\ref{rema:badquotient2}).
\end{rema}


\subsection{A prototypal example and the $q$-Borel transformation}
\label{subsection:prototypal}
\index{prototypal example}

To give a feeling of the situation encountered we introduce at once a 
prototypal example. To begin with, we only assume that $K := \Ka$ or 
$K := \Kf$ without choosing. This will allow us to treat both the formal 
and the analytical situations, and to see clearly the difference. \\

Let $P_{1} := (K,\Phi_1) = (K,\sq)$ and $P_{2} := (K,\Phi_z) = (K,z^{-1} \sq)$. 
These pure isoclinic modules of respective slopes $0$ and $1$ correspond
to the rank $1$ systems $\sq f = f$ and $\sq f = z f$. The pure module
$M_0 := P_1 \oplus P_2$ defines a formal class and we intend to compute
the set $\F(M_0) = \F(P_1,P_2)$ of isoformal analytic classes. We thus 
consider $q$-difference modules $M$ endowed with a submodule $M_1$ 
and two isomorphisms: $M_1 \rightarrow P_1$ and $M/M_1 \rightarrow P_2$.
Such a module $M$ has the form $M_u = (K^2,\Phi_A)$ for some $u \in K$,
where:
$$
A = A_u := \begin{pmatrix} 1 & u \\ 0 & z \end{pmatrix}.
$$
(Thus, setting $u := 0$ indeed yields $M_0$.)
Note that this description encompasses the additional structure provided 
by the isomorphism $\gr\ M_u \simeq M_0$. Equivalence of two such objects 
$M_u$ and $M_v$ is given by an isomorphism described in matricial form as:
$$
F := \begin{pmatrix} 1 & f \\ 0 & 1 \end{pmatrix}:
\begin{pmatrix} 1 & u \\ 0 & z \end{pmatrix} \simeq
\begin{pmatrix} 1 & v \\ 0 & z \end{pmatrix}
\Longleftrightarrow (\sq F) A_u = A_v F.
$$
The fact that $F$ has this precise shape (unipotent and upper triangular)
corresponds to the condition $g = g' \circ \gr(f)$ in the definition of
our equivalence relation (definition \ref{defi:analyticisoformalclass}). \\

The relation $(\sq F) A_u = A_v F$ is in turn equivalent to the first order
inhomogeneous $q$-difference equation: 
\index{$q$-difference equation (inhomogeneous)}
$$
z \sq f - f = v - u.
$$
So we are actually interested in the cokernel of the mapping 
$\Psi: f \mapsto z \sq f - f$ from $K$ to itself. \\

Now, if $K = \Kf$, this mapping is bijective and admits as inverse the
mapping $w \mapsto - \sum_{n \geq 0} (z \sq)^n w$ (this is well defined
in the $z$-adic topology). Therefore, for all $u,v \in \Kf$, the equation 
$z \sq f - f = v - u$ admits a unique solution $f$ in $\Kf$, so that there 
is a unique class in $\F(P_{1},P_{2})$, that of the pure module 
$M_0 = P_1 \oplus P_2$. \\

However, if $K = \Ka$, one proves that $\Psi$ is injective (this is
obvious) and that its cokernel is isomorphic to $\C$. More precisely, 
for each $u \in \Ka$, there exists a unique $u_0 \in \C$ such that the 
equation $z \sq f - f = u - u_0$ admits a solution; and $u_0$ can be
explicitly computed, as we shall now see.

\label{not25}
\begin{defi}
\label{defi:qBorel}
\index{$q$-Borel transform}
The \emph{$q$-Borel transform} of $f(z) := \sum f_{n} z^{n} \in \Ka$ 
is defined as:
$$
\Bq f(\xi) := \sum \dfrac{f_{n}}{q^{n(n-1)/2}} \xi^{n}.
$$
\end{defi}

\begin{rema}
Actually, this is the transformation at \emph{level $1$}; for general
properties of $\Bq$ see \cite{RamisGrowth}. Also note the customary use 
of the new variable $\xi$ in the so-called \emph{Borel plane}.
\end{rema}

Clearly, the radius of convergence of $f$ being $> 0$, that of $\Bq f$ 
is infinite (under our general convention that $\lmod q \rmod > 1$). 
More precisely, the $\C$-linear mapping $\Bq$ sends $\Ka$ isomorphically 
to the following subspace of $\O(\C)[\xi^{-1}]$:
\label{not26}
$$
\C(\{\xi\})_{q,1} := \left\{\sum \phi_{n} \xi^{n} \in \C((\xi)) 
\tq \phi_{n} \prec q^{-n(n-1)/2}\right\}.
$$
Here, $u_{n} \prec v_{n}$ means: $u_{n} = O(A^{n} v_{n})$ for some $A > 0$. \\

Introduction of the $q$-Borel transform is motivated by the following 
easily checked property:
\begin{align*}
z \sq f - f = g 
& \Longleftrightarrow \forall n ~,~ q^{n-1} f_{n-1} - f_{n} = g_{n} \\
& \Longleftrightarrow \forall n ~,~ 
\dfrac{f_{n-1}}{q^{(n-1)(n-2)/2}} - \dfrac{f_{n}}{q^{n(n-1)/2}} =
\dfrac{g_{n}}{q^{n(n-1)/2}} \\
& \Longleftrightarrow (\xi - 1) \Bq f(\xi) = \Bq g(\xi).
\end{align*}
We therefore get a commutative diagram:
$$
\begin{CD}
  \Ka    @>{1 - z \sq}>>      \Ka         \\
@VV{\Bq}V         @VV{\Bq}V        \\
  \C(\{\xi\})_{q,1}  @>{\times(1-\xi)}>>  \C(\{\xi\})_{q,1}  
\end{CD}
$$
The vertical arrows are isomorphisms and the right one sends $\C \subset \Ka$ 
onto $\C \subset \C(\{\xi\})_{q,1}$. (The fact that $\times(1-\xi)$ does send 
$\C(\{\xi\})_{q,1}$ into itself is easily checked.)

\begin{lemm}
\label{lemm:qBorel}
The image of the lower horizontal arrow is the subspace
$\{\phi \in \C(\{\xi\})_{q,1} \tq \phi(1) = 0\}$.  
\end{lemm}
\begin{proof}
It is obvious that $\phi(1)$ makes sense for any $\phi \in \C(\{\xi\})_{q,1}$ 
and so that any $\phi$ in the image satisfies $\phi(1) = 0$. 
If $\phi(1) = 0$, it is also clear that 
$\psi := \dfrac{\phi}{1 - \xi} \in \O(\C)[\xi^{-1}]$, and we are left
to prove that $\psi \in \C(\{\xi\})_{q,1}$. So we write 
$\phi = \sum a_{n} \xi^{n}$,
$\psi = \sum b_{n} \xi^{n}$, so that we have $\sum a_{n} = 0$ and:
$$
b_{n} = \sum_{p \leq n} a_{p} = - \sum_{p > n} a_{p}.
$$
Suppose that $\lmod a_{n} \rmod \leq C A^{n} \lmod q \rmod^{-n(n-1)/2}$ for 
all $n$, where $C,A > 0$. Then:
\begin{align*}
\lmod b_{n} \rmod \leq C \sum_{p > n} A^{p} \lmod q \rmod^{-p(p-1)/2} & \leq 
C A^{n} \lmod q \rmod^{-n(n-1)/2} 
\sum_{\ell > 0} A^{\ell} \lmod q \rmod^{-\ell(2 n + \ell -1)/2} \\
& \leq C' A^{n} \lmod q \rmod^{-n(n-1)/2},
\end{align*}
where 
$C' := \sum\limits_{\ell > 0} A^{\ell} \lmod q \rmod^{-\ell(\ell -1)/2} < \infty$.
\end{proof}

\begin{prop}
\label{prop:qBorel}
One has a splitting of $\C$-vector spaces:
$$
\Ka = \C \oplus \text{Im~} \Psi.
$$
The corresponding projection operators are $u \mapsto \Bq u(1)$
and $u \mapsto u - \Bq u(1)$
\end{prop}
\begin{proof}
It follows from the lemma that the image of the upper horizontal arrow 
in the commutative diagram above is $\{g \in \Ka \tq \Bq g(1) = 0\}$.
The conclusion follows easily.
\end{proof}

Thus, the only \emph{obstruction} in solving $z \sq f - f = g$ analytically 
is $\Bq g(1)$ (we saw that solving formally is always possible). Therefore, 
each $u \in \Ka$ is equivalent to a unique $u_0 \in \C$ modulo $\text{Im} \Psi$,
and $u_0 = \Bq u(1)$. \\

This implies that the map $u \mapsto M_u$ induces a bijection of $\C$ 
with $\F(P_{1},P_{2})$. Moreover, we have representatives in ``normal 
form'' $\begin{pmatrix} 1 & u \\ 0 & z \end{pmatrix}$ with $u_0 \in \C$, 
a particular case of the ``Birkhoff-Guenther normal form'' 
\index{Birkhoff-Guenther normal form}
(subsection 
\ref{subsection:BGNF}). Note that this normal form rests on a transcendental 
computation.

\begin{exem}
If we take for instance $u := 1$ and $v := 0$, we see that the unique 
formal solution of the equation $z \sq f - f = v - u$ is the Tshakaloff
series 
\index{Tshakaloff series}
defined in \eqref{eqn:Tshakaloff}, page \pageref{eqn:Tshakaloff},
and considered as a $q$-analogue of the Euler series $\sum n! z^n$. 
\index{Euler Series}
In the same way as the latter is the simplest non trivial object in the 
analytical classification of complex linear differential equations so is 
the former for $q$-difference equations. As such, it will follow us, see 
\emph{e.g.} example \ref{exem:archetypal} page \pageref{exem:archetypal},
example \ref{exem:Tshakaloff7} page \pageref{exem:Tshakaloff7} and the 
more detailed study in section \ref{section:qEuler} of chapter 
\ref{chapter:examplesStokes}.
\end{exem}

\begin{rema}
\label{rema:badquotient2}
\index{moduli (space of)}
\index{quotient (badly behaved)}
In the notations of remark \ref{rema:badquotient1}, it is clear that 
$\Aut(P_{1}) = \Aut(P_{2}) = \C^{*}$ and it is easy to see that the
action of $\Aut(P_{1}) \times \Aut(P_{2}) = \C^{*} \times \C^{*}$ on
$\F(P_{1},P_{2}) = \C$ is given by $(t_1,t_2).u := (t_1/t_2)u$. Thus
one is led to the quotient $\C/\C^{*}$, which is rather badly behaved: 
it consists of two points $0,\overline{1}$, the second being dense.
\end{rema}


\subsection{The case of two slopes}
\label{subsection:thecaseoftwoslopes}

The case $k = 1$ is of course trivial. The case where $k = 2$ is 
``linear'' or ``abelian'': the set of classes is naturally a finite 
dimensional vector space over $\C$. We call it ``one level case'',
\index{one level case}
\index{two slopes case}
\index{abelian case}
because the \emph{$q$-Gevrey level} 
\index{$q$-Gevrey level (equation)}
$\mu_{2} - \mu_{1}$ (see the
paragraph \ref{subsection:qGevreylevels} of the general notations 
in the introduction for its definition) is the fundamental parameter.
This will be illustrated in section \ref{section:qGevreyinterpolation},
in chapter \ref{chapter:asymptotictheory}. Let $P,P'$ be pure analytic 
$q$-difference modules with ranks $r,r'$ and slopes $\mu < \mu'$. 

\label{not27}
\begin{prop}
\label{prop:classesaunniveau=Ext}
There is a natural one-to-one correspondance:
$$
\F(P,P') \rightarrow \Ext(P',P).
$$
\end{prop}
\begin{proof}
Here, $\Ext$ denotes the space of extension classes in the category
of left $\Dq$-modules; the homological interpretation is discussed in
the remark below. \\
Note that an extension of $\Dq$-modules of finite length has finite length,
so that an extension of $q$-difference modules is a $q$-difference module.
To give $g: \gr M \simeq P \oplus P'$ amounts to give an isomorphism
$M_{\leq \mu} \simeq P$ and an isomorphism $M/M_{\leq \mu} \simeq P'$,
\ie\ a monomorphism $i: P \rightarrow M$ and an epimorphism
$p: M \rightarrow P'$ with kernel $i(P)$, \ie\ an extension of $P'$
by $P$. Reciprocally, for any such extension, one automatically has
$M_{\leq \mu} = i(P)$, thus an isomorphism $g: \gr M \simeq P \oplus P'$.
The condition of equivalence of pairs $(M,g)$ is then exactly the condition 
of equivalence of extensions.
\end{proof}

\begin{rema}
By the classical identification of $\Ext$ spaces with $\Ext^{1}$ modules,
we thus get a description of $\F(P,P')$ as a $\C$-vector space. We shall
have use for an explicit description of this structure in terms of matrices:
this will be given in section \ref{section:extdiffmod} of the appendix.
\end{rema}


\subsection{Matricial description}
\label{subsection:matricialdescriptionofF(P1,...,Pk)}
\index{matricial description of $\F(P_1,\ldots,P_k)$}

We now go for a preliminary matricial description, valid without any
restriction on the slopes (number or integrity). Details can be found
in appendix \ref{chapter:ISOGRAD}. 
From the theorem \ref{theo:slopefiltration}, we 
deduce that each analytic $q$-difference module $M$ with 
$S(M) = \{\mu_{1},\ldots,\mu_{k}\}$, these slopes being indexed in 
increasing order: $\mu_{1} < \cdots < \mu_{k}$ and having multiplicities 
$r_{1},\ldots,r_{k} \in \N^{*}$, can be written $M = (\Ka^{n},\Phi_{A})$ with:
\label{not28}
\begin{equation}
\label{eqn:formestandardtriangulaire}
A = A_{U} := \begin{pmatrix}
B_{1}  & \ldots & \ldots & \ldots & \ldots \\
\ldots & \ldots & \ldots  & U_{i,j} & \ldots \\
0      & \ldots & \ldots   & \ldots & \ldots \\
\ldots & 0 & \ldots  & \ldots & \ldots \\
0      & \ldots & 0       & \ldots & B_{k}    
\end{pmatrix},
\end{equation}
where, for $1 \leq i \leq k$, $B_{i} \in \GL_{r_{i}}(\Ka)$ encodes a pure 
isoclinic module $P_{i} = (K^{r_{i}}, \Phi_{B_{i}})$ of slope $\mu_{i}$ and 
rank $r_{i}$ and where, for $1 \leq i < j \leq k$,
$U_{i,j} \in \Mat_{r_{i},r_{j}}(\Ka)$. Here, $U$ stands short for 
$(U_{i,j})_{1 \leq i < j \leq k} \in \prod\limits_{1 \leq i < j \leq k}
\Mat_{r_{i},r_{j}}(\Ka)$. Call $M_{U} = M$ the module thus defined: it is 
implicitly endowed with an isomorphism from $\gr M_{U}$ to
$M_{0} := P_{1} \oplus \cdots \oplus P_{k}$, here identified with
$(K^{n},\Phi_{A_{0}})$. If moreover the slopes are integral: 
$S(M) \subset \Z$, then one may take, for $1 \leq i \leq k$,
$B_{i} = z^{\mu_{i}} A_{i}$, where $A_{i} \in \GL_{r_{i}}(\C)$. \\

Now, a morphism from $M_{U}$ to $M_{V}$ compatible with the graduation 
(as in definition \ref{defi:analyticisoformalclass}) is a matrix:
\begin{equation}
\label{eqn:isoplat}
F := \begin{pmatrix}
I_{r_{1}}  & \ldots & \ldots & \ldots & \ldots \\
\ldots & \ldots & \ldots  & F_{i,j} & \ldots \\
0      & \ldots & \ldots   & \ldots & \ldots \\
\ldots & 0 & \ldots  & \ldots & \ldots \\
0      & \ldots & 0       & \ldots & I_{r_{k}}    
\end{pmatrix}, \text{~with~} 
(F_{i,j})_{1 \leq i < j \leq k} \in \prod_{1 \leq i < j \leq k}
\Mat_{r_{i},r_{j}}(\Ka),
\end{equation}
such that $(\sq F) A_{U} = B_{U} F$. The corresponding relations for 
the $F_{i,j},U_{i,j},V_{i,j}$ will be detailed in subsection 
\ref{subsection:BGNF}. Here, we just note 
that the above form of $F$ characterizes a unipotent algebraic subgroup
\index{unipotent algebraic group}
\label{not29}
\label{endroit:introductiondugroupeG}
$\G$ of $\GL_{n}$, which is completely determined by the Newton polygon
of $M_{0}$. The condition of equivalence of $M_{U}$ and $M_{V}$ reads:
\label{not30}
$$
M_{U} \sim M_{V} \Longleftrightarrow
\exists F \in \G(\Ka) ~:\: F[A_{U}] = A_{V}.
$$
The set of classes $\F(P_{1},\ldots,P_{k})$ may therefore be identified 
with the quotient of $\prod\limits_{1 \leq i < j \leq k} \Mat_{r_{i},r_{j}}(\Ka)$
by that equivalence relation, \ie\ by the action of $\G(\Ka)$. \\

For a more formal description, we introduce the block-diagonal part of $A$,
\ie\ the matrix $A_0$ corresponding to the pure module $M_0$:
\begin{equation}
\label{eqn:formestandarddiagonale}
A_{0} := \begin{pmatrix}
B_{1}  & \ldots & \ldots & \ldots & \ldots \\
\ldots & \ldots & \ldots  & 0 & \ldots \\
0      & \ldots & \ldots   & \ldots & \ldots \\
\ldots & 0 & \ldots  & \ldots & \ldots \\
0      & \ldots & 0       & \ldots & B_{k}    
\end{pmatrix},
\end{equation}
along with the following set:
\label{not31}
$$
\G^{A_{0}}(\Kf) := 
\{\hat{F} \in \G(\Kf) \tq \hat{F}[A_{0}] \in \GL_{n}(\Ka)\}.
$$
If $\hat{F} \in \G^{A_{0}}(\Kf)$ and $F \in \G(\Ka)$, then
$(F \hat{F})[A_{0}] = F\bigl[\hat{F}[A_{0}]\bigr] \in \GL_{n}(\Ka)$,
so that $F \hat{F} \in \G^{A_{0}}(\Kf)$: therefore, the group $\G(\Ka)$
operates on the set $\G^{A_{0}}(\Kf)$.

\begin{prop}
\label{prop:groupeformelsuranalytique}
The map $\hat{F}\mapsto \hat{F}[A_{0}]$ induces a one-to-one correspondance:
$$
\G^{A_{0}}(\Kf)/\G(\Ka) \rightarrow \F(P_{1},\ldots,P_{k}).
$$
\end{prop}
\begin{proof}
From the formal case in \ref{subsection:slopefiltration}, it follows that,
for any $U \in \prod\limits_{1 \leq i < j \leq k} \Mat_{r_{i},r_{j}}(\Ka)$ there 
exists a unique $\hat{F} \in \G(\Kf)$ such  that $\hat{F}[A_{0}] = A_{U}$. 
The equivalence of $\hat{F}[A_{0}]$ with $\hat{F}'[A_{0}]$ is then just
the relation $\hat{F}' \hat{F}^{-1} \in \G(\Ka)$.
\end{proof}

\begin{rema}
\label{not32}
Write $\hat{F}_{U}$ for the $\hat{F}$ in the above proof. More generally, 
for any $U,V \in \prod\limits_{1 \leq i < j \leq k} \Mat_{r_{i},r_{j}}(\Ka)$ there 
exists a unique $\hat{F} \in \G(\Kf)$ such  that $\hat{F}[A_{U}] = A_{V}$;
write it $\hat{F}_{U,V}$. Then $\hat{F}_{U,V} = \hat{F}_{V} \hat{F}_{U}^{-1}$
and the condition of equivalence of $M_{U}$ and $M_{V}$ reads:
$$
M_{U} \sim M_{V} \Longleftrightarrow \hat{F}_{U,V} \in \G(\Ka).
$$
Giving analyticity conditions for a formal object strongly hints towards
a resummation problem!
\index{resummation problem}
(See chapters \ref{chapter:qMalgrangeSibuya} and
\ref{chapter:asymptotictheory}.)
\end{rema}


 
\chapter{The affine space of isoformal analytic classes}
\label{chapter:isoformalclasses}


\section{Isoformal analytic classes of analytic $q$-difference modules}
\label{section:isoformalclasses}
\index{isoformal analytic classes of analytic $q$-difference modules}

We shall now specialize the results of the appendix \ref{chapter:ISOGRAD}, 
in particular its section \ref{section:extensionsofdifferencemodules}, to 
the case of $q$-difference modules\footnote{Reading the appendix is not
a prerequisite to reading that chapter, except for some particular explicit
references.}. 


\subsection{Extension classes of analytic $q$-difference modules}
\label{subsection:extensionsofqdifferencemodules}
\index{extension classes of analytic $q$-difference modules}

From here on, we consider only analytic $q$-difference modules and
the base field is $\Ka$ (except for brief indications about the
formal case).

\begin{theo}
\label{theo:dimensiondeuxpentes}
Let $M,N$ be pure modules of ranks $r,s \in \N^{*}$ and slopes
$\mu < \nu \in \Q$. Then $\dim_{\C} \F(M,N) = r s(\nu - \mu)$.
\end{theo}
\begin{proof}
Since $\F(M,N) \simeq \Ext^{1}(N,M) \simeq \Gamma^{1}(N^{\vee} \otimes M)$
(proposition \ref{prop:GammaetExt}) and since $N^{\vee} \otimes M$ is pure 
isoclinic of rank $r s$ and slope $\mu - \nu < 0$, the theorem is an 
immediate consequence of the following lemma.
\end{proof}

\begin{lemm}
\label{lemm:dimensiondeuxpentes}
Let $M$ be a pure module of ranks $r$ and slope $\mu < 0$. Then 
$\dim_{\C} \Gamma^{1}(M) = - r \mu$.
\end{lemm}
\begin{proof}
We give four different proofs, of which two require that $\mu \in \Z$.

1) Write $d := - r \mu \in \N^{*}$. If $M = \Dq/\Dq P$, the module $M^{\vee}$
is pure of rank $r$ and slope $- \mu$, and can be written as $\Dq/\Dq P^{\vee}$
for some dual $P^{\vee} = a_{0} + \cdots + a_{r} \sq^{r}$ of $P$, such that
$v_{0}(a_{0}) = 0$, $v_{0}(a_{r}) = d$ and $v_{0}(a_{i}) \geq i d/r$ for all $i$
(lemma \ref{lemm:puremodules}). We want to apply proposition 2.5 of
\cite{Bezivin}, but the latter assumes $\lmod q \rmod < 1$, so we 
consider $L := P^{\vee} \sq^{-r} = b_{0} + \cdots + b_{r} \sigma^{r}$, where 
$\sigma := \sq^{-1}$ and $b_{i} := a_{r-i}$. After \emph{loc.cit}, the operator
$L: \Ra \rightarrow \Ra$ has index $d$. More generally, for all $m \in \N^{*}$,
the operator $L: z^{-m} \Ra \rightarrow z^{-m} \Ra$ is conjugate to
$z^{m} L z^{-m} = \sum q^{i m} b_{i} \sigma^{i}: \Ra \rightarrow \Ra$,
which also has index $d$. Hence $L: \Ka \rightarrow \Ka$ has index $d$,
and so has $P^{\vee}$. After corollary \ref{coro:complexedessolutionsbis},
this index is $\dim_{\C} \Gamma^{1}(M) - \dim_{\C} \Gamma(M)$. But, $M$
being pure isoclinic of non null slope, $\Gamma(M) = 0$, which ends
the first proof.

2) In the following proof, we assume $\mu \in \Z$. From theorem 1.2.3 
of \cite{JSAIF}, we know that we may choose the dual operator such that:
$$
P^{\vee} = (z^{-\mu} \sq - c_{1}) u_{1} \cdots (z^{-\mu} \sq - c_{r}) u_{r},
$$
where $c_{1},\ldots,c_{r} \in \C^{*}$ and $u_{1},\ldots,u_{r} \in \Ra$, 
$u_{1}(0) = \cdots = u_{r}(0) = 1$. We are thus left to prove that 
each $z^{-\mu} \sq - c_{1}$ has index $-\mu$:
the indexes add up and the sum will be $- r \mu = d$. Since the kernels
are trivial, we must compute the cokernel of an operator $z^{m} \sq - c$,
$m \in \N^{*}$. But it follows from lemma \ref{lemm:supplementairedeuxpentes}
that the image of $z^{m} \sq - c: \Ka \rightarrow \Ka$ admits the
supplementary space $\C \oplus \cdots \oplus \C z^{m-1}$.

3) In the following proof, we assume again $\mu \in \Z$. After 
lemma \ref{lemm:puremodules}, we can write $M = (\Ka^{r},\Phi_{z^{\mu} A})$, 
$A \in \GL_{r}(\C)$. After corollary \ref{coro:complexedessolutions}, we
must consider the cokernel of 
$\Ka^{r} \overset{\Phi_{z^{\mu} A} - \Id}{\longrightarrow} \Ka^{r}$. But, after 
lemma \ref{lemm:supplementairedeuxpentes}, the image of $\Phi_{z^{\mu} A} - \Id$
admits the supplementary space $(\C \oplus \cdots \oplus \C
z^{-\mu-1})^{r}$.

4) A similar proof, but for arbitrary $\mu$, can be deduced from \cite{vdPR}.
It follows indeed from this paper that each isoclinic module of slope $\mu$
can be obtained by successive extensions of modules admitting a dual of the
form $\Dq/\Dq (z^{a} \sq^{b} - c)$, where $b/a = - \mu$ and $c \in \C^{*}$.
\end{proof}

\begin{lemm}
\label{lemm:supplementairedeuxpentes}
Let $d,r \in \N^{*}$ and $A \in \GL_{r}(\C)$. Let $D \subset \Z$ be any set
of representatives modulo $d$, for instance $\{a,a+1,\ldots,a+d-1\}$ for
some $a \in \Z$. Then, the image of the $\C$-linear map 
$F: X \mapsto z^{d} A \sq X - X$ from $\Ka^{r}$ to itself admits as a
supplementary $\left(\sum\limits_{i \in D} \C z^{i}\right)^{r}$.
\end{lemm}
\begin{proof}
For all $i \in \Z$, write $K_{i} := z^{i} \C(\{z^{d}\})$, so that
$\Ka = \bigoplus_{i \in D} K_{i}$. Each of the $K_{i}^{r}$ is stable under $F$.
We write $w := z^{d}$, $L := \C(\{w\})$, $\rho := q^{d}$ and define $\sigma$
on $L$ by $\sigma f(w) = f(\rho w)$. Multiplication by $z^{i}$ sends $L^{r}$
to $K_{i}^{r}$ and conjugates the restriction of $F$ to $K_{i}^{r}$ to the 
mapping $G_{i}: Y \mapsto w q^{i} A \sigma Y - Y$ from $L^{r}$ to itself. 
We are left to check that the image of $G_{i}$ admits $\C^{r}$ as a 
supplementary. But this is just the case $d = 1$, $D = \{0\}$ of the lemma. 
So we tackle this case under these assumptions with the notations of the lemma.

So write $F_{A}: X \mapsto z A \sq X - X$ from $\Ka^{r}$ to itself. Also
write $X = \sum X_{n} z^{n}$ and $Y = F_{A}(X) = \sum Y_{n} z^{n}$ (all sums
here have at most a finite number of negative indices), so that
$Y_{n} = q^{n-1} A X_{n-1} - X_{n}$. Putting $\tilde{X} := \sum A^{-n} X_{n} z^{n}$
and $\tilde{Y} := \sum A^{-n} Y_{n} z^{n}$, from the relation
$A^{-n} Y_{n} = q^{n-1} A^{-(n-1)} X_{n-1} - A^{-n} X_{n}$, we draw
$\tilde{Y} = z \sq \tilde{X} - X$. Of course, $X \mapsto \tilde{X}$ is an
automorphism of $\Ka^{r}$ and we just saw that it conjugates $F_{A}$ to the
map $F: X \mapsto z \sq X - X$ from $\Ka^{r}$ to itself. Moreover, the same
automorphism leaves invariant the subspace $\C^{r} \subset \Ka^{r}$, so the 
question boils down to prove that the image of $F$ has supplementary $\C^{r}$. 
And this, in turn, splits into $r$ times the same problem for $r = 1$.

So we are left with the case of the map $\Psi: f \mapsto z \sq f - f$ 
from $\Ka$ to itself, which is the very heart of the theory ! (See for 
instance example \ref{exem:archetypal} page \pageref{exem:archetypal} and 
section \ref{section:qEuler} of chapter \ref{chapter:examplesStokes}.)
This case has been studied in detail in paragraph \ref{subsection:prototypal}.
We saw there that $\Psi$ is injective and that its image is in direct sum 
with $\C \subset \Ka$. More precisely, for any $u \in \Ka$, the unique
$u_0 \in \C$ such that $u - u_0 \in \text{Im~} \Psi$ is 
$u_0 := \Bq \, u(1)$ (the $q$-Borel transform $\Bq \, u$ was defined
in paragraph \ref{subsection:prototypal}). This achieves the proof.
\end{proof}

Note that from the proof, one draws explicit projection maps from $\Ka^{r}$
to the image of $F_{A}$ and its supplementary $\C^{r}$: these are the maps
$Y \mapsto Y - \Bq Y(A^{-1})$ and $Y \mapsto \Bq Y(A^{-1})$, where:
$$
\Bq Y(A^{-1}) := \sum q^{-n(n-1)/2} A^{-n} Y_{n}.
$$


\subsection{The affine scheme  $\F(P_{1},\ldots,P_{k})$}
\label{subsection:theaffineschemeF(M0)}
\index{affine scheme}

Now let $P_{1},\ldots,P_{k}$ be pure isoclinic $q$-difference modules 
over $\Ka$, of ranks $r_{1},\ldots,r_{k} \in \N^{*}$ and slopes 
$\mu_{1} < \cdots < \mu_{k}$. From section \ref{section:modulispace} of 
the appendix \ref{chapter:ISOGRAD}, we have a functor from commutative 
$\C$-algebras to sets:
$$
C \leadsto F(C) := \F(C \otimes_{\C} P_{1},\ldots,C \otimes_{\C} P_{k}).
$$
Here, the base change $C \otimes_{\C} -$ means that we extend the
scalars of $q$-difference modules from $\Ka$ to $C \otimes_{\C} \Ka$.

\begin{theo}
\label{theo:schemadesmodules}
\index{affine space}
The functor $F$ is representable and the corresponding affine scheme 
is an affine space over $\C$ with dimension: 
$$
\dim \F(P_{1},\ldots,P_{k}) =
\sum\limits_{1 \leq i < j \leq k} r_{i} r_{j} (\mu_{j} - \mu_{i}).
$$
\end{theo}
\begin{proof}
We apply theorem \ref{theo:themodulispaceabstract}. We need two check
the assumptions:
$$
\forall i,j, \text{~~s.t.~~} 1 \leq i < j \leq k ~,~
\Hom(P_{j},P_{i}) = 0 \text{~~and~~}
\dim_{\C}\Ext(P_{j},P_{i}) = r_{i} r_{j} (\mu_{j} - \mu_{i}).
$$
The first fact comes from theorem \ref{theo:slopefiltration} and the
second fact from theorem \ref{theo:dimensiondeuxpentes}; the fact that
the extension modules are free is here obvious since $\C$ is a field.
\end{proof}

\begin{rema}
\index{moduli (space of)}
The use of base changes $C \otimes_{\C} \Ka$ is too restrictive to
consider this as a true space of moduli, as was done for differential
equations in \cite[chap. III]{BV}: we would like to consider families 
of modules the coefficients of which are arbitrary analytic functions
 of some parameter. This will be done in \cite{JSmoduli}.
\end{rema}

Now, to get coordinates on our affine scheme, we just have to find bases 
of the extension modules $\Ext(P_{j},P_{i})$ for $1 \leq i < j \leq k$:
this follows indeed from corollary \ref{coro:themodulispaceabstract}.
An example will be given by corollary \ref{coro:schemadesmodules}.


\section{Index theorems and irregularity}
\label{section:indextheorems}
\index{index theorems}
\index{irregularity}

We shall now give some complements to the results in
\ref{subsection:extensionsofqdifferencemodules}, in the form of index 
theorems. They mostly originate in the works \cite{Bezivin} of 
B\'ezivin\footnote{The index computations in \cite{Bezivin} are more 
general in two directions: they apply to ``difference'' automorphisms 
that are more general that $z \mapsto q z$; and the equations are not 
linear.} and also in \cite{RamisGrowth}: the difference here is merely
the adaptation to our formalism. Although not strictly necessary for what 
follows, these results allow for an interpretation of
$\dim \F(P_{1},\ldots,P_{k})$ (theorem \ref{theo:schemadesmodules})
in terms of ``irregularity'', as in the paper \cite{Ma3} of Malgrange. \\

In most of this section, we write $K$ for (indifferently) $\Kf$ or $\Ka$
and respectively speak of the \emph{formal} or \emph{convergent} case.
\index{formal case}
\index{convergent case}
We intend to compute the index of an analytic $q$-difference operator
$P \in \Dq$ acting upon $\Kf$ and $\Ka$ (and, in the end, upon $\Kf/\Ka$).
To begin with, we do not assume $P$ to have analytic coefficients.


\subsection{Kernel and cokernel of $\sq  - u$}

We start with the case $\deg P = 1$. Up to an invertible factor
in $\Dq$, we may assume that $P = \sq  - u$, where
$u = d q^{k} z^{\nu} v$, $v \in K, \; v(0) = 1$, with $k \in \Z$ 
and $d \in \C^{*}$ such that $1 \leq |d| < \lmod q \rmod$.
We consider $P$ as a $\C$-linear operator on $K$.

\begin{enonce*}{Fact}
The dimensions over $\C$ of the kernel and cokernel of $P$ depend
only of the class of $u \in K^{*}$ modulo the subgroup
$\{\frac{\sq (w)}{w} \tq w \in K^{*}\}$ of $K^{*}$. This class
is equal to that of a $d z^{\nu}$. 
\end{enonce*}
\begin{proof}
Conjugating the $\C$-linear endomorphism $\sq  - u$ of $K$
by $w \in K^{*}$ (undestood as the automorphism $\times w$), one finds:
$$
w \circ (\sq  - u) \circ w^{-1} = 
\frac{w}{\sq (w)} \circ (\sq  - u'),
\text{~with~} u' = u \frac{\sq (w)}{w},
$$
whence the first statement. Then, $q^{k} v = \frac{\sq (w)}{w}$, where
$w = z^{k} \underset{i \geq 1}{\prod} \sq ^{-i}(v) \in K$ (also well
defined in the convergent case), whence the second statement.
\end{proof}

So we now set $u := d z^{\nu}$. 

\begin{enonce*}{Fact}
The kernel of $\sq  - u : K \rightarrow K$ has dimension $1$ if
$(\nu,d) = (0,1)$, and $0$ otherwise.
\end{enonce*}
\begin{proof}
The series $f = \sum f_{n} z^{n}$ belongs to the kernel if, and only 
if $\forall n \;,\, q^{n} f_{n} = d f_{n - \nu}$. Since we deal with
Laurent series having poles, this implies $f = 0$ except maybe
if $\nu = 0$. In the latter case, it implies $f = 0$, except maybe if
$d \in q^{\Z}$. From the condition on $|d|$, this is only possible
if $d = 1$, thus $u = 1$. In that case, the kernel is plainly $\C$.
\end{proof}

\begin{enonce*}{Fact}
The map $\sq  - u : K \rightarrow K$ is onto if $\nu >0$ and also if 
$\nu = 0, \; d \not= 1$.
\end{enonce*}
\begin{proof}
Let $g \in K$. We look for $f$ solving the equation
$\sq  f - d z^{\nu} f = g$ in $K$. This is equivalent to:
$\forall n \;,\, q^{n} f_{n} - d f_{n - \nu} = g_{n}$,
that is, 
$\forall n \;,\, f_{n} = q^{-n} (d f_{n - \nu} + g_{n})$.
If $\nu \geq 1$, one computes the coefficients by induction from
the $\nu$ first among them. Moreover, in the convergent case,
inspection of the denominators show that $f$ converges if $g$ 
does; and then, the functional equation ensures meromorphy.
If $\nu = 0$ and $d \not= 1$ (so that $d \notin q^{\Z}$), 
one gets rightaway $f_{n} = \frac{g_{n}}{q^{n} - d}$. Convergence 
(resp. meromorphy) is then immediate in the convergent case.
(When $\nu \geq 1$, one can also consider the fixpoint equation
$F(f) = f$, where $F(f) := \sq ^{-1}(g + d z^{\nu} f)$: it is easy
to see that this is a contracting operator as well for the formal
topology, \ie\ for $z$-adic convergence, as for the transcendant 
topology, \ie\ for usual convergence.)
\end{proof}

\begin{enonce*}{Fact}
If $(\nu,d) = (0,1)$, the cokernel of $\sq  - u : K \rightarrow K$ 
has dimension $1$.
\end{enonce*}
\begin{proof}
One checks that $\sq  - u$ vanishes on $\C$ and induces
an automorphism of the subspace $K^{\bullet}$ of $K$ made up of
series without constant term.
\end{proof}

\begin{enonce*}{Fact}
Assume $\nu < 0$. Then $\sq  - u : K \rightarrow K$ is onto in
the formal case.
\end{enonce*}
\begin{proof}
Put $F'(f) = d^{-1} z^{- \nu} (\sq (f) - g)$. Since $\nu > 0$, this
is a $z$-adically contracting operator, whence the existence (and
unicity) of a fixed point. 
\end{proof} 

\begin{rema}
This is the first place where the formal and convergent cases differ.
The operator is (rather strongly) \emph{expanding} for the transcendant
topology, it produces Stokes phenomena! 
\index{Stokes phenomenon}
So this is where we need an 
argument from analysis in the convergent case. 
\end{rema}

\begin{enonce*}{Fact}
Assume $\nu = - r, r \in \N^{*}$. Then, in the convergent 
case, the cokernel of  $\sq  - u : K \rightarrow K$ has dimension $r$.
\end{enonce*}
\begin{proof}
This is a consequence of lemma \ref{lemm:dimensiondeuxpentes} (and
therefore relies on the use of the $q$-Borel transformation).
\end{proof}
We now summarize our results:

\begin{prop} 
\label{prop:indices}
Let $u = d z^{\nu} v$, $v \in K, \; v(0) = 1$, with
$d \in \C^{*}$. Write $\overline{d}$ the class of $d$
modulo $q^{\Z}$. The following table shows the ranks of 
the kernel and cokernel, as well as the index 
\index{index}
$\chi(P) := \dim \Ker P - \dim \text{Coker~} P$ of the
$\C$-linear operator $P := \sq  - u: K \rightarrow K$: \\
\begin{tabular}[t]{||l|l||l|l||l||}
\hline
\multicolumn{2}{||c||}{$(\nu,\overline{d})$} & Kernel & Cokernel & 
Index  \\
\hline
\multicolumn{2}{||c||}{$(0,1)$}   &  $1$  &  $1$ & $0$ \\
\multicolumn{2}{||c||}{$(0,\not=1)$} &  $0$  &  $0$ & $0$ \\
\multicolumn{2}{||c||}{$(> 0,-)$} &  $0$  &  $0$ & $0$ \\
\multicolumn{2}{||c||}{$(< 0,-)$} &  $0$  &  
$\begin{cases} 0 \text{ (formal case) } \\ 
- \nu \text{ (convergent case) } \end{cases}$ &  
$\begin{cases} 0 \text{ (formal case) } \\ 
\nu \text{ (convergent case) } \end{cases}$ \\
\hline
\end{tabular}
\end{prop}


\subsection{The index in the general case}
\index{index}

Generally speaking, there is no simple formula for the dimensions
of the kernel and cokernel of a $q$-difference operator $P$. However,
if $P$ has integral slopes\footnote{This assumption will be dropped 
in the four next corollaries.}, we can factor it into operators of degree
$1$, and then use proposition \ref{prop:indices} and linear algebra to
deduce:
\begin{enumerate}
\item{The $\C$-linear map $P: f \mapsto P.f$ from $K$ to itself has
an index, that is finite dimensional kernel and cokernel.}
\item{The index of $P$, that is the integer 
$\chi(P) := \dim \Ker P - \dim \text{Coker~} P$ is the sum of indices
of factors of $P$.}
\end{enumerate}

Note that $\chi(P)$ is the Euler-Poincar\'e characteristic of the
complex of solutions of $P$.

\begin{coro}
Let $P$ be an operator of order $n$ and pure of slope $\mu \neq 0$. \\
(i) In the formal case, $\dim \Ker P = \dim \text{Coker~} P = 0$. \\
(ii) In the convergent case $\dim \Ker P = 0$ and
$\dim \text{Coker~} P = n \max(0,\mu)$.
\end{coro}
\begin{proof}
If $\mu \in \Z$, all the factors of $P$ have the form $(z^{\mu} \sq -c).u$. 
If $\mu \neq 0$, they are all bijective in the formal case, whence the
first statement. They are also all injective in the convergent case, 
and onto if $\mu < 0$, whence the second statement in this case. If
$\mu > 0$, they all have index $\mu$ (in the convergent case), whence 
the second statement in this case by additivity of the index. \\
The case of an arbitrary slope is easily reduced to this case by
ramification.
\end{proof}

By similar arguments:

\begin{coro}
Let $P$ be an operator pure of slope $0$. Then, the index of $P$
(in the formal or convergent case) is $0$.
\end{coro}

Now, using the additivity of the index:

\begin{coro}
The index of an arbitrary operator $P$ is $0$ in the formal case
and $- \sum\limits_{\mu > 0} r_{P}(\mu) \mu$ in the convergent case.
\end{coro}

\begin{coro}
Let $P \in \Dq$ be an analytic $q$-difference operator. The index
of $P$ acting as a $\C$-linear endomorphism of $\Kf/\Ka$ is equal
to $\sum\limits_{\mu > 0} r_{P}(\mu) \mu$.
\end{coro}

Of course, $r_{P}$ denotes here the Newton function of $P$
introduced in subsection \ref{subsection:Newtonpolygon}.


\subsection{Irregularity and the dimension of $\F(P_{1},\ldots,P_{k})$}
\label{subsection:algebraicirregularity}
\index{irregularity}

The following definition is inspired by that of \cite{Ma3}.

\begin{defi}
\label{defi:irregularity}
The \emph{irregularity} of a $q$-difference operator $P$, resp. of
$q$-difference module $M$, is defined by the formulas:
\label{not33}
$$
\text{Irr~}(P) := \sum\limits_{\mu > 0} r_{P}(\mu) \mu, \quad
\text{Irr~}(M) := \sum\limits_{\mu > 0} r_{M}(\mu) \mu.
$$
\end{defi}

Graphically, this is the height of the right part of the Newton 
polygon, from the bottom to the upper right end. It is clearly a 
formal invariant. From its interpretation as an index in $\Kf/\Ka$, 
the irregularity of operators is additive with respect to the product. 
From theorem \ref{theo:Newtonpolygon}, the irregularity of modules is
additive with respect to exact sequences. Moreover, with the notations
of \ref{theo:slopefiltration}, writing
$$
M^{>0} := M/M_{\leq 0}
$$
the ``positive part'' of $M$, one has: 
$$
\text{Irr~}(M) = \text{Irr~}(M^{>0}).
$$

\label{not34}
Write $M_{0} := P_{1} \oplus \cdots \oplus P_{k}$. 
\index{internal End}
Then the ``internal
End'' $E := \underline{\End}(M_{0}) = M_{0}^{\vee} \otimes M_{0}$ of 
this module (subsection \ref{subsection:generalitiesdiffmods}) has as
Newton function:
$$
r_{E}(\mu) = \sum_{\mu_{j} - \mu_{i} = \mu} r_{i} r_{j}.
$$
(This follows from theorem \ref{theo:Newtonpolygon}.) From theorem 
\ref{theo:schemadesmodules} and definition \ref{defi:irregularity}, 
we therefore get the equality:
$$
\dim \F(P_{1},\ldots,P_{k}) = 
\text{Irr~} \left(\underline{\End}(M_{0})\right).
$$
In subsection \ref{subsection:sheaftheoreticirregularity}, we shall
give a sheaf theoretical interpretation of this formula.


\section{Explicit description of $\F(P_{1},\ldots,P_{k})$ 
in the case of integral slopes}
\label{section:explicitdescription}
\index{integral slopes}
\index{slopes (integrality assumption)}

\noindent
\emph{\textbf{\large FROM NOW ON, WE ASSUME THAT THE SLOPES ARE INTEGRAL:}}
$\mu_{1},\ldots,\mu_{k} \in \Z$. \\

It is then possible to make the results of sections 
\ref{section:isoformalclasses} and \ref{section:indextheorems}
more algorithmic by using the matricial description of paragraph
\ref{subsection:matricialdescriptionofF(P1,...,Pk)}. Indeed, as 
explained in paragraph \ref{subsubsection:integralslopes}, our 
construction of normal forms and of explicit Stokes operators depends 
on the normal form \eqref{eqn:formestandardtriangulaireentiere} given 
herebelow for pure modules. \\

For $i = 1,\ldots,k$, we thus write $P_{i} = (\Ka^{r_{i}},\Phi_{z^{\mu_{i}} A_{i}})$, 
with $A_{i} \in \GL_{r_{i}}(\C)$. Then, putting $n := r_{1} + \cdots + r_{k}$,
we have $M_{0} := P_{1} \oplus \cdots \oplus P_{k} = (K^{n},\Phi_{A_{0}})$, with
$A_{0}$ as in equation \eqref{eqn:formestandarddiagonale} page
\pageref{eqn:formestandarddiagonale} from paragraph
\ref{subsection:matricialdescriptionofF(P1,...,Pk)}:
\begin{equation}
\label{eqn:formestandarddiagonaleentiere}
A_{0} := \begin{pmatrix}
z^{\mu_{1}} A_{1}  & \ldots & \ldots & \ldots & \ldots \\
\ldots & \ldots & \ldots  & 0 & \ldots \\
0      & \ldots & \ldots   & \ldots & \ldots \\
\ldots & 0 & \ldots  & \ldots & \ldots \\
0      & \ldots & 0       & \ldots & z^{\mu_{k}} A_{k}    
\end{pmatrix}.
\end{equation}
Any class in $\F(P_{1},\ldots,P_{k})$ can be represented by a module
\label{not35}
$M_{U} := (K^{n},\Phi_{A_{U}})$ for some
$U := (U_{i,j})_{1 \leq i < j \leq k} \in 
\prod\limits_{1 \leq i < j \leq k} \Mat_{r_{i},r_{j}}(\Ka)$,
with, $A_{U}$ as in equation (\ref{eqn:formestandardtriangulaire}):
\begin{equation}
\label{eqn:formestandardtriangulaireentiere}
A_{U} := \begin{pmatrix}
z^{\mu_{1}} A_{1} & \ldots & \ldots & \ldots & \ldots \\
\ldots & \ldots & \ldots  & U_{i,j} & \ldots \\
0      & \ldots & \ldots   & \ldots & \ldots \\
\ldots & 0 & \ldots  & \ldots & \ldots \\
0      & \ldots & 0       & \ldots & z^{\mu_{k}} A_{k} 
\end{pmatrix},
\end{equation}
and $M_{U},M_{V}$ represent the same class if, and only if, the equation
$(\sq F) A_{V} = A_{U} F$ can be solved with $F \in \G(\Ka)$. Such an $F$
is then unique. 


\subsection{Analytic classification with $2$ integral slopes}
\label{subsection:twointegralslopes}

From the results of \ref{subsection:extensionsofqdifferencemodules},
we get the following result for the case of two slopes. Here, and after,
we write, for $\mu < \nu \in \Z$:
$$
\label{not36}
K_{\mu,\nu} := \sum_{i=0}^{\nu - \mu - 1} \C z^{i} \text{~or, at will:~}
K_{\mu,\nu} := \sum_{i=\mu}^{\nu - 1} \C z^{i}.
$$
(The second choice is motivated by good tensor properties, see for
instance \cite{RS1}.) Then, for $r,s \in \N^{*}$, we denote
$\Mat_{r,s}(K_{\mu,\nu})$ the space of $r \times s$ matrices with coefficients
in $K_{\mu,\nu}$.

\begin{prop}
\label{prop:espaceadeuxpentes}
Let $U \in \Mat_{r_{1},r_{2}}(\Ka)$. Then, there exists a unique pair:
$$
\label{not37}
Red(\mu_{1},A_{1},\mu_{2},A_{2},U) := 
(F_{1,2},V) \in \Mat_{r_{1},r_{2}}(\Ka) \times \Mat_{r_{1},r_{2}}(K_{\mu_{1},\mu_{2}})
$$
such that:
$$
(\sq F_{1,2}) z^{\mu_{2}} A_{2} - z^{\mu_{1}} A_{1} F_{1,2} = U - V,
$$
that is:
$$
\begin{pmatrix} I_{r_{1}} & F_{1,2} \\ 0 & I_{r_{2}} \end{pmatrix}:
\begin{pmatrix}
z^{\mu_{1}} A_{1} &  V            \\
0              & z^{\mu_{2}} A_{2} 
\end{pmatrix} \simeq
\begin{pmatrix}
z^{\mu_{1}} A_{1} &  U            \\
0              & z^{\mu_{2}} A_{2} 
\end{pmatrix}.
$$
\end{prop}
\begin{proof}
Using the reductions of lemma \ref{lemm:dimensiondeuxpentes}, this is just
a rephrasing of lemma \ref{lemm:supplementairedeuxpentes}.
\end{proof}

We consider $V$ as a \emph{polynomial normal form}
\index{polynomialnormal form}
for the class of $M_{U}$ and $F_{1,2}$ as the corresponding reduction datum. 
Of course, they depend on the choice of the space of coefficients 
$K_{\mu_{1},\mu_{2} }$. \\

\begin{exem}
\label{exem:archetypal}
\index{archetypal example}
This is the archetypal example (and strongly related to the prototypal
example of subsection \ref{subsection:prototypal}), from which much of 
the theory is built. Fix $c \in \C^{*}$. For any $f,u,v \in \Ka$, the 
isomorphism:
$\begin{pmatrix}1 & f \\ 0 & 1 \end{pmatrix}:
\begin{pmatrix} 1 & v \\ 0 & cz \end{pmatrix} \simeq
\begin{pmatrix} 1 & u \\ 0 & cz \end{pmatrix}$ is equivalent to 
the inhomogeneous $q$-difference equation: $c z \sq f - f = u - v$.
Writing $f = \sum f_{n} z^{n}$, etc, we find the conditions:
$$
c q^{n-1} f_{n-1} - f_{n} = u_{n} - v_{n} \Longleftrightarrow
\dfrac{f_{n-1}}{c^{n-1} q^{(n-1)(n-2)/2}} - \dfrac{f_{n}}{c^{n} q^{n(n-1)/2}} = 
\dfrac{u_{n} - v_{n}}{c^{n} q^{n(n-1)/2}},
$$
which admit an analytic solution $f$ if, and only if, 
$\Bq u(c^{-1}) = \Bq v(c^{-1})$ (recall that the $q$-Borel transform $\Bq$
was defined in definition \ref{defi:qBorel} page \pageref{defi:qBorel}).
The analytic class of the module 
with matrix $\begin{pmatrix} 1 & u \\ 0 & cz \end{pmatrix}$ within the 
formal class $\begin{pmatrix} 1 & 0 \\ 0 & cz \end{pmatrix}$ is therefore
$\Bq u(c^{-1}) \in \C$. Writing $f_{u}$ the unique solution in $\Ka$ of 
the equation $c z \sq f - f = u - \Bq u(c^{-1})$, we have
$Red(0,1,1,c,u) = \bigl(f_{u},\Bq u(c^{-1})\bigr)$. \\
Note that equation $c z \sq f - f = u$ always admits a unique formal
solution $\hat{f}_{u}$. For instance, for $u \in \C$, we find that
$\hat{f}_{u} = - u \tsh(c z)$, where $\tsh$ is the Tshakaloff series
\label{not38}
\index{Tshakaloff series}
defined by \eqref{eqn:Tshakaloff} page \pageref{eqn:Tshakaloff}. 
We infer that, for a general
$u \in \Ka$, one has $\hat{f}_{u} = f_{u} - \Bq u(c^{-1}) \tsh(c z)$.
\end{exem}

\begin{exem}[A direct computation]
In proposition \ref{prop:espaceadeuxpentes}, we meet the equation
$(\sq F_{1,2}) z^{\mu_{2}} A_{2} - z^{\mu_{1}} A_{1} F_{1,2} = U - V$. Here is
an explicit resolution that does not go through all the reductions of 
\ref{subsection:extensionsofqdifferencemodules}. We write 
$X = \sum X_{n} z^{n}$ for $F_{1,2}$ and $Y = \sum Y_{n} z^{n}$ for $U - V$. 
We have:
\begin{align*}
(\sq X) z^{\mu_{2}} A_{2} - z^{\mu_{1}} A_{1} X = Y & \Longleftrightarrow
\forall n ~,~ 
q^{n - \mu_{2}} X_{n - \mu_{2}} A_{2} - A_{1} X_{n - \mu_{1}} = Y_{n} \\
& \Longleftrightarrow
\forall n ~,~ q^{n - d} X_{n - d} A_{2} - A_1 X_{n} = 
Y_{n + \mu_{1}}, \\
& \Longleftrightarrow
\forall n ~,~ q^{n - d} A_{1}^{-1} X_{n - d} A_{2} - X_{n} = 
Z_{n} := A_{1}^{-1} Y_{n + \mu_{1}},
\end{align*}
where we denote $d := \mu_{2} - \mu_{1} \in \N^{*}$ the \emph{level} 
\index{level of an inhomogeneous equation} 
of the above inhomogeneous equation. (The second equivalence is obtained
by a mere translation of indices.) It follows from \cite{Bezivin},
\cite{RamisGrowth} that, for any analytic $Y$, this has a formal
solution of $g$-Gevrey level $d$ (this was defined in paragraph
\ref{subsection:qGevreylevels} of the general notations, in the
introduction); it also follows
from \emph{loc.cit} that, if there is a solution of $q$-Gevrey level
$d' > d$ (which is a stronger condition), then, there is an analytic 
solution. To solve our equation, we introduce a $q$-Borel transform 
of level $d$. 
\index{$q$-Borel transform of level $d$}
This depends on an arbitrary family $(t_{n})_{n \in \Z}$ 
in $\C^{*}$, subject to the condition:
$$
\forall n ~,~ t_{n} = q^{n-d} t_{n-d}.
$$
For $d = 1$, the natural choice is $t_{n} := q^{n(n-1)/2}$. In general,
such a family can be built from Jacobi theta functions as in \cite{RS1}.
At any rate, $t_{n}$ has the order of growth of $q^{n^{2}/2 d}$. The $q$-Borel
transform is then defined by:
$$
\label{not39}
\B_{q,d} \left(\sum f_{n} z^{n}\right) := \sum \dfrac{f_{n}}{t_{n}} \xi^{n}.
$$
We must also choose $d^{th}$ roots $B_{1},B_{2}$ of $A_{1},A_{2}$. Then,
writing $\B_{q,d} X = \sum \tilde{X}_{n} z^{n}$, etc, we find the relations:
$$
\forall n ~,~ 
B_{1}^{-d} \tilde{X}_{n-d} B_{2}^{d} - \tilde{X}_{n} = \tilde{Z}_{n} 
\Longleftrightarrow
\forall n ~,~ B_{1}^{n-d} \tilde{X}_{n-d} B_{2}^{-(n-d)} - 
B_{1}^{n} \tilde{X}_{n}  B_{2}^{-n} = B_{1}^{n} \tilde{Z}_{n} B_{2}^{-n}.
$$
Thus, a necessary condition for the existence of an analytic solution $X$
is that, for all $i$ in a set of representatives modulo $d$, one has:
$$
\sum_{n \equiv i \pmod{d}} B_{1}^{n} \tilde{Z}_{n} B_{2}^{-n} = 0.
$$
This provides us with $d$ obstructions in $\Mat_{r_{1},r_{2}}(\C)$ and it is
not hard to prove, along the same lines as what has been done, that these
form a complete set of invariants. More precisely, the map which sends
$U$ to the $d$-uple of matrices 
$\sum\limits_{n \equiv i \pmod{d}} \dfrac{1}{t_{n}} B_{1}^{n-d} U_{n + \mu_{1}} B_{2}^{-n}$
yields an isomorphism of $\F(P_{1},P_{2})$ with $\Mat_{r_{1},r_{2}}(\C)^{d}$.
Different choices of the family $(t_{n})$ and the matrices $B_{1},B_{2}$
induce different isomorphisms.
\end{exem}


\subsection{The Birkhoff-Guenther normal form}
\label{subsection:BGNF}
\index{Birkhoff-Guenther normal form}

Going from two slopes to general case rests on the following remark. 
The functor $M \leadsto M' := M_{\leq \mu_{k-1}}$ induces an onto mapping
$\F(P_{1},\ldots,P_{k}) \rightarrow \F(P_{1},\ldots,P_{k-1})$. The inverse
image in $\F(P_{1},\ldots,P_{k})$ of the class of $M'$ in 
$\F(P_{1},\ldots,P_{k-1})$ is in natural one-to-one correspondance with
the space $\Ext^{1}(P_{k},M')$. The latter is a $\C$-vector space of
dimension $\sum\limits_{1 \leq j < k} r_{j} r_{k} (\mu_{k} - \mu_{j})$. 
This follows from the slope filtration 
$0 = M_{0} \subset \cdots \subset M_{k-1} = M'$
and the resulting exact sequences, for $1 \leq j < k$:
$$
0 \rightarrow M_{j-1} \rightarrow M_{j} \rightarrow P_{j} \rightarrow 0
\Longrightarrow
$$
$$
0 \rightarrow \Ext^{1}(P_{k},M_{j-1}) \rightarrow \Ext^{1}(P_{k},M_{j}) 
\rightarrow \Ext^{1}(P_{k},P_{j}) \rightarrow 0.
$$
(Recall that $\Hom(P_{k},P_{j}) = 0$.) Actually, there is a non canonical
$\C$-linear isomorphism 
$\Ext^{1}(P_{k},M') \simeq \bigoplus_{1 \leq j < k} \Ext^{1}(P_{k},P_{j})$.
One can go further using the matricial description of paragraph
\ref{subsection:matricialdescriptionofF(P1,...,Pk)}.
We keep the notations recalled at the beginning of subsection
\ref{subsection:twointegralslopes} and moreover denote 
\label{not40}
$\overline{M_{U}}$ the class in $\F(P_{1},\ldots,P_{k})$ of the module 
$M_{U} := (K^{n},\Phi_{A_{U}})$ for 
$U := (U_{i,j})_{1 \leq i < j \leq k} \in 
\prod\limits_{1 \leq i < j \leq k} \Mat_{r_{i},r_{j}}(\Ka)$.

\begin{prop}
\label{prop:kintegralslopesspace}
The map $U \mapsto \overline{M_{U}}$ from 
$\prod\limits_{1 \leq i < j \leq k} \Mat_{r_{i},r_{j}}(K_{\mu_{i},\mu_{j} })$ 
to $\F(P_{1},\ldots,P_{k})$ is one-to-one.
\end{prop}
\begin{proof}
We already now that the map is onto and the conclusion will follow from
the following fact: for all 
$U \in \prod\limits_{1 \leq i < j \leq k} \Mat_{r_{i},r_{j}}(\Ka)$,
there exists a unique pair 
$(F,V) \in \G(\Ka) \times 
\prod\limits_{1 \leq i < j \leq k} \Mat_{r_{i},r_{j}}(K_{\mu_{i},\mu_{j}})$
such that $F[A_{V}] = A_{U}$. Writing $F$ as in equation (\ref{eqn:isoplat})
and $V = (V_{i,j})$, this is equivalent to the following system:
$$
\forall (i,j), 1 \leq i < j \leq k ~,~
V_{i,j} + \sum_{\ell = i+1}^{j-1} (\sq F_{i,\ell}) V_{\ell,j} + 
(\sq F_{i,j}) (z^{\mu_{j}} A_{j}) =
(z^{\mu_{i}} A_{i}) F_{i,j} + \sum_{\ell = i+1}^{j-1} U_{i,\ell} F_{\ell,j} + U_{i,j}.
$$
It is understood that an empty sum vanishes: if $(t_{\ell})$ is any sequence,
$\sum\limits_{\ell = i+1}^{j-1} t_{\ell} = 0$ when $j = i+1$.
Then, $U$ being given, the system is (uniquely) solved by induction
on $j - i$ by the following formulas:
$$
(F_{i,j},V_{i,j}) := Red\left(\mu_{i},A_{i},\mu_{j},A_{j},U_{i,j} +
\sum_{\ell = i+1}^{j-1} U_{i,\ell} F_{\ell,j} - 
\sum_{\ell = i+1}^{j-1} (\sq F_{i,\ell}) V_{\ell,j}\right).
$$
\end{proof}

The following is a particular case of theorem \ref{theo:schemadesmodules},
page \pageref{theo:schemadesmodules} (itself a particular case of corollary 
\ref{coro:themodulispaceabstract}, page \pageref{coro:themodulispaceabstract}):
the only novelty being that, in the case of integral slopes, we obtain an
explicit coordinate system.

\begin{theo}
\label{theo:kintegralslopesspace}
\index{affine space}
The set $\F(P_{1},\ldots,P_{k})$ is an affine space of dimension
$\sum\limits_{1 \leq i < j \leq k} r_{i} r_{j} (\mu_{j} - \mu_{i})$.
\end{theo}
\begin{proof}
This is an immediate consequence of the proposition.
\end{proof}

Note that this is the area contained in the lower finite part of
the Newton polygon. This result is the natural continuation of a 
normalisation process 
\index{normalisation process}
found by Birkhoff and Guenther in
\cite{Birkhoff3}.

\begin{defi}
\label{defi:BirkhoffGuenthernormalform}
A matrix $A_{U}$ such that $U := (U_{i,j})_{1 \leq i < j \leq k} \in 
\prod\limits_{1 \leq i < j \leq k} \Mat_{r_{i},r_{j}}(K_{\mu_{i},\mu_{j} })$ 
will be said to be in \emph{Birkhoff-Guenther normal form}.
\index{Birkhoff-Guenther normal form}
\end{defi}

Now, using the argument that follows theorem \ref{theo:schemadesmodules},
we draw from corollary \ref{coro:themodulispaceabstract}:

\begin{coro}
\label{coro:schemadesmodules}
The components of the matrices $(U_{i,j})_{1 \leq i < j \leq k}$ make up
a complete system of coordinates on the space $\F(P_{1},\ldots,P_{k})$.
\index{coordinates on $\F(P_{1},\ldots,P_{k})$}
\end{coro}

A generalisation of Birkhoff-Guenther normal form to the case of 
arbitrary slopes has recently been obtained by Virginie Bugeaud \cite{VB}.


\subsection{Computational consequences of the Birkhoff-Guenther normal form}
\label{subsection:computnormalform}

Let $A_{0}$ be as in \eqref{eqn:formestandarddiagonaleentiere}
and $A_{U} = A$ as in \eqref{eqn:formestandardtriangulaireentiere}.
Looking for $\hat{F} \in \G(\Kf)$ such that $\hat{F}[A_{0}] = A$
amounts, with notation \eqref{eqn:isoplat}, to solving the system:
$$
\forall (i,j), 1 \leq i < j \leq k ~,~
(\sq F_{i,j}) (z^{\mu_{j}} A_{j}) =
(z^{\mu_{i}} A_{i}) F_{i,j} + \sum_{\ell = i+1}^{j-1} U_{i,\ell} F_{\ell,j} + U_{i,j}.
$$
This is triangular in the sense that the equation in $F_{i,j}$
depends on previously found $F_{\ell,j}$ with $\ell > i$. \\

Assume that $A$ is in Birkhoff-Guenther normal form (definition
\ref{defi:BirkhoffGuenthernormalform}). Then we can write
$U_{i,j} = z^{\mu_{i}} U'_{i,j}$, where $U'_{i,j}$ has coefficients
in $\C[z]$. The previous equation becomes:
$$
(\sq F_{i,j}) (z^{\mu_{j} - \mu_{i}} A_{j}) =
A_{i} F_{i,j} + \sum_{\ell = i+1}^{j-1} U'_{i,\ell} F_{\ell,j} + U'_{i,j}.
$$
Assume by induction that all $F_{\ell,j}$ have coefficients
in $\Rf$. Then one may write the above equation as:
$$
F_{i,j} = U'' + z^{\delta} A_{i}^{-1} (\sq F_{i,j}) A_{j},
$$
where $U''$ has coefficients in $\Rf$ and $\delta \in \N^{*}$.
This is a fixpoint equation for an operator that is contracting
in the $z$-adic topology of $\Rf$ and so it admits a unique
formal solution. Therefore we find a unique $\hat{F} \in \G(\Rf)$ 
such that $\hat{F}[A_{0}] = A$. \\

A noteworthy consequence is that the inclusions $\Ra \subset \Ka$
and $\Rf \subset \Kf$ induce a natural identification:
$$
\G^{A_{0}}(\Rf)/\G(\Ra) \simeq \G^{A_{0}}(\Kf)/\G(\Ka),
$$
and, as a corollary, a bijection:
$$
\G^{A_{0}}(\Rf)/\G(\Ra) \rightarrow \F(P_{1},\ldots,P_{k}).
$$


\section{Interpolation by $q$-Gevrey classes}
\label{section:qGevreyinterpolation}

We shall now extend the previous results to $q$-Gevrey classification.
We use here notations from paragraph \ref{subsection:qGevreylevels} 
of the general notations in the introduction.


\subsection{$q$-Gevrey extension spaces}
\label{subsection:qGevreyExt}
\index{$q$-Gevrey extension spaces}

\subsubsection{$q$-Gevrey extensions for arbitrary slopes}

If one repaces the field $\Ka$ by the field $\Kfs$ of $q$-Gevrey series 
of level $s >0$ and $\Dq$ by $\Dqs := \Kfs \langle T,T^{-1} \rangle$, 
one gets the abelian category $\DMfs$ and the following extension 
of the results of \ref{subsection:extensionsofqdifferencemodules}.

\begin{prop}
\label{prop:dimensiondeuxpentesqGevrey}
Let $M,N$ be pure modules of ranks $r,s \in \N^{*}$ and slopes
$\mu < \nu \in \Q$ in $\DMfs$. Then one has:
$$
\dim_{\C} \Ext^{1}(M,N) = \begin{cases}
0 \text{~if~} \nu - \mu \geq 1/s, \\
r s (\nu - \mu) \text{~if~} \nu - \mu < 1/s.
\end{cases}
$$
\end{prop}
\begin{proof}
This follows indeed, by the same arguments as before, from
propositions 3.2 and 3.3 of \cite{Bezivin}.
\end{proof}

This remains true in the extreme case that $s = \infty$, since, over
the field $\Kf$, the slope filtration splits and $\Ext^{1}(M,N) = 0$;
and also in the extreme case that $s = 0$, by theorem
\ref{theo:dimensiondeuxpentes}. Therefore, the above proposition
is an interpolation between the analytic and formal settings.

\subsubsection{$q$-Gevrey extensions for integral slopes}

Here, we extend the results of \ref{subsection:twointegralslopes}.
If one replaces $\Ka$ by $\Kf$, one gets $V = 0$ in all relations
$(F,V) = Red(\mu_{1},A_{1},\mu_{2},A_{2},U)$ and the corresponding space
of classes has dimension $0$. The previous algorithm then produces
$F = \hat{F}_{U}$ and the ``formal normal form'' $A_{0}$ of $A$.
If one replaces $\Ka$ by $\Kfs$, one gets $V = 0$ and a matrix $F$
of $g$-Gevrey level $\mu_{2} - \mu_{1}$ if $\mu_{2} - \mu_{1} \geq 1/s$,
and the same $Red(\mu_{1},A_{1},\mu_{2},A_{2},U)$ as before otherwise.


\subsection{Isoformal $q$-Gevrey  classification}
\label{subsection:qGevreymoduli}
\index{$q$-Gevrey classification}

\subsubsection{$q$-Gevrey classification for arbitrary slopes}

One can state the moduli problem for $q$-Gevrey classification 
at order $s$, that is, over the field $\Kfs$ of $q$-Gevrey series 
of order $s >0$ (see subsection \ref{subsection:qGevreyExt}). 
From proposition \ref{prop:dimensiondeuxpentesqGevrey}, one gets, 
by the same argument as before:

\begin{prop}
\label{prop:schemadesmodulesqGevrey}
Over $\Kfs$, the functor $F$ is representable and the corresponding 
affine scheme is an affine space over $\C$ with dimension
$\sum\limits_{1 \leq i < j \leq k \atop \mu_{j} - \mu_{i} < 1/s} 
r_{i} r_{j} (\mu_{j} - \mu_{i})$.
\end{prop}

This remains true for $s = \infty$ (formal setting, the space of
moduli is a point) and for $s = 0$ (analytic setting, this is 
theorem \ref{theo:schemadesmodules}).

\subsubsection{$q$-Gevrey classification for integral slopes}

We fix $s > 0$ and write for short $q$-Gevrey for $q$-Gevrey of 
order $s$. Every matrix is $q$-Gevrey equivalent to a matrix $A_{U}$ 
such that $U_{i,j} = 0$ for $\mu_{j} - \mu_{i} \geq 1/s$. The slopes 
being assumed to be integral, there is moreover a unique normal form 
with $U_{i,j} = 0$ for $\mu_{j} - \mu_{i} \geq 1/s$ 
and $U_{i,j} \in \Mat_{r_{i},r_{j}}(K_{\mu_{i},\mu_{j} })$ for 
$\mu_{j} - \mu_{i} < 1/s$.


\subsection{Another kind of $q$-Gevrey interpolation}
\index{$q$-Gevrey interpolation}

A somewhat symmetric problem is to describe the space of analytic
classes within a fixed $q$-Gevrey class. This can be done in a similar
way. We fix a matrix $A_{U_{0}}$ where 
$U_{0} \in 
\prod\limits_{1 \leq i < j \leq k}\Mat_{r_{i},r_{j}}(K_{\mu_{i},\mu_{j} })$
is such that $U_{i,j} = 0$ for $\mu_{j} - \mu_{i} \geq 1/s$
(any $q$-Gevrey class contains such a matrix). This characterizes
a well defined $q$-Gevrey class and the space of analytic classes
within this $q$-Gevrey class is an affine space of dimension
$\sum\limits_{1 \leq i < j \leq k \atop \mu_{j} - \mu_{i} \geq 1/s} 
r_{i} r_{j} (\mu_{j} - \mu_{i})$. 
If the slopes are integral, one can assume that $U_{0}$ is in normal
form, \ie\ each component $U_{i,j}$ such that $\mu_{j} - \mu_{i} < 1/s$
belongs to $\Mat_{r_{i},r_{j}}(K_{\mu_{i},\mu_{j} })$. Then each analytic class
admits a unique normal form $A_{U}$ where $U$ is in normal form
and its components such that $\mu_{j} - \mu_{i} < 1/s$ are the same
as those of $U_{0}$.



\chapter{The $q$-analogs of Birkhoff-Malgrange-Sibuya theorems}
\label{chapter:qMalgrangeSibuya}


\section{Asymptotics}
\label{section:asymptotics}

We shall need a $q$-analogue of asymptotics \emph{in the sense
of Poincar\'e}. 
\index{$q$-analogue of Poincar\'e asymptotics}
In chapter \ref{chapter:asymptotictheory} we
shall develop a more restrictive notion of asymptotics. \\

The underlying idea, coming from the thesis \cite{JLM}
of Jose-Luis Martins, is that an asymptotic theory is
related to a dynamical system. In the ``classical'' case
of Poincar\'e asymptotics for ordinary differential equations,
say locally at $0$ in $\C^{*}$, the dynamical system is given
by the action of the semi-group $\Sigma := e^{]- \infty,0[}$;
\index{semi-group}
in our case, the semi-group is plainly $\Sigma := q^{- \N}$ (in
the case of difference equations there are two semi-groups:
$\Sigma :=\N$ and $\Sigma :=-\N$; this was used by J. Roques
in \cite{Ro}).
Likewise, the sheaves of functions admitting an asymptotic
expansion will be defined on the \emph{horizon}, 
\index{horizon}
that is on the quotient space $\C^{*}/\Sigma$. In the classical case,
this is the circle $S^{1}$ of directions (rays from $0$)
in $\C^{*}$; in our case, this is the elliptic curve
\index{elliptic curve $\Eq$}
$\Eq = \C^{*}/q^{\Z} = \C^{*}/q^{-\N}$ (in the difference case
this is a pair of cylinders).


\subsection{$q$-Asymptotics}
\label{subsection:qAsymptotics}
\index{$q$-asymptotics}

Recall from the introduction that we denote $p: \C^* \rightarrow \Eq$
the quotient map. We set 
$$
\label{not41}
\Sigma := q^{- \N} \; \text{~and, for~} A \subset \C, \;
\Sigma (A) := \bigcup\limits_{a\in A}  q^{- \N}\{ a\}.
$$
Let $U$ be an open set of $\C$ invariant by the semi-group
\index{stable open set}
\index{invariant open set}
$\Sigma := q^{- \N}$, that is, $q^{-1} U \subset U$ (we shall also call 
it ``stable''). We shall say that a $\Sigma$-invariant subset $K$ of 
$U$ is a {\it stable strict subset} 
\index{stable strict subset}
\index{strict subset (stable)}
if there exists a compact subset 
$K'$ of $U$ such that $K = \Sigma (K')$; then $K \cup \{ 0\}$ is compact 
in $\C$. An invariant subset $K \subset U$ is strict if and only if $p(K)$ 
is a compact subset of $\Eq$. \\

Let $f$ be a function holomorphic on the $\Sigma$-invariant open subset
$U$ of $\C^*$ and let $\hat f =\sum\limits_{n\geq 0} a_nz^n \in \Rf$. We 
shall say that $f$ is \emph{$q$-asymptotic} to $\hat f$ on $U$ if, for 
\index{$q$-asymptotics}
all $n \in \N$ and every $\Sigma$-invariant strict subset $K \subset U$:
$$
z^{-n}\bigl(f(x)-S_{n-1} \hat{f}\bigr)
$$
\label{not42}
is bounded on $K$. Here, $S_{n-1} \hat{f} :=\sum\limits_{p=0}^{n-1} a_p z^p$ 
stands for "truncation to $n$ terms" of $\hat f$. If $f$ admits a
$q$-asymptotic expansion it is plainly unique. In the following
we will say for simplicity ``asymptotic expansion" for ``$q$-asymptotic
expansion". For $n \in \N$ we shall denote $f^{(n)}(0) := n! \, a_n$. \\
\index{$q$-asymptotic expansion}
\index{asymptotic expansion}

We denote:
\label{not43}
\begin{itemize}
\item{$\Aa(U)$ the space of holomorphic functions admitting a $q$-asymptotic 
expansion on $U$,}
\item{$\Aa_{0}(U)$ the subspace of $\Aa(U)$ made up of (infinitely) flat
\index{flat functions (infinitely)}
functions, \ie\ those such that $\hat{f} = 0$ and}
\item{$\Bb(U)$ the space of holomorphic functions bounded on every stable 
strict subset of $U$.}
\end{itemize}
If $f$ is holomorphic on $U$, then $f \in \Aa(U)$ if and only if
$z^{-n} \bigl(f(z) -S_{n-1} \hat{f}(z)\bigr) \in \Bb(U)$ for all 
$n \in \N$. We check easily that $\Aa(U)$ is a $\C$-algebra and 
that $\Aa_{0}(U)$ is an ideal of $\Aa(U)$. \\

Let again $U$ be a $\Sigma$-invariant open subset of $\C^*$ and let
$U_{\infty} := p(U)$; we think of $U_{\infty}$ as the \emph{horizon} of $U$.
\index{horizon} \index{elliptic curve $\Eq$} Every open subset of $\Eq$ 
is the horizon of some $\Sigma$-invariant open subset of $\C^*$ (actually, 
of a great many of them~!) and, from now on, we shall often write 
$U_{\infty}$ an arbitrary open subset of $\Eq$ and $U$ some adequate 
$\Sigma$-invariant open subset of $\C^*$ such that $U_{\infty} := p(U)$. \\

We introduce three sheaves on $\Eq$, also denoted $\Aa$, $\Aa_{0}$ 
and $\Bb$; 
\index{sheaves on $\Eq$ of asymptotic functions}
no confusion should be caused by this overloading of notations.
The reader is invited to check that these are indeed sheaves:
\begin{itemize}
\label{not44}
\item{$\Aa(U_{\infty})$ is the direct limit of all $\Aa(U)$,}
\item{$\Aa_{0}(U_{\infty})$ is the direct limit of all $\Aa_{0}(U)$ and}
\item{$\Bb(U_{\infty})$ is the direct limit of all $\Bb(U)$.}
\end{itemize}
In each case the limit is taken for all $\Sigma$-invariant open subset $U$
of $\C^*$ such that $U_{\infty} := p(U)$. Then the $\Aa(U_{\infty})$ form 
a sheaf of $\C$-algebras and the $\Aa_{0}(U_{\infty})$ a sheaf of ideals.

\begin{lemm}
\label{lemm:asderive}
Let $U \subset \C$ be a  $\Sigma$-invariant open set and $K \subset V$
be a $\Sigma$-invariant strict subset of  $U$. There exists a real
number $\rho > 0$ such that the closed disc of radius
$\rho \lmod z \rmod$ and center $z$ lies in $U$ for all $z \in K$.
Moreover it is possible to choose $\rho >0$ such that
 $\displaystyle\bigcup_{z\in K} \bar D(z,\rho\lmod z\rmod)$
be a $\Sigma$-invariant strict subset of $U$.
\end{lemm}
\begin{proof}
There exists by definition a subset $K'$ of $K$, compact in $U$,
such that $K = \Sigma (K')$. By compacity the result of the lemma
is true for every $z\in K'$; using the action of $\Sigma$ we get
that it is also true for all $z\in K$.
\end{proof}

Let $f \in \Aa(U_{\infty})$ and $\hat f := \sum\limits_{n\geq 0} a_n z^n$
its asymptotic expansion. We set
$R_n := z^{-n}\bigl(f-S_{n-1} \hat{f}\bigr)$.
We have $\displaystyle f=\sum_{p=0}^n a_pz^p+z^{n+1}R_{n+1}$; therefore
the derivative $f'$ is given by:
$$
f' = \sum_{p=1}^n p a_p z^{p-1} + (n+1) z^n R_{n+1} + z^{n+1} R'_{n+1}.
$$
If $K\subset U$ is a $\Sigma$-invariant strict subset, then $R_{n+1}$ 
is bounded by $M_{K,n+1}$ on $K$. Therefore, using Cauchy formula for 
the derivative and lemma \ref{lemm:asderive}, $R'_n$ is bounded by some 
positive constant $\frac{M_{K,n+1}}{\rho\lmod z\rmod}$ on $K$. \\

We have $\displaystyle f'=\sum_{n\geq 1} na_nz^{n-1}$ and:
$$
z^{-n}\bigl(f'(x)-S_{n-1} \hat{f'}\bigr) =
z^{-n}\left(f'(x)- \sum_{p=1}^n n a_p z^{p-1}\right) =
(n+1)R_{n+1} + z R'_{n+1}.
$$
Since
$\lmod (n+1)R_{n+1}+zR'_{n+1} \rmod \leq
\bigl(n + 1 + \frac{1}{\rho}\bigr)M_{K,n+1}$, then
$z^{-n} \bigl(f'(x)-S_{n-1} \hat{f'}\bigr)$ is bounded on $K$,
so that $f'$ admits $\hat f'$ as a $q$-asymptotic expansion on $U$.
Hence $\Aa(V_{\infty})$ is a differential algebra and
$\Aa_{0}(V_{\infty})$ is a differential ideal of $\Aa(V_{\infty})$.
\index{differential (algebra, ideal)}


\subsection{Asymptotic expansions and ${\mathcal C}^{\infty}$
functions in Whitney sense}
\label{subsubtion:asymptoticsandWhitney}

\begin{lemm}
\label{lemm:ascle}
Let $U$ be a $\Sigma$-invariant open set of $\C$ and
$f \in \Aa(U_{\infty})$. Let $K\subset U$ be an invariant
strict subset.
\begin{itemize}
\item[(i)]
{There exists a positive constant $C_K$ (depending on $f$)
such that, for all $z_1,z_2 \in K \cup \{ 0\}$:
$$
\lmod  f(z_1)-f(z_2) \rmod \leq C_K \lmod z_1 - z_2 \rmod.
$$}
\item[(ii)]
{For all $n\in\N$ there exists a positive constant $C_{K,n}$
(depending on $f$) such that, for all $z_1,z_2\in K\cup\{ 0\}$:
$$
\lmod  f(z_1) -\sum_{p=0}^n \frac{f^{(p)}(z_2)}{p!} (z_1-z_2)^p \rmod
\leq C_{K,n} \lmod z_1 - z_2\rmod^{n+1}.
$$}
\item[(iii)]
{$$
\lim_{z_1 \rightarrow 0, z_2 \rightarrow 0
\atop ~z_1,z_2 \in K\cup\{ 0\},z_1 \neq z_2}
\dfrac{1}{\lmod z_1 - z_2 \rmod^n}
\left(f(z_1)-\sum_{p=0}^n \frac{f^{(p)}(z_2)}{p!} (z_1-z_2)^p \right) = 0.
$$}
\end{itemize}
\end{lemm}
\begin{proof}
(i) We choose $\rho >0$ (depending on $K$) as in lemma
\ref{lemm:asderive}; then
$K_1:=\displaystyle\bigcup_{z\in K} \bar D(z,\rho\lmod z\rmod)$
is an invariant strict subset of $U$. We set:
$$
M_K := \sup_{z\in K} \frac{\lmod f(z)-f(0)\rmod}{\lmod z\rmod}~,\quad
M'_{K_1} := \sup_{z\in K_1} \lmod f'(z)\rmod~.
$$

Let $z_1,z_2\in K\cup \{ 0\}$. We shall consider separately three cases:
(a) $z_2=0$;
(b) $z_2\neq 0$ and $\lmod z_1 - z_2 \rmod \leq \rho \lmod z_2 \rmod$;
(c) $z_2\neq 0$ and $\lmod z_1 - z_2 \rmod > \rho \lmod z_2 \rmod$.

The cases (a) and (b) are easy. In fact, if $z_2=0$, then:
$$
\lmod f(z_1) - f(z_2) \rmod = \lmod f(z_1)-  f(0) \rmod \leq
M_K \lmod z_1 \rmod =M_K \lmod z_1 - z_2\rmod.
$$
If $z_2 \neq 0$ and $\lmod z_1 - z_2 \rmod \leq \rho \lmod z_2 \rmod$,
then $z_2 \in K$ and the closed interval $[z_1,z_2]$ is contained in
$K_1$; therefore:
$$
\lmod  f(z_1) - f(z_2) \rmod \leq M'_{K_1} \lmod z_1 - z_2\rmod.
$$

Now consider case (c). If  $z_2 \in K$ and
$\lmod z_1 - z_2 \rmod > \rho \lmod z_2 \rmod$, then we write
$f(z_1) - f(z_2) = f(z_1) - f(0) -\bigl(f(z_2)-f(0)\bigr)$ and we have:
$$
\lmod f(z_1) - f(0) \rmod \leq M_K \lmod z_1 \rmod \leq
M_K\bigl(\lmod z_2 \rmod + \lmod z_1 - z_2\rmod\bigr)
\leq \left(\frac{1}{\rho}+1 \right) M_K \lmod z_1 - z_2 \rmod
$$
and
$$
\lmod f(z_2) - f(0) \rmod \leq M_K \lmod z_2 \rmod \leq
\frac{1}{\rho} M_K \lmod z_1 - z_2\rmod,
$$
whence:
$$
\lmod  f(z_1) - f(z_2) \rmod \leq
\frac{2+\rho}{\rho} M_K \lmod z_1 - z_2\rmod.
$$

Last, let $C_K := \max \left(M'_{K_1}, \frac{2+\rho}{\rho}M_K \right)$;
then, for all $z_1,z_2 \in K \cup \{0\}$, we have:
$$
\lmod  f(z_1) - f(z_2)\rmod \leq C_K \lmod z_1 - z_2\rmod.
$$

(ii) We define $\rho$, $K_1$, $C_K$ as in (i). We set:
$$
C_{K,n} := \sup_{z\in K}
\lmod  f(z) - \sum_{p=0}^n \frac{f^{(p)}(0)}{p!} z^p \rmod \lmod z^{-n-1}\rmod~,
\quad
M'_{K_1,n} := \sup_{z\in K_1} \lmod f^{(n+1)}(z) \rmod~.
$$
The result is obviously true if $f$ is a polynomial. Therefore,
fixing $n \in \N$ it suffices to prove the result when $f = z^{n+1}g$,
$g \in \Aa(U_{\infty})$.

As in (i) we shall consider separately the three cases (a), (b), (c).
If $z_2 = 0$, then the result merely comes from the following relation:
$$
\lmod  f(z_1) - \sum_{p=0}^n \frac{f^{(p)}(z_2)}{p!} (z_1-z_2)^p \rmod =
\lmod  f(z_1) - \sum_{p=0}^n \frac{f^{(p)}(0)}{p!} z_1^p \rmod \leq
C_{K,n} \lmod z_1\rmod^{n+1}.
$$

Next, if $z_2 \neq 0$ and
$\lmod z_1 - z_2\rmod \leq \rho \lmod z_2 \rmod$, then the closed
interval $[z_1,z_2]$ is contained in $K_1$, so that:
$$
\lmod  f(z_1) - \sum_{p=0}^n \frac{f^{(p)}(z_2)}{p!} (z_1-z_2)^p \rmod  \leq
M'_{K_1,n} \lmod z_1-z_2\rmod^{n+1}.
$$

It remains to deal with the case (c), that is, when $z_2 \neq 0$
and $\lmod z_1 -z_2\rmod > \rho \lmod z_2 \rmod$. If we set:
$$
\lambda_{K,n} :=
\sup_{z\in K \atop p=0,\ldots,n} \lmod \frac{g^{(p)}(z)}{p!}\rmod~,
\quad
\mu_K := \sup_{z\in K} \lmod z \rmod,
$$
then, for any integer $p \in [0,n]$, we find:
$$
\lmod \frac{ f^{(p)}(z_2)}{p!} \rmod =
\lmod \frac{(z^{n+1}g)^{(p)}(z_2)}{p!} \rmod \leq
\lambda_{K,n}~\sum_{k=0}^p \binom{n+1}{k} \lmod z_2 \rmod^{n+1-k} \leq
\tilde{\lambda}_{K,n}~ \lmod z_2 - z_1\rmod^{n+1-p},
$$
where:
$$
\tilde{\lambda}_{K,n} :=
\lambda_{K,n}~\rho^{p-n-1}~\sum_{k=0}^p\binom {n+1}k \mu_K^{p-k}~.
$$
From this, it follows:
\begin{align*}
\lmod f(z_1) - \sum_{p=0}^n \frac{f^{(p)}(z_2)}{p!} (z_1-z_2)^p\rmod
\leq& ~
\lmod z_1^{n+1} g(z_1) \rmod +
(n+1) \tilde{\lambda}_{K,n}~ \lmod z_2-z_1\rmod^{n+1} \\
\leq&
\left(\left(1+\frac{1}{\rho}~\right)^{n+1}~
\lambda_{K,n} +
(n+1) \tilde{\lambda}_{K,n}\right)~\vert z_2-z_1\vert^{n+1}~.
\end{align*}

To summarize, in the three cases cases we have:
$$
\lmod f(z_1) -\sum_{p=0}^n \frac{f^{(p)}(z_2)}{p!} (z_1-z_2)^p  \rmod \leq ~
C_{K,n} \lmod z_1-z_2 \rmod^{n+1},
$$
where
$$
C_{K,n} :=
\max \left(M'_{K_1,n},
\left(1 + \frac{1}{\rho}~\right)^{n+1}~\lambda_{K,n} +
(n+1) \tilde{\lambda}_{K,n}
\right)~.
$$

(iii) The result follows immediately from (ii).
\end{proof}

Before stating the next proposition, we first recall Whitney conditions
\index{Whitney conditions}
\index{Whitney theorem}
and Whitney theorem \cite[chap. I, \S 4]{Ma1}, \cite{Ma2}. We give them
here in the case of dimension $2$. Let $f$ be a function of $(x,y)$ with 
complex values, $\mathcal{C}^{\infty}$ in an open neighborhood $U'$ of some 
compact subset $K'$ of $\C$. (We keep the letters $U$ and $K$ to denote 
respectively a $\Sigma$-invariant open subset of $\C$ and an invariant 
strict subset of $U$.) We shall use the following standard conventions; 
for $k := (k_1,k_2) \in \N^2$, we set:
$$
f^{(k)} := \dfrac{\partial^{k_1}}{\partial x^{k_1}}
\dfrac{\partial^{k_2}}{\partial y^{k_2}} \,f \cdot
$$ 
We also write $k ! := k_1 ! k_2 !$, $(x,y)^k := x^{k_1} y^{k_2}$, 
$\lmod k \rmod := k_1 + k_2$, etc. We associate to $f$ the family
$(f^{(k)})_{k \in \N^2}$ of its ``Taylor fields" \emph{restricted to $K'$}. \\

We agree to identify the $\mathcal{C}^{\infty}$ function $f$ on $U' \supset K'$ 
with the $\mathcal{C}^{\infty}$ function $g$ on $V' \supset K'$ if their
associated families $(f^{(k)})_{k \in \N^2}$, $(g^{(k)})_{k \in \N^2}$ are equal.
Such a class of functions is called a \emph{$\mathcal{C}^{\infty}$ function 
in the Whitney sense} on $K'$. 
\index{$\mathcal{C}^{\infty}$ function in the Whitney sense}
\index{Whitney ($\mathcal{C}^{\infty}$ function in the - sense)}
\label{not45}
We write $\mathcal{C}^{\infty}_{Whitney}(K')$
the space of $\mathcal{C}^{\infty}$ functions\footnote{They are not really 
functions, but it usually causes no problem; moreover, in our case, they
will be (according to the next proposition).} in the Whitney sense on $K'$.
As is customary when dealing with germs, we sometimes write $f$ the class 
of $f$; likewise, we write $(f^{(k)})$ the family of restrictions to $K'$ 
of the derivatives $f^{(k)}$: by definition, they depend only on the class 
of $f$. \\

So let $f$ be a $\mathcal{C}^{\infty}$ function in the Whitney sense on $K'$.
The functions $f^{(k)}$ are continuous on $K'$ and for all $n \in \N$, the 
family $\bigl(f^{(k)}\bigr)$ (restricted to $K'$) satisfies the following 
conditions, that are called \emph{Whitney conditions of order $n \in \N$}
\index{Whitney conditions}
\label{not46}
and that we shall denote $W_n(K')$:
\begin{equation*}
\lim~\dfrac{
\left(f^{(k)}(x_1,y_1) -
\sum_{\lmod h \rmod \leq n} \frac{f^{(k+h)}(x_2,y_2)}{h!}
(x_1 - x_2,y_1 - y_2)^h \right)
}
{\lnorm (x_1-x_2,y_1 - y_2) \rnorm^{n - \lmod k\rmod}}
 = 0,
\end{equation*}
where the limit is taken for $(x_1,y_1) \to (x_0,y_0) \in K'$, 
$(x_2,y_2) \to (x_0,y_0) \in K'$, with $(x_1,y_1),(x_2,y_2) \in K'$ 
and $(x_1,y_1) \neq (x_2,y_2)$. \\

Conversely a family $\bigl(f^{(k)}\bigr)_{k\in\N^2}$ of continuous
functions on $K'$ is the family of Taylor jets
\index{Taylor jets}
of an element of $\mathcal{C}^{\infty}_{Whitney}(K')$ if and
only if it satisfies Whitney conditions $W_n(K')$ for all $n \in \N$
(this is Whitney's theorem, see \cite{Ma1,Ma2}). Note that, with the usual 
identification of $\C$ with $\R^2$, setting $z := x + \ii y$ and
using $\frac{\partial}{\partial z}$,
$\frac{\partial}{\partial \bar z}$ instead of
$\frac{\partial}{\partial x}$, $\frac{\partial}{\partial y}$,
for all $n \in \N$ we can replace Whitney conditions $W_n(K')$
\label{not47}
by equivalent conditions $\tilde W_n(K')$, that we do not state
more explicitly here. In our case of interest (functions satisfying
the Cauchy-Riemann equation $\partial f/\partial \overline{z} = 0$),
they take the form of condition (iii) of the lemma above.

\begin{prop}
\label{prop:asymptotiqueetWhitney}
Let $U$ be a $\Sigma$-invariant open set of $\C$. Then,
$f \in \Aa(U_{\infty})$ if, and only if, for every invariant
strict subset $K \subset U$, one has:
$$
f_{\vert K}\in \Ker ~ \left(\frac{\partial}{\partial \bar z}:
\mathcal{C}^{\infty}_{Whitney} \bigl({K\cup\{ 0\}},\C\bigr)
\rightarrow \mathcal{C}^{\infty}_{Whitney}\bigl({K\cup\{ 0\}},\C\bigr)\right).
$$
\end{prop}
\begin{proof}
Let $f\in \mathcal{C}^{\infty}_{Whitney}\bigl({K\cup\{ 0\}},\C\bigr)$.
If for all invariant strict subset $K\subset U$, one has:
$$
f_{\vert K}\in \Ker ~
\left(\frac{\partial}{\partial \bar z}:
\mathcal{C}^{\infty}_{Whitney}\bigl({K\cup\{ 0\}},\C\bigr)
\rightarrow
\mathcal{C}^{\infty}_{Whitney}\bigl({K\cup\{ 0\}},\C\bigr)\right),
$$
then, the function $f$ is holomorphic on $V$, its Taylor jet
$\hat f$ at $0$ is formally holomorphic (\ie\ it satisfies the
Cauchy-Riemann equation),
$\hat f = \sum\limits_{n\geq 0} a_n z^n$
and $f$ is asymptotic to $\hat f$ on $V$, $f\in\Aa(U_{\infty})$.

Conversely let $f\in \Aa(U_{\infty})$; then, for all
$p\in\N$, $f^{(p)} \in \Aa(U_{\infty})$ and for
all $r \in \N^*$, we have $ \frac{\partial^r}{\partial \bar {z}^r}f^{(p)}=0$.

We can apply lemma \ref{lemm:ascle} to each function
$f^{(p)}$, $p\in\N$, to the effect that $f$ satisfies
conditions $\tilde W_n(K)$ for all $n\in\N$ (this is assertion (iii)
of the lemma); therefore it satisfies also Whitney conditions $W_n(K)$
for all $n\in\N$ and, using Whitney theorem, we get
$f\in\mathcal{C}^{\infty}_{Whitney}\bigl({K\cup\{ 0\}},\C\bigr)$.
\end{proof}


\subsection{Asymptotics and sheaves on $\Eq$}
\label{subsubsection:asymptoticsandsheaves}

Recall from paragraph \ref{subsection:qAsymptotics} the definition 
of the sheaves $\Aa$ and $\Aa_{0}$ on $\Eq$. There is a natural map 
from $\Aa$ to the constant sheaf $\Rf$ sending a function to its 
asymptotic expansion. Although much weaker than the $q$-Gevrey theorem 
of Borel-Ritt \ref{theo:BorelRitt}, the following result does not directly
flow from it, and we give here a direct proof.
\index{sheaves on $\Eq$ of asymptotic functions}
\index{Borel-Ritt theorem (weak $q$-analogue)}

\begin{theo}[Weak $q$-analogue of Borel-Ritt]
\label{theo:BorelRittfaible}
The natural map from the sheaf $\Aa$ to the constant sheaf $\Rf$ is onto.
\end{theo}
\begin{proof}
Starting from an open set $V \subset \Eq$ small enough that
$U := p^{-1}(V) \subset \C^{*}$ is a disjoint union of open
sets $U_{n} := q^{n} U_{0}$, where $U_{0}$ is mapped homeomorphically
to $V$ by $p$, we divide $\C^{*}$ in a finite number of sectors
$S_{i}$ such that each $U_{n}$ is contained in one of the $S_{i}$.
Now, let $\hat{f} \in \Rf$. Apply the classical
\index{Borel-Ritt theorem (classical)}
Borel-Ritt theorem \cite[chap. XI, \S 1]{HS} within each sector $S_{i}$, 
yielding a map $f_{i}$ holomorphic and admitting the asymptotic expansion 
$\hat{f}$. Glueing the $f_{i}$ provides a section of $\Aa$ over $V$ having
image $\hat{f}$.
\end{proof}


\section{Existence of asymptotic solutions}
\label{section:existasymptsols}

We call \emph{adequate} an open subset $\C^{*}$ of the form
\index{adequate open subset} 
$q^{-\N} U_{0}$, where the $q^{-k} U_{0}, k \geq 0$ are pairwise 
disjoint. Clearly, $\Aa_{0}(U_{\infty})$ is a $\Ka$-vector space 
stable under $\sq$. In particular, any $q$-difference operator
$P := \sq^{n} + a_{1} \sq^{n-1} + \cdots + a_{n} \in \Dq$ defines a
$\C$-linear endomorphism 
$P: f \mapsto P f = \sq^{n}(f) + a_{1} \sq^{n-1} + \cdots + a_{n} f$ 
of $\Aa_{0}(U_{\infty})$. Likewise, any matrix $A \in \GL_{n}(\Ka)$, 
defines a $\C$-linear endomorphism 
$\sq - A: X \mapsto \sq X - A X$ of $\Aa_{0}(U_{\infty})^{n}$.

\index{asymptotic solutions (existence theorem)}
\begin{theo}[Existence of asymptotic solutions]
\label{theo:TFESA}
Let $A \in \GL_{n}(\Ka)$ and let $\hat{X} \in \Rf^{n}$ be a formal solution 
of the system $\sq X = A X$. For any adequate open subset $U$, there exists
a solution $X \in \Aa(U_{\infty})^{n}$ of that system which is asymptotic
to $\hat{X}$.
\end{theo}
\begin{proof}
From theorem \ref{theo:BorelRittfaible}, there exists
$Y \in \Aa(U_{\infty})^{n}$ which is asymptotic to $\hat{X}$
(without however being a solution of the system). Set:
$$
W := \sq Y - A Y = \sq (Y - \hat{X}) - A (Y - \hat{X})
\in \Aa_{0}(U_{\infty})^{n}.
$$
From proposition \ref{prop:surjplats} below, there exists 
$Z \in \Aa_{0}(U_{\infty})^{n}$ such that $\sq Z - A Z = W$. Then
$X := Y - Z \in \Aa(U_{\infty})^{n}$ is asymptotic to $\hat{X}$
and solution of the system.
\end{proof}

\begin{prop}
\label{prop:surjplats}
Let $A \in \GL_{n}(\Ka)$ and let $U$ be an adequate open subset. Then,
the endomorphism $\sq - A$ of $\Aa_{0}(U_{\infty})^{n}$ is onto.
\end{prop}
\begin{proof}
Let $F \in \GL_{n}(\Ka)$ and $B := F[A]$. From the commutative diagram:
$$
\begin{CD}
\Aa_{0}(U_{\infty})^{n}   @>{\sq - A}>>  \Aa_{0}(U_{\infty})^{n}        \\
@VV{F}V         @VV{\sq F}V        \\
\Aa_{0}(U_{\infty})^{n}   @>{\sq - B}>>  \Aa_{0}(U_{\infty})^{n}
\end{CD}
$$
in which the vertical arrows are isomorphisms, we just have to prove
that $\sq - B$ is onto. After the cyclic vector lemma (lemma 
\ref{lemm:cyclicvector}), we can take $B := A_{P}$, the companion 
matrix of $P := \sq^{n} + a_{1} \sq^{n-1} + \cdots + a_{n}$ described
in \ref{subsubsection:functorofsolutions}. We then apply lemma 
\ref{lemm:surjsystouequa} (herebelow) and theorem \ref{theo:surjplats}
(further below).
\end{proof}

\begin{lemm}
\label{lemm:surjsystouequa}
Let $R$ be a $\Ka$-vector space on which $\sq$ operates. With the
notations of the previous proposition, $P: R \rightarrow R$ is onto
if, and only if, $\sq - A_{P}: R^{n} \rightarrow R^{n}$ is onto.
\end{lemm}
\begin{proof}
(i) Assume $P$ is onto. To solve
$(\sq - A_{P}) \begin{pmatrix} x_{1} \\ \vdots \\ x_{n} \end{pmatrix} = 
\begin{pmatrix} y_{1} \\ \vdots \\ y_{n} \end{pmatrix}$, put
$x_{i+1} := \sq x_{i} - y_{i}$; one sees that $x_{1}$ just has
to be solution of an equation $P x_{1} = z$, where $z$ is some
explicit expression of the $y_{i}$. \\
(ii) Assume $\sq - A_{P}$ is onto. To solve $P f = g$, it is enough
to solve 
$\sq X - A_{P} X = \begin{pmatrix} 0 \\ \vdots \\ 0 \\ g \end{pmatrix}$ 
and to set $f$ to the first component of $X$.
\end{proof}

\begin{rema}
This computation is directly related to the homotopy of complexes
discussed in \ref{subsection:complexhomotopy}.
\end{rema}

\begin{theo}
\label{theo:surjplats}
The endomorphism $P$ of $\Aa_{0}(U_{\infty})$ is onto.
\end{theo}
\begin{proof}
We give the proof under the assumption that the slopes are integral,
which is the only case we need in this paper; the reader will easily
deduce the general case by ramification. The counterpart of the slope
filtration on the side of $q$-difference operators is the existence
of a factorisation \cite{MarotteZhang,JSFIL,zh1} of $P$ as a product
of non commuting factors: elements of $\Ra^{*}$ (which induce
automorphisms of $\Aa_{0}(U_{\infty})$) and first order operators
$z^{\mu} \sq - a$, $\mu \in \Z$, $a \in \C^{*}$. We are therefore
reduced to the following lemma. (This is where some analysis comes in !)
\end{proof}

\begin{lemm}
\label{lemm:surjrang1}
The endomorphism $f \mapsto z^{\mu} \sq f - a f$ of $\Aa_{0}(U_{\infty})$ 
is onto.
\end{lemm}
\begin{proof}
Write $L$ this endomorphism. Let $h \in \Aa_0(U_\infty)$; for all
$n \in \N$ and each strict stable subset $V$ of $U$, there exists
$C_{n,V}>0$ such that $\lmod h(z)\rmod < C_{n,V} \lmod z \rmod^n$ over $V$. 
We distinguish two cases: (i) $\mu \leq 0$; (ii) $\mu > 0$.

Case (i): $\mu \leq 0$ (easy case). Let $\Phi$ be the automorphism of 
$\Aa_0(U_\infty)$ defined by $\Phi(f) := a q^\mu z^{-\mu} \sigma_{q^{-1}} f$; 
one has $\Phi^m f = a^m q^{\mu m(m+1)/2} z^{-m\mu}\sigma_{q^{-m}} f$ for all 
$m \geq 0$. Since $\lmod h(z)\rmod \leq  C_{n,V} \lmod z\rmod^n$ 
on each stable strict subset $V$ of $U$ and all integers $n \geq 0$, 
one finds:
$$
\forall z \in V \;,\; \sum_{m \geq 0} \lmod \Phi^m h(z) \rmod \leq 
C_{n,V} \,\sum_{m\ge 0} \lmod a \rmod^m \lmod q \rmod^{\mu m(m+1)/2} \, 
\lmod q\rmod^{-nm}\, \lmod z \rmod^n.
$$
One deduces that the series $ \sum_{m \geq 0} \Phi^m h$ defines a flat
function over $U$, write it $H \in \Aa_0(U_\infty)$. Observing that 
$L H = z^\mu \sq(H-\Phi H) = z^\mu \sq h$, we see that $L$ is onto
over $\Aa_0(U_\infty)$.

Case (ii): $\mu > 0$ (hard case). For all $j \in \N$, set
$U_j := q^{-j} U_0 = \{z \in \C^* \tq q^j z \in U_0\}$; these are 
pairwise disjoint since $U$ was assumed to be adequate. We are going 
to build successively $f_0,f_1,f_2, \ldots$ over $U_0,U_1,U_2,\ldots$, 
which will give rise to a function $f$ defined over $U = q^{-\N} U_0$, 
with $f_j = f_{|U_j}$ and $L f = h$.

For all $j \in \N$, let  $h_j := h_{|{U_j}} : U_j \rightarrow \C$. 
Since $(\sq f)_{|U_j}=\sq f_{j-1}$, equation $L f = h$ is equivalent 
to system:
$$
\forall j \in \N^{*} \;,\; z^\mu \sigma_q f_{j-1} - a f_j = h_j,
$$
also written:
\begin{equation}
\label{equa:Lfh}
\forall j \in \N^{*} \;,\; f_j = (z^\mu \sigma_q f_{j-1}-h_j)/a.
\end{equation}

We agree that $f_{-1} \equiv 0$ and $f_0 = h_0/a $. Iterating the
last relation \eqref{equa:Lfh} above, we find:
\begin{equation}
\label{equa:Lfhj}
\forall j \in \N^{*} \;,\; f_j = 
-\sum_{\ell=0}^j q^{\mu\ell(\ell-1)/2} z^{\ell\mu} \sq^\ell h_{j-\ell}\,a^{-\ell-1}.
\end{equation}
Write $f$ the function defined on $U$ by the relations
$f_{|U_j} = f_j$, $j \in \N$; Clearly, $Lf = h$ on $U$. 
We are left to show that $f\in \Aa_0(U_\infty)$.

Let $n \in \N$ and $V$ a strict stable subset of $U$; let 
$V_j := U_j \cap V$ and choose a $K > 0$ such that 
$\lmod z \rmod \leq K \lmod q \rmod^{-j}$ for all $z \in U_j$; 
one has $|z|\le K|q|^{-j}$ on each $V_j$. Bounding by above
the right hand side of \eqref{equa:Lfhj}, one gets:
\begin{equation}
\label{equa:Lfhj||}
\forall z \in V_j \;,\;
\lmod f_j(z) \rmod \leq 
\dfrac{C_{n,V}}{a} \, P_j(\lmod q \rmod^n\alpha)\,\lmod z \rmod^n,
\end{equation}
where $\alpha := K^\mu/a$ and where $P_j$ denotes the polynomial
with positive coefficients:
$$ 
P_j(X) := \sum_{\ell=0}^j \lmod q \rmod^{\mu\ell(\ell-1)/2-j\ell\mu}\,X^\ell.
$$
Taking in account the relation
$$
P_{j+1}(X) = 
P_j(\lmod q \rmod^{-\mu}X)+ \lmod q \rmod^{-\mu(j+1)(j+2)/2}\,X^{j+1}\,,
$$
one finds that $P_j(X) \leq P(X;\mu,q)$ for all $X > 0$, where
we write $P(z;\mu,q)$ the entire function defined by:
$$
P(z;\mu,q) := \sum_{j=0}^\infty \lmod q \rmod^{-\mu j(j+1)/2}\,z^{j}.
$$
Said otherwise, \eqref{equa:Lfhj||} can be written in the form:
$$
\forall z \in V_j \;,\;
\lmod f_j(z) \rmod \leq 
\frac{C_{n,V}}a\,P(\lmod q \rmod^n\alpha;\mu,q)\,\lmod z \rmod^n,
$$
which allows us to conclude that $f$ is flat.
\end{proof}

We now apply theorem \ref{theo:TFESA} to a situation that we shall 
meet when we prove theorem \ref{theo:secondqMalgrangeSibuya}. \\

Let $A_{0} \in \GL_{n}(\Ka)$ be the matrix of a pure module and let
$A \in \GL_{n}(\Ka)$ a matrix formally equivalent to  $A_{0}$.
Let $\hat{F} \in \GL_{n}(\Kf)$ be a gauge transformation proving this
equivalence, \ie\ $\hat{F}[A_{0}] = A$. (Note that we do not
assume $\hat{F} \in \G(\Kf)$.) The conjugation relation is
equivalent to $\sq \hat{F} = A \hat{F} A_{0}^{-1}$, which we 
regard as a $q$-difference system of rank $n^{2}$ in the space 
$\Mat_{n}(\Ka) \simeq \Ka^{n^{2}}$. After theorem \ref{theo:TFESA}, 
there exists over each adequate open set $U$ a solution
$F_{U} \in \Mat_{n}(\Ka)$ of the system $\sq F = A F A_{0}^{-1}$
that is asymptotic to $\hat{F}$. Since the latter is invertible,
and since $\det F_{U}$ is obviously asymptotic to $\det \hat{F}$, 
whence non zero, $F_{U}$ is invertible. \\

Now let $U'$ another adequate open set which meets $U$. Then
$U'' := U \cap U'$ is adequate and the matrix $G := F_{U}^{-1} F_{U'}$ 
satisfies the two following properties:
\begin{enumerate}
\item{It is an automorphism of $A_{0}, \ie\ G[A_{0}] = A_{0}$.}
\item{It is asymptotic to $\hat{F}^{-1} \hat{F} = I_{n}$,
\ie\ the coefficients of $G - I_{n}$ belong to $\Aa_0(U''_\infty)$.}
\end{enumerate}

\begin{lemm}
\label{lemm:cocyclasympt}
The matrix $G$ belongs to $\G(\Aa_0(U''_\infty))$.
\end{lemm}
\begin{proof}
With the usual notations for $A_{0}$, each block of $G$ satisfies:
$$
\sq G_{i,j} = 
\left(z^{\mu_{j}} A_{j}\right)^{-1} G_{i,j} \left(z^{\mu_{i}} A_{i}\right) = 
z^{\mu_{i} - \mu_{j}} A_{j}^{-1} G_{i,j} A_{i}.
$$
The order of growth or decay of $G_{i,j}$ near $0$ is therefore the
same as $\thq^{\mu_{i} - \mu_{j}}$. Since $G_{i,j}$ is flat for $i \neq j$,
this implies $G_{i,j} = 0$ for $i > j$. For $i = j$, we apply the same
argument to $G_{i,i} - I_{r_{i}}$, which therefore vanishes.
\end{proof}


\section{The fundamental isomorphism}
\label{section:fundamentalisomorphism}

\label{not48}
We write $\Lambda_I := I_n + \Mat_n(\Aa_0)$ for the subsheaf
of groups of $\GL_n(\Aa)$ made up of matrices infinitely
tangent to the identity and, if $G$ is an algebraic subgroup
of $\GL_n(\C)$, we put $\Lambda_I^G := \Lambda_I \cap G(\Aa)$.
In particular $\Lambda_I^{\G}$ is a sheaf of triangular matrices
of the form \eqref{eqn:isoplat} with all the $F_{i,j}$ infinitely
flat. \\

For a $q$-difference module $M = (\Ka^n,\Phi_A)$ (section
\ref{subsection:generalitiesdiffmods}), we consider the set of 
\index{infinitely tangent to the identity (automorphisms)}
\emph{automorphisms of $M$ infinitely tangent to the
identity}: this is the subsheaf $\Lambda_I(M)$ of $\Lambda_I$
whose sections satisfy the equality: $F[A] = A$ (recall
that by definition $F[A] = (\sq F) A F^{-1}$);  this is
\index{Stokes sheaf}
also called the \emph{Stokes sheaf} of the module $M$. \\

If the matrix $A$ has the form \eqref{eqn:formestandardtriangulaire}
(see paragraph \ref{subsection:matricialdescriptionofF(P1,...,Pk)})
then $\Lambda_I(M)$ is a subsheaf of $\Lambda_I^{\G}$. Up to
\index{gauge transformation}
an analytic gauge transformation, we can assume that we are
in this case. As a consequence,  $\Lambda_I(M)$ is a sheaf
of \emph{unipotent} (and even triangular) groups.


\subsection{Study of $H^1(\Eq;\Lambda_I)$}
\label{subsection:glue}
\label{not49}

We choose $\tau \in \C$, $\Im(\tau) > 0$, such that $q = e^{-2 \ii\pi\tau}$. 
Since $(1,\tau)$ is a basis of the $\R$-vector space $\C$, the group morphism: 
$$
p_0: (u,v) \mapsto e^{2 \ii \pi (u+v\tau)}
$$
from $\R^2$ to $\C^*$ is onto with kernel $\Z \times \{0\}$, and the
preimage of $q^{\Z}$ is 
$$
\label{not50}
p_0^{-1}(q^{\Z}) = \Z \times \Z.
$$
Therefore, the composition $p \circ p_0$ of $p_0$ with the natural 
projection $p: \C^* \rightarrow \Eq = \C^{*}/q^{\Z}$ is a group epimorphism
\index{elliptic curve $\Eq$} 
with discrete kernel $\Z^2$; it is clearly a local diffeomorphism.

\begin{rema}
The morphism $p_0$ induces an isomorphism between $\R^2/\Z^2$ 
and $\Eq = \C^{*}/q^{\Z}$. Through this isomorphism, the natural 
splitting of $\R^2/\Z^2$ is carried to the splitting:
$$
\C^* = \U \times q^{\R},
$$
where $\U$ denotes the unit circle in $\C^*$ and, for $v \in \R$,
we denote $q^v = e^{-2 \ii \pi\tau v}$
\end{rema}

\index{open parallelogram}
We will call \emph{open parallelogram} of $\Eq$ the image by the map 
$p \circ p_0: \R^2 \rightarrow \Eq$ of an open parallelogram
$\left]u_1,u_2\right[ \times \left]v_1,v_2\right[ \subset \R^2$
where $u_2 - u_1 \leq 1, v_2 - v_1 \leq 1$. This is the same as 
the image by $x \mapsto e^{2 \ii \pi x}$ of:
$$
\{u + v\tau \in \C \tq u_1 < u < u_2, v_1 < v< v_2\},
$$
\index{closed parallelogram}
We will call \emph{closed parallelogram} of $\Eq$ the closure of 
an open parallelogram such that $u_2 - u_1 < 1,v_2 - v_1 < 1$. We 
\index{small parallelogram}
will say that a parallelogram is \emph{small} if  $u_2 - u_1 < 1/4$
and $v_2 - v_1 < 1/4$. \\

For every open covering $\VV$ of $\Eq$, there exists a finite open 
covering $\UU = (U_i)_{i\in I}$, finer than $\VV$, such that:
\begin{enumerate}
\item[(i)]{the $U_i$ are open parallelograms,}
\item[(ii)]{the open sets $U_i$ are four by four disjoint.}
\end{enumerate}
\index{good covering}
We will call \emph{good} such a covering. There is a similar
definition for closed coverings. It is then clear that:
\begin{itemize}
\item{Every element of $H^j(\Eq;\GL_n(\Aa)),j=0,1$ can be
represented by an element of $H^j(\UU;\GL_n(\Aa))$, where
$\UU$ is a good covering.}
\item{The natural map
$H^j(\UU;\GL_n(\Aa)) \rightarrow H^j(\Eq;\GL_n(\Aa))$
is always injective.}
\end{itemize}

\label{not51}
Let $\UU$ be a good open covering of $\Eq$; we denote by
$\Cc^0(\UU;\GL_n(\Aa))$ the subspace of the space of $0$-cochains
$(g_i) \in C^0(\UU;\GL_n(\Aa))$ whose coboundary $(g_{ij} := g_i g_j^{-1})$
belongs to $C^1(\UU;\Lambda_I)$. Thus, a cochain $C^0(\UU;\GL_n(\Aa))$
is any family $(g_i)$ such that $g_i \in \Aa(U_i)$, and it belongs to
$\Cc^0(\UU;\GL_n(\Aa))$ if, and only if, the $g_i$ ($i\in I$) admit the 
same asymptotic expansion $\hat g \in \GL_n\bigl(\Rf\bigr)$. It is also
equivalent to say that the natural image of the cochain $(g_i)$ in
$C^0(\UU; \GL_n\bigl(\Rf\bigr))$ is a cocycle.

\begin{lemm}
\begin{itemize}
\item[(i)]
Let $\UU$ be a good open covering of $\Eq$.
The coboundary map $\partial$ induces an injection
on the double quotient:
\label{not52}
$$
C^0(\UU;\Lambda_I) \backslash
\Cc^0(\UU;\GL_n(\Aa))
/H^0\bigl(\UU;\GL_n(\Aa)\bigr)
\rightarrow H^1(\UU;\Lambda_I).
$$
\item[(ii)]
We have a canonical injection:
$$
{\lim_{\longrightarrow}}_{\UU}~C^0(\UU;\Lambda_I)
\backslash \Cc^0(\UU;\GL_n(\Aa))
/\GL_n\bigl(\Ra\bigr)
\rightarrow H^1(\Eq;\Lambda_I),
$$
where the direct limit is indexed by the good coverings.
\end{itemize}
\end{lemm}
\begin{proof}
We have two natural maps:
$$
\Cc^0(\UU;\GL_n(\Aa)) \rightarrow Z^1(\UU;\Lambda_I), 
$$
which sends the family $(g_i)$ to the family $(g_{ij})$ and
$$
\Cc^0(\UU;\GL_n(\Aa)) \rightarrow \GL_n\bigl(\Rf\bigr),
$$
which sends the family $(g_i)$ to $\hat g \in \GL_n\bigl(\Rf\bigr)$,
the common asymptotic expansion of all the $g_i$.

Two cochains $(g_i),(g'_i)\in \Cc^0(\UU;\GL_n(\Aa))$ define
the same element of $H^1(\UU;\Lambda_I)$ if and
only if:
$$
g'_i{g'}_j^{-1}=h_ig_ig_j^{-1}h_j^{-1}, \ \text{with} \
(h_i)\in C^0(\UU;\Lambda_I).
$$
This is equivalent to ${g'}_j^{-1} h_j g_j = {g'}_i^{-1} h_i g_i$,
which thus defines an element $f \in H^0\bigl(\UU;\GL_n(\Aa)\bigr)$
and we have $g' = h g f^{-1}$.
\end{proof}

\begin{prop}
For every good open covering $\UU$ of $\Eq$ we have
a natural isomorphism:
$$
C^0(\UU:\Lambda_I)
\backslash \Cc^0(\UU;\GL_n(\Aa))
/H^0\bigl(\UU;\GL_n(\Aa)\bigr)
\rightarrow
\GL_n\bigl(\Rf\bigr)/\GL_n\bigl(\Ra\bigr).
$$
(Note that the double quotient on the left hand side is the same as
the double quotient on the left hand side of assertion (i) in the
lemma above.)
\end{prop}
\begin{proof}
Let  $\UU$ be a good open covering of $\Eq$.
Let $\hat g \in \GL_n\bigl(\Rf\bigr)$. Using the $q$-analog
of Borel-Ritt (theorem \ref{theo:BorelRittfaible}), we can
represent it by an element $g_i$ on each open set $U_i$ such
that $g_i$ admits $\hat g$ as asymptotic expansion, therefore
the natural map:
$$
\Cc^0(\UU;\GL_n(\Aa)) \rightarrow \GL_n\bigl(\Rf\bigr)
$$
is onto. The proposition follows.
\end{proof}

\begin{coro}
For every good open covering $\UU$ of $\Eq$, we have natural
injections:
$$
\GL_n\bigl(\Rf\bigr)/\GL_n\bigl(\Ra\bigr)
\rightarrow
H^1(\UU;\Lambda_I) \rightarrow H^1(\Eq;\Lambda_I).
$$
\end{coro}

We will see later (subsection \ref{subsection:theformaltrivialisation})
that these maps are bijections.


\subsection{Geometric interpretation of the elements
of  $H^1(\Eq;\Lambda_I)$}
\label{subsection:geometricinterpretation}

We need some results on regularly separated subsets of
the euclidean space $\R^2$. We first recall some definitions.

If $X \subset \R^n$ is a closed subset of the euclidean
space $\R^n$, we denote by $\Ee(X)$ the algebra
\label{not53}
of $\mathcal{C}^{\infty}$ functions in Whitney sense on $X$
(\cf\ \cite{Ma1}).

\index{Whitney extension theorem}
We recall the \emph{Whitney extension theorem}: if $Y \subset X$
are closed subset of $\R^n$, the restriction map
$\Ee(X) \rightarrow \Ee(Y)$ is surjective.

Let $A,B$ be two closed sets of $\R^n$; we define the following maps:
$$
\alpha: f \in \Ee(A \cup B) \mapsto
(f_{\vert A},f_{\vert B}) \in \Ee(A) \oplus \Ee(B),
$$
$$\beta:  (f,g)\in \Ee(A) \oplus \Ee(B) \mapsto
f_{\vert A\cap B} - g_{\vert A\cap B} \in \Ee(A \cap B),
$$
and we consider the sequence:

\begin{equation*}
0 \rightarrow \Ee(A\cup B) \overset{\alpha}{\rightarrow}
\Ee(A) \oplus \Ee(B) \overset{\beta}{\rightarrow}
\Ee(A \cap B)\rightarrow 0.
\end{equation*}

\begin{defi}
Two subsets $A$ and $B$ of the euclidean space $\R^n$
are regularly separated if the above sequence is exact.
\end{defi}

The following result is due to \L ojasiewicz (\cf\ \cite{Ma1}).

\begin{prop}
\label{prop:regsep}
Let $A$,$B$ be two closed subsets $A$ and $B$ of the euclidean
space $\R^n$; the following conditions are equivalent:
\begin{itemize}
\index{regularly separated subsets}
\item[(i)]{$A$ and $B$ are regularly separated;}
\item[(ii)]{$A \cap B = \emptyset$ or for all $x_0\in A\cap B$,
there exists a neighborhood $U$ of $x_0$ and constants $C > 0$,
$\lambda > 0$ such that for all $x\in U$:
$$
d(x,A) + d(x,B) \geq C d(x,A \cap B)^{\lambda};
$$}
\item[(iii)]{$A \cap B = \emptyset$ or for all $x_0\in A\cap B$,
there exists a neighborhood $U$ of $x_0$ and a constant $C'>0$
such that for all $x\in A\cap U$:
$$
d(x,B) \geq C' d(x,A \cap B)^{\lambda}.
$$}
\end{itemize}
\end{prop}

If we can choose $\lambda = 1$ in the above condition(s), we will
say that the closed subsets $A$ and $B$ are \emph{transverse}.
\index{transverse subsets}
Transversality is in fact a {\it local} condition.

\begin{lemm}
\label{lemm:transverse}
Let $K_1$, $K_2$ be two closed small parallelograms of $\Eq$.
The sets $p^{-1}(K_1)$ and  $p^{-1}(K_2)$ are transverse.
\end{lemm}
\begin{proof}
Using the map $p_0$ (introduced at the beginning of \ref{subsection:glue}),
we set:
\begin{align*}
A := p^{-1}(K_1) \subset \C^*
& \text{~and~} 
A_0 := p_0^{-1}(A) = (p \circ p_0)^{-1}(K_1) \subset \R^2, \\
B := p^{-1}(K_2) \subset \C^*
& \text{~and~} 
B_0 := p_0^{-1}(B) = (p \circ p_0)^{-1}(K_2) \subset \R^2.
\end{align*}
Since $K_1$ and $K_2$ are small parallelograms, both $A_0$
and $B_0$ are formed of a family of disjoint rectangles,
that is to say:
$$
A_0 = \bigcup_{m,n\in\Z} A_{m,n} \text{~and~} B_0 =\bigcup_{m,n\in\Z} B_{m,n},
$$
where the families $(A_{m,n})_{m,n\in\Z}$, $(B_{m,n})_{m,n\in\Z}$ are
such that $d(A_{m,n}, A_{m',n'})\ge 3/4$ and $d(B_{m,n},B_{m',n'})\ge 3/4$
for any $(m',n') \not= (m,n)$. We only need to deal with the case
when $A  \cap B \not= \emptyset$, that is, when there exists $(m,n)$,
$(m',n')$ such that $A_{m,n} \cap B_{m',n'}\not=\emptyset$.

Consider the function $f$ defined over $\C^* \setminus A \cap B$ by
$$
f(x)=\frac{d(x,A)+d(x,B)}{d(x,A\cap B)}
$$
and observe that $f$ is $q$-invariant; therefore, it is enough to
prove the following statement: for any $x_0 \in  A \cap B$ such that
$1 < \lmod x_0 \rmod \le  \lmod q \rmod$, there exists an open disk
$D(x_0,r)$, $r>0$, and a constant $C>0$ such that $f(x)\ge C$ for
all $x\in D(x_0,r) \setminus (A \cap B)$.

On the other hand, the mapping $p_0$ is a local diffeomorphism from $\R^2$
onto $\C^*$ inducing a local equivalence between the natural metrics.
Therefore the proof of the lemma boils down to the transversality
between $A_0$ and $B_0$, which is merely deduced from the transversality
of each pair of rectangles $(A_{m,n},B_{m',n'})$ of the euclidean plane.
\end{proof}

Let $\gamma \in H^1(\Eq;\Lambda_I)$, we can represent
it by an element $g\in Z^1(\UU;\Lambda_I)$
where $\UU$ is a good open covering  of $\Eq$.
We can suppose that all the parallelograms of the covering
are small. We will associate to $g$ a germ of fibered
space $\bigl(M_g,\pi,(\C,0)\bigr)$. We will see that
if we change the choice of $g$, then these germs of
linear fibered space correspond by a canonical isomorphism.
 This construction mimicks a construction of \cite{MR}
inspired by an idea of Malgrange (\cf\ \cite{Ma2}), it is
a central point for the proof of one of our main results. \\

Let $g \in Z^1(\UU;\Lambda_I)$. We interpret
the $g_i$ as germs of holomorphic functions on the germ
of $q^{-\N}$-invariant open set defined by $U_i$, that is
the germ of $p^{-1}(U_i) \cap D$ ($D$ being an open disk
centered at the origin). For sake of simplicity we will
denote also by $U_i$ this germ.  We consider the disjoint
sum $\displaystyle \prod_{i\in I} U_i \times \C^n$ and (using
representatives of germs) we identify the points
$(x,Y) \in U_i \times \C^n$ and $(x,Z) \in U_j \times \C^n$
if $x \in U_{ij}$ and $Z =g_{ij}(x)(Y)$ (we verify that we have
an equivalence relation using the cocycle condition),
we denote the quotient by ${\mathbf M}_g$. If $x = 0$, $g_i(0)$ is
the identity of $\C^n$ therefore the quotient ${\mathbf M}_g$
is a germ of topological space along $\{0\} \times \C^n$.
\label{not54}

\begin{lemm}
\begin{itemize}
\item[(i)]
The projections $U_i \times \C^n \rightarrow U_i$
induce a germ of continuous fibration $\pi: {\mathbf M}_g\rightarrow (\C,0)$,
the fiber $\pi^{-1}(0)$ is identified with $(\C^n,0)$.
\item[(ii)]
The germ $\overset{\bullet}{\mathbf M}_g := {\mathbf M}_g \setminus \pi^{-1}(0)$ admits
a natural structure of germ of complex manifold of
dimension $n+1$, it is the unique structure such that
the natural injections
$(U_i \setminus \{0\}) \times \C^n \rightarrow \overset{\bullet}{\mathbf M}_g$
are holomorphic. The restriction of $\pi$ to $\overset{\bullet}{\mathbf M}_g$
is holomorphic.
\end{itemize}
\end{lemm}

The proof of the lemma is easy.

\begin{lemm}
There exists on ${\mathbf M}_g$ a unique structure of germ
of differentiable manifold ($\mathcal{C}^{\infty}$)
such that the natural injections
$(U_i\setminus \{0\}) \times \C^n \rightarrow \overset{\bullet}{\mathbf M}_g$
are $\mathcal{C}^{\infty}$ (in Whitney sense).
The induced structure on $\overset{\bullet}{\mathbf M}_g$ is the underlying
structure of the holomorphic structure on $\overset{\bullet}{\mathbf M}_g$.
The map $\pi$ is $\mathcal{C}^{\infty}$ of rank one,
$\bigl({\mathbf M}_g,\pi,(\C,0)\bigr)$ is a germ of vector bundle.
\end{lemm}
\begin{proof}
This lemma is the central point. The proof uses
the Whitney extension theorem and the notion of
regularly separated closed sets (\cf\ above).

We replace the good open covering $(U_i)_{i\in I}$ by
a good closed covering $(V_i)_{i\in I}$ with $V_i\subset U_i$.
We interpret the $V_i$ as $q^{-\N}$-invariant germs of
compact sets $p^{-1}(V_i) \cap \bar D \subset \C^*$
(we take $\bar D :=\bar D(r)$, $r>0$). The differentiable
functions on $V_i \times \C^n$ are by definition the
$\mathcal{C}^{\infty}$ functions in Whitney sense.
We can built a germ of set as a quotient of
$\displaystyle \prod_{i\in I} V_i \times \C^n$
using $g$ as above, the natural map from
this quotient to ${\mathbf M}_g$ is an homeorphism and
we can identify these two sets. If we consider
a representative of the germ ${\mathbf M}_g$ and $x \in \C^*$
sufficiently small, there exists a structure of
differentiable manifold along $\pi^{-1}(x)$ and
it satisfies all the conditions. Let $Y_0\in\C^n$,
we set $m_0=(0,Y_0)$ and we denote by $\Ee_{m_0}$
the ring of germs of real functions on ${\mathbf M}_g$
represented by germs at $m_0\in V_i \times \C^n$ of
functions $f_i$ compatible with glueing applications,
that is $f_j \circ g_{ji}=f_i$.

We will prove that each cordinate $y_h$ ($h =1,\ldots,n$)
on $\C^n$ extends in an element $\eta_h$ of $\Ee_{m_0}$.
It suffices to consider $y_1$.

We start with $i\in I$ We extend $y_1$ in a $\mathcal{C}^{\infty}$
function $f_i$ on $V_i \times \C^n$ (we can choose $f_i = y_1$).

Let $j \in I$, $j \neq i$. The glueing map $g_{ij}$ gives
a $\mathcal{C}^{\infty}$ function $h_{j}$ on
$V_{ij} \times \C^n \subset V_j \times \C^n$, using Whitney
extension theorem, we can extend it in a $\mathcal{C}^{\infty}$
function $f_j$ on $V_j\times\C^n$.

Let $k \in I$, $k \neq i,j$. The glueing maps $g_{ik}$ and $g_{jk}$
give  functions $h_{ik}$ and  $h_{jk}$ respectively on
$V_{ik} \times \C^n \subset V_k \times \C^n$ and
$V_{jk} \times \C^n \subset V_k \times \C^n$;
these functions coincide
on $V_{ijk} \times \C^n \subset V_k \times \C^n$. The closed sets
$V_{ik} \times \C^n$ and $V_{jk} \times \C^n$ are regularly separated
($V_{ik}$ and $V_{jk}$ are, as subsets of $\Eq$, closed parallelograms
and we can apply lemma \ref{lemm:transverse}), therefore
the $\mathcal{C}^{\infty}$ functions $h_{ik}$ and  $h_{jk}$ define
a $\mathcal{C}^{\infty}$ function $h_k$ on
$(V_{ik} \cup V_{jk})\times \C^n \subset V_k \times \C^n$,
using Whitney extension theorem, we can extend it in
a $\mathcal{C}^{\infty}$ function $f_k$ on $V_k \times \C^n$.

Let  $h \in I$, $h \neq i,j,k$, then
$ V_{ih} \cap  V_{jh} \cap  V_{kh} =V_{ijkh} =\emptyset$ and we can do
as in the preceding step, using moreover the fact that to be
$\mathcal{C}^{\infty}$ is a local property.

We can end the proof along the same lines.
\end{proof}

The $n$ functions $\eta_1,\ldots,\eta_m\in \Ee_{m_0}$
and $\pi$, interpreted as an element of  $\Ee_{m_0}$,
define a system of local coordinates on ${\mathbf M}_g$ in a neighborhood
of $m_0$; an element of $\Ee_{m_0}$ is a $\mathcal{C}^{\infty}$
function of these coordinates. Therefore ${\mathbf M}_g$ admits a structure
of differentiable manifold, this structure is fixed on 
$\overset{\bullet}{\mathbf M}_g$,
therefore unique.

\begin{prop}
\label{prop:manifold}
\begin{itemize}
\item[(i)]{The germ of differentiable manifold ${\mathbf M}_g$
admits a unique structure of germ of complex analytic manifold
extending the complex analytic structure of $\overset{\bullet}{\mathbf M}_g$.
For this structure the map $\pi$ is holomorphic and its rank is one,
$\bigl({\mathbf M}_g,\pi,(\C,0)\bigr)$ is a germ of holomorphic
vector bundle.}
\label{not55}
\item[(ii)]{We consider the formal completion $\widehat{\mathbf M}_g$
of ${\mathbf M}_g$ along $\pi^{-1}(0)$ (for the holomorphic structure
defined by (i)) and the formal completion $\hat{\mathbf F}$ of
$(\C,0)\times\C^n$ along $\{0\}\times\C^n$. We denote by
$\hat\pi: \widehat{\mathbf M}_g \rightarrow \bigl((\C,0){\hat\vert}\{0\}\bigr)$
the completion of $\pi$ and by
$\hat\pi_0: \hat{\mathbf F} \rightarrow  \bigl((\C,0){\hat\vert}\{0\}\bigr)$
the natural projection. The formal vector bundle
$\Bigl(\widehat{\mathbf M}_g,\hat\pi,\bigl((\C,0){\hat\vert}\{0\}\bigr)\Bigr)$
is naturally isomorphic to the formal vector bundle
$\bigl(\hat{\mathbf F},\hat\pi_0,((\C,0){\hat\vert}\{0\})\bigr)$.}
\end{itemize}
\end{prop}
\begin{proof}
\index{Newlander-Niremberg integrability theorem}
The proof of (i) is based on the Newlander-Niremberg
integrability theorem, it is similar, mutatis mutandis
to a proof in \cite{MR}, page 77. The proof of (ii) is
easy (the $\hat g_{ij}$ are equal to identity).
\end{proof}

If $g,g'$ are two representatives of $\gamma$,
the complex analytic manifolds ${\mathbf M}_g$ and ${\mathbf M}_{g'}$
are isomorphic (\cf\ \cite{MR}, page 77)


\subsection{The formal trivialisation of the elements
of $H^1(\Eq;\Lambda_I)$ and the fundamental isomorphism}
\label{subsection:theformaltrivialisation}
\index{formal trivialisation}

Using the geometric interpretation of the cocycles
of $Z^1(\UU;\Lambda_I)$ we will build a map:
$$
Z^1(\UU;\Lambda_I) \rightarrow  \GL_n\bigl(\Rf\bigr)
$$
defined up to composition on the right by an element
of $\GL_n\bigl(\Ra\bigr)$. Hence we will get a map:
$$
H^1(\Eq;\Lambda_I) \rightarrow
\GL_n\bigl(\Rf\bigr)/\GL_n\bigl(\Ra\bigr)
$$
and we will verify that it inverses the natural map:
$$
\GL_n\bigl(\Rf\bigr)/\GL_n\bigl(\Ra\bigr)
\rightarrow H^1(\Eq;\Lambda_I).
$$

We consider the germ of trivial bundle $(\C,0) \times\C^n,\pi_0,(\C,0)$.
We choose a holomorphic trivialisation $H$ of the germ
of fiber bundle $\bigl({\mathbf M}_g,\pi,(\C,0)\bigr)$:
$\pi_0\circ H = \pi$ ($H$ is defined up to composition on
the right by an element of  $\GL_n\bigl(\Ra\bigr)$).
From the proposition \ref{prop:manifold} (ii) and $H$ we get
an \emph{automorphism} $\hat\varphi$ of the formal vector
bundle $\bigl(\hat{\mathbf F},\hat\pi_0,(\C{\hat\vert}\{0\}\bigr)$.
We can interpet $\hat\varphi$ as an element of
$\GL_n\bigl(\Rf\bigr)$, this element
is well defined up to the choice of $H$, that is up
to composition on the left by an element of
$\GL_n\bigl(\Ra\bigr)$. Hence $\hat{\varphi}^{-1}$
is defined up to composition on the right by an element
of $\GL_n\bigl(\Ra\bigr)$. \\

The map:
$$
Z^1(\UU;\Lambda_I) \rightarrow
\GL_n\bigl(\Rf\bigr)/\GL_n\bigl(\Ra\bigr)
$$
induced by $g\mapsto \hat\varphi^{-1}$ induces a map:
$$H^1(\Eq;\Lambda_I) \rightarrow
\GL_n\bigl(\Rf\bigr)/\GL_n\bigl(\Ra\bigr).
$$
We will verify that this map is the inverse of the natural map:
$$
\GL_n\bigl(\Rf\bigr)/\GL_n\bigl(\Ra\bigr)
\rightarrow H^1(\Eq;\Lambda_I).
$$

From the natural injections
$(U_i,0) \times \C^n \rightarrow M_g$ ($i\in I$)
and the trivialisation $H$, we get automorphisms of
the vector bundles $\bigl((U_i,0) \times \C^n \rightarrow (U_i,0)\bigr)$,
that we can interpret as elements
$\varphi_i \in \Gamma\bigl(U_i;\GL_n(\Aa)\bigr)$
compatible with the glueing maps $g_{ij}$, \ie\
$\varphi_i=\varphi_j\circ g_{ji}$.
We have:
$$
g_{ij} = \varphi_i^{-1} \circ \varphi_j =
\varphi_i^{-1} \circ (\varphi_j^{-1})^{-1}\
\text{~and~} \partial(\varphi_i^{-1}) = (g_{ij}).
$$
The element $\hat \varphi_i \in \GL_n\bigl(\Rf\bigr)$
is independant of $i\in I$; we will denote it by $\hat\varphi$. \\

We have $(\varphi_i^{-1}) \in \Cc^0(\UU;\GL_n(\Aa)\bigr)$ (remember that
this space of cochains was defined in subsection \ref{subsection:glue}).
The coboundary of $(\varphi_i^{-1})$ is $g$ and its natural image
in $\GL_n\bigl(\Rf\bigr)$ is $\hat\varphi^{-1}$.
This ends the proof of our contention. \\

It is possible to do the same constructions
replacing $\GL_n(\C)$ by an algebraic
subgroup (vector bundles are replaced by vector
bundles admitting $G$ as structure group).
We leave the details to the reader.

\index{$q$-Birkhoff-Malgrange-Sibuya theorem (first)}
\begin{theo}[First $q$-Birkhoff-Malgrange-Sibuya theorem]
\label{theo:firstqMalgrangeSibuya}
Let  $G$ be an algebraic subgroup of $\GL_n(\C)$.
The natural maps:
$$
\GL_n\bigl(\Rf\bigr)/\GL_n\bigl(\Ra\bigr)
\rightarrow H^1(\Eq;\Lambda_I),
$$
$$
G\bigl(\Rf\bigr)/ G\bigl(\Ra\bigr)\rightarrow H^1(\Eq;\Lambda_I^G)
$$
are bijective.\hfill $\Box$
\end{theo}


\section{The isoformal classification by the cohomology
of the Stokes sheaf}
\label{section:theisoformalclassificationJP}
\index{isoformal classification}
\index{Stokes sheaf (cohomology)}

We will see that the constructions of the preceding paragraph are
compatible with discrete dynamical systems data and deduce a new
version of the analytic isoformal classification. We use notations
from section \ref{section:analyticisoformalclassification}.


\subsection{The second main theorem}
\label{subsection:thesecondmaintheorem}

Before going on, let us make some preliminary remarks. Let
$M_{0} := P_{1} \oplus \cdots \oplus P_{k}$ be a pure module.
Represent it by a matrix $A_0$ with the same form as in
\eqref{eqn:formestandarddiagonaleentiere}. Assume moreover
that $A_{0}$ is in Birkhoff-Guenther normal form as found
in \ref{subsection:BGNF}. \\

We saw in proposition \ref{prop:groupeformelsuranalytique} that
there is a natural bijective mapping between
$\F(M_0) = \F(P_{1},\ldots,P_{k})$ and the quotient
$\G^{A_{0}}(\Kf)/\G(\Ka)$. It also follows from the first part of
\ref{subsection:computnormalform} that this, in turn, is equal
to $\G^{A_{0}}(\Rf)/\G(\Ra)$. Last, the inclusion $\G \subset \GL_n$
induces a bijection from $\G^{A_{0}}(\Rf)/\G(\Ra)$ to the quotient
$\GL^{M_0}_n\bigl(\Rf\bigr)/\GL_n\bigl(\Ra\bigr)$, where we take
\label{not56}
$\GL^{M_0}_n\bigl(\Rf\bigr)$ to denote the set of those formal gauge
transformation matrices that send $A_{0}$ into $\GL_n\bigl(\Ra\bigr)$.

\index{$q$-Birkhoff-Malgrange-Sibuya theorem (second)}
\begin{theo}[Second $q$-Birkhoff-Malgrange-Sibuya theorem]
\label{theo:secondqMalgrangeSibuya}
Let  $M_0$ be a pure module represented by a matrix $A_{0}$ in
Birkhoff-Guenther normal form. There is a natural bijective mapping:
$$
\GL^{M_0}_n\bigl(\Rf\bigr)/\GL_n\bigl(\Ra\bigr)
\rightarrow H^1\bigl(\Eq;\Lambda_I(M_0)\bigr),
$$
whence a natural bijective mapping:
$$
\F(M_0) \rightarrow H^1\bigl(\Eq;\Lambda_I(M_0)\bigr).
$$
\end{theo}
\begin{proof}

\noindent
\underline{\emph{The map $\lambda$}}

Let $\UU :=(U_i)_{i\in I}$ be a good open covering of $\Eq$.

Let $\hat F \in \GL^{M_0}_n \subset \GL_n\bigl(\Rf\bigr)$; by
definition there exists an unique $A \in \GL_n\bigl(\Ra\bigr)$
such that $\hat F[A_0] = A$.

For all $i\in I$, there exists $g_i \in \GL_n\bigl({\mathcal
A}(U_i)\bigr)$ asymptotic to $\hat F$ on $U_i$ (\cf\ theorem
\ref{theo:TFESA}). The coboundary $(g_{ij})$ of the cochain $(g_i)$ 
belongs to $Z^1(\UU,\Lambda_I(M_0))$ (the reason for that is lemma 
\ref{lemm:cocyclasympt} which can be applied because the condition
of good covering implies that the open sets are adequate); it does 
not change if we multiply $\hat F$ on the right by an element of 
$\GL_n\bigl(\Ra\bigr)$. If we change the representatives $g_i$, 
we get another cochain in $Z^1(\UU,\Lambda_I(M_0))$ inducing 
the same element of $H^1(\UU,\Lambda_I(M_0))$. We thus get 
a natural map:
$$
\lambda:\GL^{M_0}_n\bigl(\Rf\bigr)/\GL_n\bigl(\Ra\bigr)
\rightarrow H^1\bigl(\Eq;\Lambda_I(M_0)\bigr).
$$
This map is injective.

\noindent
\underline{\emph{Surjectivity of $\lambda$}}
Let $\UU :=(U_i)_{i\in I}$ be a good open covering of $\Eq$; assume
moreover that the compact covering $(\overline U_i)_{i\in I}$ is good
and that the $U_i$ are small parallelograms.

Let $\gamma \in H^1(\Eq;\Lambda_I(M_0))$; we can represent it by
an element $g \in Z^1(\UU;\Lambda_I(M_0)$, $\UU =(U_i)_{i\in I}$ being
a good open covering  of $\Eq$.  Moreover we can choose $\UU$ such
that the compact covering $(\overline U_i)_{i\in I}$ is good and such
that the $U_i$ are small parallelograms.

As in section \ref{subsection:glue}, using $g$,
we can build a germ of analytic
manifold ${\mathbf M}_g$ and a holomorphic fibration
$\pi:{\mathbf M}_g\rightarrow (\C,0)$. But now there is a new
ingredient: we have a holomorphic automorphism
$\Theta_0:(\C,0)\times\C^n\rightarrow (\C,0)\times\C^n$ of the
trivial bundle $\bigl((\C,0)\times\C^n,\pi_0,(\C,0)\bigr)$ defined
by:
$$
(z,X)\mapsto (qz,A_0 X).
$$
It induces, for all $i\in I$ a holomorphic automorphism of fibre
bundles $\Theta_{0,i}: U_i\times\C^n\rightarrow U_i$.

By definition the $g_{ij}\in \Lambda_I(M_0)(U_{ij})$ commute with
$\Theta_0$, therefore the automorphisms $\Theta_{0,i}$ glue together
into a holomorphic automorphism $\Psi$ of the germ of fibered manifold
$\bigl({\mathbf M}_g,\pi, (\C,\{ 0\})\bigr)$. The map  $\Psi$ is
linear on the fibers and the map on the basis is the germ of
$z\mapsto qz$.

We choose a holomorphic trivialisation $H$ of the germ
of fiber bundle $\bigl({\mathbf M}_g,\pi,(\C,0)\bigr)$:
$\pi_0\circ H = \pi$. The map $\Phi :=H\circ\Psi\circ H^{-1}$
is a holomorphic automorphism of the germ of trivial bundle
$\bigl((\C,0)\times\C^n,\pi_0,(\C,0)\bigr)$, it corresponds
to an element $A\in \GL_n\bigl(\Ra\bigr)$. Using the results
and notations of section \ref{subsection:glue}, we see that there exists
an element $\hat{\varphi} \in \GL_n\bigl(\Ra\bigr)$ such that
$g \mapsto \hat{\varphi}^{-1}$ induces the inverse of the natural map
$\GL_n\bigl(\Rf\bigr)/ \GL_n\bigl(\Ra\bigr)\rightarrow
H^1\bigl(\Eq;\Lambda_I\bigr)$. If we set $\hat F := \hat\varphi^{-1}$,
it is easy to check that $\hat F[A_0]=A$. We have $\hat F \in \GL^{M_0}_n$
and $\lambda(\hat F)=g$, therefore $\lambda$ is surjective.
\end{proof}



\chapter{Summation and asymptotic theory}
\label{chapter:asymptotictheory}

To find analytic solutions having a given formal solution as 
asymptotic expansion is a rather old subject in the theory of
$q$-difference equations: see for instance the paper \cite{Tr}
by Trjitzinsky in 1933. The first author had already suggested
in \cite{RamisPanorama} the use of the gaussian function to
formulate a $q$-analogue for Laplace transform; 
\index{$q$-analogue for Laplace transform}
this allowed to find in \cite{zh1} the 
$Gq$-summation,
\index{$Gq$-summation}
a $q$-analog of the exponential summation method of Borel-Laplace. 
\index{$q$-analog of the exponential summation method of Borel-Laplace}
In \cite{MarotteZhang},
the $Gq$-summation was extended to the case of multiple levels,
and it was shown that any formal power series solution of a linear
equation with analytic coefficients is $Gq$-multisummable, 
and thus can be seen as the asymptotic expansion of an analytical 
solution in a sector with infinite opening in the Riemann surface of 
the logarithm. The work in \cite{rz} \cite{RSZ2}, \cite{zh2}, \cite{zh3}
was undertaken with the goal of obtaining a summation over the
elliptic curve, or a finite covering of it\footnote{To be more precise:
note \cite{RSZ2} can be seen as an abridged version of the present
chapter. It summarizes and completes \cite{rz}, the departure point,
which announces the first part of \cite{zh3}. These first works dealt
with divisors supported at a single point, and \cite{zh2} started the
extension to more general divisors (and more slopes).}.


\section{Some preparatory notations and results} 
\label{section:preparatorymaterial}

In this section, we will introduce some notations related to
elliptic curves $\Eq$ and to Jacobi theta function.
\index{elliptic curve $\Eq$}
\index{Jacobi theta function}
\index{theta function}


\subsection{Divisors  and sectors on 
the elliptic curve $\Eq  =  \C^*/q^\Z$}
\label{subsection:divisorsonEq}

The projection $p: \C^{*} \rightarrow \Eq$ and the discrete logarithmic 
$q$-spiral $[\lambda;q]$ 
\index{$q$-spiral (discrete logarithmic)}
were defined in the general notations, section
\ref{section:generalnotations}. In this chapter, we shall usually 
shorten the latter notation into $[\lambda] := [\lambda;q]$. We call 
\emph{divisor} a finite formal sum of weighted such $q$-spirals:
\index{divisors on $\Eq$}
\label{not57}
$$
\Lambda = \nu_1 [\lambda_1] + \cdots + \nu_m [\lambda_m], 
\text{~where~} \nu_i\in \N^* \text{~and~}
[\lambda_i] \neq [\lambda_j] \text{~if~} i \neq j.
$$
This can be identified with the \emph{effective} divisor
\index{effective divisor}
$\sum \nu_{i} \left[\overline{\lambda_{i}}\right]$ on $\Eq$ 
and, to simplify notations, we shall also write 
$\sum \nu_i [\lambda_i]$ that divisor on $\Eq$. 
The \emph{support} of $\Lambda$ is the union
\index{support of a divisor}
of the $q$-spirals $[\lambda_1],\ldots,[\lambda_m]$. 
We write:
\label{not58}
$$
\lmod \Lambda \rmod := \nu_1  + \cdots + \nu_m 
$$
the \emph{degree} of $\Lambda$ (an integer) and:
\index{degree of a divisor}
$$
\Vert \Lambda \Vert := 
(-1)^{|\Lambda|}\lambda_1^{\nu_1} \cdots \lambda_m^{\nu_m} \pmod{q^\Z}
$$
the \emph{weight} of $\Lambda$, an element of $\Eq$. It is
\index{weight of a divisor}
equal, in additive notation, to $\sum \nu_{i} p(- \lambda_{i})$
evaluated in $\Eq$. \\ 
\label{not59}
For any two non zero complex numbers $z$ and $\lambda$, we put:
$$
d_q(z,[\lambda]) := \inf_{\xi \in [\lambda]} \lmod 1-\frac z\xi \rmod.
$$
This is a kind of ``distance'' from $z$ to the $q$-spiral
$[\lambda]$: one has $d_q(z,[\lambda]) = 0$ if, and only if    
$z \in [\lambda]$. 

\begin{lemm}
\label{lemm:a1}
Let $\Vert q \Vert_1 := \inf\limits_{n\in\Z^*}|1-q^n|$ and 
$M_q := \frac{\Vert q\Vert_1}{2+\Vert q\Vert_1}$. 
Let $a \in \C$ be such that $\lmod 1-a \rmod \leq M_q$. Then: 
$$
d_q(a;[1]) = \lmod 1 - a \rmod.
$$
\end{lemm}
\begin{proof} 
This follows at once from the inequalities: 
$$
\lmod 1-aq^n \rmod \geq 
\lmod a \rmod ~ \lmod 1-q^{n} \rmod - \lmod 1-a \rmod \geq 
(1-M_q) \Vert q\Vert_1 - \lmod 1-a \rmod\ \geq \lmod 1 - a  \rmod.
$$
\end{proof} 

For $\rho > 0$, we put:
\label{not60}
\begin{align*}
D([\lambda];\rho) &:= \{z\in\C^* \tq d_q(z,[\lambda])\leq \rho\}, \\
D^c([\lambda];\rho) &:= \C^*\setminus D([\lambda];\rho),
\end{align*}
so that $D([\lambda];1)=\C^*$ and  $D^c([\lambda];1) = \emptyset$.
From lemma \ref{lemm:a1}, it follows that for any $\rho < M_q$:
$$
D([\lambda];\rho) = \bigcup_{n\in\Z} q^n D(\lambda;|\lambda|\rho),
$$
where $D(\lambda;|\lambda|\rho)$ denotes the closed disk with center
$\lambda$ and radius $|\lambda|\rho$. \\

Let $\lambda,\mu\in\C^*$; since:
$$
\lmod 1 - \dfrac z\mu q^n \rmod \geq 
\lmod \dfrac{z}{\lambda} \rmod ~ \lmod 1-\dfrac{\lambda}{\mu} q^n \rmod
- \lmod 1- \dfrac{z}{\lambda} \rmod,
$$
one gets:
\begin{equation}
\label{eqn:lambdamuz}
d_q(z,[\mu]) \geq d_q(\lambda,[\mu]) - \rho \bigl(1+d_q(\lambda,[\mu])\bigr)
\end{equation}
when $z\in D([\lambda];\rho)$ and $\rho < M_q$.

\begin{prop}
\label{prop:lambdamu}
Given two distinct $q$-spirals $[\lambda]$ and $[\mu]$, there exists 
a constant $N > 1$ such that, for all $\rho > 0$ near enough from $0$, 
if $d_q(z,[\lambda]) > \rho$ and $d_q(z,[\mu]) > \rho$, then:
$$
d_q(z,[\lambda]) \, d_q(z,[\mu]) > \dfrac\rho N \cdot
$$
\end{prop}
\begin{proof} 
By contradiction, assume that, for all $ N > 1$, there exists 
$\rho_N > 0$ and $z \in \C^*$ such that $d_q(z,[\lambda]) > \rho_N$ 
and $d_q(z,[\mu]) > \rho_N$ but 
$d_q(z,[\lambda]) \, d_q(z,[\mu]) \leq \dfrac{\rho_N}{N}$; this would
entail:
$$
\rho_N < d_q(z,[\lambda]) < \dfrac{1}{N},
$$
whence, after \eqref{eqn:lambdamuz}:
$$
d_q(z,[\mu])\geq d_q(\lambda,[\mu])-\bigl(1+d_q(\lambda,[\mu])\bigr)/N.
$$
In other words, when $N\to \infty$, one would get:
$$
d_q(z,[\lambda]) \, d_q(z,[\mu]) \gtrsim d_q(\lambda,[\mu])~\rho_N
$$  
(\ie\ there is an inequality up to an $o(1)$ term),
contradicting the assumption
$d_q(z,[\lambda])~d_q(z,[\mu]) < \frac{\rho_N} N$. 
\end{proof} 

Let $\Lambda := \nu_1[\lambda_1] + \cdots + \nu_m[\lambda_m]$ be
a divisor and let $\rho>0$. We put:
\label{not61}
\begin{equation}
\label{eqn:dzLambda}
d_q(z,\Lambda) := \prod_{1\le j\le m}\left(d_q(z,[\lambda_j])\right)^{\nu_j}
\end{equation} 
and
\begin{equation*}
D(\Lambda;\rho) := \{z\in\C^* \tq d_q(z,\Lambda)\le \rho\}~,\quad
D^c(\Lambda;\rho) := \C^* \setminus D(\Lambda;\rho).
\end{equation*}

Proposition \ref{prop:lambdamu} implies that, if
$\Lambda = [\lambda] + [\mu]$ with $[\lambda] \neq [\mu]$, one has:
$$
D^c([\lambda];\rho) \cap D^c([\mu];\rho)
\subset D^c\left(\Lambda;\dfrac{\rho}{N}\right)
$$
as long as $\rho$ is near enough from $0$.

\begin{prop}
\label{prop:Lambdad}
Let  $\Lambda := \nu_1[\lambda_1] + \cdots + \nu_m[\lambda_m]$ be 
a divisor such that $[\lambda_i] \neq [\lambda_j]$ for $i \neq j$. 
There exists a constant $N>1$ such that, for all $\rho>0$ near 
enough from $0$, one has:
\begin{equation}
\label{eqn:Dc}
D^c(\Lambda;\rho) \subset
\bigcap_{1\le j\le m} D^c([\lambda_j];\rho^{1/\nu_j}) \subset 
D^c\left(\Lambda;\dfrac{\rho}{N}\right).
\end{equation}
\end{prop}
\begin{proof}
The first inclusion comes from the fact that, if
$d_q(z,[\lambda_j]) \leq \epsilon^{1/\nu_j}$ for some index $j$, 
then $d_q(z,\Lambda) \leq \epsilon$, because 
$d_q(z,[\lambda])\leq 1$ for all $\lambda \in \C^*$.

To prove the second inclusion, we use \eqref{eqn:lambdamuz} in 
the same way as in the proof of proposition \ref{prop:lambdamu};
details are left to the reader.
\end{proof} 

Taking the complementaries in \eqref{eqn:Dc}, one gets: 
\begin{equation}
\label{eqn:D}
D\left(\Lambda;\frac\rho N\right) \subset
\bigcup_{1\le j\le m} D\left([\lambda_j];\rho^{1/\nu_j}\right) 
\subset D(\Lambda;\rho).
\end{equation}
Modifying the definition \eqref{eqn:dzLambda} of
$d_q(z,\Lambda)$ as follows:
\begin{equation}
\label{not62}
\delta(z,\Lambda) := \min_{1\le j\le m}\{d_q(z,[\lambda_j])\}^{\nu_j},
\end{equation}
one gets, for all small enough $\rho$:
$$
D_\delta(\Lambda;\rho) := \{z\in\C^*: \delta(z,\Lambda)\le\rho\} =
\bigcup_{1\le j\le m} D([\lambda_j];\rho^{1/\nu_j})
$$
Relation \eqref{eqn:D} shows that the maps
$z \mapsto d_q(z,\Lambda)$ and $z\mapsto \delta(z,\Lambda)$ 
define equivalent systems of small neighborhoods 
relative to the divisor $\Lambda$; precisely, for $\rho \to 0$:
$$
D\left(\Lambda; \dfrac{\rho}{N}\right) \subset 
D_\delta(\Lambda;\rho) \subset D(\Lambda;\rho).
$$

\begin{defi}
\label{defi:secteur}
Let $\epsilon \geq 0$. We call \emph{germ of sector within
distance $\epsilon$ from divisor $\Lambda$} any set:
\index{germ of sector within distance $\epsilon$ from divisor $\Lambda$}
$$
\label{not63}
S(\Lambda,\epsilon;R) =
\{z\in\C^*: d_q(z,\Lambda)>\epsilon, \lmod z \rmod < R\},
\text{~where~} R > 0.
$$
When $R = +\infty$, we write $S(\Lambda,\epsilon)$ instead of
$S(\Lambda,\epsilon;+\infty)$.
\end{defi}

Note that $S(\Lambda,0;R)$ is the punctured open disk 
$\{0 < \lmod z \rmod < R\}$ deprived of the $q$-spirals
pertaining to divisor $\Lambda$. Since $d_q(z,\Lambda) \leq 1$ 
for all $z\in\C^*$, one has $S(\Lambda,1;R) = \emptyset$;
that is why we shall assume that $\epsilon\in]0,1[$.
Last, each sector within a short enough distance to $\Lambda$
represents, on the elliptic curve $\Eq = \C^*/q^\Z$, the curve
$\Eq$ deprived of a family of ovals centered on the $\lambda_j$
with sizes varying like $\epsilon^{1/\nu_{j}}$. \\

By a \emph{(germ of) analytic function in $0$ out of $\Lambda$},
\index{analytic function in $0$ out of $\Lambda$ (germ of)}
we shall mean any function defined and analytic in some germ of
sector $S(\Lambda,0;R)$ with $R>0$.  Such a function $f$ is said
\index{bounded (at $0$)}
to be \emph{bounded} (at $0$) if, for all $\epsilon >0$ and all 
$r \in ]0,R[$, one has:
$$
\sup_{z\in S(\Lambda,\epsilon;r)} |f(z)| < \infty.
$$
\label{not64}
We write $\BB^{\Lambda}$ the set of such functions.

\begin{prop}
\label{prop:BB}
If $f \in \BB^\Lambda$, then for any $R > 0$ small enough and for all 
$a \in S(\Lambda,0;R)$ such that 
$\lmod a \rmod \notin 
\bigcup\limits_{n\in\N} \lmod \lambda_i \rmod \lmod q^{-n} \rmod$, 
there exists $K > 0$ such that
$$
\sup_{n\in\N} ~ \max_{\lmod z \rmod=|aq^{-n}|} \lmod f(z) \rmod < K.
$$
\end{prop}

\begin{proof}
From the assumption on $a$, it follows that:
$$
\epsilon := \frac{\max_{\tau\in[0,2\pi]}d_q(ae^{i\tau},\Lambda)}2 > 0,
$$
whence all circles centered at $0$ and with respective radii 
$|a q^{-n}|$, $n \in \N$ are contained in the sector 
$S(\Lambda,\epsilon;R)$, which concludes the proof.
\end{proof}


\subsection{Jacobi Theta function, $q$-exponential growth 
and $q$-Gevrey series}
\label{subsection:JacobiTheta}
\index{Jacobi theta function}
\index{theta function}
\index{$q$-Gevrey series}

Jacobi's theta function $\theta(z;q) = \theta(z)$ was defined by equation
\eqref{eqn:tripleproduit} page \pageref{eqn:tripleproduit}.
From its functional equation, one draws:
\begin{equation*}
\forall (z,n) \in \C^* \times \Z ~,~
\theta(q^nz)=q^{n(n+1)/2}~z^{n}~\theta(z).
\end{equation*}
Moreover, for all $z\in\C^*$, one has:
$$
\label{not65}
\lmod \theta(z) \rmod \leq 
e(z) := e(z;q) :=\theta\left(\lmod z\rmod;\lmod q \rmod\right).
$$

\begin{lemm}
\label{lemm:thetamin0}
There exists a constant $C>0$, depending on $q$ only, such that,
for all $\epsilon>0$:
$$
\lmod \theta(z) \rmod \geq {C}~ {\epsilon} ~ e(z),
$$
as long as $d_q(z,[-1])>\epsilon$.
\end{lemm}
\begin{proof} 
It is enough to see that the function 
$z \mapsto \lmod \theta(z) \rmod/e(z)$ is $q$-invariant, which 
allows one to
restrict the problem to the (compact) closure of a fundamental
annulus, for instance
$\{z \in \C \tq 1\leq \lmod z \rmod \leq \lmod q \rmod\}$. 
See \cite[Lemma 1.3.1]{zh3} for more details.
\end{proof} 

Recall that $\theta(z) = \theta\left(\dfrac{1}{qz}\right)$; 
one draws from this that $e(z) = e\left(\dfrac{1}{qz}\right)$. 
To each subset $W$ of $\C$, one associates two sets $W_{(\infty)}$, 
$W_{(0)}$ as follows:
\label{not66}
$$
W_{(\infty)} := \bigcup_{n\ge 0} q^n~ W~,\quad
W_{(0)} := \bigcup_{n\le 0}q^n~ W~.
$$

\begin{defi}
\label{defi:croissanceqexp}
Let  $f$ be a function analytic over an open subset $\Omega$ 
of $\C^*$ and let $k \geq 0$. 
\begin{enumerate}
\item{We say that \emph{$f$ has $q$-exponential growth of order 
(at most) $k$ at infinity (resp. at $0$) over $\Omega$} if,
\index{$q$-exponential growth of order (at most) $k$}
for any compact $W \subset \C$ such that $W_{(\infty)} \subset \Omega$ 
(resp. such that $W_{(0)} \subset \Omega$), there exists $C > 0$ and
$\mu > 0$ such that $\lmod f(z) \rmod \leq C \bigl(e(\mu z)\bigr)^k$ 
for all $z \in W_{(\infty)}$ (resp. $z\in W_{(0)}$).}
\item{We say that \emph{$f$ has $q$-exponential decay of order 
(at least) $k$ at infinity (resp. at $0$) over $\Omega$} if, 
\index{$q$-exponential decay of order (at least) $k$}
for any compact $W \subset \C$ such that $W_{(\infty)} \subset \Omega$ 
(resp. such that $W_{(0)} \subset \Omega$), there exists  $ C> 0$ and 
$\mu > 0$ such that $\lmod f(z) \rmod \leq C \bigl(e(\mu z)\bigr)^{-k}$
for all $z \in W_{(\infty)}$ (resp. $z \in W_{(0)}$).}
\end{enumerate}
\end{defi} 

\begin{lemm}
\label{lemm:croissanceqexp}
Let $k > 0$ and $f$ an entire function with Taylor expansion
$f(z) = \sum\limits_{n\ge 0}a_nz^n$ at $0$. Then $f$ has $q$-exponential 
growth of order (at most) $k$ at infinity over $\C$, if, and only if,
the sequence $(a_n)$ is $q$-Gevrey of order $(-1/k)$.
\end{lemm}
\begin{proof}  
The $q$-Gevrey order of a sequence was defined in paragraph 
\ref{subsection:qGevreylevels} of the general notations, in the
introduction. A variant of this result was given in 
\cite[Proposition 2.1]{RamisGrowth}. The proof shown below
rests on the functional equation $\theta(qz) = q~ z~ \theta(z)$. 

If $(a_n)$ is $q$-Gevrey of order $(-k)$, one has 
$\lmod f(z) \rmod \leq C~\theta(\mu \lmod z \rmod; \lmod q \rmod^{1/k})$, 
which shows that $f$ has $q$-exponential growth of order (at most) $k$ 
at infinity over $\C$. 

On the other hand, by Cauchy formula, one has:
$$
\lmod a_n \rmod \leq 
\min_{m\in\Z} 
\left(\max_{\lmod z \rmod=\mu \lmod q \rmod^m} \lmod f(z)\rmod (\mu ~ \lmod q \rmod^{m})^{-n}\right),
$$
where $\mu > 0$ is an arbitrary parameter. If $f$ is an entire function
with $q$-exponential growth of order (at most) $k$ at infinity, one has
$\lmod f(z) \rmod \leq C~\bigl(e(z)\bigr)^k$ for all $z\in\C^*$. From
relation $e(q^mz)=\lmod q \rmod^{m(m+1)/2}~\lmod z\rmod^m~e(z)$, one draws,
for all integers $n \geq 0$:
$$
\lmod a_n \rmod \leq C~e^k(\mu)~\mu^{-n}~
\bigl(\mu^k~\lmod q \rmod^{-n+k/2}\bigr)^m~\lmod q \rmod^{km^2/2}.
$$
 
Last, we note that the function 
$t\mapsto A^t \lmod q \rmod^{kt^2/2}$ reaches its minimum at 
$t = -\dfrac{\ln A}{k\ln \lmod q \rmod}$, with:
$$
\min_{t\in\R} A^t \lmod q \rmod^{kt^2/2} = e^{-\frac{\ln^2A}{2k\ln \lmod q \rmod}}. 
$$
From this, one deduces:
\begin{equation}
\label{eqn:minqgevrey}
\min_{m\in\Z} A^m \lmod q\rmod^{km^2/2} \leq 
\lmod q \rmod^{\frac k{8}}~e^{-\frac{\ln^2A}{2k\ln \lmod q \rmod}}
\end{equation}
for all $A > 0$. Therefore, if we set $A=\mu^k |q|^{-n+k/2}$, we get the 
wanted $q$-Gevrey estimates for $|a_n|$, which concludes the proof.
\end{proof} 


\section{Asymptotics relative to a divisor}
\label{section:asymptoticsreldivisor}

In this section, we fix a divisor 
$\Lambda := \nu_1[\lambda_1] + \cdots +\nu_m[\lambda_m]$ and 
we write $\nu := \lmod \Lambda \rmod = \nu_1 + \cdots + \nu_m$. 
Without loss of generality, we make the following \\

\noindent
\textbf{ASSUMPTION:} \\
The complex numbers $\lambda_i$ are pairwise distinct and
such that:
\begin{equation}
\label{hypothese:lambdas}
1\leq \lmod \lambda_1 \rmod \leq \lmod \lambda_2 \rmod \leq 
\cdots \leq  \lmod \lambda_m \rmod < \lmod q\lambda_1 \rmod.
\end{equation}

\begin{defi}
\label{defi:ALambda} 
Let $f \in \BB^{\Lambda}$. We write\footnote{The set $\AA_{q}^{\Lambda}$
was previously written $\AA_{q;|\Lambda|}^{\Lambda}$ in our Note \cite{RSZ2}.} 
$f \in \AA_{q}^{\Lambda}$ 
\label{not67}
and we say that $f$ admits a \emph{$q$-Gevrey
asymptotic expansion of level $\nu$ (or order $s := 1/\nu$) along divisor 
$\Lambda$ at $0$}, 
\index{$q$-Gevrey asymptotic expansion along a divisor}
if there exists a power series $\sum\limits_{n \ge0} a_n z^n$ and constants 
$C$, $A > 0$ such that, for any $\epsilon > 0$ and any integer $N \ge 0$, 
one has, for some small enough $R > 0$ and for all 
$z\in S(\Lambda,\epsilon;R)$:
\begin{equation}
\label{eqn:fN}
\lmod f(z) - \sum_{0\le n<N} a_n z^n \rmod <
\dfrac{C}{\epsilon} A^N \lmod q \rmod^{N^2/(2\nu)} \lmod z \rmod^N~.
\end{equation}
\end{defi}

Recall that  condition $z\in S(\Lambda,\epsilon;R)$ means that
$0 < \lmod z \rmod < R$ and $d_q(z,\Lambda) > \epsilon$; see 
definition  \ref{defi:secteur}. One sees at once that, 
if $f \in \AA_{q}^{\Lambda}$, its asymptotic expansion, written
$J(f)$, belongs to the space of $q$-Gevrey series of level $\nu$, 
whence the following linear map:
$$
\label{not68}
J: \quad \AA_{q}^{\Lambda} \rightarrow \Rf_{q;1/\nu}, \quad f \mapsto J(f).
$$
(For the notation $\Rfs$, see paragraph \ref{subsection:qGevreylevels}
of the general notations, in the introduction.)
In section \ref{section:BorelRitt}, we shall see that this map is onto,
providing a "meromorphic $q$-Gevrey version" of classical Borel-Ritt
theorem. In order to get the surjectivity of $J$, we shall obtain
a characterization of $\AA_{q}^{\Lambda}$ in terms of the residues along 
poles belonging to the divisor $\Lambda$. This description will be
shared among the following two paragraphs.


\subsection{Asymptotics and residues (I)}
\label{subsection:residusI}
\index{residues}

When $\nu=1$, divisor $\Lambda$ reduces to $[\lambda]$ and the above 
definition coincides with the definition of $\AA_{q;1}^{[\lambda;q]}$ in
\label{not69}
\cite{zh3}, \cite{rz}. After \cite{rz}, one has $f \in \AA_q^{[\lambda]}$ 
if, and only if, there exists an integer $N \in \N$, a $q$-Gevrey
sequence $(c_n)_{n\ge N}$ of order $(-1)$ and a function $h$ holomorphic
near $z=0$ such that, for any $z\in S([\lambda],0;|\lambda q^{-N}|)$:
\begin{equation*}
f(z) = \sum_{n\ge N} ~ \dfrac{c_n}{z-q^{-n}\lambda} + h(z).
\end{equation*}

In the sequel, we shall give a caracterisation of all elements of
$\AA_{q}^{\Lambda}$ with help of its partial fraction decomposition
\index{partial fraction decomposition}
along divisor $\Lambda$.

\begin{theo}
\label{theo:fdecomposition} 
Let $f\in\BB^{\Lambda}$. The following assertions are equivalent:
\begin{description}
\item[(i)]{One has $f \in \AA_{q}^{\Lambda}$.}
\item[(ii)]{There exists an integer $N_0$ and $\nu$ sequences 
$(\alpha_{i,j,n})_{n\ge N_0}$ ($1 \leq i \leq m$, $0 \leq j< \nu_i$)
of $q$-Gevrey level $(-\nu)$ and a function $h$ holomorphic at 
$0 \in \C$ such that the following equality holds over the sector 
$S(\Lambda,0;|q^{-N_0}\lambda_m|)$:
\begin{equation}
\label{eqn:fLambda}
f(z) = \sum_{n\ge N_0} \sum_{1\le i\le m} \sum_{0\le j<\nu_i}
\dfrac{\alpha_{i,j,n}}{(z-\lambda_i q^{-n})^{j+1}} + h(z).
\end{equation}
}
\end{description}
\end{theo}
\begin{proof}  
Recall that $\nu := \lmod \Lambda \rmod \in \N^{*}$.
Let $f \in \BB^{\Lambda}$. Under assumption \eqref{hypothese:lambdas},
page \pageref{hypothese:lambdas} about the $\lambda_i$, we can choose
\label{pourrefplusloin}
a real $R>0$ and an integer $N_0$ such that, on the one hand, one has
$\lmod q^{-N_0}\lambda_m \rmod < R < \lmod q^{-N_0+1}\lambda_1 \rmod$
and, on the other hand, $f$ is defined and analytic in the open 
sector $S(\Lambda,0;R)$. Therefore, the only possible singularities 
of $f$ in the punctured disc $0 < \lmod z \rmod < R$ belong to the 
half $q$-spirals $\lambda_i q^{-N_0-\N}$, $1 \le i \le m$. \\

\noindent
\underline{\textbf{Proof of (i) $\Rightarrow$ (ii):}}
It will rest on lemmas \ref{lemm:poles}, \ref{lemm:eta} 
and \ref{lemm:h}.

\begin{lemm}
\label{lemm:poles}
We keep the above notations $R$, $N_0$, and moreover assume that
$f \in \AA_q^\Lambda$. Then $f$ has a pole with multiplicity at 
most $\nu_i$ at each point of the half $q$-spiral $\lambda_i q^{-N_0-\N}$.
\end{lemm}
\begin{proof}
Let $n \geq N_0$ be an integer. When $z$ tends to 
$z_{i,n} := \lambda_i q^{-n}$, one has 
$d_q(z,\Lambda) = \lmod 1- \dfrac{z}{z_{i,n}} \rmod^{\nu_i} \to 0$. 
Taking $\epsilon := \dfrac{d_q(z,\Lambda)}{2}$ and $N := 0$ in
relation \eqref{eqn:fN}, definition \ref{defi:ALambda},
one finds:
$$
\lmod f(z) \rmod < \dfrac{2C}{\lmod 1-\frac z{z_{i,n}}\rmod^{\nu_i}},
$$
which allows one to conclude.
\end{proof} 

Write $P_{i,n}$ the polar part of $f$ at point $\lambda_i q^{-n}$:
\begin{equation}
\label{eqn:polaire}
P_{i,n}(z) = \sum_{0\le j<\nu_i} 
\dfrac{\alpha_{i,j,n}}{(z-\lambda_{i} q^{-n})^{j+1}} \cdot
\end{equation}
We are going to study the growth of coefficients with respect to
index $n$ when $n$ goes to infinity. For that, we choose some small
$\eta \in \left(0,\frac {\lmod q \rmod}{1+\lmod q \rmod}\right)$ so that, 
putting $r_{i,n} := \lmod \lambda_i q^{-n} \rmod \eta$, all the open disks
\label{not70}
$D^\eta_{i,n} := \{z\in\C \tq \lmod z - \lambda_i q^{-n} \rmod <r_{i,n}\}$ 
($1 \leq i \leq m$, $n \geq N_0$) are pairwise disjoint. Also write
$\CC^\eta_{i,n} = \partial^+ D^\eta_{i,n}$ 
the boundary of disk $D^\eta_{i,n}$, with positive orientation.

\begin{lemm}
\label{lemm:eta} 
Let $\RR_N := f(z) - \sum\limits_{k=0}^{N-1} a_n z^n$, the $N$-th 
remainder of $f$. For all $z \in \CC^\eta_{i,n}$, one has:
\begin{equation}
\label{eqn:RN}
\lmod \RR_N(z) \rmod < 
\dfrac{2C}{\eta^\nu}~\left(2A \lmod \lambda_i q^{-n} \rmod \right)^N
\lmod q \rmod^{N^2/(2\nu)}~,
\end{equation}
where $C$ and $A$ are the constants in relation \eqref{eqn:fN} 
of definition \ref{defi:ALambda}.
\end{lemm}
\begin{proof} 
Since $\eta < \frac {\lmod q \rmod}{1+\lmod q \rmod}$, one has $d_q(z,[\lambda_i]) = \eta$ 
and $d_q(z,[\lambda_j]) \geq \eta$ for $j \neq i$. One deduces that
$d_q(z,\Lambda) \geq \eta^{\max(\nu_1,\ldots,\nu_m)} \geq \eta^\nu$, which
allows one to write, taking $\epsilon := \dfrac{\eta^\nu}{2}$ in
\eqref{eqn:fN}:
$$
\lmod \RR_N(z) \rmod < 
\dfrac{2C}{\eta^\nu}~A^N~ \lmod q \rmod^{N^2/(2\nu)}~ \lmod z\rmod^N.
$$
Relation \eqref{eqn:RN} follows, using the fact that
$\lmod z \rmod \leq (1+\eta) \lmod \lambda_i q^{-n} \rmod <
2 \lmod \lambda_i q^{-n} \rmod$.
\end{proof} 

We now express the coefficients $\alpha_{i,j,n}$ in 
\eqref{eqn:polaire} with the help of Cauchy formula
in the following way:
$$
\alpha_{i,j,n} = 
\dfrac{1}{2 \ii \pi} \oint_{\CC_{i,n}^\eta} f(z)~(z-\lambda_i q^{-n})^j ~dz,
$$
that is, for any integer $N \geq 0$:
$$
\alpha_{i,j,n}=
\frac1{2\ii \pi} 
\oint_{\CC_{i,n}^\eta}\RR_N(z)~(z-\lambda_i q^{-n})^j ~dz.
$$
Taking in account estimation \eqref{eqn:RN} above entails: 
$$
\lmod \alpha_{i,j,n} \rmod \leq 
2 C~\eta^{j+1-\nu}~\lmod \lambda_i q^{-n} \rmod^{j+1}~X^N~Q^{N^2/2},
$$
where $X := 2A \lmod \lambda_iq^{-n} \rmod$ and $Q := \lmod q \rmod^{1/\nu}$. 
Relation \eqref{eqn:minqgevrey} implies that, if $X < 1$, then:
$$
\min_{N\in\N} \bigl(X^N~Q^{N^2/2}\bigr) \leq ~Q^{1/8}~e^{-\frac{\ln^2X}{2\ln Q}},
$$ 
which implies that:
$$
\lmod \alpha_{i,j,n} \rmod 
\leq 2 C \eta^{j+1-\nu} \,\lmod \lambda_i \rmod^{j+1}\,
\lmod q \rmod^{1/(8\nu)-(j+1)n}\,
e^{-\frac{\nu \ln^2 (2 A |\lambda_i q^{-n}|)}{2\ln |q|}},
$$
so that all the sequences $(\alpha_{i,j,n})_{n\ge N_0}$ 
are $q$-Gevrey of level $(-\nu)$. Thus, for each pair $(i,j)$
of indexes, the series:
$$
S_{i,j}(z) :=\sum_{n\ge N_0} \dfrac{\alpha_{i,j,n}}{(z-\lambda_iq^{-n})^{j+1}}
$$
converges normally on each compact subset of the sector $S(\Lambda,0)$. 

Last, if $0 < \lmod z \rmod < R$ and if
$z \notin \bigcup\limits_{1\le i\le m}\lambda_iq^{-N_0+\N}$, 
let us put:
$$
h(z) := f(z) - \sum_{1\le i\le m} \sum_{0\le j<\nu_i} S_{i,j}(z).
$$
From normal convergence of the series $S_{i,j}(z)$ on every compact
subset of $S(\Lambda,0)$, 
we draw:
\begin{equation}
\label{eqn:hP}
h(z) = f(z) - \sum_{n\ge N_0} \sum_{1\le i\le m} P_{i,n}(z).
\end{equation}

\begin{lemm}
\label{lemm:h}
The function $h$ represents the germ of an analytic function at $z=0$. 
\end{lemm}
\begin{proof} 
Since all non zero singularities of $f$ in the disk
$\lmod z \rmod < R$ are necessarily situated on the union
of half $q$-spirals $\bigcup\limits_{1\le i\le m}\lambda_i q^{-N_0+\N}$
and each $P_{i,n}(z)$ represents its polar part at each point 
$\lambda_iq^{-n}$ of the latter, the function $h$ has an analytic
continuation to the punctured disk $0 < \lmod z \rmod < R$. 
Moreover, if $\Omega := \bigcup\limits_{n>N_0}\CC(0;R\lmod q \rmod^{-n})$ 
denotes the union of the circles centered at $z=0$ and with
respective radii $R\lmod q \rmod^{-n}$ with $n > N_0$, all the series $S_{i,j}(z)$ 
converge normally over $\Omega$, and therefore remain bounded on
that family of circles. Besides, from proposition \ref{prop:BB}, 
$f$ is also bounded on $\Omega$ because
$f\in\BB^\Lambda$. It follows that $h$ is bounded on
$\Omega$, which entails the analyticity of $h$ at $z=0$. 
\end{proof} 

Combining relations \eqref{eqn:polaire} and \eqref{eqn:hP}, 
we get decomposition \eqref{eqn:fLambda} of $f$, thus achieving
the proof of (i) $\Rightarrow$ (ii). \\

\noindent
\underline{\textbf{Proof of (ii) $\Rightarrow$ (i):}}
It will rest on lemma \ref{lemm:reste}.

\label{not71}
\begin{lemm}
\label{lemm:reste}
Let $j\in\N$, $N\in\N$ and write $\RR_{j,N}$ the rational
fraction defined by relation:
\begin{equation}
\label{eqn:resteN}
\dfrac{1}{(1-t)^{j+1}} =
\sum_{n=0}^{N-1} \binom{n+j}n ~ t^n + \RR_{j,N}(t).
\end{equation}
Let $\delta>0$. If $t\in\C$ satisfies $\lmod 1-t \rmod \geq \delta$,
then:
\label{not72}
\begin{equation}
\label{eqn:RjN}
\lmod \RR_{j,N}(t) \rmod < 
C_{j,N}^{\delta,r}~ \dfrac{\lmod t \rmod^N}{\delta^{j+1}},
\quad 
C_{j,N}^{\delta,r} =
\dfrac{1}{2-r}~
\bigl(\frac{1+r\delta}{r-1}\bigr)^j~
\bigl(\frac{1+r\delta}{1+\delta}\bigr)^N
\end{equation}
for all $r\in(1,2)$.
\end{lemm}
\begin{proof}
The relation \eqref{eqn:RjN} being trivially satisfied at
point $t = 0$, we assume that $t \neq 0$. Taking the $j$-th 
derivative for each side of relation 
$\dfrac{1}{1-t} = \sum\limits_{n=0}^{N+j-1} t^n + \dfrac{t^{N+j}}{1-t}$, 
one finds:
$$
\RR_{j,N}(t) =\dfrac{1}{j!}~\bigl(\frac{t^{N+j}}{1-t}\bigr)^{(j)} =
\dfrac{1}{2\ii \pi}~\oint_{\CC(t,\rho)}\frac{s^{N+j}}{1-s}~\frac{ds}{(s-t)^{j+1}}~,
$$
where $\CC(t,\rho)$ denotes the positively oriented boundary of
the circle centered at $s = t$ and with radius $\rho \lmod t\rmod$,
with $\rho < \lmod 1 - \dfrac{1}{t} \rmod$. One deduces that:
\begin{equation}
\label{eqn:Rrho}
\lmod \RR_{j,N}(t) \rmod \leq 
\dfrac{(1+\rho)^{N+j}~\lmod t \rmod^N}
{\bigl(\lmod 1-t \rmod -\rho \lmod t\rmod\bigr)~\rho^j},
\end{equation}
where $\rho$ denotes a real number located between $0$ and
$\lmod 1-\dfrac{1}{t} \rmod$.

If $\lmod 1 - t \rmod > \delta$, with $t = 1 + Re^{i\alpha}$ 
($\alpha\in\R$), one gets: 
$\lmod 1-\dfrac{1}{t} \rmod \geq 
\dfrac{R}{1+R} > \dfrac{\delta}{1+\delta}$. 
Choose $\rho := \dfrac{r'\delta}{1+\delta}$ with $r' \in (0,1)$, 
so that 
$\lmod 1-t \rmod - \rho \lmod t \rmod \geq
 \lmod 1-t \rmod (1-r') > (1-r')~\delta$;
relation \eqref{eqn:RjN} follows at once (along with lemma
\ref{lemm:reste}) with the help of \eqref{eqn:Rrho}, with 
$\rho := \dfrac{(r-1)~\delta}{1+\delta}$ and $r := r'+1 \in (1,2)$.
\end{proof} 

We shall now use lemma \ref{lemm:reste} to prove that
(ii) $\Rightarrow$ (i); by linearity, it suffices to check
that a series of the form:
\begin{equation}
\label{eqn:Salpha}
S(z) := \sum_{n\ge 0}\dfrac{\alpha_n}{(z-\lambda q^{-n})^{j+1}}
\end{equation}
defines a function in space $A_q^\Lambda$, where
$\lambda = \lambda_i$ for some $i\in\{1,2,\ldots,m\}$, 
$0 \leq j < \nu_i$ and where $(\alpha_n)$ denotes a $q$-Gevrey 
sequence of order $(-\nu)$: there are $C$, $A > 0$ such that:
\begin{equation}
\label{eqn:alphan}
\forall n \in \N ~,~
\lmod \alpha_n \rmod \leq C A^n \lmod q\rmod^{-\nu n(n-1)/2}.
\end{equation}

Let $z\in\C^*$ such that $d_q(z,\Lambda) > \epsilon$; we have
$d_q(z,[\lambda])^{\nu_i} >\epsilon$, hence
$\lmod 1-\dfrac{z}{\lambda q^{-n}} \rmod > \sqrt[\nu_i]\epsilon$. 
Let $N \in \N$ and apply to each term of the series $S(z)$ 
formula \eqref{eqn:resteN} with $t = \dfrac{z}{\lambda q^{-n}}$
to obtain formally $S(z) = S_N(z) + R_N(z)$, where one puts:
$$
S_N(z) :=
\sum_{n\ge 0} \sum_{k=0}^{N-1} \alpha_n \binom{k+j}{k}~
(-\lambda q^{-n})^{-j-1}~\left(\dfrac{z}{\lambda q^{-n}}\right)^k
$$
and:
$$
R_N(z) := \sum_{n\ge 0} \alpha_n~(-\lambda q^{-n})^{-j-1}~
\RR_{j,N}(\dfrac{z}{\lambda q^{-n}}).
$$
Since $(\alpha_n)$ has $q$-Gevrey decay of order $\nu$ (see 
\eqref{eqn:alphan}), we see that $S_N(z)$ and $R_N(z)$ are 
normally convergent series on any domain $\Omega$ such that 
$\inf\limits_{z\in\Omega} d_q(z,[\lambda]) > 0$. 

Define:
\begin{equation}
\label{eqn:Az}
{\aa}(z) := \sum_{n\ge 0} \alpha_n~z^n,
\quad
\tilde {\aa}(z) := \sum_{n\ge 0} \lmod \alpha_n \rmod ~z^n.
\end{equation}
From relation \eqref{eqn:alphan}, one gets:
$$
\lmod {\aa}(z) \rmod \leq \tilde {\aa}(\lmod z \rmod) \leq 
C~\sum_{n\ge 0} \lmod q\rmod^{-\nu n(n-1)/2}~\lmod Az \rmod^n <
C~e(Az;{\lmod q \rmod}^\nu).
$$
On the other hand, one can express the series $S_N(z)$ in the following
way: 
\begin{equation}
\label{eqn:an}
S_N(z) = \sum_{n=0}^{N-1} a_n z^n,\quad
a_n := (-\lambda)^{-j-1}\binom{n+j}n~{\aa}(q^{j+1+n})~\lambda^{-n}~;
\end{equation}
besides, the relation \eqref{eqn:RjN} implies that, putting
$\delta := \sqrt[\nu_i]\epsilon$: 
\begin{equation}
\label{eqn:RNz}
\lmod R_N(z) \rmod \leq K_{j,N}^{\delta,r}~\lmod z\rmod^N,\quad
K_{j,N}^{\delta,r} :=
\dfrac{C_{j,N}^{\delta,r}}{\lmod \lambda \rmod^{j+1+N}~\delta^{j+1}}~
\tilde{\aa}(\lmod q\rmod^{j+1+N})~,
\end{equation}
with $\tilde{\aa}(z)$ as in \eqref{eqn:Az} and $r \in (1,2)$. 

To sum up, we have seen that, if $d_q(z,\Lambda) > \epsilon > 0$ 
and $\delta = \sqrt[\nu_i]\epsilon$, then:
$$
\lmod S(z) - \sum_{n=0}^{N-1} a_n~z^n \rmod \leq K_{j,N}^{\delta,r}~\lmod z \rmod^N,
$$
the coefficients $a_n$ and $K_{j,N}^{\delta,r}$ being defined as in
\eqref{eqn:an} and \eqref{eqn:RNz} respectively. To conclude
that $S(z)$ admits $\sum\limits_{n\ge 0} a_n z^n$ as a $q$-Gevrey asymptotic
expansion of level $\nu$ in $\Lambda$ at $0$, it only remains to be
checked that one can find $C_0>0$, $A_0>0$ with:
\begin{equation}
\label{eqn:CA}
\min_{r\in(1,2)} K_{j,N}^{\delta,r} \leq 
\dfrac{C_0}{\epsilon}~A_0^N~\lmod q \rmod^{N^2/(2\nu)}
\end{equation}
for all $N\in\N$ and $j=0,\ldots,\nu_i-1$. To do that, note that,
after lemma \ref{lemm:croissanceqexp} and relation 
\eqref{eqn:minqgevrey}, one finds:
$$
\tilde{\aa}\left(\lmod q \rmod^{j+1+N}\right) \leq 
C'~\left(\lmod q \rmod^{\frac{j+1}\nu-\frac12}~A^{\frac1\nu}\right)^N~q^{N^2/(2\nu)}~,
$$
where $C'$ denotes a constant independant of $N$; this, along 
with definition \eqref{eqn:RNz} of $K_{j,N}^{\delta,r}$, entails
\eqref{eqn:CA} and we leave the details to the reader. The
proof of theorem \ref{theo:fdecomposition} is now complete.
\end{proof}


\subsection{Asymptotics and residues (II)}
\label{subsection:residusII}

After \cite{rz}, for any $f \in \AA_q^{[\lambda]}$, the function
$z \mapsto F(z) := \theta(-\frac z\lambda)~f(z)$ is analytic
in some punctured disk $0 < \lmod z \rmod < \lmod \lambda q^{-N} \rmod$, 
with (at most) a first order $q$-exponential growth when $z$ goes to $0$,
and such that its restriction to the $q$-spiral $[\lambda]$ has at
most a growth of $q$-Gevrey order $0$\footnote{That is, there exists 
constants $C,A > 0$ such that, for any integer $n \gg 0$, one has
$\lmod F(q^{-n}\lambda) \rmod \leq CA^n$.}. 
In order to extend this result to a general divisor $\Lambda$ 
(with degree $\nu > 1$), we introduce the following definition.

\begin{defi}
\label{defi:LambdaF} 
Let $F$ be a function defined and analytic in a neighborhood
of each of the points of the support of the divisor $\Lambda$ 
near $0$.
\begin{enumerate}
\item{We call \emph{values of $F$ on $\Lambda$ at $0$}, and
\index{values of $F$ on $\Lambda$ at $0$}
write $\Lambda F(0)$, the $\nu$ (germs at infinity of) sequences:
$$
\label{not73}
\Lambda F(0) :=
\left\{\left(F^{(j)}(\lambda_i q^{-n})\right)_{n \gg 0} ~:~
1\le i\le m, 0\le j<\nu_i\right\},
$$
where $F^{(j)}$ is the $j$-th derivative of $F$ (with $F^{(0)} = F$) 
and where "$n \gg 0$" means "$n$ great enough".}
\item{We say that $F$ has \emph{$q$-Gevrey order $k$ along divisor $\Lambda$ 
at $0$} 
\index{$q$-Gevrey order along a divisor}
if all its values are $q$-Gevrey sequences of order $k$.}
\item{We write $F \in \EE_0^{\Lambda}$ if $F$ is analytic in a
\label{not74}
neighborhood of $0$ punctured at $0$, has $q$-Gevrey order 
$\nu$ at $0$ and $q$-Gevrey order $0$ along $\Lambda$ at $0$.}
\end{enumerate}
\end{defi}

To simplify, we shall write:
\begin{equation}
\label{eqn:thetalambda}
\theta_\Lambda(z) := 
\prod_{j=1}^m\left(\theta\left(-\frac z{\lambda_j}\right)\right)^{\nu_j}~.
\end{equation}
Thus, if $f\in\BB^{\Lambda}$, $\theta_\Lambda f$ 
represents an analytic function in a punctured disk 
$0 < \lmod z \rmod < R$.

\begin{theo}
\label{theo:fF} 
Let $f \in \BB^{\Lambda}$ and $F := \theta_\Lambda f$. Then, 
$f \in \AA_{q}^{\Lambda}$ if, and only if, $F \in \EE_0^{\Lambda}$.
\end{theo}
\begin{proof}  
We prove separately the two implications. \\

\noindent
\underline{\textbf{Proof of 
$f \in \AA_q^\Lambda \Rightarrow F \in \EE_0^\Lambda$:}} 
It will rest on lemma \ref{lemm:theta'}.

\begin{lemm}
\label{lemm:theta'}
Let $\lambda \in \C^*$, $k \in \N^*$ and let  
$g(z) := (\theta_\lambda(z))^k$ for $z \in \C^*$. 
For $\ell \geq 0$, consider the $\ell$-th derivative
$g^{(\ell)}$ of $g$.
\begin{enumerate}
\item{The function $g^{(\ell)}$ has $q$-exponential growth
of order $k$ both at $0$ and infinity.}
\item{If $\ell <  k$, then $g^{(\ell)}(\lambda q^{-n}) = 0$ 
for all $n\in\N$; else, $\bigl(g^{(\ell)}(\lambda q^{-n})\bigr)_{n\ge 0}$ 
is a $q$-Gevrey sequence of order $k$.}
\item{Let $n \in \Z$ and put 
$u(z) := \dfrac{g(z)}{(z-\lambda q^{-n})^k}$ for all
$z \in \C^* \setminus \{\lambda q^{-n}\}$. Then $u$ has an analytic
continuation at $\lambda q^{-n}$, with:
\begin{equation}
\label{eqn:uell}
u^{(\ell)}(\lambda q^{-n}) = \dfrac{(\ell+k)!}{k!}~g^{(\ell+k)}(\lambda q^{-n}).
\end{equation}
In particular:
\begin{equation}
\label{eqn:u0}
u(\lambda q^{-n}) =
(-1)^{(n-1)k}~(q^{-1};q^{-1})_\infty^{3k}~\bigl(\frac1\lambda\bigr)^k~q^{kn(n+1)/2}~,
\end{equation}
\label{not75}
where $(q^{-1};q^{-1})_\infty = \prod\limits_{r\ge 0}(1-q^{-r-1})$ 
($q$-Pochhammer symbol).}
\index{$q$-Pochhammer symbol}
\end{enumerate}
\end{lemm}
\begin{proof} 
Writing $g^{(\ell)}$ as the sum of a power series in $z$ and one
in $\frac{1}{z}$, one draws from lemma \ref{lemm:croissanceqexp} 
assertion (1), because derivation does not affect the $q$-Gevrey 
character of a series. Assertion (2) is obvious.

The function $u$ can be analytically continuated at $\lambda q^{-n}$, because 
$\theta_\lambda$ has a simple zero at each point of the $q$-spiral
$[\lambda]$; formula \eqref{eqn:uell} comes by differentiating 
$(\ell+k)$ times the equality $g(z) = u(z)(z-\lambda q^{-n})^k$, while
\eqref{eqn:u0} follows from a direct evaluation using Jacobi's
triple product formula \eqref{eqn:tripleproduit}.
\index{Jacobi's triple product formula}
\end{proof} 

We now come to the proof of the direct implication; that is, 
we assume condition (ii) of theorem \ref{theo:fdecomposition}.
Since $\theta_\Lambda h\in\EE_0^\Lambda$, by linearity, we just have
to check that the series $S(z)$ defined by relation 
\eqref{eqn:Salpha} satisfies $\theta_\Lambda S\in\EE_0^\Lambda$. 
To that end, put $T(z) := \theta_\Lambda(z)~S(z)$ for all 
$t \in \C^* \setminus \lambda q^{-\N}$; this clearly has an analytic
continuation to $\C^*$. On the other hand, 
$T^{(\ell)}(\lambda_{i'}q^{-n}) = 0$ for all $i'\neq i$, $\ell < \nu_{i'}$
and $n\in\Z$.

In order to evaluate $T^{(\ell)}(\lambda q^{-n})$, put
$\Lambda' := \Lambda-(j+1)[\lambda]$ and:
$$ 
T_0(z) :=
\sum_{n\ge 0}\alpha_n~
\dfrac{(\theta_\lambda(z))^{j+1}}{(z-\lambda q^{-n})^{j+1}};
$$
one has $T(z) = \theta_{\Lambda'}(z)~T_0(z)$, where $T_0$ is analytic 
on $\C^*$. First note that, if $\ell < \nu_i - j - 1$, one has 
$T^{(\ell)}(\lambda q^{-n}) = 0$ for all $n \in\N$; 
when $\nu_i - j - 1 \leq \ell < \nu_i$, using relation 
\eqref{eqn:uell}, we get:
$$
T^{(\ell)}(\lambda q^{-n}) = \alpha_n~
\sum_{k=0}^{\ell-\nu_i+j+1} \binom{\ell}k~ \theta_{\Lambda'}^{(\ell-k)}(\lambda q^{-n})~
\dfrac{(k+j+1)!}{(j+1)!}~ g^{(k+j+1)}(\lambda q^{-n}),
$$
where $n \in \N$, $g = (\theta_\lambda)^{j+1}$. Taking in account
assertions (1) and (2) of lemma \ref{lemm:theta'}, the sequences
$(\theta_{\Lambda'}^{(\ell-k)}(\lambda q^{-n}))_{n\ge 0}$ and
$(g^{(k+j+1)}(\lambda q^{-n}))_{n\ge 0}$ respectively have $q$-Gevrey order
$(\nu-j-1)$ and $(j+1)$, which implies that 
$(T^{(\ell)}(\lambda q^{-n}))_{n\ge 0}$ is bounded above by a geometric
sequence, for $(\alpha_n)_{n\ge 0}$ has $q$-Gevrey order $(-\nu)$. 

Moreover, like $\theta_\Lambda$, the function $T$ has $q$-exponential
growth of order $\nu$ at $0$. Indeed, let $a \in (1,\lmod q \rmod)$ and put
$\delta := \min\limits_{\lmod z \rmod=a}d_q(z,[1])$; one has $\delta>0$, so that:
\begin{equation}
\label{eqn:Sbornee}
\max_{\lmod z \rmod=a|\lambda q^{-k}|}|S(z)| \leq
\dfrac{1}{\delta~\lmod \lambda\rmod} \sum_{n\ge 0} \lmod \alpha_n q^n \rmod 
< \infty.
\end{equation}
This shows that $S(z)$ stays uniformly bounded over the family of
circles centered at $0$ and such that each one goes through a point
of the $q$-spiral with basis $a\lambda$. It follows first that the 
restriction of function $T$ to the circles has $q$-exponential growth 
of order $\nu$ at $0$; then that function $T$ itself has such growth.
To summarize, we see that $T \in \EE_{0}^\Lambda$, whence 
$F \in \EE_0^\Lambda$. \\

\noindent
\underline{\textbf{Proof of 
$F \in \EE_0^\Lambda \Rightarrow f \in \AA_q^\Lambda$:}}
It will rest on lemma \ref{lemm:thetamin}, which gives 
a minoration on $\theta$ as in lemma \ref{lemm:thetamin0}.

\begin{lemm}
\label{lemm:thetamin}
Let $a > 0$ and $k \in \Z$. Then:
\begin{equation}
\label{eqn:thetamin}
\min_{\lmod z \rmod=a\lmod q \rmod^{-k}} \lmod \theta(z) \rmod \geq  
\dfrac{\lmod (q^{-1};q^{-1})_\infty \rmod}
{(\lmod q \rmod^{-1};\lmod q \rmod^{-1})_\infty}~
\lmod\theta(-a;\lmod q \rmod) \rmod ~a^{-k}~ \lmod q \rmod^{k(k-1)/2}.
\end{equation}
\end{lemm}
\begin{proof} 
Let $z \in \C$ be such that $\lmod z \rmod=a\lmod q \rmod^{-k}$. From a minoration of each 
factor of the triple product formula \eqref{eqn:tripleproduit}, we get:

\begin{eqnarray*}
\min_{\lmod z \rmod=a\lmod q \rmod^{-k}} \lmod \theta(z) \rmod & \geq & 
\lmod (q^{-1};q^{-1})_\infty \rmod
\prod_{n\ge  0} \lmod 1-a\lmod q \rmod^{-n-k}\rmod 
\left\vert1-\dfrac{\lmod q \rmod^{-n-1}}{a\lmod q \rmod^{-k}}\right\vert \\ 
& = & 
\dfrac{\lmod (q^{-1};q^{-1})_\infty \rmod}{(\lmod q \rmod^{-1};\lmod q \rmod^{-1})_\infty}~
\lmod \theta(-a\lmod q \rmod^{-k};\lmod q \rmod) \rmod,
\end{eqnarray*}
yielding relation \eqref{eqn:thetamin}.
\end{proof} 

Now we assume $F \in \EE_0^\Lambda$. We put 
$f := \dfrac{F}{\theta_\Lambda}$ and keep the notations $R$, $N_0$ 
introduced at the beginning of the proof of the theorem 
(page \pageref{pourrefplusloin}); 
also, we choose $a \in \left(\lmod \lambda_m \rmod,R\right)$. 
According to relation \eqref{eqn:thetamin}, there is 
$C > 0$ and $\mu > 0$ such that, if $\lmod z \rmod = a q^{-k}$
and $k \geq N_0$, then $\lmod f(z) \rmod \leq C \lmod z \rmod^\mu$; 
that is, $f$ has moderate growth at $0$ on the circles
$\lmod z \rmod = a \lmod q \rmod^{-k}$, $k \geq N_0$. 

Partial fraction decomposition of $f$ 
at each point of $\Lambda$ contained in the punctured disk 
$0 < \lmod z\rmod < R$, will produce an expression of the form
\eqref{eqn:fLambda}, in which the coefficients $\alpha_{i,j,n}$
have $q$-Gevrey order $(-\nu)$. Indeed, let:
$$
F_{i,k,n} := \dfrac{F^{(k)}(\lambda_iq^{-n})}{k!},\quad
\Theta_{i,k,n} := \dfrac{\theta_\Lambda^{k+\nu_i}(\lambda_iq^{-n})}{(k+\nu_i)!};
$$
then:
$$
\alpha_{i,\nu_i-1,n} = \dfrac{F_{i,0,n}}{\Theta_{i,0,n}}~,\quad
\alpha_{i,\nu_i-2,n} = \dfrac{F_{i,1,n}-\alpha_{i,\nu_i-1,n}~\Theta_{i,1,n}}
{\Theta_{i,0,n}}~,\quad \ldots,
$$
$$
\alpha_{i,0,n} = 
\dfrac{F_{i,\nu_i-1,n}-\sum_{k=1}^{\nu_i-1}\alpha_{i,\nu_i-k,n}~\Theta_{i,\nu_i-k,n}}
{\Theta_{i,0,n}} \cdot
$$
According to relation \eqref{eqn:u0}, one finds:
$$
\Theta_{i,0,n}=
(q^{-1};q^{-1})_\infty^{3\nu_i}~(-\lambda_iq^{-n})^{-\nu_i}~
\theta_{\Lambda_i'}(\lambda_i)~
\prod_{j=1}^m
\left(-\dfrac{\lambda_j}{\lambda_i}\right)^{n\nu_j}~q^{\nu n(n-1)/2}.
$$
Since $(F_{i,0,n})_{n\ge N_0}$ has $q$-Gevrey order $0$, it follows 
that the sequence $(\alpha_{i,\nu_i-1,n})_{n\ge N_0}$ has $q$-Gevrey 
order $(-\nu)$. As for the other coefficients $\alpha_{i,j,n}$, 
$0 \leq j < \nu_i-1$, since the sequences $(\Theta_{i,j,n})_{n\ge N_0}$ 
all have $q$-Gevrey order $\nu$ (see lemma \ref{lemm:theta'}), 
one successively shows that the sequences $(\alpha_{i,\nu_i-2,n})_{n\ge N_0}$, 
$\ldots$, $(\alpha_{i,0,n})_{n\ge N_0}$ all have $q$-Gevrey order $(-\nu)$.

Last, note that after \eqref{eqn:Sbornee}, the function $P$
defined by:
$$
P(z) := \sum_{n\ge N_0} \sum_{1\le i\le m} \sum_{0\le j<\nu_i}
\dfrac{\alpha_{i,j,n}}{(z-\lambda_iq^{-n})^{j+1}}
$$
is uniformly bounded over the family of circles 
$\lmod z \rmod = a \lmod q \rmod^{-k}$, $k \in \Z$ if 
$a \notin \bigcup\limits_{1\le i\le m}[\lambda_i]$. Since
$f\in\BB^\Lambda$, from proposition \ref{prop:BB} we deduce 
that $h(z) := f(z) - P(z)$ is analytic in some open disk 
$0 < \lmod z \rmod < R$ and remains bounded over these circles, 
which implies that $h$ is holomorphic in a neighborhood of $0$, 
achieving the proof of the converse implication and of theorem
\ref{theo:fdecomposition}.
\end{proof}


\section{$q$-Gevrey theorem of Borel-Ritt and flat functions}
\label{section:BorelRitt}

The classical theorem of Borel-Ritt \cite[chap. XI, \S 1]{HS}
says that any formal power series is the asymptotic expansion, 
in Poincar\'e's sense, of some germ of analytical function in 
a sector of the complex plane with vertex $0$ (the Gevrey version 
for differential equations is to be found in \S 2 of the same chapter
of \emph{loc. cit.}). We are going to prove a $q$-Gevrey version of 
that theorem involving the set of functions $\AA_q^\Lambda$ and then 
give a characterization of the corresponding flat functions. 


\subsection{From a $q$-Gevrey power series to an entire function}
\label{subsection:qGevreytoentire}

As before, let $\Lambda = \sum\limits_{1\le i\le m}\nu_i[\lambda_i]$ 
be a divisor and let $\theta_\Lambda$ be the associated theta function 
given by the relation \eqref{eqn:thetalambda}.

\begin{lemm}
\label{lemm:thetatronque}
Consider the Laurent series expansion 
$\sum\limits_{n\in\Z} \beta_n~z^n$ of $\theta_\Lambda$.
\begin{enumerate}
\item{For all $(k, r) \in \Z^2$, one has:
\begin{equation}
\label{eqn:beta}
\beta_{k\nu+r} = \left(\frac L{q^r}\right)^k~q^{-k(k+1)\nu/2}~\beta_r~.
\end{equation}
(The constant $L$ is defined at the beginning of the proof.)
}
\item{For all $n \in \Z$, write $T_n$ the function defined by:
$$
T_n: = \theta_{\Lambda,n}(z) = \sum_{\ell\le n} \beta_\ell~z^\ell~.
$$
If $j\in\{1,\ldots,m\}$, one has:
\begin{equation}
\label{eqn:thetatronque}
 K_{n,j} := \sup_{z\in[\lambda_j]} \lmod z^{-n}~T_n(z) \rmod < +\infty~.
\end{equation}
}
\end{enumerate}
\end{lemm}
\begin{proof} 
From the fundamental relation $\theta(qz) = qz\theta(z)$, one draws:
$$
\theta_\Lambda(qz) = L~z^\nu~\theta_\Lambda(z),\quad
L := \prod_{j=1}^m\left(-\dfrac{q}{\lambda_j}\right)^{\nu_j}~,
$$
whence, for all $k \in \Z$:
$$
\theta_\Lambda(q^kz) = L^k~z^{k\nu}~q^{k(k-1)\nu/2}~\theta_\Lambda(z)~.
$$
The relation \eqref{eqn:beta} follows immediately.

Since $z^{-n}~T_n(z)$ represents an analytic function over 
$\C^* \cup \{\infty\}$, one obtains \eqref{eqn:thetatronque} 
by noticing that, if $z \in [\lambda_j]$, then:
$$
T_n(z) = -\sum_{\ell>n}\beta_\ell~z^\ell~,
$$
which implies that $z^{-n}~T_n(z)$ tends to zero as $z$ tends $0$ 
along the half-spiral $\lambda_i q^{-\N}$.
\end{proof} 

Consider a $q$-Gevrey series 
$\hat f = \sum\limits_{n\ge 0} a_n z^n \in \Rf_{q;1/\nu}$ with order
$1/\nu$ and, for all $\ell \in \Z$, put
$c_\ell := \sum\limits_{n\ge 0}a_n\beta_{\ell-n}$. 
Set $\ell := k \nu + r$ with $k \in \Z$ and $r \in \Z$, and write:
$$
c_\ell = \sum_{j=0}^{\nu - 1} \sum_{m \ge 0} a_{m \nu + j}\,\beta_{(k-m)\nu+(r-j)}.
$$
As we shall see now, the above infinie series converges when 
the integer $-\ell$ is large enough. Indeed, using relation 
\eqref{eqn:beta}, one gets:
$$
\beta_{(k-m)\nu+(r-j)} = (Lq^{-r})^{k} \, q^{-k(k+1)\nu/2}\,q^{kj} \,
\left(\frac{q^{r-j+k\nu}}L\right)^m \, q^{-m(m-1)\nu/2}.
$$
Since $\hat f$ has $q$-Gevrey order $1/\nu$, there exists $C_1 > 0$,
$A_1 > 0$ such that $|a_n| \le C_1\,A_1^n\,|q|^{n(n-1)/(2\nu)}$ for all
$n\in\N$. One has:
\begin{align*}
\lmod a_{m\nu+j}\,\beta_{(k-m)\nu+(r-j)} \rmod &\leq \\
C_1 (|L||q|^{-r})^{k} \, & |q|^{-k(k+1)\nu/2} \, (|q|^k A_1)^{j} \,
|q|^{j(j-1)/(2\nu)} \left(\frac{|q|^{r+k\nu+(\nu-1)/2}}{|L|A_1^{-\nu}}\right)^m\,.
\end{align*}
Therefore, if we set:
$$
C_2 := \sum_{j=0}^{\nu-1} A_1^{j} |q|^{j(j-1)/(2\nu)},
$$
then we find ($\ell = k \nu + r$):
$$
|c_\ell| \leq C_1\,C_2\,(|Lq^{-r}|)^k\,|q|^{-k(k-1)\nu/2}
\sum_{m\geq 0} \left(\frac{|q|^{\ell+(\nu-1)/2}}{|L|A_1^{-\nu}}\right)^m\,.
$$
It is easy to see that the last series converges if 
$|q|^{\ell+(\nu-1)/2} < |L| A_1^{-\nu}$; moreover, when $\ell\to-\infty$, 
the sequence $(c_\ell)$ is $q$-Gevrey of level $(-\nu)$.

\begin{prop}
\label{prop:ftheta}
Let $\hat f := \sum\limits_{n\ge 0} a_nz^n \in \Rf_{q;1/\nu}$ and
let $N_0$ 
be a negative integer such that any series defining $c_\ell$ converges for
$\ell \leq N_0$. Let $F(z) := \sum\limits_{\ell\le N_0}c_\ell z^\ell$. Then,
$F\in E_0^\Lambda$.
\end{prop}
\begin{proof} 
Since $(c_\ell)$ is a $q$-Gevrey sequence of order $(-1/\nu)$ for 
$\ell \to -\infty$, the sum $F$ represents an analytic function in
$\C^* \cup \{\infty\}$, with $q$-exponential growth of order $\nu$ 
at zero. 

We are left to estimate $F(z)$ and its derivatives along the spirals 
$[\lambda_j]$ belonging to the divisor $\Lambda$. Taking in account 
relation \eqref{eqn:beta}, one has:
$$
T_{-n-k\nu}(z) = (Lz^\nu)^{-k}~q^{-k(k-1)\nu/2}~T_{-n}(q^kz).
$$
Therefore, when $z\in [\lambda_j]$, relation \eqref{eqn:thetatronque} 
allows writing $F(z)$ in the following form:
$$
F(z) = \sum_{k\ge
  0}\sum_{n=0}^{\nu-1}a_{k\nu+n}~z^{k\nu+n}~T_{N_0-k\nu-n}(z),
$$
and also:
\begin{equation}
\label{eqn:FT}
F(z) = \sum_{n=0}^{\nu-1}z^n\sum_{k\ge 0}L^{-k}a_{k\nu+n}~q^{-k(k-1)\nu/2}~T_{N_0-n}(q^kz).
\end{equation}
One deduces that $F$ has moderate growth when $z \to 0$ along spirals
$[\lambda_j]$ belonging to the divisor $\Lambda$. 

As for the derivatives $F^{(\ell)}$, along each spiral $[\lambda_j]$, 
with $\ell<\nu_j$, one may proceed all the same: the function 
$T_{n}^{(\ell)}$ is indeed bounded over $[\lambda_j]$ for all $\ell < \nu_j$ 
and one checks that relation \eqref{eqn:FT} remains true up to
order $\ell$, that is, if $z\in[\lambda_j]$:
$$
F^{(\ell)}(z) =
\sum_{n=0}^{\nu-1} \sum_{k\ge 0}
L^{-k}a_{k\nu+n}~q^{-k(k-1)\nu/2}~\bigl[z^n~T_{N_0-n}(q^kz)\bigr]^{(\ell)}.
$$ 
\end{proof} 


\subsection{The $q$-Gevrey theorem of Borel-Ritt}
\label{subsection:qGevreyBorelRitt}

The main result of this paragraph is the following\footnote{Theorem
\ref{theo:BorelRitt} is proved with the help of theorem \ref{theo:fF},
where we gave conditions for a function to have a certain type Taylor 
expansion; here, we deal with the \emph{surjectivity} of the ``Taylor 
expansion mapping''.}.

\index{Borel-Ritt theorem ($q$-Gevrey analogue)}
\index{$q$-Gevrey theorem of Borel-Ritt}
\begin{theo}[$q$-Gevrey theorem of Borel-Ritt]
\label{theo:BorelRitt}
The mapping $J$ sending an element $f$ to its asymptotic expansion
is a surjective linear map from the $\C$-vector space
$\AA_{q}^{\Lambda}$ to the $\C$-vector space $\Rf_{q;1/\nu}$.
\end{theo}
\begin{proof} 
Let $\hat f = \sum\limits_{n\ge 0}a_nz^n$ be a $q$-Gevrey series of 
order $1/\nu$ and let $F$ be the function defined in proposition 
\ref{prop:ftheta}. Maybe at the cost of taking a smaller $N_0$, 
we may assume that formula \eqref{eqn:FT} remains true for 
all $z \in \C^*$; this being granted, one has:
\begin{equation}
\label{eqn:FTa}
F(z) = \sum_{\ell\ge 0}a_\ell~z^\ell~T_{N_0-\ell}(z)~.
\end{equation}
Put $f := \dfrac{F}{\theta_\Lambda}$, so that $f \in \AA_{q}^{\Lambda}$, 
after theorem \ref{theo:fF}. We are going to prove that $f$ admits 
$\hat f$ as an asymptotic expansion. To that aim, it is enough to show
that, if $\lambda \notin \Lambda$, one has:
\begin{equation}
\label{eqn:fell}
\lim_{n\to+\infty}f^{(\ell)}(\lambda q^{-n}) = \ell!~a_\ell 
\end{equation}
for any integer $\ell \geq 0$. 

Since $\lmod T_{N_0-n}(z) \rmod \leq e_\Lambda(z)$ and 
$\lmod \theta_\Lambda(z) \rmod \sim e_\Lambda(z)$ (see lemma 
\ref{lemm:thetamin}), relation \eqref{eqn:FTa} immediately leads
us to the limit \eqref{eqn:fell} for $\ell =0 $. If $\ell \geq
1$, 
one has:
$$
f^{(\ell)}(z) = \sum_{n\ge 0} \sum_{j=0}^\ell 
a_n~\binom nj~\frac{j!}{(\ell-j)!}~z^{n-j}~
\left[\dfrac{T_{N_0-n}(z)}{\theta_\Lambda(z)}\right]^{(\ell-j)}~.
$$
Through direct estimates, one gets the limit \eqref{eqn:fell} 
(omitted details are left to the reader).
\end{proof} 

The theorem \ref{theo:BorelRitt} can also be interpreted with help
of interpolation of a $q$-Gevrey sequence by entire functions at
points in a geometric progression. Indeed, if $f \in \AA_q^\Lambda$ 
is asymptotic to the $q$-Gevrey series $\sum\limits_{k\ge  0}a_kz^k$
of level $\nu$, then, from the decomposition \eqref{eqn:fLambda}, 
one draws:
$$
a_k = \sum_{n\ge N_0}\sum_{1\le i\le m}\sum_{0\le j<\nu_i}(-1)^{j+1}~
\binom{k+j}j~\dfrac{\alpha_{i,j,n}}{(\lambda_iq^{-n})^{k+j+1}}+h_k~,
$$
where $h_k$ denotes the coefficient of $z^k$ in the Taylor series of $h$. 
Put:
$$f_{i,j}(z) := \sum_{n\ge N_0}\alpha_{i,j,n}~z^n~;
$$
then:
\begin{equation}
\label{eqn:afij}
a_k = \sum_{1\le i\le m}\sum_{0\le j<\nu_i}
\binom{k+j}j~\dfrac{(-1)^{j+1}}{\lambda_i^{k+j+1}}~f_{i,j}(q^{k+j+1})+h_k~.
\end{equation}

\begin{coro}
\label{coro:interpol}
Let $(a_n)$ be a $q$-Gevrey sequence of order $\nu$ and let 
$\Lambda = \sum\limits_{1\le i\le m}\nu_i[\lambda_i]$ be a divisor of degree
$\nu := \nu_1 + \cdots + \nu_m$. Then, there exists $\nu$ entire functions 
$A_{i,j}$, $1 \leq i \leq m$, $0 \leq j < \nu_i$ satisfying the
following properties:     
\begin{enumerate}
\item{Each function $A_{i,j}$ has $q$-exponential growth of order (at most) 
$\nu$ at infinity.}
\item{There exists $C>0$, $A>0$ such that, for all $n\in\N$:
\begin{equation}
\label{eqn:interpol2}
\lmod a_n - 
\sum_{1\le i\le m} \sum_{0\le j<\nu_i}
\binom{n+j}j~\dfrac{(-1)^{j+1}}{\lambda_i^{n+j+1}}~A_{i,j}(q^{n+j+1})\rmod < CA^n~.
\end{equation}
}
\end{enumerate}
\end{coro}
\begin{proof} 
The functions $A_{i,j}$ are defined by replacing $N_0$ by $\max(0,N_0)$ 
in definition \eqref{eqn:afij}; one then uses relation 
\eqref{eqn:afij} to get the bounds \eqref{eqn:interpol2}.
\end{proof} 

Note that if $\Lambda$ is solely made of simple spirals ($\nu_i = 1$), 
then corollary \ref{coro:interpol} spells as follows:
$$
\lmod a_n + \sum_{1\le i\le \nu}{\lambda_i^{-n-1}}~A_{i}(q^{n+1}) \rmod < C A^n~,
$$
where $C>0$, $A>0$ and where the $A_i$ are $\nu$ entire functions
with $q$-exponential growth of order at most $\nu$ at infinity. 
In other words, one can interpolate a $q$-Gevrey sequence of order $\nu$, 
up to a sequence with at most geometric growth, by $\nu$ entire functions
on a half-spiral $\lambda q^\N$.


\subsection{Flat functions with respect to a divisor}
\label{subsection:flatfunctions}
\index{flat functions with respect to a divisor}

In the proof of theorem \ref{theo:BorelRitt}, function $F$ 
was defined up to a function of the form $\dfrac{P}{\theta_\Lambda}$, 
where $P$ is a polynomial in $z$ and $z^{-1}$. We shall see that, 
if $h$ denotes a meromorphic function at $z=0$, then 
$\dfrac{h}{\theta_\Lambda}$ represents a function with trivial 
asymptotic expansion, that is, $\dfrac{h}{\theta_\Lambda} \in \Ker J$
with notations from theorem \ref{theo:BorelRitt}.

\begin{theo}
\label{theo:Lambdaplate}
A function $f\in\AA_{q}^{\Lambda}$ is \emph{flat}, \ie\ has trivial
expansion, if, and only if, $\theta_\Lambda f$ is meromorphic in
a neighborhood of $0 \in \C$.
\end{theo}
\begin{proof} 
If $f \in \AA_{q}^{\Lambda}$ has trivial asymptotic expansion, then,
for all $\epsilon > 0$ and $z \in \C^*$ with small enough modulus:
\begin{equation}
\label{eqn:plateN}
d_q(z,\Lambda) > \epsilon \Longrightarrow
\lmod f(z) \rmod < \dfrac{C}{\epsilon}~A^N~\lmod q \rmod^{N^2/(2\nu)}~\lmod z \rmod^N
\end{equation}
for any integer  $N \geq 0$. Using relation \eqref{eqn:minqgevrey}, 
one deduces that $\theta_\Lambda f$ has moderate growth at zero, thus
is at worst meromorphic at $z = 0$. 

Conversely, note that if $f(z) = \dfrac{h(z)}{\theta_\Lambda(z)}$ with
$h(z) = O(z^\mu)$, $\mu\in\Z$, then lemma \ref{lemm:thetamin0} entails
the following relation for all $z$ with small enough modulus:
$$
d_q(z,\Lambda) > \epsilon > 0 \Longrightarrow
\lmod f(z) \rmod < 
\dfrac{M}{\epsilon}~{\lmod z \rmod^\mu}~
\prod_{j=1}^m\left(e\left(\dfrac{z}{\lambda_j}\right)\right)^{-\nu_j}~,
$$
where $M$ denotes a constant $>0$. But for any integer $n \in \Z$:
$$
e(z) = \sum_{k\in\Z}\lmod q \rmod^{-k(k-1)/2}~\lmod z \rmod^k>\lmod q \rmod^{-n(n-1)/2}~\lmod z \rmod^n~,
$$
so that:
\begin{equation}
\label{eqn:e-1}
\prod_{j=1}^m\left(e\left(\dfrac{z}{\lambda_j}\right)\right)^{-\nu_j} <
\Vert \Lambda \Vert^{-n}~\lmod q \rmod^{\nu n(n-1)/2}~\lmod z \rmod^{-\nu n}~.
\end{equation}
One thereby gets an estimate of type \eqref{eqn:plateN} for $f$,
for a sequence of integral values of $N$, with $N=\mu+\nu n$,
$n\to+\infty$ and, as a consequence, for all integers $N \geq 0$: 
this achieves the proof of the theorem.
\end{proof} 

\begin{coro}
\label{coro:f12}
Let $f_1, f_2 \in \AA_{q}^{\Lambda}$. Assume that $f_1$ et $f_2$ have the same
asymptotic expansion. Then $f_1 = f_2$ if, and only if, one of the following
conditions is satisfied:
\begin{enumerate}
\item{One has $f_1(z_n) = f_2(z_n)$ for some sequence of points $(z_n)$ 
tending to zero outside of divisor $\Lambda$.}
\item{On some spiral $[\lambda_i]$ belonging to divisor $\Lambda$: 
$$
\lim_{z\to\lambda_iq^{n}}(z-\lambda_iq^{n})^{\nu_i}~f_1(z) =
\lim_{z\to\lambda_iq^{n}}(z-\lambda_iq^{n})^{\nu_i}~f_2(z)
$$
as $n \to -\infty$.}
\end{enumerate}
\end{coro}
\begin{proof}
Immediate.
\end{proof} 


\section{Relative divisors and multisummable functions}
\label{section:relativedivisorsandmultisummable}

Theorems \ref{theo:BorelRitt} and \ref{theo:Lambdaplate} 
show that, for each divisor $\Lambda$ with degree $\nu \in \N^*$, 
the datum of a $q$-Gevrey series of order $1/\nu$ is equivalent to
the datum of an element in space  $\AA_q^\Lambda$ together with a
function of the form $\dfrac{h}{\theta_\Lambda}$ for some $h \in \Ka$.
When studying $q$-difference equations, we shall observe that the
function $h$ in the correction term $\dfrac{h}{\theta_\Lambda}$ will have
to be determined in some space of asymptotic functions similar to
$\AA_q^\Lambda$, and that is why we are now going to define a notion of
asymptoticity involving more than one divisor level.


\subsection{Relative divisors and two levels asymptotics}
\label{subsection:relativedivisorsandtwolevels}
\index{two levels asymptotics}

Let $\Lambda^\prime$, $\Lambda$ be two divisors. Assume that
$\Lambda^\prime < \Lambda$, meaning that
$\Lambda^\prime = \sum\limits_{1\le i\le m}\nu_i^\prime[\lambda_i]$,
that $\Lambda = \sum\limits_{1\le i\le m}\nu_i[\lambda_i]$, 
$0 \leq \nu_i^\prime \leq \nu_i$, $\nu_i > 0$ and that
$\lmod \Lambda \rmod > \lmod \Lambda^\prime \rmod$.

\begin{defi}
\label{defi:diviseurrelatif}
\index{relative divisors}
\index{values of $F$ over a relative divisor}
\label{not76}
Let $\Lambda^\prime < \Lambda$ and let $F$ be a function defined
and analytic in a neighborhood of each point of the spirals in
the support of $\Lambda$ near $0$.
\begin{enumerate}
\item{We call \emph{values of $F$ over the relative divisor
$\Lambda/\Lambda^\prime$ at $0$}, and we write
$(\Lambda/\Lambda^\prime)F(0)$, the $|\Lambda|-|\Lambda^\prime|$
(germs of) sequences:
$$
(\Lambda/\Lambda^\prime)F(0) :=
\left\{\left(F^{(j)}(\lambda_i q^{-n})\right)_{n\gg 0} ~:~ 
1 \leq i \leq m, \nu^\prime_i \leq j < \nu_i\right\}.
$$}
\index{$q$-Gevrey order over a relative divisor}
\item{We say that \emph{$F$ has $q$-Gevrey order $k$ over
$\Lambda/\Lambda^\prime$ at $0$} if all its values there constitute 
$q$-Gevrey sequences of order $k$.}
\item{Let $\Lambda = \Lambda_1 + \Lambda_2$ and $\Lambda^\prime = \Lambda_1$. 
We write $F \in \EE_{(\Lambda_1,\Lambda_2)}^{\Lambda}$ if $F$ is analytic
in a neighborhood of $0$ punctured at $0$, has $q$-Gevrey order 
$|\Lambda|$ at $0$, has $q$-Gevrey order $|\Lambda_2|$ along 
$\Lambda/\Lambda_2$ at $0$ and has $q$-Gevrey order $0$ along
$\Lambda_2$ at $0$.}
\end{enumerate}
\end{defi}
 
We get the following generalization of theorem \ref{theo:fF}.

\begin{prop}
\label{prop:deuxniveaux}
Let $\Lambda_1 <\Lambda =\Lambda_1 + \Lambda_2$. One has
the following decompositions:
\begin{eqnarray*}
\EE_{(\Lambda_1,\Lambda_2)}^{\Lambda} & = &
\EE_0^{\Lambda_2} + \theta_{\Lambda_2}\EE_0^{\Lambda_1}, \\
\left(\dfrac{1}{\theta_{\Lambda}}~\EE_{(\Lambda_1,\Lambda_2)}^{\Lambda}\right) 
\cap \BB^{\Lambda} & = &
\AA_{q}^{\Lambda_1} + \dfrac{1}{\theta_{\Lambda_1}}~\AA_{q}^{\Lambda_2}.
\end{eqnarray*}
\end{prop}
\begin{proof} 
We shall only check the first decomposition, the second one following
immediately with help of theorem \ref{theo:fF}. 

Let $\Lambda_2 = \sum\limits_{i=1}^m\nu_i[\lambda_i]$, with 
$[\lambda_i] \not= [\lambda_{i'}]$ for $i \not= i'$. Let 
$F \in \EE_{(\Lambda_1,\Lambda_2)}^\Lambda$. Since $F$ has $q$-Gevrey order $0$ 
over $\Lambda_2$ at zero, one gets the following Taylor expansion:
\begin{equation}
\label{eqn:FTaylor}
F(z) = 
\sum_{n\gg0} \sum_{i=0}^m \sum_{j=0}^{\nu_i-1}
\dfrac{F^{(j)}(\lambda_iq^{-n})}{j!}~(z-\lambda_iq^{-n})^j+R_{\Lambda_2}(z)~,
\end{equation}
where the convergence of these series for $n \gg 0$ comes from the fact
that each sequence $\bigl(F^{(j)}(\lambda_i q^{-n})\bigr)_{n\gg0}$ is bounded.
Let $f_2 := \dfrac{F}{\theta_{\Lambda_2}}$ and
$F_1 := \dfrac{R_{\Lambda_2}}{\theta_{\Lambda_2}}$ and put
$F_2 := \theta_{\Lambda_2}(f_2-F_1)$; one has $F =F_2 + \theta_{\Lambda_2}~F_1$ 
and $F_2$ is the function represented by the triple summation in the
expansion \eqref{eqn:FTaylor}. After theorem \ref{theo:fF}, 
one has $(f_2 - F_1) \in \AA_q^{\Lambda_2}$ and $F_2 \in \EE_0^{\Lambda_2}$. 

There remains to check that $F_1 \in \EE_0^{\Lambda_1}$. First note that
$R_{\Lambda_2}$ has order at least $\nu_i$ at each $z = \lambda_iq^{-n}$ 
for $n \gg 0$, which implies that $F_1$ represents a germ of analytic
function in a punctured neighborhood of $z = 0$. Moreover, since the
values of $F$ over the relative divisor $\Lambda/\Lambda_2$ at $0$ 
have $q$-Gevrey order $|\Lambda_2|$, the same is true for the function
$R_{\Lambda_2}$ because $F_2$ has $q$-Gevrey order $|\Lambda_2|$ at zero.
Using lemma \ref{lemm:thetamin}, one finds at last that
$F_1 \in \EE_0^{\Lambda_1}$, achieving the proof of the theorem.
\end{proof} 

Note that:
$$
\AA_{q}^{\Lambda_1} \subset 
\AA_{q}^{\Lambda_1} + \dfrac{1}{\theta_{\Lambda_1}} \AA_{q}^{\Lambda_2}~,
$$
the inclusion being strict for a non trivial divisor $\Lambda_2$.


\subsection{Multisummable functions}
\label{subsection:multisummable}
\index{multisummable functions}

More generally, let
$$
\Lambda_1 < \Lambda_1 + \Lambda_2 < \cdots <
\Lambda_1 + \Lambda_2 + \cdots + \Lambda_m= \Lambda.
$$
For $i$ from $1$ to $m+1$, put 
\label{not77}
\begin{align*}
\Lambda_{\ge i} &:= \sum\limits_{j\ge i}\Lambda_j \\
\text{and~}
\Lambda_{\le i} &:= \sum\limits_{j\le i}\Lambda_j.
\end{align*}
Also, we will write $\Lambda_{\ge m+1} = \Lambda_{\le 0} = {\bf O}$, 
where ${\bf O}$ stands for the null divisor, and
$\Lambda_{\ge 0} = \Lambda_{\le m} = \Lambda$. For short,
we shall call $(\Lambda_1,\Lambda_2,\ldots,\Lambda_m)$ an 
\emph{ordered partition of $\Lambda$}.

\begin{theo}
\label{theo:mniveaux}
Let
$(\Lambda_1,\Lambda_2,\ldots,\Lambda_m)$ be an ordered partition 
of $\Lambda$ and define $\EE_{(\Lambda_1,\Lambda_2,\ldots,\Lambda_m)}^{\Lambda}$ as
\label{not78}
the set of functions that are analytic in a punctured neighborhood 
of $0$, have $q$-Gevrey order $|\Lambda|$ at $0$ and, for $i$
from $1$ to $m$, having $q$-Gevrey order $|\Lambda_{\ge i+1}|$ 
over $\Lambda_{\ge i}/\Lambda_{\ge i+1}$ at $0$. Then, one has:
$$
\EE_{(\Lambda_1,\Lambda_2,\ldots,\Lambda_m)}^{\Lambda} =
\EE_0^{\Lambda_m} + \theta_{\Lambda_{\ge m}} \EE_0^{\Lambda_{m-1}} +
\cdots + \theta_{\Lambda_{\ge 2}}\EE_0^{\Lambda_1}.
$$
In other words, writing 
$\OO_{(\Lambda_1,\Lambda_2,\ldots,\Lambda_m)}^{\Lambda} := 
\left(\dfrac{1}{\theta_\Lambda}\EE_{(\Lambda_1,\Lambda_2,\ldots,\Lambda_m)}^{\Lambda}\right)
\cap \BB^{\Lambda}$, one has:
$$
\OO_{(\Lambda_1,\Lambda_2,\ldots,\Lambda_m)}^{\Lambda} =
\AA_{q}^{\Lambda_1} + \dfrac{1}{\theta_{\Lambda_{\le 1}}} \AA_{q}^{\Lambda_2} + \cdots +
\dfrac{1}{\theta_{\Lambda_{\le m-1}}}\AA_{q}^{\Lambda_m} \cdot
$$
 \end{theo}
\begin{proof} 
One proceeds by induction on the length $m$ of the partition of
divisor $\Lambda$: the cass $m = 1$ is theorem \ref{theo:fF} 
while the case $m = 2$ is dealt with in proposition 
\ref{prop:deuxniveaux}. The other cases follow directly 
from the next lemma.
\end{proof} 

\begin{lemm}
\label{lemm:m-1}
Let $m \ge 2$. Then:
$$
\EE_{(\Lambda_1,\Lambda_2,\ldots,\Lambda_m)}^{\Lambda} =
\EE_0^{\Lambda_m} + 
\theta_{\Lambda_{m}} 
\EE_{(\Lambda_1,\Lambda_2,\ldots,\Lambda_{m-1})}^{\Lambda_{\leq m-1}}~.
$$
\end{lemm}
\begin{proof}  
Let $F\in \EE_{(\Lambda_1,\Lambda_2,\ldots,\Lambda_m)}^{\Lambda}$. As in equation 
\eqref{eqn:FTaylor} in the course of the proof of proposition 
\ref{prop:deuxniveaux}, one considers the Taylor expansion
of $F$ along the divisor $\Lambda_m$, thereby finding that
$F \in \EE_0^{\Lambda_m} +
\theta_{\Lambda_{m}} 
\EE_{(\Lambda_1,\Lambda_2,\ldots,\Lambda_{m-1})}^{\Lambda_{\leq m-1}}$.
\end{proof} 

\begin{defi}
\label{defi:msommabilite}
\index{multisummable functions}
We call $\OO_{(\Lambda_1,\Lambda_2,\ldots,\Lambda_m)}^{\Lambda}$ the set of
\emph{multisummable functions at $0$ with respect to partition
$(\Lambda_1,\Lambda_2,\ldots,\Lambda_m)$ of divisor $\Lambda$}. 
\end{defi}

\label{not79}
By convention, we will set $\OO_{(\bf O)}^{\bf O} := \C\{z\}$, the ring 
of all germs of analytic functions at $z=0$.

\begin{prop}
\label{prop:OOmodule}
The sets $\EE_{(\Lambda_1,\Lambda_2,\ldots,\Lambda_m)}^{\Lambda}$ and
$\OO_{(\Lambda_1,\Lambda_2,\ldots,\Lambda_m)}^{\Lambda}$ have a structure
of module over $\Ra$ and are stable under the $q$-difference
operator $\sq$.
\end{prop}

\begin{proof} 
Immediate.
\end{proof} 

If $m \ge 2$, an element $f \in \OO_{(\Lambda_1,\Lambda_2,\ldots,\Lambda_m)}^{\Lambda}$ 
generally has no asymptotic expansion in the sense of definition
\ref{defi:ALambda}. Nonetheless, according to theorem 
\ref{theo:mniveaux}, there exists a $f_1 \in \AA_q^{\Lambda_1}$ 
such that $f(z)-f_1(z) = O\left(\dfrac{1}{\theta_{\Lambda_1}(z)}\right)$ 
as $z \to 0$ outside of $\Lambda$, which allows one to prove the
following.

\begin{theo}
\label{theo:masymptotique}
To each $f \in \OO_{(\Lambda_1,\Lambda_2,\ldots,\Lambda_m)}^{\Lambda}$
there corresponds a unique power series
$\hat f = \sum\limits_{n\ge 0}a_n~z^n\in\Rf_{q;1/|\Lambda_1|}$ 
satisfying the following property: for all $R > 0$ near enough $0$,
there are constants $C$, $A > 0$ such that, for all $\epsilon > 0$ 
and all $z\in S(\Lambda,\epsilon;R)$:
\begin{equation}
\label{eqn:masymptotique}
\lmod f(z) - \sum_{\ell=0}^{n-1}a_\ell~ z^\ell \rmod <
\dfrac{C}{\epsilon}~A^n~\lmod q \rmod^{n^2/(2|\Lambda_1|)}~\lmod z \rmod^n~.
\end{equation}
\end{theo}
\begin{proof} 
Applying theorem \ref{theo:mniveaux} to the function $f$, one gets:
$$
f = f_1 +\dfrac{f_2}{\theta_{\Lambda_1}} + \cdots +
\dfrac{f_m}{\theta_{\Lambda_{\le m-1}}},
$$
where $f_i \in \AA_q^{\Lambda_i}$ for all indexes $i$ from $1$ to $m$. 
Choose $R > 0$ in such a way that the $f_i~\theta_{\Lambda_i}$ have an
analytic continuation to the punctured disk $0 < \lmod z \rmod < 2 R$; 
from the definition of each space $\AA_q^{\Lambda_i}$, one has:
$$
d_q(z,\Lambda_i) > \epsilon \Longrightarrow
\lmod f_i(z) \rmod < \dfrac{C_i}{\epsilon}
$$
for all $z$ such that $0 < \lmod z \rmod < R$, where $C_i$ is some positive constant. 
On the other hand, from lemma \ref{lemm:thetamin0} and relation 
\eqref{eqn:e-1}, one has, for any divisor
$\Lambda' = \Lambda_1,\ldots,\Lambda_{\le m-1}$:
$$
d_q(z,\Lambda') > \epsilon \Longrightarrow
\lmod \dfrac{1}{\theta_{\Lambda'}(z)} \rmod <
\dfrac {C'}{\epsilon}~{A'}^n~\lmod q \rmod^{n^2/(2|\Lambda'|)}~\lmod z \rmod^n
$$
for all integers $n \ge 0$, where $C'$ and $A'$ are positive constants
depending only on the divisor $\Lambda'$. Noting that
$d_q(z,\Lambda) \le d_q(z,\Lambda_i)~d_q(z,\Lambda_{\le i-1})$
for $i=2,\ldots,m$, one gets that, for all integers $n \ge 0$:
\begin{equation}
\label{eqn:ff1}
d_q(z,\Lambda > \epsilon \Longrightarrow
\lmod f(z)-f_1(z) \rmod <
\dfrac {C_1}{\epsilon}~{A_1}^n~\lmod q \rmod^{n^2/(2|\Lambda_1|)}~\lmod z \rmod^n
\end{equation}
as long as $0 < \lmod z \rmod < R$. 

Relation \eqref{eqn:ff1} implies that both  $f$ and $f_1$ have $\hat f$
as an asymptotic expansion in the classical sense (Poincar{\'e}) at $0$,
outside of divisor $\Lambda$. Since $f_1$ is asymptotic to $\hat f$ in
the frame of space $\AA_q^{\Lambda_1}$, one has 
$\hat f \in \Rf_{q;1/|\Lambda_1|}$. The bounds 
\eqref{eqn:masymptotique} are directly obtained from 
\eqref{eqn:fN} and \eqref{eqn:ff1}, using the
obvious relation $d_q(z,\Lambda) \le d_q(z,\Lambda_1)$.
\end{proof} 

It is clear that a function $f$ satisfying \eqref{eqn:masymptotique} 
does not necessarily belong to the space
$\OO_{(\Lambda_1,\Lambda_2,\ldots,\Lambda_m)}^{\Lambda}$. 
To simplify, we shall adopt the following definition.

\begin{defi}
\label{defi:Lambda1Lambda}
\index{asymptotic expansion at $0$ along the divisors $(\Lambda_1,\Lambda)$}
When $f$ and $\hat f$ satisfy condition \eqref{eqn:masymptotique}, 
we shall say that \emph{$f$ admits $\hat f$ as an asymptotic expansion 
at $0$ along the divisors $(\Lambda_1,\Lambda)$.} If moreover $\hat f$
is the null series, we shall say that \emph{$f$ is flat at $0$ along
divisors $(\Lambda_1,\Lambda)$.}
\index{flat at $0$ along divisors $(\Lambda_1,\Lambda)$}
\end{defi}

The following result is straightforward from definitions 
\ref{defi:ALambda} (page \pageref{defi:ALambda}) and \ref{defi:Lambda1Lambda}.

\begin{prop}
\label{prop:3Lambda}
Let $\Lambda_1$, $\Lambda'$ and $\Lambda$ be divisors such that
$\Lambda_1 < \Lambda' < \Lambda$. If $f \in \AA_q^{\Lambda'}$, then
its expansion along $\Lambda'$ is also the expansion of $f$ along
$(\Lambda_1,\Lambda)$.
\end{prop}
\begin{proof} 
Immediate.
\end{proof} 


\section{Analytic classification of linear $q$-difference  equations} 
\label{section:analyticclassificationCZ}
\index{analytic classification}

We now return to the classification problem, with the 
notations of chapter \ref{chapter:isoformalclasses}. 
In particular, we consider a block diagonal matrice 
$A_{0}$ as in \eqref{eqn:formestandarddiagonaleentiere}
page \pageref{eqn:formestandarddiagonaleentiere},
and the corresponding space $\F(P_{1},\ldots,P_{k})$ of 
isoformal analytic classes within the formal class defined
by $A_{0}$.

The space of block upper triangular matrices $A_{U}$ as in 
\eqref{eqn:formestandardtriangulaireentiere} page 
\pageref{eqn:formestandardtriangulaireentiere} 
such that moreover each rectangular block $U_{i,j}$ belongs to
$\Mat_{r_{i},r_{j}}(K_{\mu_{i},\mu_{j}})$ (according to the notations
of the beginning of subsection \ref{subsection:twointegralslopes})
will be written for short $\CA$. 
\label{not80}
It follows from
proposition \ref{prop:kintegralslopesspace} that sending
$A_{U}$ to its class induces an isomorphism of $\CA$
with $\F(P_{1},\ldots,P_{k})$. 

Let $A \in \CA$ et consider the conjugation equation: 
\begin{equation}
\label{eqn:conjugaison}
(\sq F) A_0 = A F;
\end{equation} 
We saw in the first part of \ref{subsection:computnormalform}
that this admits a unique solution $\hat F = (\hat F_{ij})$ 
in $\G(\Rf)$. After J.P. B\'ezivin \cite{Bezivin}, 
this solution satisfies the following $q$-Gevrey condition:
\begin{equation*}
\forall i < j ~,~
\hat F_{ij} \in \Mat_{r_i,r_j}\bigl(\Rf_{q;1/(\mu_{j}-\mu_{j-1})}\bigr)~.
\end{equation*}

In this section, we shall first prove that the formal power 
series $\hat F$ is multisummable according to definition 
\ref{defi:msommabilite}. We shall deduce the existence 
of solutions that are asymptotic to the formal solution $\hat F$ 
at zero along divisors $(\Lambda_1,\Lambda)$ (see definition 
\ref{defi:Lambda1Lambda}). Last, in \ref{subsection:Stokes}, 
we shall prove that the Stokes multipliers make up a complete set 
of analytic invariants.

In the following, a matrix-valued function will be said to be 
summable or multisummable or to have an asymptotic expansion if 
each coefficient of the matrix is such a function. In this way, 
we will be led to consider spaces $\G(\AA_{q}^\Lambda)$, 
$\Mat_{n_1,n_2}(\OO_{(\Lambda_1,\ldots,\Lambda_{m})}^{\Lambda})$, etc.


\subsection{Summability of formal solutions}
\label{subsection:sommabilite}

Consider the equation \eqref{eqn:conjugaison} and keep
the notation \eqref{eqn:formestandardtriangulaireentiere}. 
To each pair $(i,j)$ such that $i < j$, we associate a
divisor $\Lambda_{i,j}$ with degree $\mu_j - \mu_i$; we assume
that $\Lambda_{i+1,j} < \Lambda_{i,j} < \Lambda_{i,j+1}$, meaning that
$\Lambda_{i,j} = \sum\limits_{\ell=i}^{j-1}\Lambda_{\ell,\ell+1}$. 
Put $\Lambda_j := \Lambda_{j-1,j}$ and
$\Lambda := \sum\limits_{j=2}^{k}\Lambda_j$; then
$\Lambda_{i,j} = \sum\limits_{\ell=i+1}^{j}\Lambda_j$ and
$|\Lambda| = \mu_k - \mu_1$. By convention, we write
$\Lambda_{i,i} = {\rm\bf O}$, the null divisor.

In order to establish the summability of the formal solution 
$\hat F$, we need a genericity condition on divisors.

\begin{defi}
\label{defi:compatible}
Let $A_{0}$ be as in \eqref{eqn:formestandarddiagonaleentiere}
and let $\Lambda$ be a divisor of degree
$\vert \Lambda\vert = \mu_k - \mu_1$. A partition
$(\Lambda_1,\ldots,\Lambda_m)$ is called 
\index{partition compatible with $A_0$}
\index{compatible with $A_0$ (partition)}
\emph{compatible 
with $A_0$} if $m = k-1$ and if, for $i$ from $1$ to $(k-1)$, 
one has $\vert\Lambda_i\vert = \mu_{i+1} - \mu_i$~; when this
is the case, the partition is called \emph{generic for $A_0$} 
\index{partition generic for $A_0$}
\index{generic for $A_0$ (partition)}
if it moreover satisfies the following non-resonancy condition:
$\Vert \Lambda_i \Vert \not\equiv \dfrac{\alpha_i}{\alpha_j} \pmod~ q^\Z$ 
for any eigenvalues $\alpha_i$ of $A_i$ and $\alpha_j$ of $A_j$, $j>i$.
\end{defi}

The following theorem makes more precise the result 
\cite[Theorem 3.7]{JSStokes}. Here, we prove that the \emph{unique} 
meromorphic solution $F = F_{(\Lambda_1,\ldots,\Lambda_{k-1})}$, found by both 
methods, is \emph{asymptotic to the unique formal solution $\hat F$}, 
in the sense of definition \ref{defi:Lambda1Lambda}). Moreover, its 
construction is different \cite{zh2} and uses the theorem of Borel-Ritt 
\ref{theo:BorelRitt}. \\
\begin{tabular}[t]{p{7cm} c}
{\begin{rema}
\label{rema:dictionnary}
The two approaches can be compared with help of the following dictionary
(see also subsection \ref{subsection:algebraicsummation}):
\end{rema}}
&
{\begin{tabular}[t]{|l|l|}
\hline
Here & \cite{JSStokes} \\
\hline
\index{summation divisor}
partition & summation divisor \\ 
\hline
compatible & adapted \\
\hline
generic & allowed \\
\hline
\end{tabular}}
\end{tabular}

\begin{theo}[Summability]
\index{summability}
\label{theo:sommabilite}
Let $A \in \CA$, let $\Lambda$ be a divisor and
$(\Lambda_1,\ldots,\Lambda_{k-1})$ a generic partition of
$\Lambda$ for $A_0$. For $i < j$, put
$\Lambda_{i,j} := \sum\limits_{\ell=i}^{j-1} \Lambda_\ell$.
Then the conjugacy equation $(\sq F) A_0 = A F$ admits a unique 
solution $F = (F_{ij})$ in $\G$ such that
$F_{ij} \in \Mat_{r_i,r_j}(\OO_{\Lambda_i,\ldots,\Lambda_{j-1}}^{\Lambda_{i,j}})$ for $i < j$
\end{theo}
 \begin{proof} 
Take $A$ in the form 
\eqref{eqn:formestandardtriangulaireentiere}. 
The conjugacy equation \eqref{eqn:conjugaison} is
equivalent to the following system; for$1 \leq i < j \leq k$:
$$
z^{\mu_j-\mu_i}(\sq F_{ij})A_j =
A_i F_{ij} + z^{-\mu_i} \sum_{\ell=i+1}^j U_{i\ell} F_{\ell j}.
$$
After J.P. B\'ezivin \cite{Bezivin}, we know that the formal
solution $\hat F_{ij}$ has $q$-Gevrey order $1/(\mu_j - \mu_{j-1})$
for $i < j$. 

Let $1 < j \leq k$ and consider the formal solution $\hat F_{j-1,j}$ 
of equation:
\begin{equation}
\label{eqn:Fj-1j}
z^{\mu_j-\mu_{j-1}}(\sq Y)~A_j = A_{j-1}~Y + z^{-\mu_{j-1}} U_{j-1,j}.
\end{equation}
After the theorem of Borel-Ritt \ref{theo:BorelRitt}, 
there exists a function $\Phi_{j-1,j} \in \AA_q^{\Lambda_{j-1}}$
with $\hat F_{j-1,j}$ as an asymptotic expansion at zero. 
Putting $Y = Z + \Phi_{j-1,j}$ in equation \eqref{eqn:Fj-1j}, 
one gets, from theorem \ref{theo:Lambdaplate}:
$$
z^{\mu_j-\mu_{j-1}}(\sq Y)~A_j = 
A_{j-1}~Y + \frac{H_{j-1,j}}{\theta_{\Lambda_{j-1}}}~,
$$
where $H_{j-1,j}$ denotes a meromorphic function in a neighborhood
of $z = 0$. Putting $Z := \dfrac{X}{\theta_{\Lambda_{j-1}}}$ one finds:
$$
\Vert\Lambda_{j-1}\Vert(\sq X)~A_j = A_{j-1}~X + H_{j-1,j}~.
$$
By the non-resonancy assumption, this yields a unique solution
that is meromorphic at $z = 0$. It follows that the equation
admits a unique solution in class $\AA_q^{\Lambda_{j-1}}$ which is
asymptotic to the series $\hat F_{j-1,j}$; the uniqueness is
obvious. This solution will be written $F_{j-1,j}$. 

Now consider the equation satisfied by $\hat F_{j-2,j}$:
\begin{equation}
\label{eqn:Fj-2j}
z^{\mu_j-\mu_{j-2}}(\sq Y)~A_j =
A_{j-2}~Y + z^{-\mu_{j-2}}~\bigl(U_{j-2,j-1} F_{j-1,j} + U_{j-2,j}\bigr).
\end{equation}
Choosing a function $\Phi_{j-2,j}$ in $\AA_q^{\Lambda_{j-1}}$ with
$\hat F_{j-2,j}$ as an asymptotic expansion, the change of
unknown function $Y = \dfrac{Z+\Phi_{j-2,j}}{\theta_{\Lambda_{j-1}}}$ 
transforms \eqref{eqn:Fj-2j} into:
$$
\Vert\Lambda_{j-1}\Vert~z^{\mu_{j-1}-\mu_{j-2}}~(\sq X)~A_j =
A_{j-2}~X+H_{j-2,j}~,
$$
where $H_{j-2,j}$ is a function meromorphic at $z = 0$. We are thus
led back to a similar situation as in equation \eqref{eqn:Fj-1j}, 
with $\mu_{j-1} - \mu_{j-2}$ instead of $\mu_j - \mu_{j-1}$. One thus
gets a solution of \eqref{eqn:Fj-2j} in $\OO_{\Lambda_{j-2},\Lambda_{j-1}}$ 
which is asymptotic to $\hat F_{j-2,j}$ along the divisors
$(\Lambda_{j-2},\Lambda_{j-1}+\Lambda_{j-2})$. By considering the
associated homogeneous equation, one checks that this solution
is unique. 

Iterating the process, one shows that there is a unique solution
$F_{i,j}$ in $\OO_{\Lambda_i,\ldots,\Lambda_{j-1}}^{\Lambda_{i,j}}$ which is
asymptotic to the formal solution $\hat F_{ij}$ for all $i < j$.
\end{proof} 

Theorem \ref{theo:sommabilite} can be extended in the following way.

\begin{coro}
\label{coro:sommabilite} 
Let $B_1, B_2 \in \G(\C\{z\})$ be such that 
$(A_0)^{-1}~B_i\in\G(\C\{z\})$, $i=1$, $2$, and consider  
the associated $q$-difference equation:
\begin{equation}
\label{eqn:YB12}
(\sigma_q Y)~B_1 = B_2Y ~.
\end{equation}
Then, for any given generic partition of divisors 
$(\Lambda_1,\ldots,\Lambda_{k-1})$ for $A_0$, there exists, in $\G$, 
a unique matrix solution $Y:=(Y_{i,j})$ of \eqref{eqn:YB12} 
and such that, for each pair $i < j$ of indices,  
$Y_{i,j} \in \Mat_{r_i,r_j}(\OO_{\Lambda_i,\ldots,\Lambda_{j-1}}^{\Lambda_{i,j}})$.
\end{coro}
\begin{proof}
Recall that \eqref{eqn:YB12} is analytically conjugated 
to an equation of form \eqref{eqn:conjugaison}; see section 
\ref{section:analyticisoformalclassification}.
The result follows from proposition \ref{prop:OOmodule} and 
the fact that 
$\OO_{\Lambda_i,\ldots,\Lambda_{j-1}}^{\Lambda_{i,j}} \subset
\OO_{\Lambda_i,\ldots,\Lambda_{j}}^{\Lambda_{i,j+1}})$ for all $i < j$.
\end{proof}

An obvious consequence of theorem \ref{theo:sommabilite} 
and corollary \ref{coro:sommabilite} is the following.

\begin{theo}[Existence of asymptotic solutions]
\index{asymptotic solutions (existence)}
\label{theo:solutionsasymptotiques} 
The functions $F = (F_{i,j})$ and $Y = (Y_{i,j})$ respectively 
considered in theorem \ref{theo:sommabilite} and corollary 
\ref{coro:sommabilite} both admit an asymptotic expansion in 
the following sense: for all pairs $(i,j)$ with $i < j$, the 
blocks $F_{i,j}$ and $Y_{i,j}$ have an asymptotic expansion at $0$ 
along divisors $(\Lambda_j,\Lambda_{i,j})$, according to definition 
\ref{defi:Lambda1Lambda}.
\end{theo}
\begin{proof} 
Immediate, with help of theorem \ref{theo:masymptotique}.
\end{proof}


\subsection{Stokes phenomenon and analytic classification}
\label{subsection:Stokes}
\index{Stokes phenomenon}
\index{analytic classification}

\label{not81}
Write $F^A_{(\Lambda_1,\ldots,\Lambda_{k-1})}$ the solution of
\eqref{eqn:conjugaison} obtained in theorem 
\ref{theo:sommabilite}, which can be seen as a sum of
the formal solution $\hat F$ with respect to generic partition
$(\Lambda_1,\ldots,\Lambda_{k-1})$. One thereby deduces
a summation process of $\hat F$ with respect to each generic
partition of divisors for $A_0$; we shall denote:
$$
{\ff}^A: (\Lambda_1,\ldots,\Lambda_{k-1})
\mapsto F^A_{(\Lambda_1,\ldots,\Lambda_{k-1})}.
$$
The $q$-analogue of Stokes phenomenon is displayed by
the existence of various "sums" of $\hat F$.

\begin{prop}
\label{prop:AA0}
In theorem \ref{theo:sommabilite}, one has $A = A_0$ 
if, and only if, the map ${\ff}^A$  is constant on 
the set of generic partitions of divisors for $A_0$.
\end{prop}
\begin{proof} 
If $A = A_0$, the formal solution boils down to the identity
matrix, which plainly coincides with $F^A_{(\Lambda_1,\ldots,\Lambda_{k-1})}$ 
for any compatible partition $(\Lambda_1,\ldots,\Lambda_{k-1})$. 

On the other hand, if $A \not= A_0$, the polynomial $U_{i,j}$ in
\eqref{eqn:formestandardtriangulaireentiere} are not all $0$.
To begin with, assume that $U_{j-1,j} \not= 0$ for some index $j$. 
Since $\hat F_{j-1,j}$ diverges, its sum along divisor $\Lambda_j$ 
will be distinct from its sum along any different divisor.
More generally, when $U_{i,j} \not= 0$, the solution 
$F^A_{(\Lambda_1,\ldots,\Lambda_{k-1})}$ will depend in a one-to-one
way on divisor $\Lambda_j$; details are left to the reader.
\end{proof} 

\label{not82}
For any $\lambda\in\C^*$, let $[\lambda;q:{A_0}]$ be the ordered 
partition of divisors generated by the spiral $[\lambda;q]$ in 
the following way:
$$
[\lambda;q:{A_0}] :=
\bigl((\mu_2-\mu_1)[\lambda;q],(\mu_3-\mu_2)[\lambda;q],\ldots,
(\mu_k-\mu_{k-1})[\lambda;q]\bigr)~,
$$
where $\mu_j$ are as in \eqref{eqn:formestandarddiagonaleentiere}, 
\ie\ $A_0 = \text{diag}\left(z^{\mu_1} A_1,\ldots,z^{\mu_k} A_k\right)$. 

If $[\lambda;q:{A_0}]$ is generic with respect to $A_0$, we will write
$\lambda \in \{A_0\}$ and say that \emph{$\lambda$ is generic for $A_0$}.
\index{generic for $A_0$ ($\lambda$ is)}

\begin{theo}[Stokes phenomenon]
\index{Stokes phenomenon}
\label{theo:Stokes}
Fix some $\lambda_0 \in \{A_0\}$ and for any $\lambda \in \{A_0\}$ 
and $A \in \CA$, set:
$$
\label{not83}
\text{St}_{\lambda_0}(\lambda;A) := 
\bigl(F_{[\lambda_0;q:{A_0}]}^A\bigr)^{-1}~F_{[\lambda;q:{A_0}]}^A~.
$$
The following conditions are equivalent.
\begin{enumerate}
\item{There exists $\lambda \in \{A_0\}$ such that 
$[\lambda;q] \neq [\lambda_0;q]$ but 
$\text{St}_{\lambda_0}(\lambda;A) = Id$.}
\item{For all $\lambda \in \{A_0\}$, the equality 
$\text{St}_{\lambda_0}(\lambda;A)=Id$ holds.}
\item{One has $A=A_0$.}
\end{enumerate}
\end{theo}
\begin{proof}
It follows immediately from proposition \ref{prop:AA0}
\end{proof}

In the previous theorem, each $\text{St}_{\lambda_0}(\lambda;A)$ 
represents an upper-triangular unipotent matrix-valued function 
that is analytic in some disk $0 < \lmod z \rmod < R$ except at spirals 
$[\lambda;q]$ and $[\lambda_0;q]$; moreover, it is infinitely 
close to the identity matrix as $z \to 0$. If 
$Y := \text{St}_{\lambda_0}(\lambda;A) - I_{n}$, then $Y$ is flat 
at zero and satifies the relation:
$$
(\sigma_q Y)~A_0 = A_0~Y.
$$

Theorem \ref{theo:Stokes} implies that each matrix $A$ 
chosen within the isoformal class $\CA$ reduces to 
the normal form $A_0$ whenever the Stokes matrix
$\text{St}_{\lambda_0}(\lambda;A)$ 
becomes trivial for some couple of generic parameters $\lambda_0$, 
$\lambda$ such that $\lambda\not=\lambda_0 \bmod q^\Z$. \\

We shall now show that for any given pair of generic parameters 
$\lambda_0$, $\lambda$ that are not $q$-congruent, the data 
$\{\text{St}_{\lambda_0}(\lambda;A), A_0\}$ constitutes a complete 
set of analytical invariants associated to the $q$-difference 
module determined by $A$. \\
 
Indeed, the notation $\text{St}_{\lambda_0}(\lambda;A)$ introduced 
for any $A \in \CA$ can be naturally generalized for any matrix 
$B \in \G(\C\{z\})$ that belongs to the same formal class as $A_0$. 
To simplify, assume that $(A_0)^{-1}B \in \G(\Ra)$; this is for
instance the case if $B$ is taken in Birkhoff-Guenther normal form
(definition \ref{defi:BirkhoffGuenthernormalform}). From corollary 
\ref{coro:sommabilite} it follows that the corresponding conjugacy 
equation:
$$
(\sigma_q Y)~B = A_0~Y
$$
admits, in $\G$, a unique solution in the space of multi-summable
functions associated to any given generic ordered partition of
divisors $[\lambda;q:A_0]$. Thus, we are led to the notations  
${\ff}^B_{[\lambda;q:A_0]}$ and $\text{St}_{\lambda_0}(\lambda;B)$ as in
the case of $A \in \CA$.

\begin{coro}
\label{coro:Stokes}
Suppose that $B_1$ and $B_2$ are two matrices such that 
$(A_0)^{-1}B_i \in \G(\Ra)$. Then the following assertions 
are equivalent.
\begin{enumerate}
\item{There exists $(\lambda_0,\lambda) \in \{A_0\}\times\{A_0\}$ 
such that $[\lambda;q] \neq [\lambda_0;q]$ but 
$\text{St}_{\lambda_0}(\lambda;B_1) = \text{St}_{\lambda_0}(\lambda;B_2)$.}
\item{The equality 
$\text{St}_{\lambda_0}(\lambda;B_1) =
\text{St}_{\lambda_0}(\lambda;B_2)$ 
holds for all $\lambda_0$, $\lambda\in \{A_0\}$.}
\item{The matrices $B_1$ and $B_2$ give rise to analytically 
equivalent $q$-difference modules.}
\end{enumerate}
\end{coro}
\begin{proof} 
The only point to prove is that (1) implies (3). To do that, 
notice that from (1) we obtain the equality
$$
{\ff}^{B_2}_{[\lambda_0;q:A_0]}~\bigl({\ff}^{B_1}_{[\lambda_0;q:A_0]}\bigr)^{-1} =
{\ff}^{B_2}_{[\lambda;q:A_0]}~\bigl({\ff}^{B_1}_{[\lambda;q:A_0]}\bigr)^{-1}~,
$$
in which the left hand side and the right hand side are both 
solution to the same equation \eqref{eqn:YB12}. Since 
these solutions are analytical in some disk $0 < \lmod z \rmod < R$
except maybe respectively on the $q$-spiral $[\lambda_0;q]$ 
or $[\lambda;q]$ and that they have the same asymptotic expansion 
at zero, they must be analytic at  $z=0$ and we arrive at assertion (3).
\end{proof}



\chapter{Geometry of the space of classes}
\label{chapter:geometryofF(M)}

In this chapter we describe to some extent the geometry of the 
space $\F(M_{0})$ of analytic isoformal classes in the formal 
class $M_{0}$. In subsection \ref{subsection:BGNF}, we already 
used the Birkhoff-Guenther normal form to find coordinates on 
$\F(M_{0})$. Here, we rather use the identification of $\F(M_{0})$ 
with $H^{1}(\Eq,\Lambda_I(M_0))$ proved in theorem
\ref{theo:secondqMalgrangeSibuya} and completed in section
\ref{section:analyticclassificationCZ}. The description given here 
will be further pursued in a separate work \cite{JSmoduli}.


\section{Privileged cocycles}
\label{section:privilegedcocycles}

In "applications", it is sometimes desirable to have explicit 
cocycles to work with instead of cohomology classes. We shall now 
describe cocycles with nice properties, that can be explicitly 
computed from a matrix in standard form. Most of the proofs of what 
follows are consequences of theorem \ref{theo:sommabilite}. However, 
they can also be obtained through the more elementary approach of 
\cite{JSStokes}, which we shall briefly summarize here in subsection
\ref{subsection:algebraicsummation}. \\

Fix $A_{0}$ as in \eqref{eqn:formestandarddiagonaleentiere} page
\pageref{eqn:formestandarddiagonaleentiere} and let $M_{0}$ be
the corresponding pure module. \emph{All notations that follow 
are relative to this $A_{0}$ and $M_{0}$}.
Recall from section \ref{section:generalnotations} at the end 
of the introduction the notations $\overline{a} \in \Eq$ for 
the class of $a \in \C^{*}$ and $[a;q] := a q^{\Z}$ for the 
corresponding $q$-spiral. Also recall from paragraph
\ref{subsection:thetafunctions} the function $\thq$,
holomorphic over $\C^{*}$ with simple zeroes on $[-1;q]$ and
satisfying the functional equation $\sq \thq = z \thq$.


\subsection{``Algebraic summation''}
\label{subsection:algebraicsummation}

To construct cocycles encoding the analytic class of 
$A \in \GL_{n}(\Ka)$, we use ``meromorphic sums'' of $\hat{F}_{A}$.
These can be obtained either by the summation process of theorem 
\ref{theo:sommabilite}, or by the ``algebraic summation process'' 
of \cite{JSStokes}. We now summarize this process. We restrict to 
the case of divisors supported by a point, since it is sufficient 
for classification. \\

Let $A$ as in \eqref{eqn:formestandardtriangulaireentiere}
(we do not assume from start that it is in Birkhoff-Guenther 
normal form). We want to solve the equation $\sq F = A F A_{0}^{-1}$ 
with $F \in \G_{A_{0}}$ meromorphic in some punctured neighborhood 
of $0$ in $\C^{*}$. Since $A,A_{0},A^{-1},A_{0}^{-1}$ are all 
holomorphic in some punctured neighborhood of $0$ in $\C^{*}$,
the singular locus\footnote{In the context of $q$-difference 
equations, the singular locus of a matrix of functions $P$ is 
made of the poles of $P$ as well as those of $P^{-1}$. For a 
unipotent matrix, like $F$, these are all poles of $F$ since 
$F^{-1}$ is a polynomial in $F$.} of $F$ near $0$ is invariant 
under $q^{-\N}$: it is therefore a finite union of germs of half 
$q$-spirals $a q^{-\N}$ (the finiteness flows from the meromorphy 
of $F$). If we take $A$ in Birkhoff-Guenther normal form (definition 
\ref{defi:BirkhoffGuenthernormalform}), the functional equation 
$\sq F = A F A_{0}^{-1}$ actually allows us to extend $F$ 
to a meromorphic matrix on the whole of $\C^{*}$ and its singular 
locus is a finite union of discrete logarithmic $q$-spirals $[a;q]$. \\

To illustrate the process, we start with an example.

\begin{exem}
\label{exem:Tshakaloff7}
Setting $A_{u} := \begin{pmatrix} z^{-1} & z^{-1} u \\ 0 & 1 \end{pmatrix}$,
with $u \in \Ka$, and $F = \begin{pmatrix} 1 & f \\ 0 & 1 \end{pmatrix}$, 
we saw that $A_{u} = F[A_{0}] \Longleftrightarrow z \sq f - f = u$. 
This admits a unique formal solution $\hat{f}$, which can be computed by 
iterating the $z$-adically contracting operator $f \mapsto - u + z \sq f$. 
If $A$ is in Birkhoff-Guenther normal form, $u \in \C$ and 
$\hat{f} = - u \; \tsh$ (the Tshakaloff series).
\index{Tshakaloff series}
We want $f$ to be meromorphic on some punctured neighborhood of $0$ 
in $\C^{*}$, which we write $f \in \M(\C^{*},0)$, and to have (in germ)
half $q$-spirals of poles. For simplicity, we shall assume that $u$ is 
holomorphic on $\C^{*}$, so that $f$ will have to be meromorphic on the
whole of $\C^{*}$ with complete $q$-spirals of poles.
\label{not84}
We set $\thqc(z) := \thq(z/c)$, which is holomorphic on $\C^{*}$, with 
simple zeroes on the $q$-spiral $[-c;q]$ and satisfies the functional 
equation $\sq \thqc = (z/c) \thqc$. We look for $f$ with simple poles
on $[-c;q]$ by writing it $f = g/\thqc$ and looking for $g$ holomorphic
on $\C^{*}$ and satisfying the functional equation:
\begin{align*}
z \sq (g/\thqc) - (g/\thqc) = u &\Longleftrightarrow c \sq g - g = u \thqc \\
&\Longleftrightarrow \forall n \in \Z \;,\; (c q^{n} - 1) g_{n} = [u \thqc]_{n},
\end{align*}
where we have introduced the Laurent series expansions 
$g = \sum\limits_{n \in \Z} g_{n} z^{n}$ and 
$u \thqc = \sum\limits_{n \in \Z} [u \thqc]_{n} z^{n}$.
If $c \not\in q^{\Z}$, we get the solution:
\label{not85}
$$
f_{\overline{c}} := 
\dfrac{1}{\thqc} \, \sum_{n \in \Z} \dfrac{[u \thqc]_{n}}{c q^{n} - 1} z^{n}.
$$
This is clearly the unique solution of $z \sq f - f = u$ with only 
simple poles on $[-c;q]$, a condition that depends on $\overline{c}$ 
only and justifies the notation $f_{\overline{c}}$. We consider it as 
\emph{the summation of the formal solution $\hat{f}$ in the ``authorized 
direction of summation'' $\overline{c} \in \Eq$}. 
\index{summation in an authorized direction of summation}
\index{authorized direction of summation}
\index{direction of summation}
Accordingly, we write 
it $S_{\overline{c}} \hat{f}$. The sum $S_{\overline{c}} \hat{f}$ is 
asymptotic to $\hat{f}$ (see further below). Note that the ``directions 
of summation'' are elements of $\Eq = \C^{*}/q^{\Z}$ which plays here 
the role of the circle of directions $S^{1} := \C^{*}/\R_{+}^{*}$ of 
the classical theory. Also note that all ``directions of summations'' 
$\overline{c} \in \Eq$ are authorized, except for one.
\end{exem}

Returning to the general equation 
$F[A_{0}] = A \Leftrightarrow \sq F = A F A_{0}^{-1}$, we want to
look for solutions $F$ that are meromorphic on some punctured
neighborhood of $0$ in $\C^{*}$, and we want to have unicity by
imposing conditions on the half $q$-spirals of poles: position
and multiplicity. Thus, for $f \in \M(\C^{*},0)$, and for a finite
family $(a_{i})$ of points of $\C^{*}$ and $(n_{i})$ of integers,
we shall translate the condition: ``The germ $f$ at $0$ has all 
its poles on $\bigcup [a_{i};q]$, the poles on $[a_{i};q]$ having 
at most multiplicity $n_{i}$'' by an inequality of divisors on $\Eq$:
\label{not86}
$$
\div_{\Eq}(f) \geq - \sum n_{i} [\overline{a_{i}}].
$$
We introduce a finite subset of $\Eq$ defined by resonancy conditions:
$$
\Sigma_{A_{0}} := \bigcup_{1 \leq i < j \leq k} S_{i,j}, \text{~where~}
S_{i,j} := \{p(-a) \in \Eq \tq 
q^{\Z} a^{\mu_{i}} \text{Sp}(A_{i}) \cap q^{\Z} a^{\mu_{j}} \text{Sp}(A_{j})
\neq \emptyset\}.
$$
Here, $\text{Sp}$ denotes the spectrum of a matrix. 

\begin{theo}[``Algebraic summation'']
\label{theo:algebraicsummation}
\index{algebraic summation}
\index{summation in an authorized direction of summation}
\index{authorized direction of summation}
\index{direction of summation}
For any matrix $A$ in form \eqref{eqn:formestandardtriangulaireentiere},
and for any ``authorized direction of summation'' 
$\overline{c} \in \Eq \setminus \Sigma_{A_{0}}$, 
there exists a unique meromorphic gauge transform 
$F \in \G_{A_{0}}\bigl(\M(\C^{*},0)\bigr)$ satisfying:
$$
F[A_{0}] = A \text{~and~} 
\div_{\Eq}(F_{i,j}) \geq - (\mu_{j} - \mu_{i}) [\overline{-c}]
\text{~for~} 1 \leq i < j \leq k.
$$
\end{theo}
\begin{proof}
See \cite[theorem 3.7]{JSStokes}.
\end{proof}

Of course, $\div_{\Eq}(F_{i,j}) \geq - (\mu_{j} - \mu_{i}) [\overline{-c}]$
means that all coefficients $f$ of the rectangular block $F_{i,j}$ are
such that $\div_{\Eq}(f) \geq - (\mu_{j} - \mu_{i}) [\overline{-c}]$.
The matrix $F$ of the theorem is considered as \emph{the summation 
of $\hat{F}_{A}$ in the authorized direction of summation 
$\overline{c} \in \Eq$} and we write it $S_{\overline{c}} \hat{F}_{A}$. \\

It is actually the very same sum as that obtained in theorem
\ref{theo:sommabilite}. Indeed, using notations from subsection
\ref{subsection:sommabilite}, if one requires the divisors $\Lambda_{i}$ 
to be supported by a single point $\overline{c} \in \Eq$, the genericity
condition of definition \ref{defi:compatible} translates here to: 
$\overline{c} \in \Eq \setminus \Sigma_{A_{0}}$, as one sees by using the 
dictionary provided in remark \ref{rema:dictionnary}. Then, from the same 
dictionary, one draws that the conditions on the sum guaranteed by the 
conclusion of theorem \ref{theo:sommabilite} immediately imply the conditions 
on $S_{\overline{c}} \hat{F}_{A}$ guaranteed by the above theorem, and the 
unicity in the above theorem allows one to conclude. As a consequence of 
subsection \ref{subsection:sommabilite}, the sum $S_{\overline{c}} \hat{F}_{A}$ 
is asymptotic to $\hat{F}_{A}$. \\

The algorithm to compute $F := S_{\overline{c}} \hat{F}_{A}$ is the
following. We introduce the gauge transform:
$$
\label{not87}
T_{A_{0},c} :=
\begin{pmatrix}
\thqc^{-\mu_{1}} I_{r_{1}}  & \ldots & \ldots & \ldots & \ldots \\
\ldots & \ldots & \ldots  & 0 & \ldots \\
0      & \ldots & \ldots   & \ldots & \ldots \\
\ldots & 0 & \ldots  & \ldots & \ldots \\
0      & \ldots & 0       & \ldots &  \thqc^{-\mu_{k}} I_{r_{k}}
\end{pmatrix}.
\quad \text{(Recall that~} \thqc(z) = \thq(z/c).)
$$
Then $F[A_{0}] = A$ , is  equivalent to $G[B_{0}] = B$, where
$B = T_{A_{0},c}[A]$, $B_{0} = T_{A_{0},c}[A_{0}]$, and
$G = T_{A_{0},c} F T_{A_{0},c}^{-1}$. Now, $B_{0}$ is block-diagonal
with blocks $c^{\mu_{i}} A_{i} \in \GL_{r_{i}}(\C)$, and we can solve
for $G[B_{0}] = B$ with $G \in \G_{A_{0}}\bigl(\O(\C^{*},0)\bigr)$,
because this boils down to a system of matricial equations:
$$
(\sq X) c^{\mu_{j}} A_{j} - c^{\mu_{i}} A_{i} X = \text{~some r.h.s.~} Y.
$$
Expanding in Laurent series $X = \sum X_{p} z^{p}$, $Y = \sum Y_{p} z^{p}$, 
we are led to:
$$
q^{p} X_{p} c^{\mu_{j}} A_{j} - c^{\mu_{i}} A_{i} X_{p} = Y_{p}.
$$
Since the spectra of $c^{\mu_{i}} A_{i}$ and $c^{\mu_{i}} A_{i}$ are non 
resonant modulo $q^{\Z}$, the spectra of $q^{p} c^{\mu_{i}} A_{i}$ and 
$c^{\mu_{i}} A_{i}$ are disjoint and the endomorphism
$V \mapsto q^{p} V c^{\mu_{j}} A_{j} - c^{\mu_{i}} A_{i} V$
of $\Mat_{r_{i},r_{j}}(\C)$ is actually an automorphism. There is
therefore a unique formal solution $X$, and convergence in $(\C^{*},0)$ 
is not hard to prove, so we indeed get $G$ with the required properties.
Last, putting $F := T_{A_{0},c}^{-1} G T_{A_{0},c}$ gives 
$S_{\overline{c}} \hat{F}_{A}$ in the desired form.


\subsection{Privileged cocycles}
\label{subsection:privilegedcocycles}

We have three ways of constructing a cocycle encoding the analytic
isoformal class of $A$ in the formal class defined by $A_{0}$: the
map $\lambda$ defined in the first part of the proof of theorem 
\ref{theo:secondqMalgrangeSibuya}; the construction of theorem
\ref{theo:Stokes}; and the one of \cite{JSStokes}, based on theorem
\ref{theo:algebraicsummation} above. Because of what was said just after 
this theorem, the last two ways give the same \emph{privileged} cocycle, 
which we shall now describe more precisely.

\begin{rema}
The construction in theorem \ref{theo:secondqMalgrangeSibuya} is not
deterministic, since it involves the choice of a good covering and of 
functions $g_i$ asymptotic to $\hat{F}$. However, among the possible 
choices, there is the possibility to take the sums defined by theorems 
\ref{theo:Stokes} and \ref{theo:algebraicsummation}. Thus, among the 
possible cocycles constructed in theorem \ref{theo:secondqMalgrangeSibuya},
there is our privileged cocycle.
\end{rema}

If $M_{0}$ is the pure module with matrix $A_{0}$, recall from section
\ref{section:fundamentalisomorphism} the sheaf $\Lambda_I(M_0)$ of
automorphisms of $A_{0}$ infinitely tangent to the identity. We shall
also denote $\UU$ the covering of $\Eq$ consisting of the Zariski open 
subsets $U_{c} := \Eq \setminus \{c\}$ 
\label{not88}
such that $c \not\in \Sigma_{A_{0}}$
(thus, we drop the upper bar denoting classes in $\Eq$).
For any $c \in \Eq \setminus \Sigma_{A_{0}}$ an authorized direction of 
summation, write for short $F_{c} := S_{c} \hat{F}_{A}$. Then $F_{c}$
is holomorphic invertible over the preimage $p^{-1}(U_{c}) \subset \C^{*}$.
If $c,d \in \Eq \setminus \Sigma_{A_{0}}$, $F_{c,d} := F_{c}^{-1} F_{d}$ is
in $\G\bigl(\O(p^{-1}(U_{c} \cap U_{d}))\bigr)$ and it is an automorphism 
of $A_{0}$ infinitely tangent to the identity:
$$
F_{c,d} \in \Lambda_I(M_0)(U_{c} \cap U_{d}).
$$
The cocycle $(F_{c,d}) \in Z^{1}(\UU,\Lambda_I(M_0))$ involves
rectangular blocks $(F_{c,d})_{i,j}$ (for $1 \leq i < j \leq k$) 
belonging to the space $E_{i,j}$ of solutions of the equation
$(\sq X) z^{\mu_{j}} A_{j} = z^{\mu_{i}} A_{i} X$. For $X \in E_{i,j}$,
it makes sense to speak of its poles on $\Eq$, and of their 
multiplicities. For $c,d \in \Eq \setminus \Sigma_{A_{0}}$
\label{not89}
distinct, we write $E_{i,j,c,d}$ the space of those $X \in E_{i,j}$
that have at worst poles in $c,d$ and with multiplicities 
$\mu_{j} - \mu_{i}$.

\begin{lemm}
\label{lemm:dimensionEijcd}
The space $E_{i,j,c,d}$ has dimension $r_{i} r_{j} (\mu_{j} - \mu_{i})$.
\end{lemm}
\begin{proof}
Writing again $\theta_{q,a}(z) := \thq(z/a)$, any $X \in E_{i,j,c,d}$ 
can be written $(\theta_{q,a} \theta_{q,b})^{-(\mu_{j} - \mu_{i})} Y$, 
where $p(-a) = c$, $p(-b) = d$ and $Y$ is holomorphic on $\C^{*}$ 
and satisfies the equation: 
$$
\sq Y = \left(\dfrac{z}{ab}\right)^{\mu_{j} - \mu_{i}} A_{i} Y A_{j}^{-1}.
$$
Taking $Y$ as a Laurent series, one sees that $(\mu_{j} - \mu_{i})$
consecutive terms in $\Mat_{r_{i},r_{j}}(\C)$ can be chosen at will.
\end{proof}

\begin{defi}
\label{defi:cocycleprivilegie}
The cocycle $(F_{c,d}) \in Z^{1}(\UU,\Lambda_I(M_0))$ is said
\index{privileged cocycle}
to be \emph{privileged} if $(F_{c,d})_{i,j} \in E_{i,j,c,d}$ for
all distinct $c,d \in \Eq \setminus \Sigma_{A_{0}}$ and all
$1 \leq i < j \leq k$. The space of privileged cocycles
is denoted $Z^{1}_{pr}(\UU,\Lambda_I(M_0))$.
\label{not90}
\end{defi}

It is not hard to see that, $i,j$ being fixed, the corresponding
component $(c,d) \mapsto (F_{c,d})_{i,j}$ of a cocycle is totally
determined by its value at any particular choice of $c,d$.
That is, fixing distinct $c,d \in \Eq \setminus \Sigma_{A_{0}}$
gives an isomorphism:
$$
Z^{1}_{pr}(\UU,\Lambda_I(M_0)) \rightarrow 
\bigoplus_{1 \leq i < j \leq k} E_{i,j,c,d}.
$$
Thus, the space of privileged cocycles has the same dimension
as $\F(M_{0})$. Actually, one has two bijections:

\begin{theo}
The following natural maps are bijections:
$$
\F(M_{0}) \rightarrow Z^{1}_{pr}(\UU,\Lambda_I(M_0)) \rightarrow
H^{1}(\Eq,\Lambda_I(M_0)).
$$
\end{theo}
\begin{proof}
See \cite[prop. 3.17, th. 3.18]{JSStokes}. The second map is
the natural one.
\end{proof}


\subsection{$q$-Gevrey interpolation}
\label{subsection:qGevreyinterpolationbis}
\index{$q$-Gevrey interpolation}

The results found in section \ref{section:qGevreyinterpolation}
can be adapted here \emph{mutatis mutandis}: the two classification
problems tackled there lead to the following spaces of classes: 
$\bigoplus\limits_{\mu_{j} - \mu_{i} < 1/s} E_{i,j,c,d}$
and $\bigoplus\limits_{\mu_{j} - \mu_{i} \geq 1/s} E_{i,j,c,d}$, which have 
exactly the desired dimensions. This can be most easily checked using 
the devissage arguments from section
\ref{section:devissageqGevrey} below.


\section{D\'evissage $q$-Gevrey}
\label{section:devissageqGevrey}

The following results come mostly from \cite{JSStokes}.


\subsection{The abelian case of two slopes}
\label{subsection:abeliancase}

\index{one level case}
\index{two slopes case}
\index{abelian case}
When the Newton polygon of $M_{0}$ has only two slopes $\mu < \nu$
with multiplicities $r,s$, so that we can write:
$$
A_{0} = \begin{pmatrix} z^{\mu} B & 0 \\ 0 & z^{\nu} C \end{pmatrix},
\quad B \in \GL_{r}(\C), \; C \in \GL_{s}(\C),
$$
then the unipotent group $\G_{A_{0}}$ is isomorphic to the vector space
$\Mat_{r,s}$ through the isomorphism:
$$
F \mapsto \begin{pmatrix} I_{r} & F \\ 0 & I_{s} \end{pmatrix}
$$
from the latter to the former. Accordingly, the sheaf of unipotent 
groups $\Lambda_I(M_0)$ can be identified with the abelian sheaf
$\Lambda$ on $\Eq$ defined by:
$$
\Lambda(U) := \{F \in \Mat_{r,s}\bigl(\O(p^{-1}(U))\bigr) \tq 
(\sq F) (z^{\nu} C) = (z^{\mu} B) F \}.
$$
\index{locally free sheaf over $\Eq$}
\index{holomorphic vector bundle over $\Eq$}
This is actually a locally free sheaf, hence the sheaf of sections 
of a holomorphic vector bundle on $\Eq$, which we also write $\Lambda$
(more on this in section \ref{section:vectorbundles}). This bundle
can be described geometrically as the quotient of the trivial bundle 
$\C^{*} \times \Mat_{r,s}(\C)$ over $\C^{*}$ by the equivariant action 
of $q^{\Z}$ determined by the action 
$(z,F) \mapsto \bigl(q z,(z^{\mu} B) F (z^{\nu} C)^{-1}\bigr)$ 
of the generator $q$:
$$
\Lambda = \dfrac{\C^{*} \times \Mat_{r,s}(\C)}
{(z,F) \sim \bigl(q z,(z^{\mu} B) F (z^{\nu} C)^{-1}\bigr)} 
\longrightarrow 
\dfrac{\C^{*}}{z \sim q z} = \Eq.
$$
For details, see \cite{JSStokes,JSfvqde}. This bundle is the tensor
product of a line bundle of degree $\mu - \nu$ (corresponding to
the ``theta'' factor $\sq f = z^{\mu - \nu} f$) and of a flat bundle 
(corresponding to the ``fuchsian'' factor $\sq F = B F C^{-1}$);
we shall say that it is \emph{pure (isoclinic)}. 
\index{pure isoclinic vector bundle}
Now, the first cohomology
group of such a vector bundle is a finite dimensional vector space
whose dimension can be easily computed from its rank $r s$ and degree
$\mu - \nu$: it is $r s (\nu - \mu)$. Actually, using Serre duality
and an explicit frame of the dual bundle (made up of theta functions),
one can provide an explicit coordinate system on $H^{1}(\Eq,\Lambda)$:
\index{coordinate system on $H^{1}(\Eq,\Lambda)$}
this is done in \cite{JSfvqde}, where the relation to the $q$-Borel 
transform is shown. In this section, we shall see how, in general,
the non-abelian cohomology set $H^{1}(\Eq,\Lambda)$ can be described
\index{non-abelian cohomology set $H^{1}(\Eq,\Lambda)$}
from successive extensions from the abelian situation.


\subsection{A sequence of central extensions}
\label{subsection:centralextensions}

\label{not91}
For simplicity, we write $\G$ for $\G_{A_{0}}$. The Lie algebra $\g$
of $\G$ consists in all nilpotent matrices of the form:
$$
\begin{pmatrix}
0_{r_{1}}  & \ldots & \ldots & \ldots & \ldots \\
\ldots & \ldots & \ldots  & F_{i,j} & \ldots \\
0      & \ldots & \ldots   & \ldots & \ldots \\
\ldots & 0 & \ldots  & \ldots & \ldots \\
0      & \ldots & 0       & \ldots & 0_{r_{k}}    
\end{pmatrix},
$$
where $0_r$ is the square null matrix of size $r$ and where 
each $F_{i,j}$ is in $\Mat_{r_{i},r_{j}}$ for $1 \leq i < j \leq k$.
For each integer $\delta$, write $\g^{\geq \delta}$ the sub-Lie
algebra of matrices whose only non null blocks $F_{i,j}$ have
level $\mu_{j} - \mu_{i} \geq \delta$; it is actually an ideal
of $\g$ and $\G^{\geq \delta} := I_{n} + \g^{\geq \delta}$ is a
normal subgroup of $\G$. Moreover, one has an exact sequence:
$$
0 \rightarrow \G^{\geq \delta}/\G^{\geq \delta+1}
\rightarrow \G/\G^{\geq \delta+1} 
\rightarrow \G/\G^{\geq \delta} \rightarrow 1,
$$
actually, a central extension. The map $g \mapsto I_{n} + g$ induces
an isomorphism from the vector space $\g^{\geq \delta}/\g^{\geq \delta+1}$
to the kernel $\G^{\geq \delta}/\G^{\geq \delta+1}$ of this exact sequence.
We write $g^{(\delta)}$ this group: it consists of matrices in $\g$ whose 
only non null blocks $F_{i,j}$ have level $\mu_{j} - \mu_{i} = \delta$. 
(This is the reason why we wrote a $0$ instead of a $1$ at the left of 
the exact sequence !) \\

\label{not92}
Restricting the above consideration to $\Lambda := \Lambda_{I}(M_{0})$ 
considered as a sheaf of subgroups of $\G$ (with coefficients in function 
fields), we get a similar central extension:
$$
0 \rightarrow \lambda^{(\delta)} \rightarrow 
\Lambda/\Lambda^{\geq \delta+1} \rightarrow 
\Lambda/\Lambda^{\geq \delta} \rightarrow 1.
$$
Now, a fundamental theorem in the non-abelian cohomology of sheaves 
says that we have an exact sequence of cohomology sets:

\begin{theo}
\label{theo:devissageH1}
One has an exact sequence of pointed sets:
\index{exact sequence of pointed sets}
$$
0 \rightarrow H^{1}(\Eq,\lambda^{(\delta)}) \rightarrow 
H^{1}(\Eq,\Lambda/\Lambda^{\geq \delta+1}) \rightarrow 
H^{1}(\Eq,\Lambda/\Lambda^{\geq \delta}) \rightarrow 1.
$$
\end{theo}
\begin{proof}
The exactness should be here understood in a rather strong sense.
The sheaf $\lambda^{(\delta)}$ is here a pure isoclinic holomorphic 
vector bundle of degree $- \delta$ and rank 
$\sum\limits_{\mu_{j} - \mu_{i} = \delta} r_{i} r_{j}$. Its first cohomology 
group $V^{(\delta)}$ is therefore a vector space of dimension 
$\delta \sum\limits_{\mu_{j} - \mu_{i} = \delta} r_{i} r_{j}$.
The theorem says that $V^{(\delta)}$ operates on the pointed set
$H^{1}(\Eq,\Lambda/\Lambda^{\geq \delta+1})$ with quotient the pointed 
set $H^{1}(\Eq,\Lambda/\Lambda^{\geq \delta})$. For a proof, see
\cite[th. 1.4 and prop. 8.1]{Frenkel}.
\end{proof}

\begin{coro}
\label{coro:devissageH1}
There is a natural bijection of $H^{1}(\Eq,\Lambda_I(M_0))$ with 
$\bigoplus\limits_{\delta \geq 1} H^{1}(\Eq,\lambda^{(\delta)})$.
\end{coro}
\begin{proof}
Indeed, the cohomology sets  being pointed, the theorem yields a bijection 
of each $H^{1}(\Eq,\Lambda/\Lambda^{\geq \delta+1})$ with
$H^{1}(\Eq,\Lambda/\Lambda^{\geq \delta}) \times 
H^{1}(\Eq,\lambda^{(\delta)})$.
\end{proof}

\label{not93}
We shall write 
$\F_{\leq \delta}(M_{0}) := H^{1}(\Eq,\Lambda/\Lambda^{\geq \delta+1})$.
This space is the solution to the following classification problem:
two matrices with diagonal part $A_{0}$ are declared equivalent 
if their truncatures corresponding to the levels $\leq \delta$ 
are analytically equivalent (through a gauge transform in $\G$);
and there is no condition on the components with levels $> \delta$.
This is exactly the equivalence under $\Kfs$ for $\delta < s \leq \delta + 1$.
The corresponding Birkhoff-Guenther normal forms have been described
in section \ref{section:qGevreyinterpolation}. The extreme cases
are $\delta > \mu_{k} - \mu_{1}$, where $\F_{\leq \delta}(M_{0})$
is the whole space $\F(M_{0})$; and $\delta < 1$, where it is trivial.


\subsection{Explicit computation}
\label{subsection:explicitcomputation}

We want to make more explicit theorem \ref{theo:devissageH1}. \\

So we assume that $A,A'$ represent two classes in 
$H^{1}(\Eq,\Lambda/\Lambda^{\geq \delta+1})$ having the same image
in $H^{1}(\Eq,\Lambda/\Lambda^{\geq \delta})$.
Up to an analytic gauge transform, the situation just described
can be made explicit as follows. Assume that $A,A'$ have graded 
part $A_{0}$ and are in Birkhoff-Guenther normal form. Assume 
moreover that they coincide up to level $\delta - 1$. Then the same
is true for $\hat{F}_{A}$ and $\hat{F}_{A'}$; therefore, it is also 
true for $\hat{F}_{A,A'}$ and $I_{n}$, that is, 
$\hat{F}_{A,A'} \in \G^{\geq \delta}(\Kf)$. The first non-trivial
upper diagonal of $\hat{F}_{A,A'}$ is at level $\delta$ and, as
a divergent series, generically it actually has $q$-Gevrey level 
$\delta$; if it has level $< \delta$, then it is convergent.
The classes of $A,A'$ in
$\F_{\leq \delta}(M_{0}) = H^{1}(\Eq,\Lambda/\Lambda^{\geq \delta+1})$
have the same image in
$\F_{\leq \delta-1}(M_{0}) = H^{1}(\Eq,\Lambda/\Lambda^{\geq \delta})$.
After theorem \ref{theo:devissageH1}, there exists 
a unique element of $V^{(\delta)}$ which carries the class of $A$ 
to the class of $A'$. The Stokes matrices 
$S_{c,d} \hat{F}_{A} := \left(S_{c,d} \hat{F}_{A}\right)^{-1} S_{d} \hat{F}_{A}$
\label{not94} 
and 
$S_{c,d} \hat{F}_{A'} := \left(S_{c,d} \hat{F}_{A'}\right)^{-1} S_{d} \hat{F}_{A'}$ 
are congruent modulo $\Lambda^{\geq \delta}$, 
so that their quotient $S_{c,d} \hat{F}_{A,A'}$ is in $\Lambda^{\geq \delta}$. 
Its first non trivial upper diagonal is at level $\delta$; call it 
$f_{c,d}$ and consider it as an element of $\lambda^{(\delta)}$. This 
defines a cocycle, hence a class in $H^{1}(\Eq,\lambda^{(\delta)})$. 
This class is the element of $V^{(\delta)}$ we look for.

\begin{exem}
\label{exem:geometrythreeslopes}
Take 
$A_{0} := \begin{pmatrix} a & 0 & 0 \\ 
0 & b z & 0 \\ 0 & 0 & c z^{2} \end{pmatrix}$
and
$A := \begin{pmatrix} a & u & v_{0} + v_{1} z \\ 
0 & b z & w z \\ 0 & 0 & c z^{2} \end{pmatrix}$
in Birkhoff-Guenther normal form, so that $u,v_{0},v_{1},w \in \C$.
Here, the total space of classes $H^{1}(\Eq,\Lambda_I(M_0))$ has
dimension $4$, while each of the components $H^{1}(\Eq,\lambda^{(1)})$
and $H^{1}(\Eq,\lambda^{(2)})$ has dimension $2$. \\
To compute the class of $A$ in $H^{1}(\Eq,\Lambda_I(M_0))$ identified 
with $H^{1}(\Eq,\lambda^{(1)}) \oplus H^{1}(\Eq,\lambda^{(2)})$,
we introduce the intermediate point
$A' := \begin{pmatrix} a & u & 0 \\ 
0 & b z & w z \\ 0 & 0 & c z^{2} \end{pmatrix}$
and the intermediate gauge transforms
$G := \begin{pmatrix} 1 & f & x \\ 0 & 1 & h \\ 0 & 0 & 1 \end{pmatrix}$
such that $G[A_{0}] = A'$ and
$H := \begin{pmatrix} 1 & 0 & g \\ 0 & 1 & 0 \\ 0 & 0 & 1 \end{pmatrix}$
such that $H[A'] = A$, so that
$F := H G =
\begin{pmatrix} 1 & f & g + x \\ 0 & 1 & h \\ 0 & 0 & 1 \end{pmatrix}$
is such that $F[A_{0}] = A$. \\
The component in $H^{1}(\Eq,\lambda^{(1)})$ is computed by considering 
$G$ alone. Its coefficients satisfy the equations:
$$
b z \sq f - a f = u \text{~and~} c z \sq h - b h = w.
$$
(The coefficient $x$ is irrelevant here.) We shall see in section
\ref{section:qEuler} how to compute the corresponding cocycles
$(f_{c,d})$ and $(h_{c,d})$, but what is clear here is that they 
are respectively linear functions of $u$ and of $w$. So in the
end, the component of the class of $A$ in $H^{1}(\Eq,\lambda^{(1)})$
is $u L_{1}(a,b) + w L_{1}(b,c)$ for some explicit basic classes
$L_{1}(a,b)$, $L_{1}(b,c)$ (the index $1$ is for the level). \\
Similarly, the component in $H^{1}(\Eq,\lambda^{(2)})$ is computed 
by considering $H$ alone. Its coefficient satisfies the equation:
$$
c z^{2} \sq g - a g = v_{0} + v_{1} z.
$$
It is therefore clear that the component we look for is
$v_{0} L_{2,0}(a,c) + v_{1} L_{2,1}(a,c)$ for some explicit 
base of classes $L_{2,0}(a,c),L_{2,1}(a,c)$. \\
This example will be pursued in section \ref{section:qEulersquare}.
\end{exem}


\subsection{Various geometries on $\F(P_{1},\ldots,P_{k})$}
\label{subsection:variousgeometries}

In subsection \ref{subsection:BGNF}, we drawed from the 
Birkhoff-Guenther normal form an affine structure on $\F(M_{0})$.
This structure is made explicit by the coordinates provided
by proposition \ref{prop:kintegralslopesspace}. The d\'evissage 
above implies that it is the same as the affine structure on
$H^{1}(\Eq,\Lambda_I(M_0))$ inherited from the vector space
structures on the $H^{1}(\Eq,\lambda^{(\delta)})$ through corollary 
\ref{coro:devissageH1}. More precisely:

\begin{theo}
The mapping from 
$\prod\limits_{1 \leq i < j \leq k} \Mat_{r_{i},r_{j}}(K_{\mu_{i},\mu_{j} })$
to $\bigoplus\limits_{\delta \geq 1} H^{1}(\Eq,\lambda^{(\delta)})$
coming from proposition \ref{prop:kintegralslopesspace}, corollary 
\ref{coro:devissageH1} and theorems \ref{theo:secondqMalgrangeSibuya} 
and \ref{section:analyticclassificationCZ} is linear.
\end{theo}
\begin{proof}
The computations are in essence the same as in the example. Details
will be written in \cite{JSmoduli}.
\end{proof}

There is a third source for the geometry on $\F(M_{0})$, namely
the identification with the space $\bigoplus E_{i,j,c,d}$ of all
privileged cocycles (where $c,d$ are fixed arbitrary) found in
subsection \ref{subsection:privilegedcocycles}. The corresponding
geometry is the same, see \emph{loc. cit.}.


\section{Vector bundles associated to $q$-difference modules}
\label{section:vectorbundles}

We briefly recall here the general construction of which the vector
bundle in subsection \ref{subsection:abeliancase} is an example, and 
then give an application. This is based on \cite{JSfvqde}.


\subsection{The general construction}
\label{subsection:vectorbundles}

For details on the following, see \cite{JSfvqde}. 
To any $q$-difference module $M$ over $\Ka$, one can associate
\label{not95}
a holomorphic vector bundle $\F_{M}$ over $\Eq$ in such a way
\index{holomorphic vector bundle associated to a $q$-difference module}
that the correspondence $M \leadsto \F_{M}$ is functorial and 
that the functor is faithful, exact and compatible with tensor 
products and duals. If one restricts to pure isoclinic modules
with a fixed slope, the functor is moreover \emph{fully} faithful
(this follows immediately from the case of fuchsian modules, which
is dealt with in \cite{JSGAL}); but this ceases to be true for 
arbitrary modules: for instance, if $M := (\Ka,\sq)$ and 
$N := (\Ka,z^{-1}\sq)$, then the function $\thq$ induces a morphism 
from $\F_{M}$ to $\F_{N}$, although $\Hom(M,N) = 0$. This is one of 
the reasons why this functor is not very important in the present work 
(see however the remark further below). \\

In order to describe $\F_{M}$, we shall assume for simplicity that
the module $M$ is related to a $q$-difference system $\sq X = A X$
such that $A$ and $A^{-1}$ are holomorphic all over $\C^{*}$, for 
instance, that $A$ is in Birkhoff-Guenther normal form. (In the 
general case, one just has to speak of germs at $0$ everywhere). 
The sheaf of holomorphic solutions over $\Eq$ is then defined by 
the relation:
$$
\label{not96}
\F_{M}(U) := \{X \in \O\bigl(p^{-1}(U)\bigr)^{n} \tq \sq X = A X\}.
$$
This is a locally free sheaf, whence the sheaf of sections of
a holomorphic vector bundle over $\Eq$ which we also write $\F_{M}$.
The bundle $\F_{M}$ can be realized geometrically as the quotient of 
the trivial bundle $\C^{*} \times \C^{n}$ over $\C^{*}$ by an equivariant 
action of the subgroup $q^{\Z}$ of $\C^{*}$:
$$
\F_{M} = \dfrac{\C^{*} \times \C^{n}}{(z,X) \sim \bigl(q z,A X)} 
\rightarrow \dfrac{\C^{*}}{z \sim q z} = \Eq.
$$
The construction of the bundle $\Lambda$ in subsection
\ref{subsection:abeliancase} corresponds to the $q$-difference
system $(\sq F) (z^{\nu} C) = (z^{\mu} B) F$, which is the
``internal Hom'' $\Homint(N,M)$ of the modules $M,N$
respectively associated to the $q$-difference systems with
matrices $z^{\mu} B,z^{\nu} C$. From the compatibilities with
tensor products and duals, one therefore draws:
$$
\Lambda = \F_{N}^{\vee} \otimes \F_{M}.
$$

\begin{rema}
From the exactness and the existence of slope filtrations, one
deduces that the vector bundle associated to a $q$-difference
module with integral slopes admits a flag of subbundles such
that each quotient is ``pure isoclinic'', that is, isomorphic
to the tensor product of a line bundle with a flat bundle. The
functor sending the module $M$ to the vector bundle $\F_{M}$
\emph{endowed with such a flag} is fully faithful.
\end{rema}


\subsection{Sheaf theoretical interpretation of irregularity}
\label{subsection:sheaftheoreticirregularity}
\index{irregularity (sheaf theoretical interpretation)}

We interpret here in sheaf theoretical terms the formula relating
irregularity and the dimension of $\dim \F(P_{1},\ldots,P_{k})$
given in subsection \ref{subsection:algebraicirregularity}. \\

Actually, the irregularity of $\underline{\End}(M_{0})$ comes from
its positive part:
$$
\label{not97}
\underline{\End}^{> 0}(M_{0}) = 
\bigoplus_{1 \leq i < j \leq k} \Homint(P_{i},P_{j}).
$$
For each pure component $P_{i,j} := \Homint(P_{j},P_{i})$, 
the computations of sections \ref{section:isoformalclasses} and
\ref{section:indextheorems} give:
$$
\dim \Gamma^{0}(P_{i,j}) - \dim \Gamma^{1}(P_{i,j}) = 
r_{i} r_{j} (\mu_{j} - \mu_{i}).
$$
As we know, $\F(P_{1},\ldots,P_{k}) = \F(M_{0})$ is isomorphic 
to the first cohomology set of the sheaf $\Lambda_I(M_{0})$.
Also, in the previous section, we have seen that this is an 
affine space with underlying vector space the first cohomology 
space of the sheaf $\lambda_I(M_{0})$ of their Lie algebras. 
And the latter, from its description as
$\bigoplus_{\delta \geq 1} \lambda^{(\delta)}$, is the vector bundle 
associated to the $q$-difference module $\underline{\End}^{> 0}(M_{0})$. 
So the irregularity is actually an Euler-Poincar\'e characteristic.



\chapter{Examples of the Stokes phenomenon}
\label{chapter:examplesStokes}

In the study of $q$-special functions, one frequently falls upon series
that are \emph{convergent} solutions of \emph{irregular} $q$-difference
equations, the latter thus also admitting \emph{divergent} solutions;
and this should be so, since Adams lemma\footnote{In its original form,
Adams lemma \cite{Adams,MarotteZhang,LDVS} says that, for any analytic 
$q$-difference operator, the power series appearing in the solutions 
related to the \emph{biggest} slope have a strictly positive radius of 
convergence: see for instance the $q$-Euler equation in paragraph
\ref{subsection:qEuler}. Adams lemma was used implicitly in paragraph
\ref{subsection:slopefiltration} through its consequences: existence
of an analytic factorisation and definition of the slope filtration.}
\index{Adams lemma}
ensures us that irregular equations always have \emph{some} convergent 
solutions. However, in most works, only
convergent solutions have been considered, although the other ones
are equally interesting\footnote{This was of course well known to 
Stokes himself when he studied the Airy equation, as well as to all
those who used divergent series in numerical computations of celestial
mechanics, leading to Poincar\'e work on asymptotics. But, of course, 
one should above all remember Euler, who used divergent series for
numerically computing $\zeta(2) = \pi^{2}/6$, in flat contradiction
to the opinion expressed by Bourbaki in ``Topologie G\'en\'erale'', 
IV, p. 71.} \\

The touchstone of any theory of the Stokes phenomenon is Euler series.
\index{Euler series}
We shall therefore concentrate on one of its $q$-analogs, the $q$-Euler
\index{$q$-Euler series} 
\index{Tshakaloff series} 
or Tshakaloff series $\tsh(z)$ (see equation \eqref{eqn:Tshakaloff}, 
page \pageref{eqn:Tshakaloff}). The simplest $q$-difference equation 
satisfied by $\tsh$ is the $q$-Euler equation $z \sq f - f = -1$.
\index{$q$-Euler equation}
\index{Stokes phenomenon}
We shall detail the Stokes phenomenon for a family of similar equations 
in \ref{section:qEuler} and we shall apply it to some confluent basic 
hypergeometric series. Then we shall show how such equations naturally 
appear in some well known historical cases: that of Mock Theta functions 
in \ref{section:MockTheta}, that of the enumeration of class numbers 
of quadratic forms in \ref{section:QuadForms}; this has been exploited 
by the third author in \cite{CZMTF,CZMORD}.

\subsubsection{Notations}

We use some notations from $q$-calculus 
\index{$q$-calculus}
for this chapter only.
Let $p \in \C$ be such that $0 < \lmod p \rmod < 1$, \eg\
$p := q^{-1}$. Let $a,a_{1},\ldots,a_{k} \in \C$ and $n \in \N$. 
The \emph{$q$-Pochhammer symbols} are:
\label{not98}
\begin{align*}
\index{$q$-Pochhammer symbol}
(a;p)_{n} & := \prod_{0 \leq i < n} (1 - a p^{i}), \\
(a;p)_{\infty} & := \prod_{i = 0}^{\infty} (1 - a p^{i}), \\
(a_{1},\ldots,a_{k};p)_{n} & := \prod_{j=1}^{k} (a_{j};p)_{n}, \\
(a_{1},\ldots,a_{k};p)_{\infty} & := \prod_{j=1}^{k} (a_{j};p)_{\infty}.
\end{align*}
Note that, for all $n \in \N$, we have:
$$
(a;p)_{n} = \dfrac{(a;p)_{\infty}}{(a;p^{n})_{\infty}},
$$
but, since the right hand side is well defined for all $n \in \Z$,
this allows us to extend the definition of $q$-Pochhammer symbols;
and similarly for $(a_{1},\ldots,a_{k};p)_{n}$. \\

\index{Jacobi theta function}
\index{theta function}
\label{not99}
With these notations, Jacobi's theta functions $\theta(z;q)$ and $\thq$ 
respectively defined in equations \eqref{eqn:tripleproduit} and
\eqref{eqn:thq} page \pageref{eqn:thq} admit the factorisations:
\begin{align}
\label{eqn:deuxtriplesproduits}
\theta(z;q) & = (q^{-1},- z,- q^{-1} z^{-1};q^{-1})_{\infty}, \\
\thq(z) & = (q^{-1},-q^{-1} z,-z^{-1};q^{-1})_{\infty}.
\end{align}
(This is again Jacobi's triple product formula.) 
\index{Jacobi's triple product formula}
In section
\ref{section:qEuler}, we use functions derived from $\theta(z;q)$,
as in chapter \ref{chapter:asymptotictheory}. 
In section \ref{section:qEulersquare}, we use functions derived from  
$\thq$, as in chapter \ref{chapter:geometryofF(M)}.
In section \ref{section:MockTheta}, we use both kinds in alternance.
Last, note that, along with these two forms, we shall also use in 
section \ref {section:QuadForms} the ``classical'' forms $\theta_{i,j}$, 
$i,j \in \{0,1\}$. 


\section{The $q$-Euler equation and confluent 
basic hypergeometric series}
\label{section:qEuler}
\index{$q$-Euler equation}
\index{confluent basic hypergeometric series}

Consider a $q$-difference system of rank $2$ and level $1$ (that is,
its slopes are $\mu,\mu+1$ for some $\mu \in \Z$). Through some analytic
gauge transformation, its matrix can be put in the form
$b z^{\mu} ~ \begin{pmatrix} 1 & u \\ 0 & a z \end{pmatrix}$,
where $a,b \in \C^{*}$, $\mu \in \Z$ and $u \in \Ka$. The $b z^{\mu}$
factor corresponds to a tensor product by a rank one object $L$, which
does not affect the Stokes phenomenon, nor the isoformal classification,
\ie\ the map $M \leadsto L \otimes M$ induces an isomorphism:
$$
\F(P_{1},\ldots,P_{k}) \rightarrow \F(L \otimes P_{1},\ldots,L \otimes P_{k}).
$$
We therefore assume that $b = 1, \mu = 0$. The associated inhomogeneous
equation is the following \emph{$q$-Euler equation}:
\begin{equation}
\label{eqn:qEuler}
\begin{pmatrix} 1 & f \\ 0 & 1 \end{pmatrix}:
\begin{pmatrix} 1 & 0 \\ 0 & a z \end{pmatrix} \simeq
\begin{pmatrix} 1 & u \\ 0 & a z \end{pmatrix}
\Longleftrightarrow
a z \sq f - f = u.
\end{equation}


\subsection{A digest on the $q$-Euler equation}
\label{subsection:qEuler}

\subsubsection{Formal solution}

We obtain it, for instance, by iterating the $z$-adically contracting
operator $f \mapsto -u + a z \sq f$. One finds the fixed point:
$$
\hat{f} = - \sum_{n \geq 0} a^{n} q^{n(n-1)/2} z^{n} \sq^{n} u.
$$
If $u$ is a constant, the right hand side is just $- u \tsh(a z)$.
If $u = \sum\limits_{k >> - \infty} u_{k} z^{k}$, then 
$\hat{f} = - \sum\limits_{k >> - \infty} u_{k} z^{k} \tsh(q^{k} a z)$.

\subsubsection{Birkhoff-Guenther normal form}

There is a unique $\alpha \in \C$ such that, setting $v := u - \alpha$,
the unique formal solution of $a z \sq f - f = v$ is convergent.
Indeed, putting $v = \sum\limits_{k >> - \infty} v_{k} z^{k}$, from
the relation $a q^{n-1} f_{n-1} - f_{n} = v_{n}$, we draw that
$\sum\limits_{n \in \Z} a_{n} q^{-n(n-1)/2} v_{n} = 0$, which also
writes: $\alpha = \Bq u(a^{-1})$, where, as usual (see definition 
\ref{defi:qBorel}, page \pageref{defi:qBorel}),
$\Bq u(\xi) = \sum\limits_{k >> - \infty} u_{k} q^{-k(k-1)/2} \xi^{k}$.
(The letter $\xi$ is traditional for the Borel plane.) Note that the
value $\Bq u(a^{-1})$ is well defined, for $\Bq u$ is an entire function.
The Birkhoff-Guenther normal form of 
$\begin{pmatrix} 1 & u \\ 0 & a z \end{pmatrix}$ is therefore
$\begin{pmatrix} 1 & \Bq u(a^{-1}) \\ 0 & a z \end{pmatrix}$.

\subsubsection{``Algebraic'' summation}

For any $a \in \C^{*}$, we shall set:
$$
\label{not100}
\theta_{q,a}(z) := \thq(z/a).
$$
(See \eqref{eqn:thq} in the general notations of section 
\ref{section:generalnotations}.)
One looks for a solution of \eqref{eqn:qEuler} in the form
$f = g/\theta_{q,\lambda}$, with $g \in \O(\C^{*})$ and $\lambda$
adequately chosen in $\C^{*}$. We are led to solve the equation:
$a \lambda \sq g - g = u \theta_{q,\lambda}$. Identifying the
coefficients of the Laurent series, one gets the unique solution:
$$
g = \sum\limits_{n \in \Z} 
\dfrac{\left[u \theta_{q,\lambda}\right]_{n}}{a \lambda q^{n} - 1} z^{n},
$$
which makes sense if, and only if, $\lambda \not\in [a^{-1};q]$. 
(Note that we write $\left[h\right]_{n} := a_{n}$
\label{not101}
when $h = \sum a_{n} z^{n}$.)
Thus, for all ``authorized directions of summation'' $\lambda$, we get 
the unique solution for \eqref{eqn:qEuler} with only simple poles
over $[-\lambda;q]$:
$$
S_{\lambda} \hat{f} := \dfrac{1}{\theta_{q,\lambda}}
\sum\limits_{n \in \Z} 
\dfrac{\left[u \theta_{q,\lambda}\right]_{n}}{a \lambda q^{n} - 1} z^{n},
$$

\subsubsection{``True'' summation}

Since the polar condition that uniquely characterizes the solution
$S_{\lambda} \hat{f}$ only depends on 
$\overline{\lambda} := \lambda \pmod{q^{\Z}} \in \Eq$,
or (equivalently) on $\Lambda := [\lambda;q]$, we shall also write:
$$
S_{\overline{\lambda}} \hat{f} := S_{\Lambda} \hat{f} := S_{\lambda} \hat{f}.
$$
We now make this dependency explicit. From the equality
$u \theta_{q,\lambda} = \sum u_{k} q^{-\ell(\ell+1)/2} z^{k} (z/\lambda)^{\ell}$,
we draw:
$$
\left[u \theta_{q,\lambda}\right]_{n} =
\sum_{k + \ell = n} u_{k} q^{-\ell(\ell+1)/2} \lambda^{-\ell} =
\sum_{k} u_{k} q^{-(n-k)(n-k+1)/2} \lambda^{k-n},
$$
so that:
$$
q^{n(n+1)/2} \lambda^{n} \left[u \theta_{q,\lambda}\right]_{n} =
\sum_{k} u_{k} q^{n k-k(k-1)/2} \lambda^{k} = \Bq u(q^{n} \lambda).
$$
On the other hand, iterating the relation 
$\theta_{q,\lambda}(z) = \frac{z}{q \lambda} \theta_{q,q \lambda}(z)$
yields 
$\theta_{q,\lambda}(z) = 
\frac{z^{n}}{q^{n(n+1)/2} \lambda^{n}} \theta_{q,q^{n} \lambda}(z)$, whence:
$$
\label{not102}
S_{\overline{\lambda}} \hat{f} = S_{\Lambda} \hat{f} =
\sum_{n \in \Z} \dfrac{q^{-n(n+1)/2} \lambda^{-n} \Bq u(q^{n} \lambda)}
{z^{n} q^{-n(n+1)/2} \lambda^{-n} \theta_{q,q^{n} \lambda}(z) (a \lambda q^{n} - 1)} 
~ z^{n} =
\sum_{\mu \in \Lambda} \dfrac{\Bq u(\mu)}{(a \mu - 1) \theta_{q,\mu}(z)}
~ z^{n}.
$$

\begin{rema}
In \cite{RS1,RS2,RS3}, one computes the residue of the meromorphic
function $\overline{\lambda} \mapsto S_{\overline{\lambda}} \hat{f}$
at the pole $\overline{a^{-1}} \in \Eq$. According to the above formula,
one finds:
$$
\text{Res}_{\overline{\lambda} = \overline{a^{-1}}} 
S_{\overline{\lambda}} \hat{f} =
\dfrac{1}{2 \ii \pi} ~ \dfrac{\Bq u(a^{-1})}{\theta_{q}(az)} \cdot
$$
Indeed, for any $b \in \C^{*}$ and for any map $f: \C^{*} \rightarrow \C$ 
analytic in a neighborhood of $[b;q]$, setting:
$$
\forall \overline{\lambda} \in \Eq ~,~
F\left(\overline{\lambda}\right) :=
\sum_{\mu \in [\lambda;q]} \dfrac{f(\mu)}{\mu - b},
$$
defines a meromorphic map $F$ with a simple pole at $\overline{b} \in \Eq$
and the corresponding residue:
$$
\text{Res}_{\overline{\lambda} = \overline{b}} 
S_{\overline{\lambda}} \hat{f} =
\dfrac{1}{2 \ii \pi} ~ \dfrac{f(b)}{b} \cdot
$$
Note that the residue of a function here makes sense, because of
the canonical generator $dx = \dfrac{1}{2 \ii \pi} ~ \dfrac{dz}{z}$
of the module of differentials, which allows one to flatly identify
maps on $\Eq$ with differentials. (Here, as usual, $z = e^{2 \ii \pi x}$, 
where $x$ is the canonical uniformizing parameter of 
$\Eq = \C/(\Z + \Z \tau)$.)
\end{rema}


\subsection{Some confluent basic hypergeometric functions}
\label{subsection:confluenthypergeom}

\index{confluent basic hypergeometric functions}
Usual basic hypergeometrics series have the form:
$$
\label{not103}
{}_{2} \Phi_{1}(a,b;c;q,z) := \sum_{n \geq 0}
\dfrac{(a,b;q^{-1})_{n}}{(c,q^{-1};q^{-1})_{n}} z^{n}~,~
\text{~where~} a,b,c \in \C^{*}.
$$
(Remember that here, $\lmod q \rmod > 1$.) Writing for short $F(z)$
this series, the rescaling $F(- q^{-1} z/c)$ degenerates, when
$c \to \infty$, into a \emph{confluent} basic hypergeometrics series:
$$
\phi(a,b;q,z) := \sum_{n \geq 0}
\dfrac{(a,b;q^{-1})_{n}}{(q^{-1};q^{-1})_{n}} q^{n(n+1)/2} z^{n}~,~
\text{~where~} a,b \in \C^{*}.
$$
Writing $u_{n} := \dfrac{(a,b;q^{-1})_{n}}{(q^{-1};q^{-1})_{n}} q^{n(n+1)/2}$
the general coefficient, we have, for all $n \geq 0$:
$$
(q^{n+1} - 1) u_{n+1} = q^{2} (q^{n} - a) (q^{n} - b) u_{n},
$$
whence, multiplying by $z^{n+1}$ and summing:
$$
(\sq - 1) \hat{f} = q^{2} z (\sq - a) (\sq - b) \hat{f},
$$
where we write for short $\hat{f}(z)$ the divergent series
$\phi(a,b;q,z)$. We shall denote the corresponding $q$-difference
operator as:
$$
L := q^{2} z (\sq - a) (\sq - b) - (\sq - 1) =
q^{2} z \sq^{2} - \bigl(1 + (a+b) q^{2} z)\bigr) \sq + (1 + ab q^{2} z).
$$
We are interested in the equation $L f = 0$. Its Newton polygon (\ie\
that of $L$) has slopes $0$ and $1$. The slope $0$ has exponent $1$ and 
gives rise to the divergent solution $\hat{f}$. To tackle the slope $1$,
we compute:
$$
(z \thq) L (z \thq)^{-1} = \dfrac{1}{q z}
\Bigl(\sq^{2} - \bigl(1 + (a+b) q^{2} z)\bigr) \sq + 
q z (1 + ab q^{2} z)\Bigr),
$$
which has slopes $0$ and $-1$, the latter having exponent $0$. According 
\index{Adams lemma}
to Adams lemma (see the footnote in the introduction to this chapter), we 
thus get a unique solution $g_{0} \in 1 + z \Ra$, whence the ``convergent'' 
solution $f_{0} := \dfrac{g_{0}}{z \thq}$ of equation $L f = 0$.

\subsubsection{Factoring $L$}

To get $L = q^{2} z (\sq - A)(\sq - B)$, we first look for $B$ such that 
$(\sq - B) f_{0} = 0$, that is, 
$B := \dfrac{\sq f_{0}}{f_{0}} = 
\dfrac{1}{q z} ~ \dfrac{\sq g_{0}}{g_{0}} \cdot$. Thus, 
$L = (\sq - A)(q z \sq - q z B) = (\sq - A)
\left(q z \sq - \dfrac{\sq g_{0}}{g_{0}}\right)$ 
and, by identification of the constant terms:
$A = \dfrac{(1 + a b q^{2} z) g_{0}}{\sq g_{0}}$ and in the end:
$$
L = \left(\sq - \dfrac{(1 + a b q^{2} z) g_{0}}{\sq g_{0}}\right)
\left(q z \sq - \dfrac{\sq g_{0}}{g_{0}}\right).
$$
The corresponding non homogeneous equation is 
$q z f \sq - \dfrac{\sq g_{0}}{g_{0}} f = v$, where $v$ is a non trivial
solution of
$\left(\sq - \dfrac{(1 + a b q^{2} z) g_{0}}{\sq g_{0}}\right) v = 0$.
An obvious choice is:
$$
v := \dfrac{1}{g_{0}} \prod_{n \geq 1} (1 + a b q^{2-n} z) =
\dfrac{1}{g_{0}} (- a b q z;q^{-1})_{\infty}.
$$
One checks easily that the above non homogeneous equation has indeed
a unique solution in $1 + z \Rf$, and this has to be $\hat{f}$. This
equation is associated with the matrix 
$\begin{pmatrix} \dfrac{\sq g_{0}}{g_{0}} & v \\ 0 & q z \end{pmatrix}$,
which can be seen to be equivalent to the matrix 
$\begin{pmatrix} 1 & u \\ 0 & z \end{pmatrix}$ through the gauge transform
$\begin{pmatrix} g_{0}^{-1} & 0 \\ 0 & z^{-1} \end{pmatrix}$.


\subsection{Some special cases}
\label{subsection:specialcases}

Taking $a := b := q^{-1}$ yields:
$$
\hat{f} = \sum_{n \geq 0} (q^{-1};q^{-1})_{n} q^{n(n+1)/2} z^{n}.
$$
The recurrence relation $\dfrac{u_{n+1}}{u_{n}} = q^{n+1} - 1$ 
immediately gives the non homogeneous equation
$q z \sq f - (1+z) f = -1$. This, in turn, boils down to the straight
$q$-Euler equation by setting $g := \dfrac{z}{(-q^{-1} z;q^{-1})_{\infty}} f$,
so that $z \sq g - g = \dfrac{- z}{(-z;q^{-1})_{\infty}} \cdot$ \\

Taking $a := b := 0$ yields:
$$
\hat{f} = \sum_{n \geq 0} \dfrac{1}{(q^{-1};q^{-1})_{n}} q^{n(n+1)/2} z^{n}.
$$
The recurrence relation $\dfrac{u_{n+1}}{u_{n}} = \dfrac{q^{2n+2}}{q^{n+1} - 1}$ 
gives the homogeneous equation $q^{2} z \sq^{2} f - \sq f + f = 0$.
The convergent solution is $f_{0} := \dfrac{1}{z \thq} g_{0}$, where
$g_{0} = \sum \gamma_{n} z^{n} \in 1 + z \Ra$ is solution of the equation
$(\sq^{2} - \sq  + q z) g = 0$. The corresponding recurrence relation 
$(q^{2 n} - q^{n}) \gamma_{n} + q \gamma_{n-1} = 0$ can be solved exactly
and entails:
$$
g_{0} = \sum_{n \geq 0} \dfrac{(-1)^{n} q^{-n^{2}}}{(q^{-1};q^{-1})_{n}} z^{n}.
$$
The corresponding factorisation is
$L = \left(\sq - \dfrac{g_{0}}{\sq g_{0}}\right)
\left(q z \sq - \dfrac{\sq g_{0}}{g_{0}}\right)$. Following 
\cite{CZCBHG}\footnote{Note that \cite{LuZ} contains similar formulas.}, 
one can prove the following formula for the Stokes operators:
$$
S_{\lambda} \hat{f} - S_{\mu} \hat{f} = (q^{-1};q^{-1})_{\infty}^{2}
\dfrac{\thq(-\lambda/\mu) \thq(z/\lambda\mu)}
{\thq(-1/\lambda) \thq(-1/\mu) \thq(\lambda/z) \thq(z/\mu)}
\sum_{n \in \Z} \dfrac{(-1)^{n} q^{-n^{2}}}{(q^{-1};q^{-1})_{n}} z^{n}.
$$
We shall not attempt to prove this, but a similar formula is checked 
in the next case. \\

Taking $a := 0, b := q^{-1}$ yields:
$$
\hat{f} = \sum_{n \geq 0} q^{n(n+1)/2} z^{n} = \tsh(q z),
$$
which is solution of the $q$-Euler equation $(q z \sq - 1) f = -1$.
A solution of the associated homogeneous equation is 
$\dfrac{1}{\thq(qz)} = \dfrac{1}{\thq(1/z)}$, so that, for any two
$\lambda,\mu \not\in [1;q]$ (authorized directions of summation):
$$
S_{\lambda} \hat{f} - S_{\mu} \hat{f} = \dfrac{K(\lambda,\mu,z)}{\thq(1/z)},
$$
where $K$ is $q$-invariant in each of the three arguments. We assume
$\overline{\lambda} \neq \overline{\mu}$; then, as a function of $z$,
the numerator $K$ is elliptic with simple zeroes over $[-1;q]$ and at 
most simple poles over $[-\lambda;q]$ and $[-\mu;q]$; thus:
$$
\label{not104}
K(\lambda,\mu,z) = K'(\lambda,\mu) ~
\dfrac{\thq(1/z) \thq(z/\lambda \mu)}{\thq(\lambda/z) \thq(z\mu)},
$$
where $K'(\lambda,\mu)$ is independent of $z$ and, as a function of
$\lambda$, has at most simple poles over $[1;q]$ and $[\mu;q]$; thus:
$$
K'(\lambda,\mu) \dfrac{\thq(z/\lambda \mu)}{\thq(\lambda/z)} = 
K'' \dfrac{\thq(z/\lambda \mu) \thq(-\lambda/\mu)}
{\thq(\lambda/z) \thq(-1/\lambda)},
$$
where $K''$ is independent of $\lambda,\mu,z$. Last, we get:
$$
S_{\lambda} \hat{f} - S_{\mu} \hat{f} = C
\dfrac{\thq(-\lambda/\mu) \thq(z/\lambda\mu)}
{\thq(-1/\lambda) \thq(-1/\mu) \thq(\lambda/z) \thq(z/\mu)} \cdot
$$
We shall now see that $C = -\thq'(-1) = (q^{-1};q^{-1})_{\infty}^{3}$.
The second equality follows immediately from Jacobi Triple Product
Formula. Note that, by simple singularity analysis, one may write:
$$
S_{\lambda} \hat{f} = 
\sum_{n \in \Z} \dfrac{\alpha_{n}}{z + \lambda q^{n}} \cdot
$$
Since $S_{\lambda} \hat{f}(0) = 1$, we have $\sum \alpha_{n} q^{-n} = \lambda$.
On the other hand, from the functional equation, taking residues yields
the recurrence relation: $\alpha_{n-1}  = - \lambda \alpha_{n} q^{n-1}$, then
$\alpha_{n} = (-1/\lambda)^{n} q^{-n(n-1)/2} \alpha_{0}$ and in the end:
$$
\lambda = \alpha_{0} \sum_{n \in \Z} (-1/\lambda)^{n} q^{-n(n-1)/2} = 
\alpha_{0} \thq(-1/\lambda) \Longrightarrow 
\alpha_{0} = \dfrac{\lambda}{\thq(-1/\lambda)} \cdot
$$
On the other hand:
\begin{align*}
\alpha_{0} &= 
\lim_{z \to -\lambda} (z + \lambda) (S_{\lambda} \hat{f} - S_{\mu} \hat{f}) \\
&= 
\lim_{z \to -\lambda} (z + \lambda) C
\dfrac{\thq(-\lambda/\mu) \thq(z/\lambda\mu)}
{\thq(-1/\lambda) \thq(-1/\mu) \thq(\lambda/z) \thq(z/\mu)} \\
&=
C \dfrac{1}{\thq(-1/\lambda)} \dfrac{-\lambda}{\thq'(-1)},
\end{align*}
whence the desired conclusion.


\section{The symmetric square of the $q$-Euler equation}
\label{section:qEulersquare}
\index{$q$-Euler equation (symmetric square)}

This will be our only example with more than two slopes. Consider 
the square $\hat Y := \tsh^{2}$ of the Tshakaloff series:
\index{Tshakaloff series (square of)}
$$
\hat Y(z) = \left(\sum_{n\ge 0}q^{n(n-1)/2} z^n\right)^2.
$$
As we shall see in section \ref{section:qEulersquareanalytic}, the
series $\hat Y = \tsh^{2}$ does not support the same process of
analytic summation as $\tsh$ itself. This comes from the fact that 
the Newton polygon of $\hat Y$ has three slopes, as we shall see, while
that of $\tsh$ has two slopes. First, however, we want to give some 
recipes to tackle such examples.


\subsection{Algebraic aspects}
\label{section:qEulersquarealgebraic}

Remember that the equation $f = 1 + z \sq f$ satisfied by $\tsh$
is nothing but the cohomological equation for 
$\begin{pmatrix} 1 & -1 \\ 0 & z \end{pmatrix}$. 


\subsubsection{Newton polygon of a symmetric square}

To find the equation satisfied by a product of two functions, one uses 
the tensor product of two systems or modules; for a square, one uses 
likewise the symmetric square. \\

Let $M = (V,\Phi)$ be a $q$-difference module and
$T^{2} M := M \otimes M = (V \otimes V,\Phi \otimes \Phi)$ its tensor 
square. The linear automorphism $x \otimes y \mapsto y \otimes x$ 
commutes with $\Phi \otimes \Phi$, so that it actually defines an 
involutive $q$-difference automorphism of $M \otimes M$, and a 
splitting:
$$
\label{not105}
T^{2} M = S^{2} M \oplus \Lambda^{2} M.
$$
If $M$ has slopes $\mu_{1},\ldots,\mu_{k}$ with multiplicities
$r_{1},\ldots,r_{k}$, then $T^{2} M$ has slopes the $\mu_{i} + \mu_{j}$ , 
$1 \leq i,j \leq k$ with multiplicities the $r_{i} r_{j}$. (Of course,
if many sums $\mu_{i} + \mu_{j}$ are equal, the corresponding 
multiplicities $r_{i} r_{j}$ should be added; the same remark will
hold for the following computations.) Said otherwise, the slopes of 
$T^{2} M$ are the $2 \mu_{i}$, $1 \leq i \leq k$ with multiplicities 
the $r_{i}^{2}$; and the $\mu_{i} + \mu_{j}$ , $1 \leq i < j \leq k$ 
with multiplicities the $2 r_{i} r_{j}$. \\

The repartition of these slopes (breaking of the Newton polygon)
among the symmetric and exterior square is as follows:
\begin{itemize}
\item{$S^{2} M$ has slopes the $2 \mu_{i}$, $1 \leq i \leq k$ 
with multiplicities the $\dfrac{r_{i}^{2} + r_{i}}{2}$; and the 
$\mu_{i} + \mu_{j}$ , $1 \leq i < j \leq k$ with multiplicities 
the $r_{i} r_{j}$.}
\item{$\Lambda^{2} M$ has slopes the $2 \mu_{i}$, $1 \leq i \leq k$ 
with multiplicities the $\dfrac{r_{i}^{2} - r_{i}}{2}$; and the 
$\mu_{i} + \mu_{j}$ , $1 \leq i < j \leq k$ with multiplicities 
the $r_{i} r_{j}$.}
\end{itemize}
If there are two slopes $\mu < \nu$, with multiplicities $r,s$,
no confusion of sums $\mu_{i} + \mu_{j}$ can arise, and we find:
\begin{itemize}
\item{$T^{2} M$ has slopes $2 \mu < \mu + \nu < 2 \nu$, with
multiplicities $r^{2},2 r s,s^{2}$;}
\item{$S^{2} M$ has the same slopes, with multiplicities 
$\dfrac{r^{2}+r}{2},r s,\dfrac{s^{2}+s}{2}$;}
\item{$\Lambda^{2} M$ has the same slopes, with multiplicities 
$\dfrac{r^{2}-r}{2},r s,\dfrac{s^{2}-s}{2} \cdot$}
\end{itemize}
We now take $M_{0} = \begin{pmatrix} a & 0 \\ 0 & b z \end{pmatrix}$
and $M = \begin{pmatrix} a & u \\ 0 & b z \end{pmatrix}$, where
$a,b \in \C^{*}$ and $u \in \Ka$. The symmetric squares admit an
obvious choice of basis and corresponding matrices:
$$
N_{0} = \begin{pmatrix} 
a^{2} & 0 & 0 \\ 
0 & a b z & 0 \\ 
0 & 0 & b^{2} z^{2} 
\end{pmatrix} 
\text{~and~}
N = \begin{pmatrix} 
a^{2} & 2 a u & u^{2} \\ 
0 & a b z & u b z \\ 
0 & 0 & b^{2} z^{2} 
\end{pmatrix}.
$$
If $F = \begin{pmatrix} 1 & f \\ 0 & 1 \end{pmatrix}$ is such that
$F[M_{0}] = M$, then we have $G[N_{0}] = N$ with $G$ given by:
$$
G = S^{2} F = \begin{pmatrix} 1 & 2 f & f^{2} \\ 
0 & 1 & f \\ 0 & 0 & 1 \end{pmatrix}.
$$
Actually, if one does not cheat, when looking for 
$G = \begin{pmatrix} 1 & f_{1} & f_{2} \\ 
0 & 1 & f_{3} \\ 0 & 0 & 1 \end{pmatrix}$
such that $G[N_{0}] = N$, one has to solve the system:
\begin{align*}
a b z \sq f_{1} &= a^{2} f_{1} + 2 a u, \\
b^{2} z^{2} \sq f_{2} &= a^{2} f_{2}  + 2 a u f_{3} + u^{2}, \\ 
b^{2} z^{2} \sq f_{3} &= a b z f_{3} + u b z.
\end{align*}
Since we know from start that $b z \sq f = a f + u$, we see that
$f_{1} := 2 f$ and $f_{3} := f$ respectively solve the first and third
equation; then, we find that $f_{2} := f^{2}$ solve the second equation. \\

Using the system above, we find a second order inhomogeneous 
equation for $f_{2}$ alone as follows:
\begin{align*}
(b z \sq - a) \, \dfrac{1}{2 a u} \, (b^{2} z^{2} \sq - a^{2}) f_{2} & =
(b z \sq - a) \, f_{3} + (b z \sq - a) \, \dfrac{u^{2}}{2 a u} 
\Longrightarrow \\
\left(\dfrac{b z}{\sq(u)} \sq - \dfrac{a}{u}\right) \, 
(b^{2} z^{2} \sq - a^{2}) f_{2} & = b z \sq(u) + a u.
\end{align*}
We leave to the reader to find a simpler form, as well as the 
corresponding third order homogeneous equation. At any rate,
in the case that $u \in \C$ (Birkhoff-Guenther normal form) 
we have:
$$
(b z \sq - a) \, (b^{2} z^{2} \sq - a^{2}) f_{2} = b z + a.
$$
In the particular case $a = b = -u = 1$ of the Tshakaloff series, 
we are led to the following equation:
\begin{equation}
\label{equation:carre}
L \hat Y = 1 + z, \text{~where~} L := q^2z^3 \sq^2 - z(1+z) \sq + 1.
\end{equation}

\begin{rema}
We refer here to the example \ref{exem:geometrythreeslopes} and
specialize it to the case of the symmetric square $N$ above.
We see that, when $u$ varies, the class of $N$ in
$H^{1}(\Eq,\lambda^{(1)}) \oplus H^{1}(\Eq,\lambda^{(2)})$ has
components $2 a u L_{1}(a^{2},a b) + u b L_{1}(a b,b^{2})$ and
$u^{2} L_{2,0}(a^{2},b^{2})$: a nice parabola.
\end{rema}


\subsubsection{Algebraic summation of the square of the Tshakaloff series }

The fact that $F[A_{0}] = A \Rightarrow S^{2} F[B_{0}] = B$ is purely 
algebraic and stays true of the sum in direction $\overline{c}$, so
that one gets the following sums:
$$
S_{\overline{c}} G = S^{2} (S_{\overline{c}} F) =
\begin{pmatrix} 1 & 2 f_{\overline{c}} & f_{\overline{c}}^{2} \\ 
0 & 1 & f_{\overline{c}} \\ 0 & 0 & 1 \end{pmatrix}.
$$
Moreover, observing that 
$x \mapsto \begin{pmatrix} 1 & 2 x & x^{2} \\ 
0 & 1 & x \\ 0 & 0 & 1 \end{pmatrix}$
is a morphism from $\C$ $\C$ to $\GL_{3}(\C)$, one gets the explicit
formula for the cocycle:
$$
S_{\overline{c},\overline{d}} G = S^{2} (S_{\overline{c},\overline{d}} F) =
\begin{pmatrix} 
1 & 2 f_{\overline{c},\overline{d}} & f_{\overline{c},\overline{d}}^{2} \\ 
0 & 1 & f_{\overline{c},\overline{d}} 
\\ 0 & 0 & 1 \end{pmatrix}.
$$
The algebraic sums $f_{\overline{c}}$ have been described explicitly
in subsection \ref{subsection:qEuler} where their notation was
$S_{\overline{\lambda}} \hat{f}$, with $\overline{\lambda} = c$. 


\subsection{Analytic aspects}
\label{section:qEulersquareanalytic}


\subsubsection{The square of the Tshakaloff series is not summable 
with one level}

Consider the square $\hat Y := \tsh^{2}$ of the Tshakaloff series:
$$
\hat Y(z) = \left(\sum_{n\ge 0}q^{n(n-1)/2} z^n\right)^2.
$$
Its $q$-Borel transform $\Bq \hat Y$ (at level $1$) can be computed 
from the following simple remark: 
$$
\hat f(z) = \sum_{n\ge 0} a_n z^n \text{~and~} \hat g \in\Rf
\Longleftrightarrow
\Bq(\hat f \hat g) =
\sum_{n \ge 0} a_n q^{-n(n-1)/2} \xi^n \Bq \hat g(q^{-n} \xi).
$$
It follows that, if  
$P(\xi) = \Bq \hat Y(\xi){\prod_{n\ge 0}(1-q^{-n}\xi)}$, 
then $P$ is an entire function such that:
\begin{equation}
\label{equation:BorelCarre}
P(q^m) = (-1)^mq^{m(3m+1)/2}(q^{-1};q^{-1})_m(q^{-1};q^{-1})_\infty.
\end{equation}
From this, we see that $P$ has at infinity $q$-exponential growth
of order $\geq 3$. \\

The $q$-Borel transform of $\hat Y$ represents a meromorphic function 
in $\C$ with (simple) poles on $q^\N$ and having at infinity in  
$\C \setminus q^\N$ $q$-exponential growth of order exactly $2$. \\

\index{$q$-Borel-Laplace summable}
Thus, $\hat Y = \tsh^{2}$ is not $q$-Borel-Laplace summable as $\tsh$ 
itself is. This comes from the fact that the Newton polygon of the former 
(resp. of the latter) has three (resp. two) slopes as we have seen.
Moreover, the problem is cousin to a nonlinear problem.


\subsubsection{Multisummability}

\begin{prop} 
Let $\lambda$, $\mu\in\C^*$ and let $f\in\AA_q^{[\lambda]}$ 
and $g\in_q^{[\mu]}$.
\begin{enumerate}
\item {If $[\lambda] = [\mu]$, then
$fg \in \OO_{([\lambda],\lambda])}^{2[\lambda]}$.}
\item{If $[\lambda] \neq [\mu]$, then, generically 
$fg \notin 
\OO_{([\lambda],[\mu])}^{[\lambda]+[\mu]}\cup 
\OO_{([\mu],[\lambda])}^{[\lambda]+[\mu]}$.}
\end{enumerate}
\end{prop}
\begin{proof}
Write $f = F/\theta_\lambda$, $g=G/\theta_\mu$; from theorem \ref{theo:fF}, 
one gets $F \in \EE_0^{[\lambda]}$ and $G\in \EE_0^{[\mu]}$. (The functions
$\theta_\lambda$, $\theta_\mu$ were introduced just before the statement
of theorem \ref{theo:fF}.) \\
If $[\lambda] = [\mu]$, then $FG$ represents an analytic function near $0$
in $\C^*$, with $q$-Gevrey growth of null order on the $q$-spiral 
$[\lambda]$ at $0$ and $q$-Gevrey growth of order $2$ globally. To obtain
the $q$-Gevrey growth order of its values on $2[\lambda]/[\lambda]$ at $0$, 
write:
$$
(FG)'(\lambda q^{-n}) = F'(\lambda q^{-n}) G(\lambda q^{-n}) +
F(\lambda q^{-n}) G'(\lambda q^{-n}),
$$
which will provide the $q$-Gevrey order. Using theorem \ref{theo:mniveaux}, 
we get the first statement. \\
For the second statement, just note that, generally, $FG$ has null
$q$-Gevrey order at $0$ neither on $[\lambda]$, nor on $[\mu]$; indeed,
the sequences $(F(\mu q^{-n}))$ and $(G(\lambda q^{-n}))$ have $q$-Gevrey
order one.
\end{proof}

Let $\lambda$, $\mu \notin [1] = q^\Z$. Write  $f_\lambda$, resp. $f_\mu$,
the solutions of the $q$-Euler equation satisfied by $\tsh$ in 
$\AA_q^{[\lambda]}$, resp. in $\AA_q^{[\mu]}$. If $[\lambda] = [\mu]$, 
from the proposition above, one has:
$$
(f_\lambda)^2 \in \OO_{([\lambda],\lambda])}^{2[\lambda]},
$$
and this is the solution provided by theorem \ref{theo:sommabilite} for
equation \eqref{equation:carre} with $\Lambda=2[\lambda]$ and 
$\Lambda_1=\Lambda_2=[\lambda]$. However, if $[\lambda] \neq [\mu]$, 
then $f_\lambda f_\mu$ is \emph{not} the solution provided by that
theorem with $\Lambda = [\lambda]+[\mu]$ and 
$\{\Lambda_1,\Lambda_2\} =\{[\lambda],[\mu]\}$.


\subsubsection{Possibility of a multisummation process}
\index{multisummation process}

No explicit algorithm is presently known to yield the solution of 
\eqref{equation:carre} in the spirit of theorem \ref{theo:sommabilite}
with $\Lambda=[\lambda]+[\mu]$. A multisummation algorithm does exist 
in a very different setting, due to the third author. The following 
result shows its similarity to classical Borel-Laplace summation.
\index{classical Borel-Laplace summation}

\begin{theo}
\label{thm:12sommable}
In \eqref{equation:carre}, we have $L = (z \sq - 1)(z^2 \sq - 1)$ and
the series $\hat Y$ is $(1,1/2)$-summable by the following process:
\begin{eqnarray*}
\hat f = \sum_{n\ge 0} a_n z^n \quad
&
\overset{\Bq}{\longrightarrow}
&
\varphi := \sum_{n\ge 0} a_n q^{-n(n-1)/2} \xi^n \\
&
\overset{{\mathcal A}^{*}_{q;1,1/2}}{\longrightarrow}
&
f^*_{1}
\overset{{\mathcal L}^{*}_{q;1/2}}{\longrightarrow}
f^*_{(1,1/2)}.
\end{eqnarray*}
\end{theo}

For a basic introduction to this method and explanation of the notations
above, see \cite{DRSZ}, which contains further references to this work.


\section{From the Mock Theta functions to the
$q$-Euler equation}
\label{section:MockTheta}
\index{Mock Theta functions}

Our source here is the famous 1935 paper 
``The final problem: an account of the Mock Theta Functions''
by G. N. Watson, as reproduced, for instance, in \cite{Rama}.
On page 330 of \emph{loc. cit.}, seven ``mock theta functions
of order three'' are considered. The first four are called 
$f,\phi,\psi,\chi$ (after Ramanujan who discovered them);
the three last are called $\omega,\nu,\rho$ (after Watson who
added\footnote{At least, so says Watson in \emph{loc. cit.}.
However, this is questioned by G.E. Andrews and B.C. Berndt
in their introduction to Ramanujan's Lost Notebook. The point
is also discussed in detail in pp. 171-172 of the article by 
Andrews in \cite{Rama}.} them to the list). In the notation of 
Ramanujan and Watson, the unique variable of these analytic 
functions is written $q$ (this tradition goes back to Jacobi) 
and it is assumed there that $0 < \lmod q \rmod < 1$. \\

In \cite{CZMTF}, a new variable $x$ is added (this tradition goes 
back to Euler) and one puts:
$$
\label{not106}
s(\alpha,\beta;q,x) := \sum_{n \geq 0} a_{n} x^{n}, 
$$
where:
\begin{align*}
a_{n} & := \dfrac{q^{n^{2}}}
{\left(\frac{q}{\alpha},\frac{q}{\beta};q\right)_{n}}
\left(\frac{1}{\alpha \beta}\right)^{n} \\
& =
\dfrac{(\alpha,\beta;q)_{\infty}}{(\alpha q^{-n},\beta q^{-n};q)_{\infty}} q^{-n} \\
& =
\dfrac{q^{n^{2}}}
{(\alpha - q) \cdots (\alpha - q^{n})(\beta - q) \cdots (\beta - q^{n})} \cdot
\end{align*}
Note for further use that the second formula makes sense for all $n \in \Z$
and allows one to define another series:
$$
F(\alpha,\beta;q,x) := \sum_{n \in \Z} a_{n} x^{n}, 
$$
\label{not107}
so that $F(\alpha,\beta;q,x) = s(\alpha,\beta;q,x) + G(\alpha,\beta;q,x)$,
where:
$$
G(\alpha,\beta;q,x) = \sum_{n < 0} a_{n} x^{n} = 
\sum_{n \geq 0} (\alpha,\beta;q)_{n} (q/x)^{n}.
$$
The formula giving $s(\alpha,\beta;q,x)$ subsumes all seven mock 
theta functions, which can be respectively recovered by setting 
$(\alpha,\beta,x)$ to be one of the following: $(-1,-1,1)$, 
$(\ii,-\ii,1)$, $(\sqrt{q},-\sqrt{q},1)$, $(j,j^{2},-q)$, 
$(1/q,1/q,1)$, $(\ii/\sqrt{q},-\ii/\sqrt{q},1)$ and $(-j/q,-j^{2}/q,1)$. 
(Note that in all cases, among other multiplicative relations, $\alpha$ 
and $\beta$ map to torsion points of $\Eq$.) \\

In the following, which is intended to motivate the study of the Stokes
phenomenon, we follow recent work by the third author \cite{CZMTF}, 
skipping most of the proofs. In \ref{subsection:functequafors} and 
\ref{subsection:backtoMockTheta}, we use the conventions of \emph{loc. 
cit.}. In particular, it is assumed (exceptionally) that $\lmod q \rmod < 1$
and we use the theta function:
$$
\theta(x;q) := \sum_{n \in \Z} q^{n(n-1)/2} x^{n} = (q,-x,-q/x;q)_{\infty}.
$$ 
(See \eqref{eqn:tripleproduit} in the general notations of section 
\ref{section:generalnotations} and equation \eqref{eqn:deuxtriplesproduits}
page \pageref{eqn:deuxtriplesproduits}.) In \ref{subsection:backtoqEuler}, 
we shall return to the general convention $\lmod q \rmod > 1$ of 
the present paper, and use the function $\thq$


\subsection{Functional equation for $s(\alpha,\beta;q,z)$}
\label{subsection:functequafors}

From the recurrence relation 
$(\alpha - q^{n+1}) (\beta - q^{n+1}) a_{n+1} = q^{2n+1} a_{n}$, one deduces
that $s$, as a function of $x$, is solution of the second order non
homogeneous $q$-difference equation:
\begin{equation}
\label{eqn:NHCBHG}
\bigl((\sq - \alpha)(\sq - \beta) - q x \sq^{2}\bigr) s = 
(1 - \alpha) (1 - \beta).
\end{equation}
The recurrence relation actually remains valid for all $n \in \Z$,
which implies that $F$ is solution of the corresponding homogeneous
equation: 
\begin{equation}
\label{eqn:HCBHG}
\bigl((\sq - \alpha)(\sq - \beta) - q x \sq^{2}\bigr) F = 0.
\end{equation}
This, in turn, entails that $-G = s - F$ is a solution defined at infinity
of \eqref{eqn:NHCBHG}. \\

We shall assume now that $\alpha,\beta \neq 0$ and that 
$\beta/\alpha \not\in q^{\Z}$. One checks easily that the equation
\eqref{eqn:HCBHG} is fuchsian at $0$ with exponents $\alpha,\beta$
(see for instance \cite{JSAIF}). Likewise, taking in account the general
conventions of this paper (\ie\ using the dilatation factor $q^{-1}$ in 
order to have a modulus $> 1$), we see that \eqref{eqn:HCBHG} is pure 
isoclinic at infinity, with slope $-1/2$; this will explain the 
appearance of $\theta(-,q^{2})$ in the following formulas. We define:
$$
M(\alpha,\beta;q,x) := 
\left(\frac{q}{\alpha},\frac{q}{\beta};q\right)_{\infty} F(\alpha,\beta;q,x).
$$
This is another solution of \eqref{eqn:HCBHG}, with more symetries.
It admits the following expansions:
\label{not108}
\begin{align*}
M(\alpha,\beta;q,x) &= 
\theta(q \alpha \beta x;q^{2}) U_{\alpha/\beta}(x) - 
\dfrac{q}{\alpha} ~ \theta(\alpha \beta x;q^{2}) V_{\alpha/\beta}(x) \\
& =
\theta(q x/\alpha \beta;q^{2}) U_{\alpha/\beta}(1/x) - 
\dfrac{q}{\alpha} ~ \theta(q^{2} x/\alpha \beta;q^{2}) V_{\alpha/\beta}(1/x),
\end{align*}
with the following definitions:
\label{not109}
\begin{align*}
U_{\lambda}(x) & := 
\sum_{m \geq 0} q^{m^{2}} S_{2 m}(-q^{-2 m} \lambda;q) (q x/\lambda)^{m}, \\
V_{\lambda}(x) & := 
\sum_{m \geq 0} q^{m(m+1)} S_{2 m+1}(-q^{-2 m-1} \lambda;q) (q x/\lambda)^{m}, \\
S_{n}(x;q) & := 
\sum_{k=0}^{n} \dfrac{q^{k^{2}}}{(q;q)_{k} (q;q)_{n-k}} (-x)^{k}.
\end{align*}
\index{Stieltjes-Wiegert polynomials}
(The $S_{n}$ are the \emph{Stieltjes-Wiegert polynomials}.) \\

\begin{rema}
The parameter $\lambda := \alpha/\beta$ is linked to monodromy.
Indeed, the local Galois group of \eqref{eqn:HCBHG} at $0$
\index{local Galois group}
is the set of matrices 
$\begin{pmatrix} \gamma(\alpha) & 0 \\ 0 & \gamma(\beta) \end{pmatrix}$,
where $\gamma$ runs through the group endomorphisms of $\C^{*}$ that
send $q$ to $1$; and the local monodromy group is the rank $2$ free
\index{local monodromy group}
abelian subgroup with generators corresponding to two particular choices
of $\gamma$ described in \cite{JSGAL}.
\end{rema}


\subsection{Back to the the Mock Theta function}
\label{subsection:backtoMockTheta}

We can for instance study $\phi,\psi,\nu,\rho$ by setting $x = 1$ in 
$s(\alpha,\beta;q,x)$. (The function $\chi$ involves $x = -q$ and
$f,\omega$ will not comply the condition $\beta/\alpha \not\in q^{\Z}$:
in \cite{CZMTF}, their study is distinct, though similar.) Up to the
knowledge of standard $q$-functions, one is reduced to the study of:
$$
\label{not110}
U(\lambda) := U_{\lambda}(1) \text{~and~} V(\lambda) := V_{\lambda}(1).
$$
These are solutions of the $q$-difference equations:
\begin{align*}
U(q \lambda) - \lambda U(\lambda/q) & =
(1 - \lambda) \dfrac{\theta(q \lambda;q^{2})}{(q,q;q)_{\infty}}, \\
V(q \lambda) - q \lambda V(\lambda/q) & =
(1 - \lambda) \dfrac{\theta(\lambda/q;q^{2})}{(q,q;q)_{\infty}} \cdot
\end{align*}
Upon setting:
$$
U(\lambda) =: \dfrac{\theta(\lambda;q^{2})}{(q,q;q)_{\infty}} Y(\lambda)
\text{~and~} 
V(\lambda) =: \dfrac{\theta(\lambda/q;q^{2})}{(q,q;q)_{\infty}} Z(\lambda),
$$
we find that both $Y$ and $Z$ are solutions of the $q$-difference equation:
$$
X(q \lambda) - \dfrac{\lambda^{2}}{q} X(\lambda/q) = 1 - \lambda.
$$


\subsection{Back to the $q$-Euler equation}
\label{subsection:backtoqEuler}

To fit this equation with the convention of this paper, we shall
put $z := \lambda$, $f(z) := X(q \lambda)$ and take $q^{-2}$ as
the new dilatation coefficient, that we shall denote by $q$, so
that $\lmod q \rmod > 1$ indeed. Our equation becomes:
$$
\sqrt{q} z^{2} \sq f - f = z - 1.
$$
(We have implicitly chosen a square root $\sqrt{q}$.) There are four
pathes of attack. The most poweful involves the summation techniques
of chapter \ref{chapter:asymptotictheory} and it is the one used in
\cite{CZMTF}. We show the other three as an easy application exercise.

\subsubsection{$q$-Borel transformation}

Following \ref{subsection:extensionsofqdifferencemodules}, we put
$Z := z^{2}$, $Q := q^{2}$ and $f(z) = g(Z) + z h(Z)$, so that:
$$
\sqrt{q} Z \sigma_{Q} g - g = -1 \text{~and~}
q \sqrt{q} Z \sigma_{Q} h - h = 1.
$$
The end of the computation, \ie\ that of the invariants
$(\mathcal{B}_{Q,1})(-1)(\sqrt{q})$ and 
$(\mathcal{B}_{Q,1})(1)(q \sqrt{q})$, is left to the reader.

\subsubsection{Birkhoff-Guenther normal form}

According to section \ref{section:explicitdescription}, we see that
our equation is already in Birkhoff-Guenther normal form. Actually,
it is the equation for $f$ such that 
$\begin{pmatrix} 1 & f \\ 0 & 1 \end{pmatrix}$ is an isomorphism from
$\begin{pmatrix} 1 & 0 \\ 0 & \sqrt{q} z^{2} \end{pmatrix}$ to
$\begin{pmatrix} 1 & z - 1 \\ 0 & \sqrt{q} z^{2} \end{pmatrix}$.

\subsubsection{Privileged cocycles}

There is here an obvious isomorphism of $\Lambda_I(M_0)$ with
the vector bundle $F_{1/\sqrt{q} z^{2}}$ 
(\cf\ \ref{subsection:extensionsofqdifferencemodules}).
(More generally, if 
$A_{0} = \begin{pmatrix} a & 0 \\ 0 & b z^{\delta} \end{pmatrix}$,
then $\Lambda_I(M_0) \simeq F_{a/b z^{\delta}}$.) The privileged cocycles
of \ref{section:privilegedcocycles} are best obtained by the
elementary approach of \cite{JSStokes} as follows. We look for a
solution $f_{\overline{c}} = \dfrac{g}{\theta_{q,c}^{2}}$ with $g$
holomorphic over $\C^{*}$. The corresponding equation is
$\sqrt{q} c^{2} \sq g - g = (z - 1)\theta_{q,c}^{2}$. Writing
$\thq^{2} =: \sum \tau_{n} z^{n}$, we see that:
$$
f_{\overline{c}} = \dfrac{1}{\theta_{q,c}^{2}} \sum_{n \in\Z} 
\dfrac{(\tau_{n-1} c - \tau_{n})c^{-n}}{\sqrt{q} c^{2} q^{n} - 1} z^{n}.
$$
This is the only solution with poles only on $[-c;q]$ and at most
double. It makes sense only for $\sqrt{q} c^{2} \not\in q^{\Z}$,
which prohibits four values $\overline{c} \in \Eq$. The components 
of the Stokes cocycle are the $(f_{\overline{c}} - f_{\overline{d}})$.


\section{From class numbers of quadratic forms to the
$q$-Euler equation}
\label{section:QuadForms}
\index{$q$-Euler equation}

This topic is related to a paper of Mordell \cite{Mordell} and to
recent work \cite{CZMORD} by the third author. We follow their
notations, except for the use of the letter $q$, and also for
the dependency on the modular parameter $\omega$, which we do not
always make explicit. \\

We shall have here use for the classical theta functions, defined
for $x \in \C$ and $\Im(\omega) > 0$:
\label{not111}
\begin{align*}
\theta_{0,1}(x) = \theta_{0,1}(x,\omega) & :=
\sum_{n \in \Z} (-1)^{n} e^{\ii \pi (n^{2} \omega + 2 n x)}, \\
\theta_{1,1}(x) = \theta_{1,1}(x,\omega) & :=
\dfrac{1}{\ii} \sum_{m~odd} (-1)^{(m-1)/2} e^{\ii \pi (m^{2} \omega/4 + m x)}, \\
& =
e^{\ii \pi (\omega/4+x-1/2)} \sum_{n \in \Z} (-1)^{n} e^{\ii \pi (n(n+1) \omega + 2 n x)}.
\end{align*}
We shall set $q := e^{- 2 \ii \pi \omega}$ (so that indeed $\lmod q \rmod > 1$)
and $z := e^{2 \ii \pi x}$. The above theta functions are related to $\thq$
through the formulas:
\begin{align*}
\theta_{0,1}(x,\omega) & = \thq(- \sqrt{q} z), \\
\theta_{1,1}(x,\omega) & = \dfrac{\sqrt{z}}{\ii q^{1/8}} \thq(-z).
\end{align*}
(Thus, the latter is multivalued as a function of $z$.)


\subsection{The generating series for the class numbers}
\label{subsection:generatingseries}
\index{class numbers (generating series)}

Consider the quadratic forms $a x^{2} + 2 h x y + b y^{2}$, with
$a,b,h \in \Z$, $a,b$ not both even (``uneven forms''), and
$D := a b - h^{2} > 0$, up to the usual equivalence. For any
$D \in \N^{*}$, write $F(D)$ the (finite) number of classes
of such forms.
\label{not112}

\begin{theo}[Mordell, 1916)]
\label{theo:Mordell1916}
Let $f_{0,1}(x) = f_{0,1}(x,\omega)$ be the unique entire function solution
of the system:
$$
\begin{cases}
f_{0,1}(x+1) = f_{0,1}(x), \\
f_{0,1}(x+\omega) + f_{0,1}(x) = \theta_{0,1}(x).
\end{cases}
$$
Then, for $\Im(\omega) > 0$:
$$
\sum_{n \in \N^{*}} F(n) e^{\ii \pi n \omega} =
\dfrac{\ii}{4 \pi} ~ \dfrac{f_{0,1}'(0)}{\theta_{0,1}(0)} \cdot
$$
\end{theo}

If one now defines $G_{0,1}(z) := \dfrac{f_{0,1}(x)}{\theta_{0,1}(x)}$
(which does make sense, since the right hand side is $1$-periodic),
one falls upon the familiar $q$-difference equation:
$$
(\sqrt{q} z \sq - 1) G_{0,1} = \sqrt{q} z.
$$
We leave it as an exercise for the reader to characterize $G_{0,1}$
as the unique solution complying some polar conditions.


\subsection{Modular relations}
\label{section:modularity}

In order to obtain modular and asymptotic properties for the generating
series of theorem \ref{theo:Mordell1916}, Mordell generalized his results
in 1933. We extract the part illustrating our point. Mordell sets:
$$
f(x) = f(x,\omega) := 
\dfrac{1}{\ii} \sum_{m~odd} 
\dfrac{(-1)^{(m-1)/2} e^{\ii \pi (m^{2} \omega/4 + m x)}}{1 + e^{\ii \pi \omega m}} \cdot
$$
This the unique entire function solution of the system:
$$
\begin{cases}
f(x+1) + f(x) = 0, \\
f(x+\omega) + f(x) = \theta_{1,1}(x).
\end{cases}
$$
The interest for Mordell is the (quasi-)modular relation:
$$
f(x,\omega) - \dfrac{\ii}{\omega} f(x/\omega,-1/\omega) =
\dfrac{1}{\ii} \theta_{1,1}(x,\omega) 
\int_{-\infty}^{+\infty} \dfrac{e^{\ii \pi \omega t^{2} - 2 \pi t x}}
{e^{2 \pi t} - 1} ~ dt.
$$
(The path of integration is $\R$ except that one avoids $0$ by below.)
The interest for us is that one can put 
$G(z) := \dfrac{f(x)}{\theta_{1,1}(x)}$
(the right hand side is $1$-periodic), and get the \emph{same} equation
as before:
$$
(\sqrt{q} z \sq - 1) G = \sqrt{q} z.
$$
This is used in \cite{CZMORD} to generalize Mordell results: the 
fundamental idea is to compare two summations of the solutions, one
along the lines of the present paper, the other along different
lines previously developped by C. Zhang.


\subsection{Related other examples}
\label{section:relatedexamples}

Mordell also mentions that the following formula of Hardy and Ramanujan:
$$
\dfrac{1}{\sqrt{- \ii \omega}} 
\int_{-\infty}^{+\infty} \dfrac{e^{-\ii \pi (t - \ii x)^{2}/\omega}} {\cosh \pi t} ~ dt
=
\int_{-\infty}^{+\infty} \dfrac{e^{\ii \pi \omega t^{2} - 2 \pi t x}}{\cosh \pi t} ~ dt
$$
can be proved along similar lines, by noting that both sides are entire
solutions of the system:
$$
\begin{cases}
\Phi(x-1) + \Phi(x) =  
\dfrac{2 e^{\ii \pi (x-1/2)^{2}/\omega}}{\sqrt{- \ii \omega}}, \\
\Phi(x+\omega) + e^{\ii \pi(2 x + \omega)} \Phi(x) = 
2 e^{\ii \pi (3\omega/4 + x)}
\end{cases}
$$
and that the latter admits only one such solution.



\appendix



\chapter{Classification of isograded filtered difference modules}
\label{chapter:ISOGRAD}

In this appendix, we tackle a general classification problem, actually
an algebraic version of the isoformal analytic classification problem
of section \ref{section:analyticisoformalclassification}. We first 
describe in section \ref{section:generalsetting} an abstract version 
of the problem, which we discuss but not solve in full generality. Then, 
in section \ref{section:ALDIFF}, we describe the category of difference 
modules over difference rings, in which we work in the following sections. 
The solution in this setting is given in section \ref{section:modulispace}
by theorem \ref{theo:themodulispaceabstract} and corollary
\ref{coro:themodulispaceabstract}. Since it involves some homological
algebra, we have taken great care in section \ref{section:extdiffmod}
to justify our explicit description of extension spaces along with their 
linear structure\footnote{This is all the more necessary since the work
\cite{RS3} by the first two authors strongly relies on this explicit
description.}. Sections \ref{section:extensionsofdifferencemodules} and
\ref{section:cohomologicalequation} provide some technical tools for the
case of difference fields.


\section{General setting of the problem}
\label{section:generalsetting}

Let $C$ a commutative ring and $\mathcal{C}$ an abelian $C$-linear
category. We fix a finitely graded object:
$$
P = P_{1} \oplus \cdots \oplus P_{k},
$$
and intend to classify pairs 
\label{not113}
$(\underline{M},\underline{u})$ 
made up of a finitely filtered object:
\index{finitely filtered object}
$$
\underline{M} = 
(0 = M_{0} \subset M_{1} \subset \cdots \subset M_{k} = M),
$$
and of an isomorphism from $\gr M$ to $P$:
$$
\underline{u} = (u_{i}: M_{i}/M_{i-1} \simeq P_{i})_{1 \leq i \leq k}.
$$
As easily checked, it amounts to the same as giving $k$ exact sequences:
$$
0 \rightarrow M_{i-1} \overset{w_{i}}{\rightarrow} M_{i} 
\overset{v_{i}}{\rightarrow} P_{i} \rightarrow 0.
$$
The pairs $(\underline{M},\underline{u})$ and 
$(\underline{M'},\underline{u'})$ are said to be equivalent if (with
obvious notations) there exists a morphism from $M$ to $M'$ which
is compatible with the filtrations and with the structural isomorphisms,
that is, making the following diagram commutative:
$$
\xymatrix{
\gr M \ar@<0ex>[rd]^{u} \ar@<0ex>[rr]^{\gr f}  &  
& \gr M'  \ar@<0ex>[ld]_{u'} \\
& P &  \\
}
$$
In the description by exact sequences, the equivalence relation
translates as follows: there should exist morphisms 
$f_{i}: M_{i} \rightarrow M'_{i}$ making the following diagrams
commutative:
$$
\begin{CD}
0 @>>> M_{i-1}   @>{w_{i}}>> M_{i}     @>{v_{i}}>>      P_{i}  @>>> 0     \\
&  &   @VV{f_{i-1}}V      @VV{f_{i}}V      @VV{\Id_{P_{i}}}V       \\
0 @>>> M'_{i-1}   @>{w'_{i}}>> M'_{i}     @>{v'_{i}}>>    P_{i} @>>> 0
\end{CD}
$$
Note that such a morphism (if it exists) is automatically strict
and an isomorphism. \\

\label{not114}
We write $\F(P_{1},\ldots,P_{k})$ the set\footnote{We admit that
$\F(P_{1},\ldots,P_{k})$ is indeed a set; from the devissage arguments 
that follow, (see \ref{subsection:devissage}), it easily seen to be 
true if all the $\Ext$ spaces are sets, \emph{e.g.} in the category of 
left modules over a ring.} of equivalence classes of pairs 
$(\underline{M},\underline{u})$.


\subsection{Small values of $k$}
\label{subsection:smallvalues}

\label{not115}
For $k = 1$, the set $\F(P_{1})$ is a singleton. For $k = 2$, the set
$\F(P_{1},P_{2})$ has a natural identification with the set 
$\Ext(P_{2},P_{1})$ of classes of extensions of $P_{2}$ by $P_{1}$,
which carries a structure of $C$-module\footnote{Note however that
the latter comes by identifying $\Ext$ with $\Ext^{1}$, and that there
are two opposite such identifications, see for instance \cite{cartEil},
exercice 1, p. 308. We shall systematically use the conventions of
\cite{BAH} and, from now on, make no difference between $\Ext$ and 
$\Ext^{1}$.}. The identification generalizes the one that was described 
in proposition \ref {prop:classesaunniveau=Ext} and can be obtained as 
follows. To give a filtered module 
$\underline{M} = (0 = M_{0} \subset M_{1} \subset M_{2} = M)$ endowed with 
an isomorphism from $\gr M$ to $P_{1} \oplus P_{2}$ amounts to give an 
isomorphism from $M_{1}$ to $P_{1}$ and an isomorphism from $M/M_{1}$ to 
$P_{2}$, that is, a monomorphism $i$ from $P_{1}$ to $M$ and an epimorphism
$p$ from $M$ to $P_{2}$ with kernel $i(P_{1})$, that is, an exact sequence:
$$
0 \rightarrow P_{1} \overset{i}{\rightarrow} M 
\overset{p}{\rightarrow} P_{2} \rightarrow 0,
$$
that is, an extension of $P_{2}$ by $P_{1}$. One then checks easily that
our equivalence relation thereby corresponds with the usual isomorphism
of extensions. \\

\label{not116}
When $k = 3$, the description of $\F(P_{1},P_{2},P_{3})$ amounts to the
classification of \emph{blended extensions} (``extensions panach\'ees'').
\index{blended extensions}
These were introduced by Grothendieck in \cite{Grothendieck}. We refer
to the studies \cite{DB1,DB2} by Daniel Bertrand\footnote{Also see the
revised version ``Extensions panach\'ees autoduales'', to be found on
\verb?http://www.math.jussieu.fr/~bertrand/Recherche/recherche.html?.}, 
whose conventions we use. Start from a representative of a class 
in $\F(P_{1},P_{2},P_{3})$, in the form of three exact sequences:
$$
0 \rightarrow M_{0} \overset{w_{1}}{\rightarrow} M_{1} 
\overset{v_{1}}{\rightarrow} P_{1} \rightarrow 0, \quad
0 \rightarrow M_{1} \overset{w_{2}}{\rightarrow} M_{2} 
\overset{v_{2}}{\rightarrow} P_{2} \rightarrow 0, \quad
0 \rightarrow M_{2} \overset{w_{3}}{\rightarrow} M_{3} 
\overset{v_{3}}{\rightarrow} P_{3} \rightarrow 0.
$$
Also recall that $M_{0} = 0$ and $M_{3} = M$. These give rise to two
further exact sequences; first:
$$
\xymatrix{
0 \ar@<0ex>[r] & P_{1} \ar@<0ex>[r]^{w_{2} \circ v_{1}^{-1}} &
M_{2} \ar@<0ex>[r]^{v_{2}} & P_{2} \ar@<0ex>[r] & 0}
$$
Indeed, $v_{1}: M_{1} \rightarrow P_{1}$ is an isomorphism, so that
$v_{1}^{-1}: P_{1} \rightarrow M_{1}$ is well defined; and the exactness 
is easy to check. To describe the second exact sequence, note that
$v_{2}: M_{2} \rightarrow P_{2}$ induces an isomorphism
$\overline{v_{2}}: M_{2}/w_{2}(M_{1}) \rightarrow P_{2}$, whence
$(\overline{v_{2}})^{-1}: P_{2} \rightarrow M_{2}/w_{2}(M_{1})$; then 
$w_{3}: M_{2} \rightarrow M_{3} = M$ induces a morphism
$\overline{w_{3}}: M_{2}/w_{2}(M_{1}) \rightarrow M'$, where
we put $M' := M/w_{3} \circ w_{2}(M_{1})$, and, by composition, a morphism 
$\overline{w_{3}} \circ (\overline{v_{2}})^{-1}: P_{2} \rightarrow M'$; last, 
$v_{3}: M_{3} \rightarrow P_{3}$ is trivial on $w_{3}(M_{2})$, therefore
on $w_{3} \circ w_{2}(M_{1})$, so that it induces a morphism
$\overline{v_{3}}: M' \rightarrow P_{3}$. Now we get the sequence:
$$
\xymatrix{
0 \ar@<0ex>[r] & P_{2} \ar@<0ex>[rr]^{\overline{w_{3}} \circ (\overline{v_{2}})^{-1}} & &
M' \ar@<0ex>[r]^{\overline{v_{3}}} & P_{3} \ar@<0ex>[r] & 0}
$$
We leave for the reader to verify the exactness of this sequence. 
These two sequences can be \emph{blended} (``panach\'ees'') to give
the following commutative diagram of exact sequences:
$$
\xymatrix{
               & &                 
               & 0 \ar@<0ex>[d]        
               & 0 \ar@<0ex>[d] & \\
0 \ar@<0ex>[r] & P_{1} \ar@<0ex>[rr]^{w_{2} \circ v_{1}^{-1}} \ar@{=}[d] & &
                 M_{2} \ar@<0ex>[r]^{v_{2}} \ar@<0ex>[d]^{w_{3}} & 
P_{2} \ar@<0ex>[r] \ar@<0ex>[d]^{\overline{w_{3}} \circ (\overline{v_{2}})^{-1}} & 0 \\
0 \ar@<0ex>[r] & P_{1} \ar@<0ex>[rr]^{w_{3} \circ w_{2} \circ v_{1}^{-1}} & &
                 M \ar@<0ex>[r]^{can} \ar@<0ex>[d]^{v_{3}} & 
                 M' \ar@<0ex>[r] \ar@<0ex>[d]^{\overline{v_{3}}} & 0 \\
               & &
               & P_{3} \ar@{=}[r] \ar@<0ex>[d] 
               & P_{3} \ar@<0ex>[d] \\
               & &                 
               & 0        
               & 0 &
}
$$
(We call $can$ the canonical projection from $M$ to $M'$.) Conversely, 
starting from such a diagram of blended extensions, we recover a 
representative $(\underline{M},\underline{u})$ of a class in 
$\F(P_{1},P_{2},P_{3})$ as follows. The module $M$ is the one sitting 
at the center of the diagram. The submodule $M_1$ is the image of $P_1$ 
by the monomorphism $w_{3} \circ w_{2} \circ v_{1}^{-1}$. Then, $M_2$ is 
the preimage by $can$ of the submodule of $M'$ defined as the image of 
$P_2$ by $\overline{w_{3}} \circ (\overline{v_{2}})^{-1}$. Of course,
$M_0 := \{0\}$ and $M_3 := M$. The structural morphisms $u_i$, $i = 1,2,3$, 
are easy to define. It can then be proven that one thus gets a bijective 
mapping from $\F(P_{1},P_{2},P_{3})$ to the set of equivalence classes of 
diagrams of blended extensions. (Note however that, in \emph{loc. cit.},
D. Bertrand is interested by a slightly different equivalence relation.)
Actually, one can even define a category on both sides of this correspondence 
and get an equivalence of these categories, see \cite{JSmoduli}.


\subsection{The devissage}
\label{subsection:devissage}

Our goal is to give conditions ensuring that $\F(P_{1},\ldots,P_{k})$ 
carries the structure of an affine space over $C$, and to compute its
dimension. The case $k = 2$ suggests that we should assume the $C$-modules 
$\Ext(P_{j},P_{i})$ to be free of finite rank. (As we shall see, the only
pairs that matter are those with $i < j$.) Then, aiming at an induction
argument, one invokes a natural onto mapping:
$$
\F(P_{1},\ldots,P_{k}) \longrightarrow \F(P_{1},\ldots,P_{k-1}),
$$
sending the class of $(\underline{M},\underline{u})$ defined as above
to the class of $(\underline{M'},\underline{u'})$ defined by:
$$
\underline{M'} = 
(0 = M_{0} \subset M_{1} \subset \cdots \subset M_{k-1} = M')
\text{~and~}
\underline{u'} = (u_{i})_{1 \leq i \leq k-1}.
$$
The preimage of the class of $(\underline{M'},\underline{u'})$ described
above is identified with $\Ext(P_{k},M')$; and note that $\Ext(P_{k},M')$ 
indeed only depends (up to a canonical isomorphism) on the class of
$(\underline{M'},\underline{u'})$. Under the assumptions we shall choose,
we shall see that $\Ext(P_{k},M')$ can in turn be unscrewed (d\'eviss\'e)
in the $\Ext(P_{k},P_{i})$ for $i < k$, and we expect to get a space with 
dimension $\sum\limits_{1 \leq i < j \leq k} \dim \Ext(P_{j},P_{i})$.

\begin{rema}
Once described the space $\F(P_{1},\ldots,P_{k})$, one can ask for the
seemingly more natural problem of the classification of those objects
$M$ such that $\gr M \simeq P$ (without prescribing the ``polarization'' 
$u$). One checks that the group $\prod \Aut(P_{i})$ operates on the space
$\F(P_{1},\ldots,P_{k})$: actually, $(\phi_{i}) \in \prod \Aut(P_{i})$ acts 
on the ``class'' of all pairs $(\underline{M},\underline{u})$ through left 
compositions $\phi_{i} \circ u_{i}$. Then, our new classification comes
by quotienting $\F(P_{1},\ldots,P_{k})$ by this action. We shall not deal
with that problem.
\end{rema}

The use of homological algebra in classification problems for functional 
equations is ancient, but it seems that the first step, the algebraic
modelisation, has sometimes been tackled rather casually: for instance,
when identifying a module of extensions with a cokernel, the explicit
description of a map is almost always given; the proof of its bijectivity
comes sometimes; the proof of its additivity seldom; the proof of its
linearity (seemingly) never. For that reason, very great care has been
given here to detailed algebraic constructions and proofs of ``obvious''
isomorphisms.


\section{Difference modules over difference rings}
\label{section:ALDIFF}

Our classification of isograded filtered difference modules over a 
\index{isograded filtered difference modules}
difference field will involve an affine moduli space that is actually
\index{moduli space}
\index{affine scheme} 
an affine scheme. To obtain that, we will have to consider in section 
\ref{section:extscal} a functor in commutative $C$-algebras that involves 
difference modules over difference \emph{rings}. Thus, in order to study 
this ``relative'' situation, we must generalize our basic constructions 
to difference rings. 


\subsection{Difference rings, difference operators}
\label{subsection:diffrings}
\index{difference rings}
\index{difference operators}

Let $K$ be a commutative ring and $\sigma$ a ring automorphism
of $K$. As noticed in \ref{subsection:generalitiesdiffmods}, most 
constructions and statements about difference fields and modules
remain valid over the \emph{difference ring} $(K,\sigma)$. 
\label{not117}
Our favorite examples are, of course, $(\Ra,\sq)$, $(\Ka,\sq)$, $(\Rf,\sq)$ 
and $(\Kf,\sq)$. \\

More precisely, we shall assume the following situation. Let $C$ be a 
commutative ring (we may think of the the field $\C$ of complex numbers, 
or else some arbitrary field of ``constants''). Assume that $K$ is a 
commutative $C$-algebra and $\sigma$ a $C$-algebra automorphism, 
that is, a $C$-linear ring automorphism. \emph{We shall always assume 
that the automorphism $\sigma$ is of infinite order}, \ie\ the iterates
$\sigma^k$, $k \in \Z$, are all different. (In the case of $q$-differences,
this means that $q$ is not a root of unity.) Then the ring of constants:
$$
\label{not118}
K^{\sigma} := \{x \in K \tq \sigma x = x\}
$$
is actually a sub $C$-algebra of $K$. The ring of $\sigma$-difference
operators is defined as the \"Ore-Laurent ring:
$$
\D_{K,\sigma} := K<T,T^{-1}>
$$
with $T$ an invertible indeterminate, subject to the twisted 
commutation relations: 
$$
\forall k \in \Z \;,\; \forall \lambda \in K \;,\;
T^{k}.\lambda = \sigma^{k}(\lambda) T^{k}.
$$
Its elements are the non commutative Laurent polynomials 
$\sum\limits_{-\infty \ll k \ll + \infty} a_k T^k$. The ring  $\D_{K,\sigma}$
is actually is a $C$-algebra with center $K^{\sigma}$. As a consequence,
the category $\DKMod$ of left $\D_{K,\sigma}$-modules is a $C$-linear 
abelian category. \\

Mapping the element $\sum a_k T^k \in \D_{K,\sigma}$ to the $C$-linear 
endomorphism $\sum a_k \sigma^k$ of $K$, one defines a morphism of 
$C$-algebras from $\D_{K,\sigma}$ to $\Lin_C(K)$. The image of this morphism 
is the sub-$C$-algebra $K<\sigma,\sigma^{-1}>$ generated by $K$ (that is,
the operators of left multiplication by elements of $K$), $\sigma$  and 
$\sigma^{-1}$. Beware however that neither $K<\sigma,\sigma^{-1}>$ nor
$\D_{K,\sigma}$ is a $K$-algebra, since $K$ is not central in either case. \\

Since we have assumed that $\sigma$ is of infinite order, one deduces 
from Artin-Dedekind's lemma on independence of characters that, when 
$K$ is a field (more generally, when it is an integral ring), the morphism 
is injective, so that we can identify:
$$
\D_{K,\sigma} = K<\sigma,\sigma^{-1}>.
$$
For this reason, we shall rather sometimes write $K<\sigma,\sigma^{-1}>$
for $\D_{K,\sigma}$ when no confusion thus arises, even if $K$ is not an
integral ring.


\subsection{Difference modules over difference rings}
\label{subsection:moddiff}

We are interested in the category $\DKMod$ of left $\D_{K,\sigma}$-modules.
We shall give two alternative description of these modules. \\

\index{$\sigma$-linear automorphism}
\index{semi-linear automorphism}
To a left $\D_{K,\sigma}$-module $M$, we associate the pair $(V,\Phi)$,
\label{not119}
where the $K$-module $V$ is obtained from $M$ by restriction of scalars
and where $\Phi$ is the $\sigma$-linear (or semi-linear) automorphism of 
$V$ defined by
$x \mapsto \sigma.x$ (external multiplication in the $\D_{K,\sigma}$-module 
$M$). The property of $\sigma$-linearity means:
$$
\forall \lambda \in K \;,\; \forall x \in V \;,\;
\Phi(\lambda x) = \sigma(\lambda) \Phi(x).
$$
Conversely, from such a pair $(V,\Phi)$, we can recover $M$ as the 
$\D_{K,\sigma}$-module having $V$ as underlying group and the external product 
defined by $(\sum a_k \sigma^k).x := \sum a_k \Phi^k(x)$. The fact that we do 
get a $\D_{K,\sigma}$-module depends on the $\sigma$-linearity of $\Phi$. 
If the $\D_{K,\sigma}$-module $N$ likewise corresponds to the pair $(W,\Psi)$, 
then $\D_{K,\sigma}$-linear maps from $M$ to $N$ are exactly the same as 
$K$-linear maps $f: V \rightarrow W$ such that $\Psi \circ f = f \circ \Phi$.
The category $\DKMod$ is therefore equivalent to the following category:
objects are the pairs $(V,\Phi)$, where $V$ is a $K$ -module and $\Phi$ 
is a $\sigma$-linear automorphism of $V$; morphisms from $(V,\Phi)$ to 
$(W,\Psi)$ are $K$-linear maps $f: V \rightarrow W$ such that 
$\Psi \circ f = f \circ \Phi$. From now on, we shall move freely from 
$\D_{K,\sigma}$-modules to pairs $(V,\Phi)$ and conversely. \\

\label{not120}
For any $K$-module $V$, write $\sigma^* V$ the $K$-module obtained by 
restriction of scalars through $\sigma: K \rightarrow K$. This means 
that, writing $\mu.x$ the external multiplication on $V$, the external 
multiplication on $\sigma^* V$ is defined by the formula: 
$(\lambda, x) \mapsto \sigma(\lambda).x$. It follows that a $\sigma$-linear 
automorphism $\Phi: V \rightarrow V$ is just a $K$-linear isomorphism
from $V$ to $\sigma^* V$. This allows for yet another description of the 
category $\DKMod$. Noting that $\Lin_K(V,W) = \Lin_K(\sigma^* V,\sigma^* W)$, 
we will find that many linear properties of $\D_{K,\sigma}$-modules (like
the definition and study of tensor products in paragraph 
\ref{subsection:tensorproduct}; or the definition and study of ``internal
Hom'' in paragraph \ref{subsection:internalhom}) are then immediate 
consequences of general linear algebra over $K$. \\

From general properties of the restriction of scalars, it follows that a 
sequence $M' \rightarrow M \rightarrow M''$ in $\DKMod$ is exact if, and 
only if, it is exact as a sequence of $K$-modules. However, if an exact 
sequence $0 \rightarrow M' \rightarrow M \rightarrow M'' \rightarrow 0$
splits in $\DKMod$ it splits in $Mod_K$, but not conversely. Any such
exact sequence is isomorphic to one where $M = (V,\Phi)$, $M' = (V',\Phi')$,
$M'' = (V'',\Phi'')$, where $V'$ is a sub-$K$-module of $V$ and 
$\Phi' = \Phi_{|V}$, and where $V'' = V/V'$ and $\Phi''$ is induced by $\Phi$.

\begin{defi}
\index{difference modules}
\index{difference $C$-algebra}
A \emph{difference module} over the \emph{difference $C$-algebra}
$(K,\sigma)$ (more shortly, over $K$) is a left $\D_{K,\sigma}$-module 
which, by restriction of scalars to $K$, yields a finite rank projective 
$K$-module.
\end{defi}

In our second description, this is a pair $(V,\Phi)$ such that $V$ is a 
finite rank projective $K$-module. We require that the module be projective
for technical reasons (such that good behaviour under extension of scalars).
The finiteness condition is natural in the context of functional equations. 

\label{not121}
\begin{exem}
\label{exem:unitetPhiA}
\index{unit difference module}
We write $\1$ and call \emph{unit difference module} the $\D_{K,\sigma}$-module 
$(K,\sigma)$. As a $\D_{K,\sigma}$-module, is isomorphic to 
$\D_{K,\sigma}/\D_{K,\sigma}(\sigma-1)$. \\
For any invertible matrix $A \in \GL_n(K)$, putting 
$\Phi_A(X) := A^{-1} (\sigma X)$, we define a difference module
$(K^n,\Phi_A)$. For instance, if $n = 1$ and $A = (1)$, this is $\1$.
Any difference module $(V,\Phi)$ such that $V$ is free over $K$ is
isomorphic to some $(K^n,\Phi_A)$.
\end{exem}

\label{not122}
We write $\DM$ the full subcategory of $\DKMod$ with objects the 
difference modules over $K$. Like $\DKMod$, it is abelian and $C$-linear. 
Thus, monomorphisms, epimorphisms, isomorphisms and exact sequences in
$\DM$ are the same thing as in $\DKMod$. Moreover, since an extension of 
finite rank projective $K$-modules is a finite rank projective $K$-module 
it follows that $\DM$ is a thick subcategory of $\DKMod$, so that the 
calculus of extensions in $\DM$ is the restriction of the calculus of 
extensions in $\DKMod$. For instance, computing $\Ext^i(M,N)$ in $\DM$ 
or in $\DKMod$ is the same.

\begin{rema}
Many constructions still make sense, and many properties remain true,
and moreover useful, without the projectiveness or finiteness conditions.
For instance, one may want to consider \emph{difference extension rings}
\index{difference extension rings}
of $(K,\sigma)$, that is, difference rings $(K',\sigma')$ such that
$K \subset K'$ and $\sigma'_{|K} = \sigma$; in some interesting cases,
$K'$ is not a finite $K$-module: for instance, $\Ka \subset \Kf$. We will 
sometimes briefly mention these generalisations, for future reference. 
\end{rema}


\subsection{The functor $\Gamma$}
\label{subsection:foncteurGamma}

\index{functor $\Gamma$}
\label{not123}
Let $M := (V,\Phi)$ and $N := (W,\Psi)$ two difference modules. Write
$\Lin_{\sigma}(V,W) := \Lin_K(V,\sigma^* W)$ the $K$-module of $\sigma$-linear 
maps from $V$ to $W$. For simplicity, we shall also (improperly) write 
$\Lin_K(M,N) := \Lin_K(V,W)$ and $\Lin_{\sigma}(M,N) := \Lin_{\sigma}(V,W)$. \\

For any $K$-linear map $f: V \rightarrow W$, the map 
$\Psi \circ f - f \circ \Phi$ is $\sigma$-linear from $V$ to $W$.
\label{not124}
The map $t_{\Phi,\Psi}: f \mapsto \Psi \circ f - f \circ \Phi$ from 
$\Lin_K(M,N)$ to $\Lin_{\sigma}(M,N)$ is $C$-linear and, by definition, 
$\Hom(M,N)$ is its kernel. We thus have an exact sequence of $C$-modules:
$$
0 \rightarrow \Hom(M,N) \rightarrow \Lin_K(M,N) 
\overset{t_{\Phi,\Psi}}{\longrightarrow} \Lin_{\sigma}(M,N).
$$
This sequence is clearly functorial in $M$ and in $N$ (contravariant
in $M$, covariant in $N$). We shall see later that, under appropriate
assumptions, its cokernel is the module $\Ext(M,N)$ (theorem
\ref{theo:descriptionextensions}).

\begin{defi}
\index{functor of solutions}
The \emph{functor $\Gamma$ of solutions} is defined as:
$$
\label{not125}
M \leadsto \Gamma(M) := \Hom(\1,M).
$$
\end{defi}

\begin{prop}
The functor $M \leadsto \Gamma(M)$ from difference modules to 
$C$-modules is left exact, and one has an identification:
$$
\Gamma(M) = \{v \in V \tq \Phi(v) = v\}.
$$
\end{prop}
\begin{proof}
The functoriality and exactness are plain. The identification comes
as follows: under the usual identification $f \mapsto v := f(1)$ of 
$\Lin_K(K,V)$ with $V$, the condition $\Phi \circ f = f \circ \sigma$
translates into $\Phi(v) = v$.
\end{proof}

The $i$-th right derived functor of $\Gamma$ is given by:
$$
\label{not126}
M \leadsto \Gamma^i(M) = \Ext^i(\1,M).
$$

\begin{rema}
The contents of this paragraph remain valid without conditions for
all $\D_{K,\sigma}$-modules.
\end{rema}


\subsection{Tensor products}
\label{subsection:tensorproduct}
\index{tensor product}

Let $M := (V,\Phi)$ and $N := (W,\Psi)$ two difference modules. According 
to \cite[\S 3, no 3]{BAL}, one has a canonical isomorphism:
$$
\sigma^*(V \otimes_K W) \simeq \sigma^* V \otimes_K \sigma^* W.
$$
It follows that there is a unique $\sigma$-linear isomorphism:
$$
\Phi \otimes \Psi: V \otimes_K W \rightarrow \sigma^*(V \otimes_K W)
$$
\label{not127}
such that $x \otimes y \mapsto \Phi(x) \otimes \Psi(y)$. This makes the
tensor product $M \otimes N := (V \otimes_K W,\Phi \otimes \Psi)$ into
a difference module. Note however that the result is not a tensor product
of $\D_{K,\sigma}$-modules.

\begin{prop}
The tensor product in $\DM$ is associative, commutative, bifunctorial
(covariant in each argument) and exact. Moreover, there are canonical 
isomorphisms:
$$
\1 \otimes M \simeq M \simeq M \otimes \1.
$$
\end{prop}
\begin{proof}
This follows immediately from the corresponding properties for projective
$K$-modules, with the same meaning (commutativity is understood up to 
canonical isomorphisms, etc).
\end{proof}

\begin{rema}
The contents of this paragraph remain valid without conditions for
all $\D_{K,\sigma}$-modules, except for the exactness in the proposition:
this remains true for $K$-projective  $\D_{K,\sigma}$-modules but reduces 
to right exactness for general $\D_{K,\sigma}$-modules.
\end{rema}


\subsection{Internal $\Hom$}
\label{subsection:internalhom}
\index{internal $\Hom$}

Let $M := (V,\Phi)$ and $N := (W,\Psi)$ two difference modules; then
$\Lin_K(V,W)$ is projective of finite rank over $K$. For any $K$-linear 
map $u: V \rightarrow W$, the map $\Psi \circ u \circ \Phi^{-1}$
is $K$-linear. The map $\Gamma: u \mapsto \Psi \circ u \circ \Phi^{-1}$ from
$\Lin_K(V,W)$ to itself is a $K$-linear automorphism, whence a difference
module:
$$
\label{not128}
\Homint(M,N) := \bigl(\Lin_K(V,W),\Gamma\bigr).
$$
This is called an ``internal Hom''. For instance, $\Homint(\1,M)$ is
canonically isomorphic to $M$.

\begin{theo}
\label{theo:adjunction}
Let $M := (V,\Phi)$, $M' := (V',\Phi')$ and $N := (W,\Psi)$ be three 
difference modules. There is a canonical functorial isomorphism:
$$
\Hom\bigl(M,\Homint(M',N)\bigr) \simeq \Hom(M \otimes M',N).
$$
\end{theo}
\begin{proof}
From \cite[\S 4.1, prop. 1]{BAL}, we have a canonical (and functorial) 
isomorphism:
$$
\Lin_K\bigl(V,\Lin_K(V',W)\bigr) \simeq \Lin_K(V \otimes_K V',W),
$$
which sends $f: V \rightarrow \Lin_K(V',W)$ to the unique linear map
$g: V \otimes_K V' \rightarrow W$ such that $g(v \otimes v') = f(v)(v')$. 
Conversely, $g: V \otimes_K V' \rightarrow W$ is sent to 
$f: V \rightarrow \Lin_K(V',W)$ defined by 
$v \mapsto (v' \mapsto g(v \otimes v'))$. This yields the isomorphism 
of the underlying $K$-modules. An easy computation then shows that:
$$
\Psi \circ g = g \circ (\Phi \otimes \Phi') \Longleftrightarrow
\Gamma \circ f = f \circ \Phi;
$$
with $f$, $g$ and $\Gamma$ defined as above. Functoriality comes from 
the corresponding assertion in linear algebra.
\end{proof}

\begin{rema}
The contents of this paragraph remain valid without conditions for
all $\D_{K,\sigma}$-modules.
\end{rema}


\subsection{Duality}
\label{subsection:duality}

We keep the same notations as in paragraph \ref{subsection:internalhom}.
According to \cite[\S 4.2]{BAL}, the canonical morphism:
$$
\nu: \Lin_K(V,V') \otimes_K W \rightarrow \Lin_K(V,V' \otimes_K W)
$$
is a monomorphism if $W$ is projective and an isomorphism if $V$ or $W$
is projective of finite rank. One checks that the following diagram
is commutative:
$$
\xymatrix{
\Lin_K(V,V') \otimes_K W \ar@<0ex>[d] \ar@<0ex>[rr]^{\nu} & &
\Lin_K(V,V' \otimes_K W)  \ar@<0ex>[d] \\
\Lin_K(V,V') \otimes_K W \ar@<0ex>[rr]^{\nu} & & 
\Lin_K(V,V' \otimes_K W)
}
$$
The left vertical map is given by:
$f \otimes w \mapsto (\Phi' \circ f \circ \Phi^{-1}) \otimes \Psi(w)$,
while the right vertical map is given by:
$g \mapsto (\Phi' \otimes \Psi) \circ g \circ \Phi^{-1}$.

\begin{prop}
For any three difference modules, $\nu$ induces a canonical and
functorial isomorphism of difference modules:
$$
\Homint(M,M') \otimes N \simeq \Homint(M,M' \otimes N).
$$
\end{prop}
\begin{proof}
Immediate from the discussion above.
\end{proof}

\begin{defi}
\index{dual}
\label{not129}
The \emph{dual} of the difference module $M$ is:
$$
M^{\vee} := \Homint(M,\1).
$$
\end{defi}

If $M = (V,\Phi)$, then $M^{\vee} = (V^*,\Phi^{\vee})$, where
$\Phi^{\vee}(u) := \sigma \circ u \circ \Phi^{-1}$. For instance, 
for $A \in \GL_n(K)$, one has $(K^n,\Phi_A)^{\vee} = (K^n,A^{\vee})$, 
where $A^{\vee} := {}^tA^{-1}$ is the contragredient of $A$.

\begin{coro}
We have a canonical and functorial isomorphism of difference modules:
$$
M^{\vee} \otimes N \simeq \Homint(M,N).
$$
\end{coro}
\begin{proof}
Set $M' := \1$ in the proposition.
\end{proof}

\begin{rema}
Part of these conclusions stay true for more general $\D_{K,\sigma}$-modules.
In the proposition, as well as in the corollary, the map is a monomorphism 
if $M$ is an arbitrary $\D_{K,\sigma}$-module and $N$ is a $K$-projective
$\D_{K,\sigma}$-module; and an isomorphism if $M$ or $N$ is a difference 
module, the other being an arbitrary $\D_{K,\sigma}$-module.
\end{rema}


\section{Finitely filtered difference modules}
\label{section:filtmoddiff}
\index{finitely filtered difference modules}


\subsection{Isograded classification of difference modules}
\label{subsection:isograddiffmod}
\index{isograded classification of difference modules}

Let $(K,\sigma)$ be a difference ring, with $K$ is a commutative 
$C$-algebra and $\sigma$ a $C$-algebra automorphism. Recall from
the previous section that a difference module over $K$ is 
a left $\D_{K,\sigma}$-module which is projective of finite rank as
a $K$-module and that the category $\DM$ of difference modules is
abelian and $C$-linear. Difference modules over $K$ can be realized 
as pairs $(E,\Phi)$, where $E$ is a projective $K$-module of finite rank
and $\Phi$ is a semi-linear automorphism of $E$. In this description, 
a morphism of difference modules from $(E,\Phi)$ to $(F,\Psi)$ is a map 
$u \in \Lin_{K}(E,F)$ such that $\Psi \circ u = u \circ \Phi$. \\

As an instance of the general problem considered in section 
\ref{section:generalsetting}, we fix difference modules $P_{i}$ 
($1 \leq i \leq k$) and consider finitely filtered difference modules
$\underline{M}$ with associated graded module 
$P = P_{1} \oplus \cdots \oplus P_{k}$. We assume that each such object
$\underline{M} = (0 = M_{0} \subset M_{1} \subset \cdots \subset M_{k} = M)$
comes equipped with an isomorphism
$\underline{u} = (u_{i}: M_{i}/M_{i-1} \simeq P_{i})_{1 \leq i \leq k}$
from $\gr M$ to $P$. \\

In accordance with the general problem described in section 
\ref{section:generalsetting}, we consider $(\underline{M},\underline{u})$ 
and $(\underline{M'},\underline{u'})$ to be equivalent if there exists 
a morphism $f: M \rightarrow M'$ that respects the filtration: 
$f(M_i) \subset M'_i$ for $1 \leq i \leq k$, and such that, writing 
$g_i: M_{i}/M_{i-1} \rightarrow M'_{i}/M'_{i-1}$ the morphisms thus induced 
by $f_{|M_i}$, one has $u_i = u'_i \circ g_i$ for $1 \leq i \leq k$. 
Since $f$ respects the filtration, one can define
$\gr f: \gr M \rightarrow \gr M'$ and the above condition is equivalent
to $\underline{u} = \gr f \circ \underline{u'}$. It is a standard fact
(and easy to prove) that $f$ is then automatically an isomorphism. \\

Our goal is to classify such pairs $(\underline{M},\underline{u})$ up to
this equivalence relation. The equivalence classes clearly form a set,
which we write $\F(P_{1},\ldots,P_{k})$. In this section, we give some 
matricial descriptions, under additional assumptions.


\subsection{Matricial description of difference modules}
\label{subsection:descmat1}
\index{matricial description of difference modules}

Here, we assume that the $K$-module $E$ is \emph{free} of finite rank $n$.
This assumption will be maintained in \ref{subsection:descmat2} and
\ref{subsection:descmat3}, where we pursue the matricial description.
We then see that the difference module $M = (E,\Phi)$ is then isomorphic
to some $(K^n,\Phi_A)$ (example \ref{exem:unitetPhiA}). Indeed, choosing
a basis $\B$ allows one to identify $E$ with $K^{n}$. It is then
clear that $\Phi(\B)$ is also a basis of $E$, whence the existence of
$A \in GL_{n}(K)$ such that $\Phi(\B) = \B A^{-1}$. If $x \in E$ has the
coordinate column vector $X \in K^{n}$ in basis $\B$, \ie\ if $x = \B X$, 
computing $\Phi(x) = \Phi(\B X) = \Phi(\B) \sigma(X) = \B A^{-1} \sigma(X)$
shows that $\Phi(x) \in E$ has the coordinate column vector 
$A^{-1} \sigma(X)$; we thus may identify $(E,\Phi) \simeq (K^{n},\Phi_{A})$, 
where $\Phi_{A}(X) := A^{-1} \sigma(X)$. \\

Morphisms from $(K^{n},\Phi_{A})$ to $(K^{p},\Phi_{B})$ are matrices 
$F \in \Mat_{p,n}(K)$ such that $(\sigma F) A = B F$ and their composition 
boils down to matrix product. In particular, isomorphism of modules is 
described by gauge transformations:
$$
(K^{n},\Phi_{A}) \simeq (K^{p},\Phi_{B}) \Longleftrightarrow
n = p \text{~and~} \exists F \in GL_{n}(K) \;:\;
B = F[A] := (\sigma F) A F^{-1}.
$$

\begin{rema}
A similar description can be given for free modules of infinite rank.
Such a $K$-module is isomorphic to $K^{(I)}$, and one should replace
$\GL_n(K)$ by $\GL_I(K)$ and $\Mat_{p,n}(K)$ by $\Mat_{J \times I}(K)$.
Infinite ``matrices'' are to be understood as follows: elements 
of $\End_K(K^{(I)})$ are families $(a_{i,j})$ of elements of $K$ indexed by 
$I \times I$ and such that each column ($j$ fixed) has finite support. 
The set $\Mat_I(K) := \End_K(K^{(I)})$ is a unitary ring and 
$\GL_I(K) := \Aut_K(K^{(I)})$ is its group of units. There is a similar 
condition for $\Mat_{J \times I}(K)$.
\end{rema}


\subsection{Matricial description of filtered difference modules}
\label{subsection:descmat2}
\index{matricial description of filtered difference modules}

Let $P_{i} = (G_{i},\Psi_{i})$ ($1 \leq i \leq k$) be difference modules
such that each $K$-module $G_{i}$ is free of finite rank $r_{i}$ . For
each $G_{i}$, choose a basis $\mathcal{D}_{i}$, and write
$B_{i} \in GL_{r_{i}}(K)$ the invertible matrix such that
$\Psi_{i}(\mathcal{D}_{i}) = \mathcal{D}_{i} B_{i}$. \\

Let $\underline{M}$ be a finitely filtered difference module with
associated graded module $P = P_{1} \oplus \cdots \oplus P_{k}$; more
precisely, 
$\underline{M} = (0 = M_{0} \subset M_{1} \subset \cdots \subset M_{k} = M)$
is equipped with an isomorphism
$\underline{u} = (u_{i}: M_{i}/M_{i-1} \simeq P_{i})_{1 \leq i \leq k}$
from $\gr M$ to $P$. Letting $M_{i} = (E_{i},\Phi_{i})$, one builds a
basis $\mathcal{B}_{i}$ of $E_{i}$ by induction on $i = 1,\ldots,k$ in such
a way that $\mathcal{B}_{i-1} \subset \mathcal{B}_{i}$ and that
$\mathcal{B}'_{i} := \mathcal{B}_{i} \setminus \mathcal{B}_{i-1}$ lifts 
$\mathcal{D}_{i}$ via $u_{i}$ (showing by the way that $K$-modules $E_{i}$ 
are free of finite ranks $r_{1} + \cdots + r_{i}$).
Write $\Phi_{i}$ (resp. $\overline{\Phi_{i}}$) the semi-linear automorphism 
induced by $\Phi$ (resp. by $\Phi_{i}$) on $E_{i}$ (resp. on $E_{i}/E_{i-1}$), 
and $\mathcal{C}_{i}$ the basis of $E_{i}/E_{i-1}$ induced by 
$\mathcal{B}'_{i}$, one draws from equality
$u_{i} \circ \overline{\Phi_{i}} = \Psi_{i} \circ u_{i}$ (due to the fact
that $u_{i}$ is a morphism) the relation:
$\overline{\Phi_{i}}(\mathcal{C}_{i}) = \mathcal{C}_{i} B_{i}$, then,
from the latter, the relation:
$$
\Phi_{i}(\mathcal{B}'_{i}) \equiv \mathcal{B}'_{i} B_{i} \pmod{E_{i-1}}.
$$
Last, one gets that the matrix of $\Phi$ in basis $\mathcal{B}$ is
block upper-triangular:
$$
\Phi(\mathcal{B}) = \mathcal{B} 
\begin{pmatrix}
B_{1} & \star  & \star \\ 
0     & \ddots & \star \\
0     & 0      & B_{k} 
\end{pmatrix}.
$$
As in \ref{subsection:descmat1}, we identify $P_{i}$ ($1 \leq i \leq k$) 
with $(K^{r_{i}},\Phi_{A_{i}})$, where $A_{i} := B_{i}^{-1} \in GL_{r_{i}}(K)$. 
Likewise, $P$ is identified with $(K^{n},\Phi_{A_{0}})$ and $M$ with
$(K^{n},\Phi_{A})$, where $n := r_{1} + \cdots + r_{k}$ and:
$$
A_{0} = \begin{pmatrix} 
A_{1} & 0      & 0 \\
0     & \ddots & 0 \\ 
0     & 0      & A_{k} 
\end{pmatrix} \text{~~et~~}
A = \begin{pmatrix} 
A_{1} & \star  & \star \\
0     & \ddots & \star \\ 
0     & 0      & A_{k} 
\end{pmatrix}.
$$
Note that these relations implicitly presuppose that a filtration 
on $M$ is given, as well as an isomorphism from $\gr M$ to $P$. If
moreover $M' = (K^{n},\Phi_{A'})$, where $A'$ has the same form as $A$ 
(\ie\ $M'$ is filtered and equipped with an isomorphism from $\gr M'$
to $P$), then, a morphism from $M$ to $M'$ respecting filtrations (\ie\ 
sending each $M_{i}$ into $M'_{i}$) is described by a matrix $F$ in the
following block upper triangular form; and the induced endomorphism of 
$P \simeq \gr M \simeq \gr M'$ is described by the corresponding block
diagonal matrix $F_{0}$
$$
F = \begin{pmatrix} 
F_{1} & \star  & \star \\
0     & \ddots & \star \\ 
0     & 0      & F_{k} 
\end{pmatrix} \text{~~et~~}
F_{0} = \begin{pmatrix} 
F_{1} & 0      & 0     \\
0     & \ddots & 0     \\ 
0     & 0      & F_{k} 
\end{pmatrix}.
$$
In particular, a morphism inducing identity on $P$ (thus ensuring that the
filtered modules $M,M'$ belong to the same class in $\F(P_{1},\ldots,P_{k})$) 
is represented by a matrix in $\G(K)$, where we denote $\G$ the algebraic
\label{not130}
subgroup of $GL_{n}$ defined by the following shape:
$$
\begin{pmatrix} I_{r_1} & \star  & \star \\
0     & \ddots & \star \\ 0     & 0      & I_{r_k} 
\end{pmatrix}.
$$
(The unipotent algebraic group $\G$ was previously introduced in paragraph
\ref{subsection:matricialdescriptionofF(P1,...,Pk)}, page
\pageref{endroit:introductiondugroupeG}.) \\

\label{not131}
Now write $A_{U}$ the block upper triangular matrix with block diagonal
component $A_{0}$ and with upper triangular blocks the 
$U_{i,j} \in \Mat_{r_{i},r_{j}}(K)$ ($1 \leq i < j \leq k$); here, $U$ is an
abrevation for the family $(U_{i,j})$. For all $F \in \G(K)$, the matrix
$F[A_{U}]$ is equal to $A_{V}$ for some family of 
$V_{i,j} \in \Mat_{r_{i},r_{j}}(K)$. Thus, the group $\G(K)$ operates on the set
$\prod\limits_{1 \leq i < j \leq k} \Mat_{r_{i},r_{j}}(K)$. The above discussion can 
be summarised as follows:

\begin{prop}
\label{prop:F(P1...Pk)isoProdMat}
The map sending $U$ to the class of $(K^{n},\Phi_{A_{U}})$ induces a
bijection from the quotient of the set 
$\prod\limits_{1 \leq i < j \leq k} \Mat_{r_{i},r_{j}}(K)$ under the action of
the group $\G(K)$ onto the set $\F(P_{1},\ldots,P_{k})$.
\end{prop}

\begin{rema}
The point of the above trivial and rather long construction of triangular
matrices is not so much that adequate bases exist, but that, by choosing 
such triangular forms, one implicitly fixes the structural isomorphisms
$M_i/M_{i-1} \simeq P_i$.
\end{rema}

 
\section{Extensions of difference modules}
\label{section:extdiffmod}
\index{extensions of difference modules}

Let $0 \rightarrow M' \rightarrow M \rightarrow M'' \rightarrow 0$
an exact sequence in $\DKMod$. If $M',M''$ are difference modules, 
so is $M$ (it is projective of finite rank because we have a split 
sequence of $K$-modules). The calculus of extensions is therefore
the same in $\DM$ as in $\DKMod$ and we will simply write $\Ext(M'',M')$ 
the group $\Ext_{\D_{K,\sigma}}(M'',M')$ of classes of extensions of $M''$ 
by $M'$. From general homological algebra, we know that $\Ext(M'',M')$
is actually endowed with a structure of $C$-module, whence a $C$-module
structure on $\F(M',M'')$ (as noted at the beginning of section
\ref{subsection:smallvalues}). \\

On the other hand, if $M'$, $M''$ are free over $K$ of respective ranks 
$n'$, $n''$, proposition \ref{prop:F(P1...Pk)isoProdMat} of the previous 
section allows for a description of $\F(M',M'')$ as a quotient of the
space $\Mat_{n',n''}(K)$ by a $C$-linear group action. This also provides
\index{module structure on $\F(M',M'')$}
a $C$-module structure on $\F(M',M'')$. We shall show in this section that 
these structures actually coincide. For any two left modules $M'$,$M''$ 
over a unitary $C$-algebra, the linear structure on $\Ext(M'',M')$ is 
completely described in \cite[\S 7]{BAH}\footnote{And, to the best of 
our knowledge, nowhere else.} and we shall rely on this description to 
make explicit the corresponding structure of $C$-module of $\F(M',M'')$ 
in the case of arbitrary difference modules $M'$ and $M''$. The main result 
is theorem \ref{theo:descriptionextensions}. \\

So let $M = (E,\Phi)$ and $N = (F,\Psi)$. Any extension
$0 \rightarrow N \rightarrow R \rightarrow M \rightarrow 0$ of $M$
by $N$ gives rise (by restriction of scalars) to an exact sequence
of $K$-modules $0 \rightarrow F \rightarrow G \rightarrow E \rightarrow 0$ 
such that, if $R = (G,\Gamma)$, the following diagram is commutative:
$$
\begin{CD}
0 @>>> F   @>{i}>> G     @>{j}>>      E  @>>> 0     \\
&  &   @VV{\Psi}V      @VV{\Gamma}V      @VV{\Phi}V                 \\
0 @>>> F   @>{i}>> G     @>{j}>>      E @>>> 0
\end{CD}
$$
(We wrote again $i,j$ the underlying $K$-linear maps.) Since $E$ is 
projective, the sequence is split and one can from start identify
$G$ with the $K$-module $F \times E$, thus writing $i(y) = (y,0)$ and 
$j(y,x) = x$. The compatibility conditions $\Gamma \circ i = i \circ \Psi$
and $\Phi \circ j = j \circ \Gamma$ then imply:
\label{not132}
$$
\Gamma(y,x) = \Gamma_{u}(y,x) := \bigl(\Psi(y) + u(x), \Phi(x)\bigr),
\text{~with~} u \in \Lin_{\sigma}(E,F),
$$
where we write $\Lin_{\sigma}(E,F)$ the set of $\sigma$-linear maps from
$E$ to $F$. (This means that $u$ is a group morphism such that
$u(\lambda x) = \sigma(\lambda) u(x)$.) Setting moreover
$R_{u} := (F \times E,\Gamma_{u})$, which is a difference module 
naturally equipped with a structure of extension of $M$ by $N$, we see
that we have defined a surjective map:
\begin{align*}
\Lin_{\sigma}(E,F) & \rightarrow \Ext(M,N), \\
u & \mapsto \theta_{u} := \text{~class of~} R_{u}.
\end{align*}
We can make precise the conditions under which $u,v \in \Lin_{\sigma}(E,F)$
have the same image $\theta_{u} = \theta_{v}$, \ie\ under which $R_{u}$
and $R_{v}$ are equivalent extensions. This happens if there exists
a morphism $\phi: R_{u} \rightarrow R_{v}$ inducing the identity map
on $M$ and $N$, that is, a linear map
$\phi: F \times E \rightarrow F \times E$ such that 
$\Gamma_{v} \circ \phi = \phi \circ \Gamma_{u}$ (since it is a morphism
of difference modules) and having the form 
$(x,y) \mapsto \bigl(y + f(x),x\bigr)$ (since it induces the identity 
maps on $E$ and on $F$). Now, the first condition becomes:
$$
\forall (y,x) \in F \times E \;,\;
\bigl(\Psi(y + f(x)) + v(x),\Phi(x)\bigr) =
\bigl(\Psi(y) + u(x) + f(\Phi(x)),\Phi(x)\bigr),
$$
that is:
$$
u - v = \Psi \circ f - f \circ \Phi.
$$
Remark by the way that, for all $f \in \Lin_{K}(E,F)$, the map
$t_{\Phi,\Psi}(f) := \Psi \circ f - f \circ \Phi$ is $\sigma$-linear from
$E$ to $F$.

\begin{theo}
\label{theo:descriptionextensions}
The map $u \mapsto \theta_{u}$ from $\Lin_{\sigma}(E,F)$ to $\Ext(M,N)$ 
is functorial in $M$ and in $N$, $C$-linear, and its kernel is the image
of the $C$-linear map:
\begin{align*}
t_{\Phi,\Psi}: \Lin_{K}(E,F) & \rightarrow \Lin_{\sigma}(E,F), \\
f & \mapsto \Psi \circ f - f \circ \Phi. 
\end{align*}
\end{theo}
\begin{proof}

\noindent \underline{Functoriality.} \\
We shall only prove (and use) it on the covariant side, \ie\ in $N$. 
We invoke \cite[\S 7.1 p. 114, example 3 and \S 7.4, p. 119, prop. 4]{BAH}.
Let $\theta$ be the class in $\Ext(M,N)$ of the extension 
$0 \overset{i}{\rightarrow} N \rightarrow R 
\overset{j}{\rightarrow} M \rightarrow 0$ and
$g: N \rightarrow N'$ a morphism in $\DM$. Let
$$
\begin{CD}
0 @>>> N   @>{i}>>  R     @>{j}>>       M  @>>> 0     \\
&  &   @VV{g}V      @VV{h}V      @VV{\Id_{M}}V                 \\
0 @>>> N'  @>{i'}>> R'    @>{j'}>>      M  @>>> 0
\end{CD}
$$
be a commutative diagram of exact sequences. If $\theta'$ is the class
in $\Ext(M,N)$ of the extension
$0 \overset{i'}{\rightarrow} N' \rightarrow R'
\overset{j'}{\rightarrow} M \rightarrow 0$,
then:
$$
\Ext(\Id_{M},g)(\theta) = g \circ \theta = 
\theta' \circ \Id_{M} = \theta'.
$$
We take:
$$
R' := R \oplus_{N} N' = 
\dfrac{R \times N'}{\{\bigl(i(n),-g(n)\bigr) \tq n \in N\}},
$$
with $i',j'$ the obvious arrows, and for $R$ the extension $R_{u}$;
with the previous notations for $N,M,R$, and also writing 
$N' = (F',\Psi')$, with compatibility condition 
$\Psi' \circ g = g \circ \Psi$, we see that the $K$-module underlying
$R \oplus_{N} N'$ is:
$$
G' := \dfrac{F \times E \times F'}{\{\bigl(y,0,-g(y)\bigr) \tq y \in F\}},
$$
endowed with the semi-linear automorphism induced by the map
$\Gamma_{u} \times \Psi'$ from $F \times E \times F'$ to itself
(the latter does fix the denominator). \\
The map $(y,x,y') \mapsto \bigl(y' + g(y),x\bigr)$ from
$F \times E \times F'$ to $F' \times E$ induces an isomorphism from $G'$ 
to $F' \times E$ and the induced semi-linear automorphism on $G'$ is 
$(y',x) \mapsto \bigl(\Psi'(y') + g\bigl(u(x)\bigr), \Phi(x)\bigr)$,
that is $\Gamma_{gu}$, from which it follows that $R' = R_{gu}$.
The arrows $i',j'$ are determined as follows: $i'(y')$ is the class 
of $(0,y')$ in $G'$, that is, under the previous identification,
$i'(y') = (y',0)$; and $j'(y',x)$ is the image of an arbitrary preimage,
for instance the class of $(0,x,y')$: that image is $j(0,x) = x$. 
We have therefore shown that the class of the extension $R_{u}$ by  
$\Ext(\Id_{M},g)$ is $R_{gu}$, which is the wanted functoriality. It is
expressed by the commutativity of the following diagram:
$$
\begin{CD}
\Lin_{\sigma}(E,F)  @>>> \Ext(M,N) \\
@VV{\Lin_{\sigma}(\Id_{M},g)}V   @VV{\Ext(\Id_{M},g)}V  \\
\Lin_{\sigma}(E,F') @>>> \Ext(M,N')
\end{CD}
$$

\noindent \underline{Linearity.} \\
According to the remark just before the theorem, the map $t_{\Phi,\Psi}$ 
indeed sends $\Lin_{K}(E,F)$ to $\Lin_{\sigma}(E,F)$. \\
\emph{Addition.}
The reference here is \cite[\S 7.6, rem. 2 p. 124]{BAH}. From the
extensions
$0 \rightarrow N \overset{i}{\rightarrow} R 
\overset{p}{\rightarrow} M \rightarrow 0$ and
$0 \rightarrow N \overset{i'}{\rightarrow} R'
\overset{p'}{\rightarrow} M \rightarrow 0$ having classes 
$\theta,\theta' \in \Ext^{1}(M,N)$, one computes $\theta + \theta'$
as the class of the extension
$0 \rightarrow N \overset{i''}{\rightarrow} R''
\overset{p''}{\rightarrow} M \rightarrow 0$, where:
$$
R'' := \dfrac{\{(z,z') \in R \times R' \tq p(z) = p'(z')\}}
{\{(-i(y),i'(y)) \tq y \in N\}},
$$
and $i''(y)$ is the class of $(0,i'(y))$, \ie\ the same as the class of
$(i(y),0)$; and $p''$ sends the class of $(z,z')$ to $p(z) = p'(z')$.
Taking $R = R_{u}$ and $R' = R_{u'}$, the numerator of $R''$ is identified
with $F \times F \times E$ equipped with the semi-linear automorphism
$(y,y',x) \mapsto (\Psi(y) + u(x),\Psi(y') + u'(x),\Phi(x))$. 
The denominator is identified with the subspace $\{(-y,y,0) \tq y \in F\}$
equipped with the induced map. The quotient is identified with
$0 \times F \times E$, through the map $(y,y',x) \mapsto (0,y'',x)$, where
$y'' := y' + y$, equipped with the semi-linear automorphism $\Phi''$ which
sends $(0,y'',x)$ to
$$
(0,\Psi(y') + u'(x) + \Psi(y) + u(x),\Phi(x)) =
(0,\Psi(y'') + (u + u')(x),\Phi(x)).
$$
This is indeed $R_{u+u'}$. \\
\emph{External multiplication.}
The reference here is \cite[\S 7.6, prop. 4 p. 119]{BAH}. 
Let $\lambda \in C$. We apply the invoked proposition to the following
commutative diagram of exact sequences:
$$
\begin{CD}
0 @>>> N   @>>>  R_{u}          @>>>       M  @>>> 0     \\
&  &  @VV{\times \lambda}V  @VV{(\times \lambda,\Id_{M})}V @VV{\Id_{M}}V \\
0 @>>> N  @>>> R_{\lambda u}    @>>>       M  @>>> 0
\end{CD}
$$
If $\theta$, $\theta'$ are the classes in $\Ext(M,N)$ of the two
extensions, one infers from \emph{loc. cit.} that:
$$
\theta' \circ \Id_{M} = (\times \lambda) \circ \theta 
\Longrightarrow \theta' = \lambda \theta.
$$
The class of the extension $R_{\lambda u}$ is therefore indeed equal
to the product of $\lambda$ by the class of the extension $R_{u}$. \\

\noindent \underline{Exactness.} \\
It follows immediately from the computation shown just before the
statement of the theorem.
\end{proof}


\subsection{The complex of solutions}
\label{subsection:complexedessolutions}

The following is sometimes considered as a difference analog of the 
de Rham complex in one variable, see for instance \cite{YA,TarVar}.

\begin{defi}
\index{complex of solutions}
We call \emph{complex of solutions of $M$ in $N$} the following
complex of $C$-modules:
\label{not133}
\begin{align*}
t_{\Phi,\Psi}: \Lin_{K}(E,F) & \rightarrow \Lin_{\sigma}(E,F), \\
f & \mapsto \Psi \circ f - f \circ \Phi. 
\end{align*}
concentrated in degrees $0$ and $1$.
\end{defi}

It is indeed clear that the source and target are $C$-modules, that
the map $t_{\Phi,\Psi}$ does send the source to the target and that
it is $C$-linear. 

\begin{coro}
\label{coro:homologiecomplexesolutions}
The homology of the complex of solutions is $H^{0} = \Hom(M,N)$ and
$H^{1} = \Ext(M,N)$, and these equalities are functorial.
\end{coro}
\begin{proof} 
The statement about $H^{1}$ is the theorem. As regards $H^{0}$, the 
kernel of $t_{\Phi,\Psi}$ is the $C$-module 
$\{f \in \Lin_{K}(E,F) \tq \Psi \circ f = f \circ \Phi\}$,
that is, $\Hom(M,N)$; and functoriality is obvious in that case.
\end{proof}

\begin{coro}
\label{coro:longuesuite}
From the exact sequence
$0 \rightarrow N' \rightarrow N \rightarrow N'' \rightarrow 0$,
one deduces the ``cohomology long exact sequence'':
\begin{align*}
0 \rightarrow \Hom(M,N') \rightarrow & \Hom(M,N) \rightarrow \Hom(M,N'') \\
& \rightarrow \Ext(M,N') \rightarrow \Ext(M,N) \rightarrow \Ext(M,N'')
\rightarrow 0.
\end{align*}
\end{coro}
\begin{proof}
We keep the previous notations (and moreover adapt them to $N',N''$). 
The exact sequence of projective $K$-modules
$0 \rightarrow F' \rightarrow F \rightarrow F'' \rightarrow 0$
being split, both lines of the commutative diagram:
$$
\begin{CD}
0 @>>> \Lin_{K}(E,F') @>>> \Lin_{K}(E,F)  @>>> \Lin_{K}(E,F'')  @>>> 0     \\
&  &  @VV{t_{\Phi,\Psi'}}V @VV{t_{\Phi,\Psi}}V @VV{t_{\Phi,\Psi''}}V  \\
0 @>>> \Lin_{\sigma}(E,F') @>>> \Lin_{\sigma}(E,F)  @>>> 
\Lin_{\sigma}(E,F'')  @>>> 0 
\end{CD}
$$
are exact, and it is enough to call to the snake lemma.
\end{proof}


\subsection{Matricial description of extensions of difference modules}
\label{subsection:descmat3}
\index{matricial description of extensions of difference modules}

We now assume $E,F$ to be free of finite rank over $K$ and accordingly
identify them with $M = (K^{m},\Phi_{A})$, $A \in GL_{m}(K)$ and
$N = (K^{n},\Phi_{B})$, $B \in GL_{n}(K)$. An extension of $N$ by $M$ 
then takes the form $R = (K^{m+n},\Phi_{C})$, where
$C = \begin{pmatrix} A & U \\ 0_{n,m} & B \end{pmatrix}$ for some 
rectangular matrix $U \in \Mat_{m,n}(K)$; we shall write $C = C_{U}$.
The injection $M \rightarrow R$ and the projection $R \rightarrow N$
have as respective matrices $\begin{pmatrix} I_{m} \\ 0_{n,m} \end{pmatrix}$
and $\begin{pmatrix} 0_{n,m} & I_{n}\end{pmatrix}$. The extension thus
defined will be denoted $R_{U}$. \\

A morphism of extensions $R_{U} \rightarrow R_{V}$ is a matrix of the form
$F = \begin{pmatrix} I_{m} & X \\ 0_{n,m} & I_{n} \end{pmatrix}$ for some
rectangular matrix $X \in \Mat_{m,n}(K)$. The compatibility condition with
the semi-linear automorphisms writes:
$$
(\sigma F) C_{U} = C_{V} F \Longleftrightarrow U + (\sigma X) B = A X + V
\Longleftrightarrow V - U = (\sigma X) B - A X.
$$

\begin{coro}
\label{coro:descmatrextabstrait}
The $C$-module $\Ext^{1}(N,M)$ is thereby identified with the cokernel
of the endomorphism $X \mapsto (\sigma X) B - A X$ of $\Mat_{m,n}(K)$.
\end{coro}
\begin{proof}
The above construction provides us with a bijection, but it follows from
theorem \ref {theo:descriptionextensions} that it is indeed an isomorphism.
\end{proof}

By taking $m = n = 1$, and $A := (1)$, $B := (cz)$, we recover example
\ref{exem:archetypal} page \pageref{exem:archetypal}. Letting $c := 1$
yields the case that was studied in paragraph \ref{subsection:prototypal}.


\section{Extension of scalars}
\label{section:extscal}
\index{extension of scalars}

We want to see $\F(P_{1},\ldots,P_{k})$ as a scheme over $C$, that is as
a representable functor
$C' \leadsto \F(C' \otimes_{C} P_{1},\ldots,C' \otimes_{C} P_{k})$ 
\index{functor
$C' \leadsto \F(C' \otimes_{C} P_{1},\ldots,C' \otimes_{C} P_{k})$}
from commutative $C$-algebras to sets. To that end, we shall extend what
we did to a ``relative'' situation. \\

Let $C'$ be a commutative $C$-algebra. We set:
$$
K' := C' \otimes_{C} K \text{~and~} \sigma' := 1 \otimes_{C} \sigma.
$$
Then $K'$ is a commutative $C'$-algebra and $\sigma'$ an automorphism
of that $C'$-algebra. Moreover:
$$
\D_{K',\sigma'} := K'<\sigma',{\sigma'}^{-1}> = 
K' \otimes_{K} \D_{K,\sigma} = C' \otimes_{C} \D_{K,\sigma}.
$$
These equalities should be interpreted as natural (functorial) 
isomorphisms. \\

From any difference module $M = (E,\Phi)$ over $(K,\sigma)$, one
gets a difference module $M' = (E',\Phi')$ over $(K',\sigma')$
by putting:
$$
E' = K' \otimes_{K} E = C' \otimes_{C} E \text{~and~}
\Phi' = \sigma' \otimes_{K} \Phi = 1 \otimes_{C} \Phi.
$$
(This is indeed a left $\D_{K',\sigma'}$-module and it is projective
of finite rank over $K'$.) We shall write it $M' = C' \otimes_{k} M$
to emphasize the dependency on $C'$. The following proposition is 
the tool to tackle the case $k = 2$.

\begin{prop}
\label{prop:supermajik}
Let $M,N$ be two difference modules over $(K,\sigma)$. One has a
functorial isomorphism of $C'$-modules:
$$
\Ext_{\D_{K',\sigma'}}(C' \otimes_{C} M,C' \otimes_{C} N) 
\simeq C' \otimes_{C} \Ext_{\D_{K,\sigma}}(M,N),
$$
and a functorial epimorphism of $C'$-modules:
$$
C' \otimes_{C} \Hom_{\D_{K,\sigma}}(M,N) \rightarrow
\Hom_{\D_{K',\sigma'}}(C' \otimes_{C} M,C' \otimes_{C} N).
$$
\end{prop}
\begin{proof}
We shall write $M' = C' \otimes_{C} M$, $E' = K' \otimes_{K} E$ etc.
The $K$-modules $E,F$ beeing projective of finite rank, there are
natural isomorphisms:
$$
C' \otimes_{C} \Lin_{K}(E,F) = \Lin_{K'}(E',F') \text{~et~}
C' \otimes_{C} \Lin_{\sigma}(E,F) = \Lin_{\sigma'}(E',F').
$$
(This is immediate if $E$ and $F$ are free, the general case follows.)
By tensoring the (functorial) exact sequence:
$$
0 \rightarrow \Hom_{\D_{K,\sigma}}(M,N) \rightarrow \Lin_{K}(E,F) 
\rightarrow \Lin_{\sigma}(E,F) \rightarrow \Ext_{\D_{K,\sigma}}(M,N)
\rightarrow 0,
$$
we get the exact sequence:
$$
C' \otimes_{C} \Hom_{\D_{K,\sigma}}(M,N) \rightarrow 
C' \otimes_{C} \Lin_{K}(E,F) \rightarrow C' \otimes_{C} \Lin_{\sigma}(E,F) 
\rightarrow C' \otimes_{C} \Ext_{\D_{K,\sigma}}(M,N) \rightarrow 0.
$$
Both conclusions then come by comparison with the exact sequence:
$$
0 \rightarrow \Hom_{\D_{K',\sigma'}}(M',N') \rightarrow \Lin_{K'}(E',F') 
\rightarrow \Lin_{\sigma'}(E',F') \rightarrow \Ext_{\D_{K',\sigma'}}(M',N')
\rightarrow 0.
$$
\end{proof}

\begin{prop}
Let $0 = M_{0} \subset M_{1} \subset \cdots \subset M_{k} = M$ be a
$k$-filtration with associated graded module 
$P_{1} \oplus \cdots \oplus P_{k}$. Then, setting 
$M'_{i} := C' \otimes_{C} M_{i}$ and
$P'_{i} := C' \otimes_{C} P_{i}$, we get a $k$-filtration 
$0 = M'_{0} \subset M'_{1} \subset \cdots \subset M'_{k} = M'$
with associated graded module $P'_{1} \oplus \cdots \oplus P'_{k}$.
\end{prop}
\begin{proof}
The $P_{i}$ being projective as $K$-modules, the exact sequences
$0 \rightarrow M_{i-1} \rightarrow M_{i} \rightarrow P_{i} \rightarrow 0$
are $K$-split, so they give rise by the base change $K' \otimes_{K}$ 
to exact sequences
$0 \rightarrow M'_{i-1} \rightarrow M'_{i} \rightarrow P'_{i} \rightarrow 0$.
\end{proof}

If $(\underline{M},\underline{u})$ denotes the pair made up of the above 
$k$-filtered object and of a fixed isomorphism from $\gr M$ to 
$P_{1} \oplus \cdots \oplus P_{k}$, we shall write
$(C' \otimes_{C} \underline{M},1 \otimes_{C} \underline{u})$ 
the corresponding pair deduced from the proposition.

\begin{defi}
We define as follows a functor $F$ from the category of commutative
$C$-algebras to the category of sets. For any commutative $C$-algebra
$C'$, we set:
$$
F(C') := \F(C' \otimes_{C} P_{1},\ldots,C' \otimes_{C} P_{k}).
$$
For any morphism $C' \rightarrow C''$ of commutative $C$-algebras,
the map $F(C') \rightarrow F(C'')$ is given by:
$$
\text{class of~} (\underline{M'},\underline{u'}) \mapsto
\text{class of~} 
(C'' \otimes_{k'} \underline{M'},1 \otimes_{C} \underline{u'}).
$$
\end{defi}

The set $F(C')$ is well defined according to the previous constructions.
The map $F(C') \rightarrow F(C'')$ is well defined on pairs thanks to
the proposition, and the reader will check that it goes to the quotient.
Last, the functoriality (preservation of the composition of morphisms)
comes from the contraction rule of tensor products:
$$
C''' \otimes_{C''} (C'' \otimes_{C'} \underline{M'}) =
C''' \otimes_{C'} \underline{M'}.
$$


\section{Our moduli space}
\label{section:modulispace}
\index{moduli space}

To simplify, herebelow, instead of saying ``the functor $F$ is 
represented by an affine space over $C$ (with dimension $d$)'',
we shall say ``the functor $F$ is an affine space over $C$ (with 
dimension $d$)''. This is just the usual identification of a
scheme with the space it represents.

\begin{theo}
\label{theo:themodulispaceabstract}
\index{affine space}
Assume that, for $1 \leq i < j \leq k$, one has $\Hom(P_{j},P_{i}) = 0$
and that the $C$-module $\Ext(P_{j},P_{i})$ is free of finite rank
$\delta_{i,j}$. Then the functor 
$C' \leadsto F(C') := \F(C' \otimes_{C} P_{1},\ldots,C' \otimes_{C} P_{k})$
is an affine space over $C$ with dimension 
$\sum\limits_{1 \leq i < j \leq k} \delta_{i,j}$.
\end{theo}
\begin{proof}
When $k = 1$, it is trivial. When $k = 2$, writing $V$ the free 
$C$-module of finite rank $\Ext(P_{2},P_{1})$, and appealing to
proposition \ref{prop:supermajik}, we see that this is the functor
$C' \leadsto C' \otimes_{C} V$, which is represented by the symetric
algebra of the dual of $V$, an algebra of polynomials over $C$.
For $k \geq 3$, we use an induction based on a lemma of Babbitt 
and Varadarajan \cite[lemma 2.5.3, p. 139]{BV}:
\begin{lemm}
\label{lemm:BabVar}
Let $u: F \rightarrow G$ be a natural transformation between two functors
from commutative $C$-algebras to sets. Assume that $G$ is an affine space
over $C$ and that, for any commutative $C$-algebra $C'$, and for any
$b \in G(C')$, the following functor ``fiber above $b$'' from
commutative $C'$-algebras to sets:
$$
C'' \leadsto u_{C''}^{-1}\bigl(G(C' \rightarrow C'')(b)\bigr)
$$
is an affine space over $C'$. Then $F$ is an affine space over $C$.
\end{lemm}
In \emph{loc. cit.}, this theorem is proved for $C = \C$, but the 
argument is plainly valid for any commutative ring. Here is its skeleton.
Choose $B = C[T_{1},\ldots,T_{d}]$ representing $G$. Take for $b$ the
identity of $G(B) = \Hom(B,B)$ (``general point''); the fiber is
represented by $B[S_{1},\ldots,S_{e}]$. One then shows that
$C[T_{1},\ldots,T_{d},S_{1},\ldots,S_{e}]$ represents $F$. This gives
by the way a computation of $\dim F$ as $\dim G + \dim $ of the
general fiber. In our case, all fibers will have the same dimension. \\

On our way to the proof of theorem \ref{theo:themodulispaceabstract},
we shall now have a closer look at these fibers.


\subsection{Structure of the fibers}

Before going to the proof of the theorem, we need an auxiliary result.
Remember that, all along, we assume that $\Hom(P_{j},P_{i}) = 0$ for 
$1 \leq i < j \leq k$.

\begin{prop}
\label{prop:structuredesfibres}
Let $C'$ be a commutative $C$-algebra and let $M'$ be a difference
module over $K' := C' \otimes_{C} K$, equipped with a $(k-1)$-filtration:
$0 = M'_{0} \subset M'_{1} \subset \cdots \subset M'_{k-1} = M'$
such that $\gr M' \simeq P'_{1} \oplus \cdots \oplus P'_{k-1}$
(as usual, $P'_{i} := C' \otimes_{C} P_{i}$). Then the functor in
commutative $C'$-algebras
$C'' \leadsto \Ext(C'' \otimes_{C} P_{k}, C'' \otimes_{C'} M')$
\index{affine space}
is an affine space over $C'$ with dimension
$\sum\limits_{1 \leq i \leq k} \delta_{i,k}$.
\end{prop}
\begin{proof}
After proposition \ref{prop:supermajik}, this is the functor
$C'' \leadsto C'' \otimes_{C'} \Ext(P'_{k},M')$. From each exact sequence
$0 \rightarrow M'_{i-1} \rightarrow M'_{i} \rightarrow P'_{i} \rightarrow 0$
one draws the cohomology long exact sequence of corollary 
\ref{coro:longuesuite}; but, from proposition \ref{prop:supermajik}, 
one draws that, for any commutative $C$-algebra $C'$, one has
$\Hom(C' \otimes_{C} P_{j},C' \otimes_{C} P_{i}) = 0$ and that the
$C'$-module $\Ext(C' \otimes_{C} P_{j},C' \otimes_{C} P_{i})$ 
is free of finite rank $\delta_{i,j}$. According to the equalities
$\Hom(P'_{j},P'_{i}) = 0$, the long exact sequence is here shortened as:
$$
0 \rightarrow \Ext(P'_{k},M'_{i-1}) \rightarrow \Ext(P'_{k},M'_{i}) 
\rightarrow \Ext(P'_{k}, P'_{i}) \rightarrow 0,
$$
and, for $i = 1,\ldots,k-1$, these sequences are split, the term at
the right being free. So, in the end:
$$
\Ext(P'_{k},M') \simeq \bigoplus_{1 \leq i \leq k} \Ext(P'_{k}, P'_{i}),
$$
which is free of rank $\sum\limits_{1 \leq i \leq k} \delta_{i,k}$.
As in the case $k = 2$ (which is a particular case of the proposition), 
the functor mentioned is represented by the symetric algebra of the dual
of this module.
\end{proof}

For all $\ell$ such that $1 \leq \ell \leq k$, let us write:
\begin{align*}
V_{\ell} := \bigoplus_{1 \leq i \leq \ell} \Ext(P_{\ell}, P_{i}), \\
W_{\ell} := \bigoplus_{1 \leq i < j \leq \ell} \Ext(P_{j}, P_{i}), \\
V'_{\ell} := \bigoplus_{1 \leq i \leq \ell} \Ext(P'_{\ell}, P'_{i}), \\
W'_{\ell} := \bigoplus_{1 \leq i < j \leq \ell} \Ext(P'_{j}, P'_{i}).
\end{align*}
We consider $V_{\ell}, W_{\ell}$ as affine schemes over $C$ and
$V'_{\ell}, W'_{\ell}$ as affine schemes over $C'$, so that:
$$
V'_{\ell} = C' \otimes_{C} V_{\ell}, \quad W'_{\ell} = C' \otimes_{C} W_{\ell}.
$$
We improperly write $C' \otimes_{C} V$ the base change of affine schemes
$\text{Spec~} C' \otimes_{\text{Spec~} C} V$. Also, we do not distinguish
between the direct sums of the free $C$-modules $\Ext(P_{j}, P_{i})$ and
the product of the corresponding affine schemes. Note that, in the proof 
of the proposition, the isomorphism $\Ext(P'_{k},M') \simeq V'_{\ell}$ is 
functorial in $C'$. This entails:

\begin{coro}
\label{coro:structuredesfibres}
Each fiber $\Ext(P'_{k},M')$ is isomorphic to $V'_{k}$ as a scheme over $C'$.
\end{coro}


\subsection{End of the proof of theorem \ref{theo:themodulispaceabstract}}

Now we end the proof of the theorem. Besides functor $F(C')$, we
consider the functor
$C' \leadsto G(C') := \F(C' \otimes_{C} P_{1},\ldots,C' \otimes_{C} P_{k-1})$,
of which we assume, by induction, that it is an affine space of dimension 
$\sum\limits_{1 \leq i < j \leq k-1} \delta_{i,j}$. The natural transformation 
from $F$ to $G$ is the one described in \ref{subsection:devissage}.
An element $b \in G(C')$ is the class of a pair
$(\underline{M'},\underline{u'})$, a $(k-1)$-filtered object over $C'$,
and the corresponding fiber is the one studied in the above auxiliary
proposition \ref{prop:structuredesfibres}. The lemma \ref{lemm:BabVar}
of Babbitt and Varadarajan then allows us to conclude.
\end{proof}

Looking at the proof of \ref{lemm:BabVar}, one moreover sees that,
if all fibers of $u: F \rightarrow G$ are isomorphic to a same affine
space $V$ (up to obvious extension of the base), then there is an
isomorphism $F \simeq V \times_{C} G$ such that $u$ corresponds to
the second projection. With the previous notations, we see that,
writing $\F_{\ell}$ the functor
$C' \leadsto \F(C' \otimes_{C} P_{1},\ldots,C' \otimes_{C} P_{\ell})$,
corollary \ref{prop:structuredesfibres} gives an isomorphism of schemes:
$$
\F_{\ell} \simeq \F_{\ell-1} \times V_{\ell}.
$$
By induction, we conclude:

\begin{coro}
\label{coro:themodulispaceabstract}
The functor in $C$-algebras
$C' \leadsto \F(C' \otimes_{C} P_{1},\ldots,C' \otimes_{C} P_{k})$
is isomorphic to the functor
$C' \leadsto 
\bigoplus_{1 \leq i < j \leq \ell} \Ext(C' \otimes_{C} P_{k}, C' \otimes_{C} P_{i})$.
That is, we have an isomorphism of affine schemes over $C$:
$$
\F(P_{1},\ldots,P_{k}) \simeq \prod_{1 \leq i < j \leq \ell} \Ext(P_{k}, P_{i}).
$$
\end{coro}


\section{Extension classes of difference modules}
\label{section:extensionsofdifferencemodules}
\index{extension classes of difference modules}

We shall now describe more precisely extension spaces in the case of 
a difference field $(K,\sigma)$. We keep all the previous notations.
Our main tools are theorem \ref{theo:descriptionextensions} and its
corollaries \ref{coro:homologiecomplexesolutions} and
\ref{coro:descmatrextabstrait}. We shall give more ``computational''
variants of these results. \\

\index{cyclic vector lemma}
We use the cyclic vector lemma (lemma \ref{lemm:cyclicvector}): the
proof given in \cite{LDV} is valid for general difference modules, under
the assumption that the characteristic of $K$ is $0$ and that $\sigma$
is not of finite order (\eg\ $q$ is not a root of unity). Although 
most of what follows remains true for $\D$-modules of arbitrary length
(because $\D$ is principal \cite{BAH}), the proofs, inspired by 
\cite[II.1.3]{Sab}, are easier for $\D/\D P$. Recall that, the center 
of $\D$ being $C$, all functors considered here are $C$-linear and produce 
$C$-vector spaces. For a difference module $M = (E,\Phi)$, we still write 
$E$ the $C$-vector space underlying $E$. 

\begin{prop}
\label{prop:complexedessolutions}
Let $M := \D/\D P$ and $N := (F,\Psi)$ two difference modules.
Then $\Ext^{i}(M,N) = 0$ for $i \geq 2$ and there is an exact sequence
of $C$-vector spaces:
$$
0 \rightarrow \Hom(M,N) \rightarrow F \overset{P(\Psi)}{\longrightarrow} 
F \rightarrow \Ext^{1}(M,N) \rightarrow 0.
$$
\end{prop}
\begin{proof}
From the presentation (in the category of left $\D$-modules):
$$
\D \overset{\times P}{\rightarrow} \D \rightarrow M \rightarrow 0,
$$
one draws the long exact sequence:
\index{long exact sequence}
\begin{align*}
0 \rightarrow \Hom(M,N) \rightarrow & \Hom(\D,N) \rightarrow \Hom(\D,N) \\
& \rightarrow \Ext^{1}(M,N) \rightarrow \Ext^{1}(\D,N) \rightarrow 
\Ext^{1}(\D,N) \rightarrow \cdots \\
& \cdots \rightarrow 
\Ext^{i}(M,N) \rightarrow \Ext^{i}(\D,N) \rightarrow \Ext^{i}(\D,N) 
\rightarrow \cdots
\end{align*}
Since $\D$ is free, $\Ext^{i}(\D,N) = 0$ for $i \geq 1$ and the portion
$\Ext^{i-1}(\D,N) \rightarrow \Ext^{i}(M,N) \rightarrow \Ext^{i}(\D,N)$
of the long exact sequence gives the first conclusion.

We are left to identify the portion $\Hom(\D,M) \rightarrow \Hom(\D,M)$.
Of course, 
$$
x \mapsto \lambda_{x} := \bigl(Q \mapsto Q.x = Q(\Psi)(x)\bigr) 
\text{~and~} f \mapsto f(1)
$$
are isomorphisms from $M$ to $\Hom(\D,M)$ and return reciprocal to each
other. Then, the map to be identified is $f \mapsto f \circ (\times P)$, 
where, of course, $\times P$ denotes the map $Q \mapsto Q P$. Conjugating
it with our isomorphisms yields
$x \mapsto \bigl(\lambda_{x} \circ (\times P)\bigr)(1) = (P.1)(\Psi)(x)$,
that is, $P(\Psi)$. 
\end{proof}

\begin{rema}
The complex $F \overset{P(\Psi)}{\longrightarrow} F$ (in degrees $0$ and $1$)
thus has cohomology $\Hom(M,N)$, $\Ext^{1}(M,N)$. This is clearly functorial
in $N$. On the other hand, there is a part of arbitrariness in the choice
of $P$, but, after \cite{BAH} \S 6.1, the homotopy class of the complex
depends on $M$ alone.
\end{rema}

\begin{coro}
\label{coro:complexedessolutions}
\index{complex of solutions}
Let $M = (E,\Phi)$. The \emph{complex of solutions of $M$}:
$$
E \overset{\Phi - \Id}{\longrightarrow} E
$$
has cohomology $\Gamma(M)$, $\Gamma^{1}(M)$.
\end{coro}
\begin{proof}
In the proposition, take $M := \1$, $P := \sq - 1$ and $N := (E,\Phi)$.
\end{proof}

Note that this is functorial in $M$. \\

\begin{rema}
Applying the corollary to $M^{\vee} \otimes N$ implies that the map
$f \mapsto \Psi \circ f \circ \Phi^{-1} - f$ has kernel $\Hom(M,N)$
(which is obvious) and that its cokernel is in one-to-one correspondance
with $\Ext^{1}(M,N)$, which is similar to the conclusion of theorem 
\ref{theo:descriptionextensions}. However, in this way, we do not get 
the identification of the operations on extensions: this would be possible 
using \cite[\S 6.3]{BAH}.
\end{rema}

\begin{coro}
\label{coro:complexedessolutionsbis}
Let $M = \D/\D P$. Then, for any dual $P^{\vee}$ of $P$, the complex:
$$
K \overset{P^{\vee}(\sigma)}{\longrightarrow} K
$$
has cohomology $\Gamma(M)$, $\Gamma^{1}(M)$.
\end{coro}
\begin{proof}
In the proposition, take $N := \1 = (K,\sigma)$. 
This gives $\Gamma(M^{\vee})$, $\Gamma^{1}(M^{\vee})$ as cohomology 
of $K \overset{P(\sigma)}{\longrightarrow} K$.
Now, replace $M$ by $M^{\vee}$. (Of course, no functoriality here!)
\end{proof}

\begin{exem}
For $u \in K^{*} = \GL_{1}(K)$, where $K = \Ka$ or $\Kf$, put 
$M_{u} := (K,\Phi_{u}) = \Dq/\Dq (\sq - u^{-1})$. A dual of $\sq - u^{-1}$
is, for instance, $\sq - u$. One has $M_{u}^{\vee} = M_{u^{-1}}$ and 
$M_{u} \otimes M_{v} = M_{u v}$. \\
From the corollaries above, we draw that 
$K \overset{u^{-1} \sq - 1}{\longrightarrow} K$, resp. 
$K \overset{\sq - u}{\longrightarrow} K$ has cohomology $\Gamma(M_{u})$, 
$\Gamma^{1}(M_{u})$. From the proposition, we draw that 
$K \overset{v^{-1} \sq - u^{-1}}{\longrightarrow} K$
has cohomology $\Hom(M_{u},M_{v})$, $\Ext^{1}(M_{u},M_{v})$.
\end{exem}


\section{The cohomological equation}
\label{section:cohomologicalequation}
\index{cohomological equation}

We here prepare the grounds for the examples of chapter 
\ref{chapter:examplesStokes} by giving some practical recipes.


\subsection{Some inhomogeneous equations}
\label{subsection:inhomogeneous}

Like in the theory of linear differential equations, many interesting
examples come in dimension $2$; see for instance example
\ref{exem:archetypal} page \pageref{exem:archetypal} , and the more
detailed studies of chapter \ref{chapter:generalnonsense}, paragraph 
\ref{subsection:prototypal} and chapter \ref{chapter:examplesStokes}, 
section \ref{section:qEuler}. \\

So let $a,b \in K^{*}$ and $u \in K$. Extensions of $M_{b} := (K,\Phi_{b})$ 
by $M_{a} := (K,\Phi_{a})$ have form $N_{a,b,u} := (K^{2},\Phi_{A_{u}})$ where
$A_{u} := \begin{pmatrix} a & u \\ 0 & b \end{pmatrix}$. An isomorphism
of extensions from $N_{a,b,u}$ to $N_{a,b,v}$ would be a matrix
$F := \begin{pmatrix} 1 & f \\ 0 & 1 \end{pmatrix}$ such that
$b \sigma f - a f = v - u$. More generally, the space $\Ext(M_{b},M_{a})$
is isomorphic to the cokernel of the $C$-linear map $b \sigma - a$ from 
$K$ to itself. We shall call \emph{cohomological equation} the following
first order inhomogeneous equation:
\begin{equation}
\label{equation:equationcohomologique}
b \sigma f - a f = u.
\end{equation}
This can be seen as the obstruction to finding an isomorphism $F$
from $A_{0}$ to $A_{u}$. \\

More generally, let $L := a_{0} + \cdots + a_{n} \sigma^{n} \in \D$.
By vectorializing the corresponding equation we get a system with
matrix:
$$
A = \begin{pmatrix}
0 & 1 & 0 & \ldots & 0 \\
0 & 0 & 1 & \ldots & 0 \\
\vdots & \vdots & \vdots & \ddots & \vdots \\
0 & 0 & 0 & \ldots & 1 \\
- a_{0}/a_{n} & - a_{1}/a_{n} & - a_{2}/a_{n} & - \ldots & - a_{n-1}/a_{n} 
\end{pmatrix}.
$$
Actually, $M := (K^{n},\Phi_{A}) \simeq (\D/\D L)^{\vee}$. From corollary
\ref{coro:complexedessolutionsbis}, we deduce:
$$
\text{Coker~} L \simeq \Gamma^{1}(M) \simeq \Ext(\1,M).
$$
This can be seen as an obstruction to finding an isomorphism 
$\begin{pmatrix} I_{n} & X \\ 0 & 1 \end{pmatrix}$ from
$\begin{pmatrix} A & 0 \\ 0 & 1 \end{pmatrix}$ to
$\begin{pmatrix} A & U \\ 0 & 1 \end{pmatrix}$. In fact,
an isomorphism 
$\begin{pmatrix} I_{n} & X \\ 0 & 1 \end{pmatrix}$ from
$\begin{pmatrix} A & U \\ 0 & 1 \end{pmatrix}$ to
$\begin{pmatrix} A & V \\ 0 & 1 \end{pmatrix}$ would correspond 
to a solution of $\sigma X - A X = V - U$. Assuming for instance
$a_{n} = 1$ and writing $x_{i}$ the components of $X$, this gives 
the equations $\sigma x_{i} - x_{i} = v_{i} - u_{i}$ ($i = 1,\ldots,n-1$)
and $\sigma x_{n} + a_{0} x_{1} + \cdots + a_{n-1} x_{n} = v_{n} - u_{n}$.
From this, one can solve trivially to get an equivalent of $U$ with
components $0,\ldots,0,u$ (exercice for the reader). And, if $U$
has this form, finding an isomorphism 
$\begin{pmatrix} I_{n} & X \\ 0 & 1 \end{pmatrix}$ from
$\begin{pmatrix} A & 0 \\ 0 & 1 \end{pmatrix}$ to
$\begin{pmatrix} A & U \\ 0 & 1 \end{pmatrix}$
is equivalent to solving $L x_{1} = u$.
Still more generally, it is easy to see that finding an isomorphism 
$\begin{pmatrix} I_{n} & X \\ 0 & I_{p} \end{pmatrix}$ from
$\begin{pmatrix} A & 0 \\ 0 & B \end{pmatrix}$ to
$\begin{pmatrix} A & U \\ 0 & B \end{pmatrix}$ is equivalent 
to solving $(\sigma X) B - A X = U$, and to obtain anew an 
identification of the cokernel of $X \mapsto (\sigma X) B - A X$
with $\Ext\bigl((K^{p},\Phi_{B}),(K^{n},\Phi_{A})\bigr)$. 


\subsection{A homotopy}
\label{subsection:complexhomotopy}

The following is intended to be an explanation of the equivalence
of various computations above. Let $(K,\sigma)$ be a difference field 
with constant field $C$. The $\sigma$-difference operator 
$P := \sigma^{n} + a_{1} \sigma^{n-1} + \cdots + a_{n}$ gives rise to 
a $C$-linear complex $K \overset{P}{\rightarrow} K$. Writing $A_{P}$ 
the companion matrix described in \ref{subsubsection:functorofsolutions},
we also have a complex $K^{n} \overset{\Delta}{\rightarrow} K^{n}$, where we
have set $\Delta := \sigma - A_{P}$. We then have a morphism of complexes:
$$
\begin{CD}
K  @>{P}>> K \\
@V{V}VV   @V{I}VV  \\
K^{n} @>{\Delta}>> K^{n}
\end{CD}, 
\text{~where~} V(f) := 
\begin{pmatrix} f \\ \sq f \\ \vdots \\ \sq^{n-1} f \end{pmatrix}
\text{~and~} I(g) := \begin{pmatrix} 0 \\ \vdots \\ 0 \\ g \end{pmatrix}.
$$
(The equality $\Delta \circ V = I \circ P$ is obvious.) We now introduce 
the operators $P_{i} := \sigma^{n-i} + a_{1} \sigma^{n-i-1} + \cdots + a_{n-i}$ 
(they are related to the Horner scheme for $P$). We then have a morphism of 
complexes in the opposite direction:
$$
\begin{CD}
K  @>{P}>> K \\
@A{\pi_{1}}AA   @A{\Pi}AA  \\
K^{n} @>{\Delta}>> K^{n}
\end{CD}, 
\text{~where~} 
\pi_{1} \begin{pmatrix} f_{1} \\ \vdots \\ f_{n} \end{pmatrix} := f_{1}
\text{~and~} 
\Pi \begin{pmatrix} g_{1} \\ \vdots \\ g_{n} \end{pmatrix} := 
\sum_{i=1}^{n} P_{i} g_{i}.
$$
(Of course, one must check that $\Pi \circ \Delta = P \circ \pi_{1}$.)
Clearly $(\pi_{1},\Pi) \circ (V,I)$ is the identity of the first complex. 
We are going to see that $(V,I) \circ (\pi_{1},\Pi)$ is homotopic to the 
identity of the second complex, whence their homological equivalence 
(see \cite[\S 2.4, def. 4,5 and prop. 5]{BAH}). To that end, we introduce 
a backward operator $K^{n} \overset{\Delta'}{\leftarrow} K^{n}$ defined by 
the following relation:
$$
\Delta'\begin{pmatrix} g_{1} \\ \vdots \\ g_{n} \end{pmatrix} := 
\begin{pmatrix} g'_{1} \\ \vdots \\ g'_{n} \end{pmatrix}, \text{~where~} 
g'_{i} := \sum_{j+k = i-1} \sigma^{j} g_{k} \Longrightarrow
\begin{cases}
V \circ \pi_{1} - \Id_{K^{n}} = \Delta' \circ \Delta, \\
I \circ \Pi - \Id_{K^{n}} = \Delta \circ \Delta'. 
\end{cases}
$$
(The computations are mechanical and again left to the reader.) 
This implies that the two complexes are indeed homotopic.


\backmatter

\bibliography{smfcal}

\chapter*{Index of notations}
\label{indexnotations}

$\Ra$, $\Ka$, $\Rf$, $\Kf$ 
\dotfill \pageref{not1}

$q \in \C$ 
\dotfill \pageref{not2}

$\sq f(z)$ 
\dotfill \pageref{not3}

$\Eq$, $p: \C^{*} \rightarrow \Eq$, $[\lambda;q]$, $\overline{\lambda}$ 
\dotfill \pageref{not4}

$\sq X = A X$ 
\dotfill \pageref{not5}

$\theta(z;q)$, $\theta(z)$ 
\dotfill \pageref{not6}

$\thq(z)$ 
\dotfill \pageref{not7}

$\Rfs$, $\Kfs$ 
\dotfill \pageref{not8}

$\D_{K,\sigma}$, $\D$ 
\dotfill \pageref{not9}

$K^{\sigma}$ 
\dotfill \pageref{not10}

$\DMod$ 
\dotfill \pageref{not11}

$(E,\Phi)$ 
\dotfill \pageref{not12}

$DiffMod(K,\sigma)$ 
\dotfill \pageref{not13}

$(K^{n},\Phi_{A})$, $\Phi_{A}(X)$ 
\dotfill \pageref{not14}

$\1$, $\Gamma(M)$, $\Gamma^{i}(M)$ 
\dotfill \pageref{not15}

$\Homint(M,N)$, $M^{\vee}$ 
\dotfill \pageref{not16} 

$M \otimes N$, $\rk M$ 
\dotfill \pageref{not17} 

$\bigl(\Ka,\sq\bigr)$, $\bigl(\Kf,\sq\bigr)$, $F[A]$, $\Dq$ 
\dotfill \pageref{not18}

$N(P)$, $S(P)$, $r_{P}$
\dotfill \pageref{not19}

$N(M)$, $S(M)$, $r_{M}$
\dotfill \pageref{not20}

$\tsh(z)$
\dotfill \pageref{not21}

$M_{(\mu)}$
\dotfill \pageref{not22}

$M_{\leq \mu}$, $M_{< \mu}$, $M_{(\mu)}$, $\gr\ M$
\dotfill \pageref{not23}

$\F(M_{0})$, $\F(P_{1},\ldots,P_{k})$, $(M,g)$
\dotfill \pageref{not24}

$\Bq f(\xi)$
\dotfill \pageref{not25}

$\C(\{\xi\})_{q,1}$
\dotfill \pageref{not26}

$\F(P,P') \rightarrow \Ext(P',P)$
\dotfill \pageref{not27}

$A_U$, $M_U$
\dotfill \pageref{not28}

$\G$
\dotfill \pageref{not29}

$M_{U} \sim M_{V}$
\dotfill \pageref{not30}

$\G^{A_{0}}$
\dotfill \pageref{not31}

$\hat{F}_{U}$, $\hat{F}_{U,V}$
\dotfill \pageref{not32}

$\text{Irr~}(P)$, $\text{Irr~}(M)$
\dotfill \pageref{not33}

$\underline{\End}(M_{0}) = M_{0}^{\vee} \otimes M_{0}$
\dotfill \pageref{not34}

$M_U$, $A_U$
\dotfill \pageref{not35}

$K_{\mu,\nu}$
\dotfill \pageref{not36}

$Red(\mu_{1},A_{1},\mu_{2},A_{2},U)$
\dotfill \pageref{not37}

$\tsh$
\dotfill \pageref{not38}

$\B_{q,d}$
\dotfill \pageref{not39}

$\overline{M_{U}}$
\dotfill \pageref{not40}

$\Sigma$, $\Sigma (A)$
\dotfill \pageref{not41}

$S_{n-1} \hat{f}$
\dotfill \pageref{not42}

$\Aa(U)$, $\Aa_{0}(U)$, $\Bb(U)$
\dotfill \pageref{not43}

$U_{\infty}$, $\Aa(U_{\infty})$, $\Aa_{0}(U_{\infty})$, $\Bb(U_{\infty})$
\dotfill \pageref{not44}

$\mathcal{C}^{\infty}_{Whitney}(K')$
\dotfill \pageref{not45}

$W_n(K')$
\dotfill \pageref{not46}

$\tilde W_n(K')$
\dotfill \pageref{not47}

$\Lambda_I$, $\Lambda_I^G$, $\Lambda_I^{\G}$, $\Lambda_I(M)$
\dotfill \pageref{not48}

$H^1(\Eq;\Lambda_I)$
\dotfill \pageref{not49}

$p_0$
\dotfill \pageref{not50}

$\Cc^0(\UU;\GL_n(\Aa))$, $C^0(\UU;\GL_n(\Aa))$, $C^1(\UU;\Lambda_I)$,
$C^0(\UU; \GL_n\bigl(\Rf\bigr))$, 
\dotfill \pageref{not51}

$C^0(\UU;\Lambda_I) \backslash \Cc^0(\UU;\GL_n(\Aa)) 
/H^0\bigl(\UU;\GL_n(\Aa)\bigr)$
\dotfill \pageref{not52}

$C^0(\UU;\Lambda_I) \backslash \Cc^0(\UU;\GL_n(\Aa)) /\GL_n\bigl(\Ra\bigr)$
\dotfill \pageref{not52}

$\Ee(X)$
\dotfill \pageref{not53}

${\mathbf M}_g$, $\overset{\bullet}{\mathbf M}_g$, 
$\bigl({\mathbf M}_g,\pi,(\C,0)\bigr)$
\dotfill \pageref{not54}

$\widehat{\mathbf M}_g$, $\hat{\mathbf F}$, $\hat\pi$,
$\hat\pi_0: \hat{\mathbf F} \rightarrow  \bigl((\C,0){\hat\vert}\{0\}\bigr)$
\dotfill \pageref{not55}

$\Bigl(\widehat{\mathbf M}_g,\hat\pi,\bigl((\C,0){\hat\vert}\{0\}\bigr)\Bigr)$,
$\bigl(\hat{\mathbf F},\hat\pi_0,((\C,0){\hat\vert}\{0\})\bigr)$
\dotfill \pageref{not55}

$\GL^{M_0}_n\bigl(\Rf\bigr)$
\dotfill \pageref{not56}

$[\lambda] := [\lambda;q]$, 
$\Lambda = \nu_1 [\lambda_1] + \cdots + \nu_m [\lambda_m]$
\dotfill \pageref{not57}

$\lmod \Lambda \rmod$, $\Vert \Lambda \Vert$
\dotfill \pageref{not58}

$d_q(z,[\lambda])$, $\Vert q \Vert_1$
\dotfill \pageref{not59}

$D([\lambda];\rho)$, $D^c([\lambda];\rho)$
\dotfill \pageref{not60}

$d_q(z,\Lambda)$, $D(\Lambda;\rho)$, $D^c(\Lambda;\rho)$
\dotfill \pageref{not61}

$\delta(z,\Lambda)$, $D_\delta(\Lambda;\rho)$
\dotfill \pageref{not62}

$S(\Lambda,\epsilon;R)$, $S(\Lambda,\epsilon)$
\dotfill \pageref{not63}

$\BB^{\Lambda}$
\dotfill \pageref{not64}

$e(z) := e(z;q) :=\theta\left(\lmod z\rmod;\lmod q \rmod\right)$
\dotfill \pageref{not65}

$W_{(\infty)}$, $W_{(0)}$
\dotfill \pageref{not66}

$\AA_{q}^{\Lambda}$, $\AA_{q;|\Lambda|}^{\Lambda}$
\dotfill \pageref{not67}

$J(f)$
\dotfill \pageref{not68}

$f \in \AA_q^{[\lambda]}$ 
\dotfill \pageref{not69}

$D^\eta_{i,n}$, $\CC^\eta_{i,n}$, $\RR_N$
\dotfill \pageref{not70}

$\RR_{j,N}$
\dotfill \pageref{not71}

$C_{j,N}^{\delta,r}$
\dotfill \pageref{not72}

$\Lambda F(0)$
\dotfill \pageref{not73}

$\EE_0^{\Lambda}$, $\theta_\Lambda(z)$
\dotfill \pageref{not74}

$(q^{-1};q^{-1})_\infty$
\dotfill \pageref{not75}

$\Lambda^\prime < \Lambda$, $\Lambda/\Lambda^\prime$, 
$(\Lambda/\Lambda^\prime)F(0)$, $\EE_{(\Lambda_1,\Lambda_2)}^{\Lambda}$
\dotfill \pageref{not76}

$\Lambda_{\ge i}$, $\Lambda_{\le i}$
\dotfill \pageref{not77}

$\EE_{(\Lambda_1,\Lambda_2,\ldots,\Lambda_m)}^{\Lambda}$,
$\OO_{(\Lambda_1,\Lambda_2,\ldots,\Lambda_m)}^{\Lambda}$, 
\dotfill \pageref{not78}

$\OO_{(\bf O)}^{\bf O}$
\dotfill \pageref{not79}

$\CA$
\dotfill \pageref{not80}

$F^A_{(\Lambda_1,\ldots,\Lambda_{k-1})}$, ${\ff}^A$
\dotfill \pageref{not81}

$[\lambda;q:{A_0}]$
\dotfill \pageref{not82}

$\text{St}_{\lambda_0}(\lambda;A)$
\dotfill \pageref{not83}

$\thqc(z)$
\dotfill \pageref{not84}

$f_{\overline{c}}$, $S_{\overline{c}} \hat{f}$
\dotfill \pageref{not85}

$\div_{\Eq}$, $\Sigma_{A_{0}}$, $\text{Sp}$
\dotfill \pageref{not86}

$T_{A_{0},c}$
\dotfill \pageref{not87}

$U_{c}$
\dotfill \pageref{not88}

$E_{i,j}$, $E_{i,j,c,d}$
\dotfill \pageref{not89}

$Z^{1}_{pr}(\UU,\Lambda_I(M_0))$
\dotfill \pageref{not90}

$\G$$\G_{A_{0}}$, $\g$, $\g^{\geq \delta}$, $\G^{\geq \delta}$, $g^{(\delta)}$ 
\dotfill \pageref{not91}

$\Lambda_{I}(M_{0})$, $\lambda^{(\delta)}$, $V^{(\delta)}$
\dotfill \pageref{not92}

$\F_{\leq \delta}(M_{0})$
\dotfill \pageref{not93}

$S_{c,d} \hat{F}_{A}$
\dotfill \pageref{not94}

$\F_{M}$
\dotfill \pageref{not95}

$\F_{M}(U)$
\dotfill \pageref{not96}

$\underline{\End}^{> 0}(M_{0})$
\dotfill \pageref{not97}

$p := q^{-1}$, $(a;p)_{n}$, $(a;p)_{\infty}$, $(a_{1},\ldots,a_{k};p)_{n}$,
$(a_{1},\ldots,a_{k};p)_{\infty}$
\dotfill \pageref{not98}

$\theta(z;q)$, $\thq$ 
\dotfill \pageref{not99}

$\theta_{q,a}(z)$
\dotfill \pageref{not100}

$\left[h\right]_{n}$
\dotfill \pageref{not101}

$S_{\overline{\lambda}} \hat{f} = S_{\Lambda} \hat{f}$
\dotfill \pageref{not102}

${}_{2} \Phi_{1}(a,b;c;q,z)$
\dotfill \pageref{not103}

$K(\lambda,\mu,z)$
\dotfill \pageref{not104}

$T^{2} M$, $S^{2} M$, $\Lambda^{2} M$
\dotfill \pageref{not105}

$s(\alpha,\beta;q,x)$ 
\dotfill \pageref{not106}

$F(\alpha,\beta;q,x)$, $G(\alpha,\beta;q,x)$
\dotfill \pageref{not107}

$M(\alpha,\beta;q,x)$
\dotfill \pageref{not108}

$U_{\lambda}(x)$, $V_{\lambda}(x)$, $S_{n}(x;q)$
\dotfill \pageref{not109}

$U(\lambda)$, $V(\lambda)$
\dotfill \pageref{not110}

$\theta_{0,1}(x)$, $\theta_{0,1}(x,\omega)$, 
$\theta_{1,1}(x)$, $\theta_{1,1}(x,\omega)$
\dotfill \pageref{not111}

$F(D)$, $f_{0,1}(x)$, $f_{0,1}(x,\omega)$
\dotfill \pageref{not112}

$(\underline{M},\underline{u})$
\dotfill \pageref{not113}

$\F(P_{1},\ldots,P_{k})$
\dotfill \pageref{not114}

$\F(P_{1},P_{2})$, $\Ext(P_{2},P_{1})$
\dotfill \pageref{not115}

$\F(P_{1},P_{2},P_{3})$
\dotfill \pageref{not116}

$(K,\sigma)$
\dotfill \pageref{not117}

$K^{\sigma}$, $\D_{K,\sigma}$, $\DKMod$
\dotfill \pageref{not118}

$(V,\Phi)$
\dotfill \pageref{not119}

$\sigma^* V$
\dotfill \pageref{not120}

$\1$, $(K^n,\Phi_A)$
\dotfill \pageref{not121}

$\DM$, $\Ext^i(M,N)$ 
\dotfill \pageref{not122}

$\Lin_{\sigma}(V,W)$, $\Lin_K(M,N)$
\dotfill \pageref{not123}

$t_{\Phi,\Psi}$
\dotfill \pageref{not124}

$\Gamma(M)$
\dotfill \pageref{not125}

$\Gamma^i(M)$, $\Ext^i(\1,M)$
\dotfill \pageref{not126}

$\Phi \otimes \Psi$, $M \otimes N$ 
\dotfill \pageref{not127}

$\Homint(M,N)$
\dotfill \pageref{not128}

$M^{\vee}$
\dotfill \pageref{not129}

$\G$
\dotfill \pageref{not130}

$A_{U}$, $A_{0}$
\dotfill \pageref{not131}

$\Gamma_{u}(y,x)$
\dotfill \pageref{not132}

$t_{\Phi,\Psi}$
\dotfill \pageref{not133}

\printindex 

\end{document}